\documentclass[reqno,11pt]{amsart}
\usepackage{amsthm,amsfonts,amssymb,euscript,mathrsfs,graphics,color,amsmath,latexsym,marginnote,hyperref}
\usepackage[latin1]{inputenc}

\theoremstyle{plain}

\numberwithin{equation}{section}

\newtheorem{theorem}{Theorem}[section]
\newtheorem{proposition}[theorem]{Proposition}
\newtheorem{lemma}[theorem]{Lemma}

\newtheorem{remark}[theorem]{Remark}
\newtheorem{remarks}[theorem]{Remark}
\newtheorem{definition}[theorem]{Definition}
\newcommand{\be}{\begin{equation}}
\newcommand{\ee}{\end{equation}}
\newcommand{\teta}{\theta}
\newcommand{\om}{\omega}

\newcommand{\e}{\varepsilon}
\newcommand{\ep}{\epsilon}
\newcommand{\al}{\alpha}

\newcommand{\ov}{\overline}
\newcommand{\wtilde}{\widetilde}
\newcommand{\R}{\mathbb R}
\newcommand{\C}{\mathbb C}

\newcommand{\Z}{\mathbb Z}

\newcommand{\N}{\mathbb N}
\newcommand{\T}{\mathbb T}
\newcommand{\sign}{{\rm sign}}

\renewcommand{\b }{\beta }
\newcommand{\s }{\sigma }
\newcommand{\ii }{{\rm i} }

\newcommand{\g }{\gamma}

\renewcommand{\l }{\lambda }

\newcommand{\m }{\mu }

\newcommand{\bral}{[ \! [} 
\newcommand{\brar}{] \! ]} 
\newcommand{\Bral}{\big[\! \! \big[} 
\newcommand{\Brar}{\big]\! \! \big]}

\newcommand{\dps}{\displaystyle}

\newcommand{\x }{\xi }

\newcommand{\pa}{\partial}

\newcommand{\opbw}{{Op^{\mathrm{BW}}}}

\def\hat{\widehat}
\def\wt{\widetilde}
\def\bar{\overline}

\def\cal{\mathcal}

\renewcommand{\Re}{\mathrm{Re}\,}

\def\ba{\begin{aligned}}
\def\ea{\end{aligned}}
\def\beginm{\begin{multline}}
\def\endm{\end{multline}}

\newcommand{\sgn}{{\rm sign}}

\providecommand{\vect}[2]{{\bigl[\begin{smallmatrix}#1\\#2\end{smallmatrix}\bigr]}} \providecommand{\sign}{\mathrm{sgn}\,}  
\providecommand{\sm}[4]{{\bigl[\begin{smallmatrix}#1&#2\\#3&#4\end{smallmatrix}\bigr]}}

\begin{document}

\title[Birkhoff normal form for periodic gravity water waves]{Birkhoff normal form and long time existence \\ 
for periodic gravity water waves} 

\date{}

\author{Massimiliano Berti}
\address{SISSA, Trieste}
\email{berti@sissa.it}

\author{Roberto Feola}
\address{SISSA, Trieste}
\email{rfeola@sissa.it}

\author{Fabio Pusateri}
\address{University of Toronto}
\email{fabiop@math.toronto.edu}

\thanks{This research was supported by PRIN 2015 ``Variational methods, with applications to problems in 
mathematical physics and geometry''.
The third author was supported in part by Start-up grants from Princeton University and the University of Toronto, 
and NSERC grant RGPIN-201}

\begin{abstract}
We consider the gravity water waves system with a periodic one-dimensional interface in infinite depth,
and prove a rigorous reduction of these equations to Birkhoff normal form up to degree four.
This proves a conjecture of Zakharov-Dyachenko  \cite{Zak2}
suggested by  the formal Birkhoff integrability of the water waves Hamiltonian  truncated at order four. 
As a consequence, we also obtain 
a long-time stability result: periodic perturbations of a flat interface that are of size $\e$ 
in a sufficiently smooth Sobolev space lead to solutions 
that remain regular and small up to times of order $\e^{-3}$. 
This time scale is expected to be optimal. 

Main difficulties in the proof are the presence of 
non-trivial resonant four-waves interactions, the so-called Benjamin-Feir resonances,
the small divisors arising from near-resonances and 
the quasilinear nature of the equations.
Some of the main ingredients that we use are: 
$(1)$ a reduction  procedure to constant coefficient operators up to 
smoothing remainders, that, together 
with the verification of key algebraic cancellations of the system,
 implies the integrability of the equations at non-negative orders; 
 $(2)$ a Poincar\'e-Birkhoff normal form of the smoothing remainders
 that deals with near-resonances;
$(3)$ an a priori algebraic identification argument
of the above Poincar\'e-Birkhoff normal form equations with 
the formal Hamiltonian  computations of  \cite{Zak2,CW, DZ, CS},
that allows us to handle the Benjamin-Feir resonances.
\end{abstract}

\maketitle

\setcounter{tocdepth}{1}
\tableofcontents 

\section{ Introduction}\label{PGWW}
We consider an incompressible and irrotational perfect fluid, under the action of gravity,
occupying at time
$t$ a two dimensional domain with infinite depth, periodic in the horizontal variable, given by 
\be\label{Deta}
    {\mathcal D}_{\eta} := \big\{ (x,y)\in \T \times\R \, ;  \ - \infty <y<\eta(t,x) \big\},  \quad 
    \T := \R \slash (2 \pi \Z) \, ,
\ee
where $\eta$ 
is a smooth enough function. 
The velocity field in the time dependent domain 
$ {\mathcal D}_{\eta} $ is the gradient of a harmonic function $\Phi$, called the velocity potential. 
The time-evolution of the fluid is determined by a system of equations for 
the two functions $(t,x)\to \eta(t,x) $, $ (t,x,y)\to \Phi(t,x,y)$.
Following Zakharov \cite{Zak1} and Craig-Sulem \cite{CrSu} we denote
$\psi(t,x) = \Phi(t,x,\eta(t,x))$ the restriction 
of the velocity potential to the free interface.
Given the shape $\eta(t,x)$ of the domain $ {\cal D}_\eta $ 
and the Dirichlet value $\psi(t,x)$ of the velocity potential at the top boundary, one can recover 
$\Phi(t,x,y)$ as the unique solution of the elliptic problem 
\begin{equation} \label{BoundaryPr}
\Delta \Phi = 0  \ \text{in } 
{\cal D}_\eta  \, , \quad 
\partial_y \Phi \to 0  \ \text{as } y \to - \infty \, , \quad 
\Phi = \psi \  \text{on } \{y = \eta(t,x)\}. 
\end{equation}
The $(\eta,\psi)$ variables then satisfy the gravity water waves system
\begin{equation} \label{eq:113}
\begin{cases}
    \partial_t \eta = G(\eta)\psi \cr
\partial_t\psi = \dps -g\eta  -\frac{1}{2} \psi_x^2 +  \frac{1}{2}\frac{(\eta_x  \psi_x + G(\eta)\psi)^2}{1+\eta_x^2}
\end{cases}
\end{equation}
where $ G(\eta)\psi $ is the Dirichlet-Neumann operator  
\begin{equation}
  \label{eq:112a}
  G(\eta)\psi := \sqrt{1+\eta_x^2}(\partial_n\Phi)\vert_{y=\eta(t,x)} = (\partial_y\Phi -\eta_x \partial_x\Phi)(t,x,\eta(t,x))
\end{equation}
and $n$ is the outward unit normal at the free interface $y= \eta(t,x)$.
$ G(\eta)$ is a pseudo-differential operator with principal symbol $|D|$,
self-adjoint with respect to the $L^2$ scalar product, positive-semidefinite, 
and its kernel contains only the constant functions. 
Without loss of generality, we set the gravity constant to $g = 1$.

It was first observed by Zakharov \cite{Zak1} that \eqref{eq:113} are the Hamiltonian system  
\be\label{HS}
\begin{aligned}
& \qquad \pa_t \eta = \nabla_\psi H (\eta, \psi) \, , \quad  \pa_t \psi = - \nabla_\eta H (\eta, \psi)  \, , 
\end{aligned}
\ee
where $ \nabla $ denotes the $ L^2 $-gradient, with Hamiltonian
\be\label{Hamiltonian}
H(\eta, \psi) := \frac12 \int_\T \psi \, G(\eta ) \psi \, dx + \frac{1}{2} \int_{\T} \eta^2  \, dx
\ee
given by the sum of the kinetic 
and potential energy of the fluid.
Recall that  the Poisson bracket between functions  $ H(\eta,\psi), F(\eta,\psi)$ is defined as
\begin{equation}\label{poissonBra}
\{F, H\}=\int_{\T}\big(\nabla_{\eta}H\nabla_{\psi}F-\nabla_{\psi}H\nabla_{\eta}F\big)dx \, .
\end{equation}
Note that the ``mass'' $\int_\T \eta \, dx$ is a prime integral of \eqref{eq:113} and, 
with no loss of generality, we can fix it to zero by shifting the $y$ coordinate. 
Moreover \eqref{eq:113} is invariant under spatial translations
and Noether's theorem implies that the momentum $ \int_{\T} \eta_x (x) \psi (x)  \, dx $
is a prime integral of \eqref{HS}.

Let $ H^s (\T) := H^s $, $ s\in \R $, 
be the Sobolev spaces of $ 2 \pi $-periodic functions of $ x $. 
It is natural to consider 
$ \eta $ in the subspace of zero average functions
$H^s_0(\T) \subset H^s(\T)$, 
and  $ \psi  \in  {\dot H}^s (\T)   $
where ${\dot H}^s (\T):= H^s (\T) \slash_{ \sim}$ 
is the homogeneous Sobolev space
obtained by the equivalence relation
$\psi_1 (x) \sim \psi_2 (x)$ if and only if $ \psi_1 (x) - \psi_2 (x) = c $ is a constant\footnote{The 
fact that $ \psi \in \dot H^s $ is coherent with the fact 
that only the velocity field $ \nabla_{x,y} \Phi $ has physical meaning, 
and the velocity potential $ \Phi $ is defined up to a constant.
For simplicity of notation 
we denote the equivalence class $ [\psi] $ by $ \psi $ and,  
since the quotient map induces an isometry of $ {\dot H}^s (\T) $ onto $ H^s_0 (\T) $, 
we will conveniently identify $ \psi $ with a function with zero average.}.
Moreover, since the space averages
$ \widehat \eta_0 (t)  := \frac{1}{2\pi} \int_{\T} \eta (t, x) \, dx $, 
$\widehat \psi_0 (t) := \frac{1}{2\pi} \int_{\T} \psi (t, x) \, dx $
evolve according to the decoupled equations\footnote{Since  the ocean has infinite depth,   if $\Phi$ solves \eqref{BoundaryPr},
then $\Phi_{c}(x,y):=\Phi(x,y-c)$ solves the same problem in $\mathcal{D}_{\eta+c}$ 
assuming the Dirichlet datum $\psi$ at the free boundary
 $\eta+c $. Therefore $G(\eta+c)=G(\eta) $, $ \forall c \in \R $, 
and $ \int_\T \nabla_\eta K \, dx = 0$ where $ K := \frac12 \int_{\T} \psi G(\eta) \psi \, dx $ denotes the kinetic energy.}
\be\label{medie00}
\pa_t {\widehat \eta}_0 (t) = 0 \, , \quad \pa_t  \widehat \psi_0 (t) = - g \widehat \eta_0 (t) \, ,  
\ee
we may restrict,  
with no loss of generality, to the invariant subspace 
$$
\int_{\T} \eta \, dx = \int_{\T} \psi \, dx = 0  \, . 
$$
The main result of this paper (Theorem \ref{BNFtheorem}) proves
a conjecture of Zakharov-Dyachenko \cite{Zak2} and Craig-Worfolk \cite{CW}, 
on the approximate integrability of the water waves system \eqref{eq:113}.
More precisely, we show that \eqref{eq:113} can be conjugated, via a bounded and invertible transformation  
in a neighborhood of the origin in phase space, to its Birkhoff normal form, up to order $4$.
This latter was formally computed in \cite{Zak2,CW}, see also \cite{CS},
and, remarkably, shown to be integrable. It is fair to say that the 
formal approach in \cite{Zak2,CW} can not be  translated into a rigorous proof
and that 
Theorem \ref{BNFtheorem} requires a completely different approach to the 
Birkhoff normal form reduction.
As a consequence, we obtain a long-time stability result  (Theorem \ref{thm:main}) 
for small periodic perturbations of flat interfaces:
periodic perturbations that are initially 
$\e$-close to the flat equilibrium in a sufficiently regular Sobolev space,
lead to solutions that remain regular and small for times of order  $\e^{-3}$. 
This result has been announced in \cite{BFP}. 

While in recent years several results
have been obtained for quasilinear equations set in an Euclidean space $ \R^d $, 
fewer results are available in the periodic setting, or on other compact manifolds. 
In this context, 
the achievement of extended stability results through rigorous reductions 
to high-order Birkhoff normal forms should be seen as a 
key step to understand the global dynamics of evolution PDEs in non-dispersive settings.

\subsection{Main results}

We denote the horizontal and vertical components of the velocity field at the free interface by
\begin{align} 
\label{def:V}
& V =  V (\eta, \psi) :=  (\pa_x \Phi) (x, \eta(x)) = \psi_x - \eta_x B \, , 
\\
\label{form-of-B}
& B =  B(\eta, \psi) := (\pa_y \Phi) (x, \eta(x)) =  \frac{G(\eta) \psi + \eta_x \psi_x}{ 1 + \eta_x^2} \, , 
\end{align}
and the ``good unknown'' of Alinhac 
\begin{equation}\label{omega0}
\omega := \psi-\opbw(B(\eta,\psi))\eta \, ,
\end{equation} 
as introduced in Alazard-Metivier \cite{AlM} 
(see Definition \ref{quantizationtotale} for the definition of the  paradifferential operator $\opbw$).

To state our first main result concerning the rigorous reduction to Birkhoff normal form of 
the 
system \eqref{eq:113},
let us assume that, for $N$ large enough and some $T>0$, we have a classical solution
\begin{align}\label{etapsi}
(\eta,\psi)\in C^{0}([-T,T]; H^{N+\frac{1}{4}}\times H^{N+\frac{1}{4}})
\end{align}
of the Cauchy problem for \eqref{eq:113} with the initial height satisfying 
\begin{align}
\label{etaave0}
\int_{\T}\eta(0,x) \, dx=0  \, .
\end{align}
The existence of such a solution, at least for small enough $T$, is guaranteed by the 
local well-posedness Theorem of Alazard-Burq-Zuily \cite{ABZ1} (see Theorem \ref{thm:local} below) 
under the regularity assumption $ (\eta,\psi,V,B)(0) \in X^{N - \frac{1}{4}}$ where we denote
\begin{align}\label{Xs}
X^s := H^{s+\frac{1}{2}} \times H^{s+\frac{1}{2}} \times H^{s} \times H^{s} \, .
\end{align}
Defining the complex scalar unknown
\begin{equation}\label{u0}
u := \frac{1}{\sqrt{2}}|D|^{-\frac{1}{4}}\eta+\frac{\ii}{\sqrt{2}}|D|^{\frac{1}{4}}\omega \, , 
\end{equation}
we deduce, by \eqref{etapsi},  that $u \in C^{0}([-T,T];H^{N})$, and $u$ solves an evolution equation of the form
\begin{align}\label{equ}
\pa_t u + \ii |D|^{\frac{1}{2}} u = M_{\geq 2}(u, \ov{u}) \, , 
\end{align}
where $ M_{\geq 2}(u, \ov{u}) $ is a fully nonlinear vector field which contains up to first order derivatives of $ u $. 
Moreover, since the zero average condition \eqref{etaave0} 
is preserved by the flow of \eqref{eq:113}, it follows that
\begin{align}
\label{uave0}
\int_{\T}u(t,x) \, dx=0 \, , \qquad \forall\; t\in [-T,T] \, .
\end{align}
This is our first main result. 

\begin{theorem}{\bf (Birkhoff normal form)} 
\label{BNFtheorem}
Let $u$ be defined as in \eqref{u0},  with $\omega$ as in \eqref{omega0}, for $(\eta,\psi)$ solution of \eqref{eq:113}
satisfying \eqref{etapsi}-\eqref{etaave0}. 
There exist $ N \gg K \gg 1 $  and $ 0 < \bar{\e} \ll 1  $,   such that, if 
\begin{equation}\label{u01}
\sup_{t\in [-T,T]} \sum_{k=0}^K \|\partial_t^k u(t)\|_{\dot{H}^{N-k}} \leq \bar{\e} \, , 
\end{equation}
then there exist a bounded 
and invertible transformation $\mathfrak{B}=\mathfrak{B}(u)$ of $ \dot H^N $, which depends (nonlinearly) on $u$, 
and a constant $C :=C(N)>0$ 
such that
\begin{equation}\label{Germe}
{\|\mathfrak{B}(u) \|}_{\mathcal{L}(\dot{H}^{N}, \dot{H}^{N})}
	+ {\|(\mathfrak{B}(u))^{-1} \|}_{\mathcal{L}(\dot{H}^{N}, \dot{H}^{N})} \leq 1+C{\|u\|}_{\dot{H}^{N}},
\end{equation}
and the variable 
$z:=\mathfrak{B}(u)u$
satisfies the equation
\begin{equation}\label{theoBireq}
\pa_{t}z = -\ii \pa_{\bar{z}}H_{ZD}(z,\bar{z}) + {\mathcal X}^{+}_{\geq 4}
\end{equation}
where:

\begin{itemize}

\smallskip
\item[(1)] the Hamiltonian $H_{ZD}$ has the form 
\begin{align}
\label{theoBirHfull}
H_{ZD} = H^{(2)}_{ZD} + H^{(4)}_{ZD} \, , 
\qquad H^{(2)}_{ZD}(z,\bar{z}) := \frac12 \int_\T \big| |D|^{\frac14} z \big|^{2} \, dx \, ,
\end{align}
with
\begin{equation}\label{theoBirH}
\begin{aligned}
H^{(4)}_{ZD} (z,\bar{z}) &:=
\frac{1}{4 \pi} \sum_{k \in \Z} |k|^3  \big(  |z_k|^4  - 2 |z_{k}|^2 |z_{-k}|^2   \big)
\\
&+  \frac{1}{\pi} \sum_{\substack{k_1, k_2 \in \Z, \, \sign(k_1) = \sign( k_2 )  \\ |k_2| < |k_1|}} |k_1| |k_2|^2  
  \big( - |z_{-k_1}|^2 |z_{k_2} |^2 + |z_{k_1}|^2  |z_{k_2}|^2 \big)
\end{aligned}
\end{equation}
where $z_k$ denotes the $k$-th Fourier coefficient of the function $z$, see \eqref{complex-uU}.

\smallskip
\item[(2)] $ {\mathcal X}^{+}_{\geq 4}  := {\mathcal X}^{+}_{\geq 4} (u,\bar{u},z,\bar{z})$ 
is a quartic nonlinear term satisfying,
for some $C :=C(N)>0$, the ``energy estimate''
\begin{equation}\label{theoBirR}
{\rm Re}\int_{\T}|D|^N {\mathcal X}^{+}_{\geq 4} \cdot \bar{|D|^N z} \, dx\leq C \|z\|_{\dot{H}^N}^{5} \, .
\end{equation} 
\end{itemize}
\end{theorem}

The main point of Theorem \ref{BNFtheorem}  is the construction of the {\it bounded} and {\it invertible} 
transformation $\mathfrak{B}(u)$ in \eqref{Germe}
which recasts the water waves system \eqref{eq:113} (in the form of the equation \eqref{equ} satisfied by $u$) into the equation \eqref{theoBireq}-\eqref{theoBirR}.
Purely formal
transformations mapping  the Hamiltonian \eqref{Hamiltonian} 
to the Hamiltonian \eqref{theoBirHfull}, up to higher order degrees of homogeneity,
were obtained by Zakharov-Dyachenko \cite{Zak2} (hence our notation $H_{ZD}$), 
Craig-Worfolk \cite{CW}, and Craig-Sulem \cite{CS}.

The main consequence of Theorem \ref{BNFtheorem} is to rigorously 
relate the flow of the full water waves system \eqref{eq:113} to the flow of the system \eqref{theoBireq},
which is made by the explicit Hamiltonian component $H_{ZD}$ plus remainders of higher homogeneity.
These remainders are under full control thanks to the energy estimates \eqref{theoBirR}.
The Hamiltonian $H_{ZD}$ is {\it integrable},
as observed in \cite{Zak2, CW},
and its flow preserves all the Sobolev norms; see Theorem \ref{Zakintegro}.
Thus, as a consequence of  Theorem \ref{BNFtheorem}, we  obtain the following long-time existence result.

\begin{theorem}{\bf (Long-time existence)}\label{thm:main}
There exists $ s_0 > 0 $ such that, 
for all $ s \geq s_0 $, there is $ \e_0 >  0 $ such that,  
for any initial data $ (\eta_0, \psi_0) $ satisfying (recall \eqref{Xs})
\be
\label{thm:main1}
\| (\eta_0, \psi_0, V_0, B_0) \|_{X^{s}} \leq \e \leq \e_0 \, , \quad
 \int_{\T} \eta_0 (x) d x =  0 \, ,  
\ee 
where $ V_0 := V(\eta_0, \psi_0) $, $ B_0 := B (\eta_0, \psi_0) $ are defined by
\eqref{def:V}-\eqref{form-of-B}, the following holds: there exist constants $c>0$, $C>0$ and a unique classical solution 
$
(\eta, \psi,V,B)  \in C^0([-T_\e, T_\e], X^{s})
$
of the water waves system \eqref{eq:113} 
with initial condition $(\eta,\psi)(0)= (\eta_0,\psi_0)$  
with   
\be
\label{time}
T_\e \geq c  \e^{-3} \, ,
\ee
satisfying 
\be
\label{thm:main2}
\sup_{[-T_\e, T_\e]} \big( \| (\eta, \psi) \|_{H^{s} \times H^{s} } + \| (V, B) \|_{H^{s-1} \times H^{s-1} } \big) \leq C \e \, , 
\qquad  \int_{\T} \eta (t,x) d x = 0  \, . 
\ee
\end{theorem}

The main conclusion of the above theorem is the existence time $ T_\e $ of order $O(\e^{-3})$.
This goes well beyond the time of $O(\e^{-1})$ which is guaranteed by the local existence theory.
It also extends past the natural time scale of $O(\e^{-2})$ which one expects for non-resonant equations,
and that has indeed been achieved for the system \eqref{eq:113} in the works of Wu \cite{Wu}, Ionescu-Pusateri \cite{IP2} 
and Hunter-Ifrim-Tataru \cite{HuIT}.
To our knowledge, this is the first $\e^{-3}$ existence result for water waves, 
or quasilinear systems, in absence of external parameters. 
This time is expected to be optimal. 

The regularity index in Theorem \ref{thm:main} is a large number $s_0$ which we did not try to optimize.
By a more careful analysis and some adjustments to our setting for the paradifferential calculus, 
one could likely pick some $s_0 \leq 30$.
In any case, it is well understood that 
the exact Sobolev regularity is generally unimportant when dealing with the long-time behavior
of classical solutions of quasilinear problems.
Finally note that, at any time $t$, 
the solution $(\eta,\psi,V,B)$ belongs to the same space $X^{s}$ as 
the initial datum (see \eqref{thm:main1}), but in  \eqref{thm:main2} 
we  control only  a weaker norm of the solution.
This is  a well known phenomenon of the pure gravity water waves equations
(see for instance \cite{ABZ1}, \cite{AlDe,AlDe1}): 
in the variables $(\eta,\omega)$ the Sobolev regularity of the solution is preserved  along the flow, 
but there is a loss of derivatives 
in passing to the unknowns $(\eta,\psi)$. 
The weaker bound \eqref{thm:main2} is still more than sufficient to apply the continuation criterion 
of Theorem \ref{thm:local}-(2) below.

Before discussing the literature on long-time existence results and normal forms,  
we briefly describe some of the key points of this paper, and refer
to Subsection \ref{sec:strategy} for a longer explanation of our strategy.

\setlength{\leftmargini}{2em}
\begin{enumerate}

\medskip
\item The long-time existence  Theorem \ref{thm:main} is obtained by a completely 
different mechanism compared to all previous works in the literature, such as \cite{Wu,IP4,HuIT} which obtain 
a shorter $\e^{-2}$ stability time,  and does not make use of energies 
as in \cite{Del1,Del3,IP4,HuIT,ifrTat}.
Instead, we develop a novel approach to obtain the {\it complete conjugation} 
of the water waves vector field \eqref{eq:113}
to its Birkhoff normal form up to order $4$. 
In particular, Theorem \ref{BNFtheorem} provides an additional
accurate description of the 
dynamics of the water waves equations \eqref{eq:113}, in a full neighborhood 
of the origin in some Sobolev space, up to order $4$. 

\medskip
\item The gravity water waves system \eqref{eq:113} 
presents a family of non-trivial quartic resonances, the so-called Benjamin-Feir resonances, 
and there are {\it no  external parameters} which can be used to modulate
the dispersion relation and avoid them, as opposed to all previous normal form results. 

\medskip
\item Besides the resonant interactions, one also needs to pay attention to 
{\it near resonances} which can prevent the boundedness of 
Poincar\'e-Birkhoff normal form transformations.
This issue is overcome by performing an iterative reduction 
to \emph{constant integrable} coefficients, modulo smoothing remainders, 
of the system \eqref{eq:113} up to order $4$.
This procedure is possible thanks to {\it specific algebraic cancellations} of the pure gravity 
water waves equations in infinite depth, see comments below \eqref{betatV}.

\medskip
\item 
Since the gravity water waves dispersion relation $\sqrt{|k|}$ is sublinear, 
the reduction procedure is very 
different from \cite{BD}, where the dispersion relation  $ \sim |k|^{3/2}$ is superlinear.
However, we still employ the paradifferential framework developed in \cite{BD} as it readily provides us with a paralinearization of the Dirichlet-Neumann map with multilinear expansions. 
Other relevant differences with respect to  \cite{BD} are the presence of exact resonances (Benjamin-Feir), the absence of external parameters, and the fact
that we do not restrict to even initial data.   

\medskip
\item Our transformations are non-symplectic and 
the final resonant Poincar\'e-Birkhoff normal form system is not a priori explicit.
Then, an important step in our proof is a {\it normal form uniqueness argument},
which allows us to identify 
the reduced Poincar\'e-Birkhoff system obtained with our procedures 
with the Hamiltonian equations associated 
to the Zakharov-Dyachenko-Craig-Worfolk Hamiltonian $H_{ZD}$ in \eqref{theoBirHfull}-\eqref{theoBirH}, 
up to degree $4$ of homogeneity. 
This algebraic identification can be seen as philosophically similar 
to Moser's indirect proof of the convergence of the Lindsted series for a KAM torus \cite{M67}.

\medskip
\item The stability time $ \sim \e^{-3} $ in Theorem \ref{thm:main} 
is expected to be {\it optimal} in view of the presence of quintic resonances 
as exhibited by Craig-Worfolk \cite{CW} and Dyachenko-Lvov-Zakharov \cite{DZ}. 
This means that we can not expect a stability time $ \sim \e^{-4} $ for all the initial data.
\end{enumerate}

\medskip
We have chosen to formulate our long-time existence result using the original symplectic variables $(\eta,\psi)$ as well
as the velocity components $(V,B)$ in \eqref{def:V}-\eqref{form-of-B} consistently with the formulation of the local
existence theorem of \cite{ABZ1}, that we reproduce 
below. 

\begin{theorem}[Local existence \cite{ABZ1}] \label{thm:local} 
Let $ s > 3/2 $ and consider 
$ (\eta_0, \psi_0)  $ such that $ (\eta_0, \psi_0, V_0, B_0)$ is in $X^s $, see \eqref{Xs}. 
Then the following holds:
\begin{itemize}
\smallskip
\item[(1)] there exists $ T_{{\rm loc}} > 0 $ such that the Cauchy problem for \eqref{eq:113} with initial data $  (\eta_0, \psi_0) $
has a unique solution $ (\eta, \psi) \in C^0([0, T_{{\rm loc}}],  H^{s + \frac12} \times H^{s + \frac12}) $  
with $ (V, B) \in C^0([0, T_{{\rm loc}}], H^{s} \times H^{s}) $; 
\smallskip
\item[(2)]  let $ T_* $ be the maximal time of existence of the solution 
$ (\eta, \psi) \in C^0([0, T_{{\rm loc}}],  H^{s + \frac12} \times H^{s + \frac12}) $. If, for some $ T_0 > 0 $,  
\be\label{thm:localcont}
\sup_{[0, T_0]} \| (\eta, \psi, V, B)(t) \|_{X^5} < + \infty 
\ee
then $ T_0 < T_* $ and  $ \sup_{[0, T_0]} \| (\eta, \psi, V, B)(t) \|_{X^s} < + \infty $. 
\end{itemize}
\end{theorem}

Part (1) of Theorem \ref{thm:local} is the local existence result \cite[Theorem 1.2]{ABZ1},
stated in the case of the torus $ \T $, for a fluid in infinite depth. 
The result is based on energy methods for hyperbolic symmetrizable quasi-linear systems, 
which are the same in $ \T^d $ and in $ \R^d$. 
A more precise version, which 
implies also the continuation criterion in (2), is Theorem 1.2 of De Poyferr\'e \cite{DPF}.
By time-reversibility, the solutions of \eqref{eq:113} are defined 
in a symmetric interval
$ [-T, T]$. 
Note that the system \eqref{sistemaDiag} which we derive in
Proposition \ref{teodiagonal} admits energy estimates.
Therefore, one could also prove a local existence result based on this, 
implementing an iterative scheme as in \cite{FI1}.

\medskip

\subsection{Literature}\label{literat}
We now present some known results on the well-posedness and normal form 
theory for the water waves  equations.

\smallskip
\emph{Local well-posedness.} 
Early results on the local well-posedness 
of the water waves system include those by 
Nalimov \cite{Nali}, Yosihara \cite{Yosi}, and Craig \cite{Crai}, 
which deal with the case of small perturbations of a flat interface. It was then proved by Wu \cite{Wu97}
 that local-in-time solutions can be constructed 
 with initial data of arbitrary size in Sobolev spaces, 
 in the irrotational case. The question of local well-posedness 
 of the water waves and free boundary Euler equations 
 has then been addressed by many authors, see for example 
 \cite{
 CrisLin, Lind1,
 Lannes, Cou, ShaZ, 
 ABZduke, ABZ1}, we refer to  \cite[Section $2$]{IP2} for a longer discussion and a more extensive list of references.  
The local well-posedness theory 
is presently well-understood: in a variety of different scenarios, 
for sufficiently nice initial data, it is possible to prove the existence of
classical smooth solutions on a small time interval that depends on the size of the initial data 
(and the arc-chord constant of the initial interface). 
In particular, for data which are $\e$ close to a flat 
interface,  solutions exist and stay regular for times of order $\e^{-1}$.

\smallskip
\emph{Long-time regularity in the Euclidean case.} 
In the Euclidean case, i.e. when the horizontal variable $x\in \R^d$, 
it is possible to construct global-in-time solutions. 
The main mechanism used in these cases is dispersion which, combined with localization 
(decay at  spatial infinity), transfers the decay of linear solutions to the nonlinear problem, 
and gives control for long times.

For 3-dimensional fluids (2d interfaces), 
the first global well-posedness results were proved by Germain-Masmoudi-Shatah \cite{GMS} and Wu
\cite{Wu2} for gravity water waves.
The more difficult question of global regularity for gravity-capillary water waves ($g>0$, $\kappa>0$,
where $\kappa$ is the surface tension coefficient when capillarity at the interface is included in the system) 
has been recently solved by Deng-Ionescu-Pausader-Pusateri \cite{DIPP}. 
For the case of a finite flat bottom see  Wang \cite{wang2}.

For 2-dimensional fluids (1d interfaces), 
Wu \cite{Wu} 
proved an almost-global existence result for gravity water waves, which was
improved to global regularity by Ionescu-Pusateri \cite{IP}, 
Alazard-Delort \cite{AlDe, AlDe1}, Hunter and Ifrim-Tataru \cite{HuIT, IFRT}.
For the capillary problem in 2d, global regularity was proved in \cite{IP3}, 
see also \cite{IP4}, and 
\cite{ifrTat}. 
We refer again to \cite{IP2} for more references.

\smallskip
\emph{Long-time existence on Tori: Normal Forms.}
In the case of the torus, $x\in \T^d$, there are no obvious dispersive effects 
that help to control solutions for long times.
In addition, the quasilinear nature of the equations 
and the lack of  conserved quantities which control high Sobolev norms,
prevent the effective use of semilinear techniques.

An important tool that can be used to extend the lifespan of solutions for quasilinear 
equations is normal form theory.
To explain the idea, let us consider a generic evolution equation of the form
\begin{equation}\label{tori1}
\pa_{t} u + \ii \omega(D) u = Q(u,\bar{u}) \, , \quad u(t=0)=u_0 \, , \quad \|u_0\|_{H^{N}}\leq \e \, , 
\end{equation}
where $\omega(D)$ is a  real-valued Fourier multiplier, 
and $Q$ is a quadratic nonlinearity, semi- or quasi-linear,  which
depends on $(u,\bar{u})$ and their derivatives.
In the case of \eqref{eq:113} the dispersion relation is $   \omega(D) =\sqrt{|k|}$.
An energy estimate for  \eqref{tori1} 
of the form
$\frac{d}{dt}E(t) \lesssim {\| u(t) \|}_{H^{N}} E(t)$,
where $E(t) \approx \|u(t)\|^{2}_{{H}^{N}}$,
allows the construction of local solutions on time scales of $O(\e^{-1})$.

To extend the time of existence one can try to obtain a {\it quartic energy inequality} of the form
\begin{equation}\label{tori2}
| E(t) - E(0) | \lesssim \int_0^t {\| u(\tau ) \|}_{H^{N}}^2 E( \tau )  \, d \tau  \, .
\end{equation}
This will then give existence for times of $O(\e^{-2})$.
For water waves, such inequalities have been proven in \cite{Wu,ToWu,IP,AlDe,HuIT} 
for the system \eqref{eq:113}, and also in the case of pure capillarity \cite{IP3,ifrTat}, 
and gravity over a flat bottom \cite{HarIT}. 
Similar results were  obtained in \cite{DS1}
for the  Klein-Gordon equation  on $\T^{d}$, 
which corresponds to the dispersion relation $\omega(k) = \sqrt{|k|^2+m^2}$ in \eqref{tori1}.
Although some delicate analysis is needed in the case of quasilinear PDEs,
the possibility of proving an inequality of the form \eqref{tori2} relies on the absence of
{\it 3-waves resonances}, that is, non-zero integers $(n_1,n_2,n_3)$ solving
\begin{equation}\label{torires2}
\s_1 \omega({n_1})+\s_2 \omega({n_2})+\s_3 \omega({n_3})=0 \, , \qquad \s_1n_1+\s_2n_2+\s_3n_3=0 \, , 
\end{equation}
for $\s_{j}\in\{+,-\}$.
In order to further extend the existence time, one can try to upgrade \eqref{tori2} to a quintic energy estimate like
\begin{equation}\label{tori3}
| E(t) - E(0) | \lesssim \int_0^t {\| u(\tau) \|}_{H^{N}}^3 E( \tau)  \, d \tau \, . 
\end{equation}
At a formal level, this is possible in the absence of \emph{4-waves resonances}, 
that is, non-trivial integer solutions of 
\begin{equation}\label{torires}
\sum_{j=1}^{4}\s_{j}\omega(n_{j})=0 \, , 
\qquad \sum_{j=1}^{4}\s_{j}n_{j}=0 \, .
\end{equation}
Here, by ``trivial'' solutions we mean those $4$-tuples where the frequencies
$n_1, \ldots, n_4 $ appear in pairs with corresponding opposite signs. 
These unavoidable resonances can then often be handled by exploiting the Hamiltonian/reversibile structure 
of the equation to show that they do not contribute to the energy inequality.

The condition on the absence of non-trivial solutions to \eqref{torires}
is, however, not satisfied 
by the gravity water waves system \eqref{eq:113}, 
see the expressions \eqref{straBFR} of the nontrivial solutions of \eqref{torires}.
We will actually prove a quintic energy inequality of the form \eqref{tori3} 
as a consequence of 
the reduction of \eqref{eq:113} to its integrable Birkhoff normal form up to order $ 4 $,  
formally obtained in \cite{DZ,CW}.
A key role here is played by a
cancellation of the resonant monomials in the normal form.

%

Existence results for longer times
can sometimes be obtained
when the dispersion relation $\omega(k)$ in \eqref{tori1} depends in a non-degenerate way on some additional {\it parameter}.
In these cases it is often possible to verify, 
for almost all values of the parameters,  
the nonexistence of integer solutions of 
\begin{equation*}
\sum_{j=1}^{N}\s_{j}\omega(n_{j})=0 \, ,  \qquad \sum_{j=1}^{N}\s_{j}n_{j}=0 \, ,
\end{equation*}
except, when $ N $ is even,  and 
$n_1, \ldots, n_N $ appear in pairs with corresponding opposite signs
(trivial resonances).  
Trivial resonances can again be handled using the the Hamiltonian (or reversible) structure of the equations.
In this direction we mention the works of
Bambusi \cite{Bam1}, Delort-Szeftel \cite{DS2}, 
Bambusi-Delort-Greb\'ert-Szeftel \cite{BDGS}, 
and Bambusi-Greb\'ert \cite{BG1} which developed normal form theory for Hamiltonian semilinear PDEs.
In the context of quasilinear PDEs, Delort \cite{Del2,Del3} 
obtained an $\e^{-M}$ existence result for arbitrary $M$, 
for almost all mass parameters $m$ for Hamiltonian Klein-Gordon equations on spheres.

For water waves, the only extended stability result proven so far, is that of Berti-Delort \cite{BD}
who obtained an $\e^{-M}$ existence result for 1d periodic and even 
gravity-capillary waves with depth $h$, 
corresponding to $\omega(k)=\sqrt{{\rm tanh}(h |k|)(g|k| + \kappa |k|^{3})}$, 
for almost all values of the surface tension parameter $\kappa $. 
This work is based on 
a paradifferential reduction of the gravity-capillary water waves equations 
to constant coefficients, up to smoothing remainders. 
We also refer to Feola-Iandoli \cite{FI2} for  
fully nonlinear reversible Schr\"odinger equations.

\smallskip
\emph{Quasi-periodic solutions.}
We finally mention that  global in time, even in $ x $,  quasi-periodic
solutions for $ 1$d space periodic water waves equations 
have been  recently constructed, using KAM techniques combined 
with a systematic use of pseudo-differential calculus, 
in Berti-Montalto \cite{BM1} for gravity-capillary waves, 
using $\kappa$ as a parameter (see \cite{Alaz-Bal} for periodic solutions),  
and in Baldi-Berti-Haus-Montalto \cite{BBHM} 
for pure  gravity waves in finite depth (see \cite{PlTo}, \cite{IPT} for periodic solutions)
using the depth or the wavelength as a parameter. 

\subsection{The Zakharov-Dyachenko conjecture and our strategy}\label{secZDstra}

In this subsection we first recall the  calculations of \cite{Zak2, CW, DZ, CS} 
concerning the formal integrability, up to  order four, 
of the pure gravity water waves Hamiltonian \eqref{Hamiltonian} in infinite depth.
We then discuss  the strategy of proof of  Theorem \ref{BNFtheorem} which rigorously justifies this integrability. 

\smallskip
\subsubsection{The formal Birkhoff normal form of Zakharov-Dyachenko \cite{Zak2}}\label{secZD}

Consider the Hamiltonian $ H $ in \eqref{Hamiltonian}.  
Introduce the complex variable
\be\label{briefu0}
w := \frac{1}{\sqrt{2}}|D|^{-\frac{1}{4}}\eta+\frac{\ii}{\sqrt{2}}|D|^{\frac{1}{4}}\psi \, ,
\ee
and let $H_{\C}$ be the Hamiltonian expressed in $(w,\ov{w})$.
By a Taylor expansion of the Dirichlet-Neumann operator for small $\eta$, see for example \cite{CrSu}, 
one can expand $H_{\C} = H_{\C}^{(2)} + H_{\C}^{(3)} + \cdots$ 
where $H_{\C}^{(\ell)}$ are $\ell$-homogeneous in 
$ (w, \ov{w})$, 
see \eqref{HVF2ham}-\eqref{cubic3}-\eqref{vectHam}. 
Note that in this Taylor expansion there is a priori no
control on the boundedness of the Hamiltonian vector fields associated to  $H_{\C}^{(\ell)}$, $\ell=3,4, \ldots$

 Applying the usual  Birkhoff normal form procedure for Hamiltonian systems 
 (see Subsection \ref{sec:formal}), it is possible to find a formal 
 symplectic transformation $\Phi$ such that
\be\label{brieftra}
H_\C \circ\Phi=H^{(2)}_{\C}+H^{(4)}_{ZD}+\cdots
\ee
where: $(1)$ all terms of homogeneity $3$ have been eliminated due to the absence 
of $3$-waves resonant interactions, 
that is, non-zero integer solutions of \eqref{torires2}, 
and $(2)$ the term $H^{(4)}_{ZD}$ is supported only on Birkhoff resonant quadruples, i.e. 
\begin{equation}\label{briefH^4}
H^{(4)}_{ZD}= \sum_{\substack{\s_1 n_1 + \s_2 n_2 + \s_3 n_3+\s_{4}n_{4} = 0 \\
\s_1 \omega({n_1})+\s_2 \omega({n_2})+\s_3 \omega({n_3})+\s_{4}\omega({n_4})=0}} 
H_{n_1, n_2, n_3,n_4}^{\s_1, \s_2, \s_3,\s_4} w_{n_1}^{\s_1} w_{n_2}^{\s_2} 
w_{n_3}^{\s_{3}} w_{n_4}^{\s_{4}}
\end{equation}
where $H_{n_1, n_2, n_3,n_4}^{\s_1, \s_2, \s_3,\s_4}\in\C$, $ \om (n) = \sqrt{|n|} $  
and $n_1, n_2, n_3,n_4\in \Z\setminus\{0\}$. 
As shown in \cite{Zak2},  
there are many solutions to the constraints
for the sum in \eqref{briefH^4}.
For example, if $\s_1 = \s_3 = 1 = -\s_2 = -\s_4$, and up to permutations, 
there are trivial solutions of the form 
$(k,k,j,j)$, which give rise to benign integrable monomials $ | w_k |^2 | w_j |^2 $
and the two parameter family of solutions, called Benjamin-Feir 
resonances,  
\be
\label{straBFR}
\bigcup_{\lambda\in\Z \setminus \{0\}, b\in\N} \Big\{ 
n_1 = -\lambda b^2, 
 \, n_2 = \lambda (b+1)^2 \, , \, n_3 = \lambda (b^2+b+1)^2, \, n_4= \lambda (b+1)^2 b^2 \Big\} \, . 
\ee
As a consequence, one could expect, a priori,  the presence in \eqref{briefH^4} of non-integrable monomials of the form
$ w_{- \l b^2} \ov{ w_{\l (b+1)^2 } } w_{\l (b^2+b+1)^2 }  \ov{ w_{\l (b+1)^2 b^2 } }$ 
and their complex conjugates. 
The striking property proved in \cite{Zak2}, see also \cite{CW,CS}, is that 
the coefficients $ H_{n_1, n_2, n_3,n_4}^{\s_1, \s_2, \s_3,\s_4}$ in  
\eqref{briefH^4} which are supported on Benjamin-Feir resonances are actually zero. 
The consequence of this ``null condition" of the gravity water waves system in infinite depth is the following remarkable result:

\begin{theorem}{\bf (Formal integrability of the water waves Hamiltonian \eqref{Hamiltonian}
at order four \cite{Zak2,DZ,CW,CS})}\label{Zakintegro}
The Hamiltonian $H^{(4)}_{ZD}$ in \eqref{briefH^4} has the form \eqref{theoBirH}. 
The Hamiltonian $H_{ZD} = H^{(2)}_{ZD} + H^{(4)}_{ZD}  $ 
in \eqref{theoBirHfull} is integrable and it can be written in action-angle variables as \eqref{new}.
The Hamiltonian vector field generated by $ H_{ZD} $,  
explicitly written in \eqref{HCS1} and \eqref{primRn}, possesses 
the actions $ | w_{n}|^{2}$, $n\in \Z\setminus\{0\}$,
as  prime integrals. 
In particular, the flow of $ H_{ZD}$ preserves all Sobolev norms.
\end{theorem}

We refer the reader to Subsection \ref{sec:formal} for further details on the structure of the Hamiltonian $ H_{ZD}$.
Unfortunately, this striking result is  a purely formal calculation because
the transformation $\Phi$ in \eqref{brieftra} is not bounded and invertible, 
and there is no control on the higher order remainder terms. 
Thus, no actual relation can be established between the flow of $H$ 
(which is well-posed for short times) and that of $H_\C \circ{\Phi}$.

\smallskip
\subsubsection{Strategy for the proof of Theorem \ref{BNFtheorem}}\label{sec:strategy}
We now describe 
how the water-waves system \eqref{eq:113} can be conjugated, through 
finitely many  well-defined, {\it bounded} and {\it invertible} transformations, to the Hamiltonian equation \eqref{theoBireq}, 
$$
\partial_t z = -\ii |D|^{1/2} z - \ii \partial_{\bar{z}} H^{(4)}_{ZD}(z,\bar{z}) + {\mathcal X}_{\geq 4}^+ \, , 
$$
where $H^{(4)}_{ZD} $ is the Hamiltonian \eqref{theoBirH} and 
the quartic {\it remainder $ {\mathcal X}_{\geq 4}^+  $
admits energy estimates} in the sense of \eqref{theoBirR}.

\medskip
{\it Step 1: Diagonalization up to smoothing remainders}. 
We begin our analysis by paralinearizing the water-waves system \eqref{eq:113}, writing it as a system in the 
complex variable 
$$
U:=(u,\bar{u}), \qquad u :=\frac{1}{\sqrt{2}}|D|^{-\frac{1}{4}}\eta+\frac{\ii}{\sqrt{2}}|D|^{\frac{1}{4}}\omega  \, , 
$$
where $\omega$ is the ``good-unknown'' defined in \eqref{omega0}.
The good unknown $\omega$ has been introduced by Alazard-Metivier \cite{AlM} and 
systematically used 
in the works  on the local existence theory of Alazard-Burq-Zuily \cite{ABZduke,ABZ1} to prove energy estimates; 
see also Alazard-Delort \cite{AlDe, AlDe1}.
In this paper  we use  the paralinearization results
proved in \cite{BD}, 
collected in Proposition \ref{laprimapara},  
 which, in addition, provide expansions in homogeneous components in $ \eta, \omega $ of the 
 paralinearized system. 
The precise form of the 
system satisfied by the complex variable $U$ is  given by  \eqref{eq:415} in Proposition \ref{WWcomplexVar}.

Our first  task is to perform a diagonalization in $(u,\bar{u})$ 
of this system up to smoothing remainders.
We remark that the highest order quasilinear transport operator in the system \eqref{eq:415} 
is already diagonal. 
Hence, as a first step in Section \ref{diagoparaprod}
we diagonalize  the sub-principal operator (which is of order $1/2$), 
as one would do to obtain local-in-time energy estimates.
We then use 
an iterative descent procedure to diagonalize the operators of order $0$, $-1/2$, and so on, 
up to a large negative order.
The outcome of this procedure is described in Proposition \ref{teodiagonal}, in which we obtain that \eqref{eq:415} 
is reduced to the system \eqref{sistemaDiag}. 

The main reason for this ``super'' diagonalization procedure,
inspired by \cite{BD}, is that, 
when combined with a reduction to constant coefficients (Step 2 below), 
it allows us to handle all the losses of derivatives that arise from quasilinear terms and small divisors
when performing Poincar\'e-Birkhoff normal form reductions (Step 3 below). 

\medskip
{\it Step 2: Reduction to constant coefficients and Poincar\'e-Birkhoff normal forms}. 
In Section \ref{riduco} we reduce all the para-differential operators in the diagonalized system 
\eqref{sistemaDiag} to constant-in-$x$ coefficients, which are integrable in the sense of
Definition \ref{defiintegro},
up to smoothing remainders of homogeneity $2$ and $3$, 
and higher order contributions that admit energy estimates of the form \eqref{theoBirR}.
The most delicate reductions concern the quasilinear components in the right-hand side of \eqref{sistemaDiag}: 
the highest order fully nonlinear transport term $\ii\opbw(V\xi)$ and the quasilinear dispersive term 
$\ii\opbw( (1 + a^{(0)}) |\xi|^{\frac12} )$.

Let us briefly describe how to deal with the transport term.
Roughly speaking, at the highest order,  system \eqref{sistemaDiag} 
(which we represent using only its first equation, the second one being the complex conjugate) 
looks like
\be
\label{stratra1} 
\pa_t w = -\ii \opbw(V \xi) w + \cdots 
\ee
where $V=V(u)$ is a real-valued function that depends on $u$, hence on $x$ and $t$, and vanishes  at $u = 0$. 
Our aim is to transform \eqref{stratra1} into an equation of the form
\begin{align}
\label{stratra1'} 
\pa_t v = -\ii \opbw(\wt{V}\xi) v + \cdots   \qquad {\rm where} \qquad  
\wt{V}= \zeta (u) + O(u^3) 
\end{align}
is a real valued function 
independent of $x$ up to cubic order in $ u $.
To do this we consider an auxiliary flow $\Phi^\theta$ obtained by solving
\be
\label{straflow} 
\partial_\theta \Phi^\theta = \mathcal{A} \Phi^\theta, \quad \Phi^{\theta=0} = {\rm Id}\, ,
\ee
where $\mathcal{A}=\ii\opbw( \beta (u) \xi)$ is a para-differential operator,
with $ \beta (u)$ a real-valued function to be determined depending on 
the solution $u$, and possibly on $\theta$.
Since $\beta$ is real valued, \eqref{straflow} is a paradifferential transport equation
which is well-posed in the auxiliary time $\theta$, 
and gives rise to a bounded and invertible flow $\Phi^\theta$, $\theta\in[0,1]$. 
The conjugation through the flow  $\Phi^{\theta=1}$ corresponds 
to a paradifferential change of variables which is approximately given by the paracomposition operator associated to 
the diffeomorphism $ x \mapsto x + \beta (u(x,t))$ of $ \T$. 

We then define a new variable through the time-$1$ flow of \eqref{straflow}, 
\[ 
v := \Phi^{\theta=1}w \, .
\]
Conjugating \eqref{stratra1} through the flow $\Phi^{\theta=1}$ one obtains (see Lemma \ref{lem:tra.Vec}) 
\be
\label{stratra3} 
\pa_t v = -\ii \opbw(V\xi) v - [\partial_t,\mathcal{A}] v   
+ \cdots = 
 -\ii \opbw \big( (V (u)  +  \pa_t \beta (u))  \xi \big) v   + \cdots 
\ee
where 
``$\cdots$'' denote paradifferential operators 
of order less than $1$, or terms satisfying the energy estimates \eqref{theoBirR}. 
Note that the contribution at the highest order $ 1$ comes from the conjugation of $ \pa_t $ 
because the dispersion relation  $ - \ii |D|^{1/2} $ has sublinear growth. 
For this reason, all our transformations are very different with respect to those performed in
 \cite{BD} for the gravity-capillary equations  where the dispersion relation $ \sim - \ii |D|^{3/2} $ 
 is superlinear. 
In Appendix \ref{sez:A2} we provide the general 
transformation rules of a paradifferential operator under
the flow generated by a paradifferential equation like \eqref{straflow}. 
In particular, a key feature is that paradifferential
operators are transformed into paradifferential ones with  symbols which can be algorithmically computed. 

In light of \eqref{stratra3} we look for $\beta$ solving
\be\label{betatV}
\partial_t  \beta (u) + V(u) = \zeta(u) + O(u^3) \, , 
\ee
where $\zeta(u)$ is constant-in-$x$. 
However, in general it is only possible to obtain
\[
\partial_t  \beta (u) + V(u) = \sum_{n \in \Z \setminus \{0\} } (\mathtt{V}_2^{(1)})^{+-}_{n,n} |u_n|^2 
+ \sum_{n \in \Z \setminus \{0\} } 
(\mathtt{V}_2^{(1)})^{+-}_{n,-n} u_n \ov{u_{- n}} e^{\ii 2 n x } + O(u^3) \, , 
\]
where $(\mathtt{V}_2^{(1)})^{+-}_{n_1n_2}$ are some coefficients depending on the function $V$.
We then verify the essential cancellation $ (\mathtt{V}_2^{(1)})^{+-}_{n,-n}\equiv 0$,
thus  reducing the equation \eqref{stratra3} to the desired form \eqref{stratra1'}.
More specifically, $ \zeta (u) $ has the ``integrable" form
\[ \zeta (u) =  \frac{1}{\pi} \sum_{n \in \Z\setminus\{0\}} n |n | |u_{n}|^2 \, , \] 
and \eqref{stratra1'} is given in Fourier by
\begin{equation}
\label{stravdot} 
\dot v_n = - \frac{\ii}{\pi} \Big(\sum_{j \in \Z\setminus\{0\}} j |j| |u_{j}|^2 \Big) \,  n v_n + \cdots 
\end{equation}
which (substitute $ u_j = v_j + \cdots $) is composed only by Birkhoff resonant cubic vector field monomials.

\begin{remark}
While we do verify explicitly several key cancellations, such as the one leading to \eqref{betatV}, 
some, but not all, of them can be derived as a consequence of the following invariance properties
of the water waves system \eqref{eq:113}:

\noindent
$(i)$ The water waves  vector field  $X(\eta, \psi)$ is reversible with respect to  the involution
 \be\label{involution}
S : \vect{\eta(x)}{\psi (x) } \mapsto  \vect{\eta(-x)}{-\psi (-x) },\;\;\;{\rm i.e.}\;\;\;
X\circ S=-S\circ X \, ;
\ee
\noindent
$(ii)$ $X$ is even-to-even, i.e. maps even functions into even functions. 
\end{remark}

The reduction described above is performed in Subsection \ref{inteord1} 
in two separate steps corresponding to the degrees of homogeneity one and two in $u$. 
Similar arguments can be used to reduce to constant coefficients -- and in Birkhoff normal form -- 
the modified dispersive term $ \ii (1 + {\mathtt a}_2)|\xi|^{1/2}$. 
Actually, thanks to additional algebraic cancellations, which appear to be intrinsic
to the water waves system \eqref{eq:113}, it turns out that the new dispersive term 
is exactly $- \ii |D|^{\frac12}$, up to lower order symbols.
The transformation which is used for the conjugation is the flow (thus bounded and invertible) 
generated by a paradifferential ``semi-Fourier integral operator'' as \eqref{straflow} with generator
${\cal A } = \ii \opbw( \beta (u) |\xi|^{\frac{1}{2}} )$ for a suitable real $\b (u)$.

All lower order symbols can also be reduced to constant coefficients in $x$ 
-- and in Poincar\'e-Birkhoff normal form -- using flows generated by Banach space ODEs. 
Eventually we obtain the system \eqref{finalsyst}, which is in Poincar\'e-Birkhoff-normal form up to cubic degree in $ u $, 
and up to a smoothing remainder and admissible symbols which satisfy energy estimates. 
We say that \eqref{finalsyst} is in Poincar\'e-Birkhoff  normal form, and not just Birkhoff,  
because it is not Hamiltonian, since we performed non-symplectic transformations.   

\medskip
{\it Step 3: Poincar\'e-Birkhoff normal form reductions}. 
By the previous transformations 
we have obtained a system of the form  
\begin{align}\label{stra10} 
\pa_t z = - \zeta (z) \partial_x z -\ii |D|^{\frac12} z  + r_{-1/2} (z; D)[z] + R(z)  + {\mathcal X}_{\geq 4} 
\end{align}
where $ r_{-1/2} $ is a constant-coefficient integrable symbol of order $ - 1/ 2 $,
up to some very regular nonlinear term $ R(z)$, 
plus an admissible remainder term $ {\mathcal X}_{\geq 4} $
of higher homogeneity satisfying energy estimates as \eqref{theoBirR}. 
Our next step, in Section \ref{sec:BNF}, is to apply Poincar\'e-Birkhoff normal form transformations to eliminate 
all non-resonant quadratic and cubic nonlinear terms in the smoothing remainder $ R $. 
Thanks to these normal forms transformations the new system becomes (Proposition \ref{cor:BNF})
\begin{equation}\label{stra11}
\pa_{t}z= -\zeta  (z) \pa_x z - \ii |D|^{\frac{1}{2}}z 
+ r_{-1/2} (z; D)[z] +  R^{\mathrm{res}}(z)+ {\mathcal X}_{\geq 4}
\end{equation}
where $ R^{\mathrm{res}}(z) $ is a cubic term of the form
\begin{align}\label{stra12}
\begin{split}
& R^{\mathrm{res}}(z) = 
  \sum_{ 
 \substack{  \sigma_1n_1+\sigma_2n_2 + \s_3 n_3 = n \\
 \sigma_1\omega(n_1) +\sigma_2\omega(n_2) + \s_3 \omega(n_3) =  \omega(n) }} 
 c_{n_1,n_2,n_3}^{\sigma_1\sigma_2, \s_3} 
  	\, z_{n_1}^{\sigma_1} z_{n_2}^{\sigma_2} z_{n_3}^{\s_3} e^{\ii n x}
\end{split}
\end{align}
with coefficients $c_{n_1,n_2,n_3}^{\sigma_1\sigma_2\s_3}\in\C$
(compare with \eqref{briefH^4}).
Solutions to the constraints in \eqref{stra12} are of two types: 

\begin{itemize}
\item[(a)] {\it Trivial Resonances}: 
These occur when, say $(\s_1,\s_2,\s_3)= (+,-,+)$,
one has $n_1=n_2$ and $n = n_3$ (or permutations) producing resonant cubic terms of the 
form $c_{n_1, n} |z_{n_1}|^2 z_n e^{\ii n x}$.

\smallskip
\item[(b)] {\it Benjamin-Feir resonances}: 
These occur when three of the frequencies $(n_1,n_2,n_3,n_4)$
have the same sign
and, in the case of $(\s_1,\s_2)=(+,-)$ say, are given by the two-parameter family in \eqref{straBFR}.
\end{itemize}

In the case of Hamiltonian systems -- or, more in general, in the presence of other
algebraic structures -- 
one can expect that the trivially resonant terms will not impact the dynamics. 
Note however the following difficulty: we have performed non-symplectic transformations so that 
the Hamiltonian nature of \eqref{stra11} is lost. 
In addition, the presence of $4$-waves resonances, such as the Benjamin-Feir,  is
a strong obstruction to 
prove bounded dynamics for times of the order $\e^{-3}$.
One may expect, in analogy with Theorem \ref{Zakintegro}, to be able to check by direct computations 
that the coefficients $c_{n_1,n_2,n_3}^{\sigma_1\sigma_2 \s_3}$ in \eqref{stra12}
vanish on the Benjamin-Feir resonances. 
However, after having  performed all the reductions described before,
this computation seems rather involved. 
As we describe in Step 4 below, in this paper 
we will prove such a property by an indirect uniqueness argument of the cubic Poincar\'e-Birkhoff normal form.  

Before moving to this last step
let us comment on the issue of {\it small divisors}.
To perform the above Birkhoff normal form reduction of the smoothing terms, 
we  need to deal with near-resonances.
Indeed, our normal form transformation is generated by a flow as in \eqref{straflow} where,
roughly speaking, the coefficients of the operator $\mathcal{A}$ are obtained through division by 
the phase  $   \sigma_1\omega(n_1) +\sigma_2\omega(n_2) + \s_3 \omega(n_3) -  \omega(n) $. 
This becomes dangerous if it degenerates rapidly close to the resonances. 
For example, if $ \s_1 = 1 = \s_3 $, $ \s_2 = - 1 $, and 
$n_1=k$, $n_2=-k$, $n_3=j$,  $n =j+2k$, with $j \gg k$
we get  $|  \omega(n_1) - \omega(n_2) + \omega(n_3) -  \omega(n) | \approx j^{-1/2}$.
Dividing by this expression then causes a 
loss of (at least) a $1/2$ derivative.
In our proof this issue is overcome thanks to  
the fact that the smoothing remainders  $ R(z) $ can tolerate losses of derivatives.

\medskip
{\it Step 4: Normal form identification}. 
In our last main step in Subsection \ref{sec:INF} we prove that the cubic terms in \eqref{stra11}-\eqref{stra12}
coincide with the Hamiltonian vector field generated by the quartic Hamiltonian in  \eqref{theoBirH}
\be\label{straide}
-\zeta (z) \partial_x z + r_{-1/2} (z; D)[z] +  R^{\mathrm{res}}(z)  = -\ii\pa_{\bar{z}}H^{(4)}_{ZD} \, .
\ee
This implies in particular that $ R^{\mathrm{res}}(z) $ is supported only on trivial resonances.
To obtain \eqref{straide} we use a normal form identification argument 
which relies on the uniqueness of solutions of the quadratic homological equation \eqref{Poisson3}.
The final outcome is that the equation \eqref{stra11}, that we have obtained through 
the  bounded and invertible transformations described in Steps 1--3, coincides 
up to quartic terms, with the Hamiltonian vector field generated by the 
Hamiltonian \eqref{theoBirHfull} formally derived by Zakharov-Dyachenko-Craig-Worfolk-Sulem. 
Note that this argument also proves that the cubic vector field of the Poincar\'e-Birkhoff 
normal form \eqref{stra11} is Hamiltonian, which was not known a priori since we 
have performed non-symplectic transformations.
This identification argument is similar in spirit 
to Moser's indirect proof of the convergence of the Lindsted series for a KAM torus 
\cite{M67}: Moser rigorously proves the existence of quasi-periodic solutions, 
which are analytic in $\e$, and then shows,  a posteriori, 
that their Taylor expansions in $\e$ coincide with the formal Lindsted power series.

\medskip
\noindent
{\bf Acknowledgements}. We thank M. Procesi and W. Craig for stimulating discussions on the topic.

\bigskip
\section{ Functional setting and para-differential calculus}\label{sec:funcsett}
In this section we introduce our notation and recall several results on para-differential calculus,
mostly following Chapter $3$ of the monograph \cite{BD}.
We find convenient the use of this set-up to obtain our initial paralinearization of the 
water waves equations \eqref{eq:113} with multilinear expansions, as stated in Proposition \ref{laprimapara},
and several tools for conjugations via paradifferential flows which are contained in Appendix \ref{sez:A2}. 

Given  an interval $I \subset \R$ symmetric with respect to $t=0$ and  $s\in \R $ we define the space
$$
C^K_{*}(I,{\dot{H}}^{s}(\T,\C^2)):=\bigcap_{k=0}^{K}C^{k}\big(I;\dot{H}^{s-k}(\T;\C^2)\big) \, , 
$$ 
endowed 
with the norm
\begin{equation}\label{normaT}
\sup_{t\in I}\|{U(t,\cdot)}\|_{K,s} \quad \mbox {where} 
\quad \|{U(t,\cdot)}\|_{K,s}:=\sum_{k=0}^{K}\|{\partial_t^k U(t,\cdot)}\|_{{\dot{H}}^{s-k}}.
\end{equation}
We denote by $C^K_{*\R}(I,{\dot{H}}^{s}(\T,\C^2))$
the space of functions $U$ in $C^K_{*}(I,{\dot{H}}^{s}(\T,\C^2))$ such that $U=\vect{u}{\bar{u}}$.
Given $r > 0 $ we set
\begin{equation}\label{palla}
B_{s}^K(I;r):=\Big\{U\in C^K_{*}(I,\dot{H}^{s}(\T;\C^{2})):\, \sup_{t\in I}\|{U(t,\cdot)}\|_{K,s}<r \Big\} \, .
\end{equation}
With similar meaning we denote $C_{*}^{K}(I;\dot{H}^{s}(\T;\C))$.
 We expand a $ 2 \pi $-periodic function $ u(x) $, with zero average in $x$ (which is identified with 
 $ u $ in the homogeneous space),  in Fourier series as 
\be\label{complex-uU}
u(x) = \sum_{n \in \Z\setminus\{0\} } \hat{u}(n)\frac{e^{\ii n x }}{\sqrt{2\pi}} \, , \qquad 
\hat{u}(n) := \frac{1}{\sqrt{2\pi}} \int_\T u(x) e^{-\ii n x } \, dx \, .
\ee
We also use the notation
\begin{equation}\label{notaFou}
u_n^+ := u_n := \hat{u}(n) \qquad  {\rm and} \qquad  u_n^- := \ov{u_n}  := \ov{\hat{u}(n)} \, . 
\end{equation}
For $n\in \N^*:= \N \! \smallsetminus \! \{0\}$ we denote by $\Pi_{n}$ the orthogonal projector from $L^{2}(\T;\C)$ 
to the subspace spanned by $\{e^{\ii n x}, e^{-\ii nx}\}$,  i.e.
\begin{equation}\label{spectralpro}
(\Pi_{n}u)(x) := \hat{u}({n}) \frac{e^{\ii nx}}{\sqrt{2\pi}}+\hat{u}({-n})\frac{e^{-\ii nx}}{\sqrt{2\pi}} \, , 
\end{equation}
and we denote by $ \Pi_n $ also the corresponding projector in  $L^{2}(\T,\C^{2})$.
If $\mathcal{U}=(U_1,\ldots,U_{p})$
is a $p$-tuple
 of functions, $\vec{n}=(n_1,\ldots,n_p)\in (\N^{*})^{p}$, we set
 \begin{equation}\label{ptupla}
 \Pi_{\vec{n}}\mathcal{U}:=(\Pi_{n_1}U_1,\ldots,\Pi_{n_p}U_p) \, .
 \end{equation}
In this paper we deal with 
 vector fields $ X $  
which satisfy 
the {\it $x$-translation invariance} property
\be\label{commutes-tau} 
X \circ \tau_\theta = \tau_\theta \circ X   \, , \quad  \forall\, \theta \in \R \, ,
\ee
where 
\be\label{tau-theta}
\tau_\theta : u(x) \mapsto (\tau_\theta u)(x) := u(x + \theta ) \, . 
\ee

\noindent{\bf Para-differential operators.}
We first give the definition of the classes of symbols that we are going to use,
collecting Definitions $ 3.1$, $ 3.2$ and $ 3.4$ in \cite{BD}.
Roughly speaking, the class $\widetilde{\Gamma}_{p}^{m}$ contains homogeneous symbols of order $m$ and homogeneity $p$ in $U$, while the class $\Gamma_{K,K',p}^{m}$ contains non-homogeneous symbols of order $m$ which
vanish at degree at least $p$ in $U$, and that are $(K-K')$-times differentiable in $t$.

\begin{definition}{\bf (Classes of  symbols)}\label{pomosimb}
Let $m\in\R$, $p, N \in \N$, $ p \leq N $,  $K'\leq K$ in $\N$, $r>0 $.

\begin{itemize}

\smallskip
\item[(i)]{\bf  $p$-homogeneous symbols.} 
We denote by $\widetilde{\Gamma}_{p}^{m}\!$ the space of symmetric $p$-linear maps
from $(\dot{H}^{\infty}(\T;\C^{2}))^{p}$ to the space of $C^{\infty}$ functions of $(x,\x)\in \T\times\R$, 
$ \mathcal{U}\to ((x,\x)\to a(\mathcal{U};x,\x)) $,  
satisfying the following. There is $\mu>0$ and,  
for any $\alpha,\beta\in \N $,  there is $C>0$ such that
\begin{equation}\label{pomosimbo1}
|\pa_{x}^{\alpha}\pa_{\x}^{\beta}a(\Pi_{\vec{n}}\mathcal{U};x,\x)|\leq C | \vec{n} |^{\mu+\alpha}\langle\x\rangle^{m-\beta}
\prod_{j=1}^{p}\|\Pi_{n_j}U_{j}\|_{L^{2}} 
\end{equation}
for any $\mathcal{U}=(U_1,\ldots, U_p)$ in $(\dot{H}^{\infty}(\T;\C^{2}))^{p}$,
and $\vec{n}=(n_1,\ldots,n_p)\in (\N^*)^{p}$. 
Moreover, we assume that, if for some $(n_0,\ldots,n_{p})\in \N\times(\N^*)^{p}$,
\begin{equation}\label{pomosimbo2}
\Pi_{n_0}a(\Pi_{n_1}U_1,\ldots, \Pi_{n_p}U_{p};\cdot)\neq0 \, ,
\end{equation}
then there exists a choice of signs $\s_0,\ldots,\s_p\in\{-1,1\}$ such that $\sum_{j=0}^{p}\s_j n_j=0$.
The property \eqref{pomosimbo2} is automatically satisfied by requiring the 
translation invariance property
\be\label{def:tr-in}
a( \tau_\teta {\cal U}; x, \xi) =  a( {\cal U}; x + \theta, \xi) \, , \quad \forall \theta \in \R \, . 
\ee
For $p=0$ we denote by $\widetilde{\Gamma}_{0}^{m}$ the space of constant coefficients symbols
$\x\mapsto a(\x)$ which satisfy \eqref{pomosimbo1} with $\al=0$
and the right hand side replaced by $C\langle \x\rangle^{m-\beta}$. 

\smallskip
\item[(ii)]{\bf Non-homogeneous symbols.} Let $ p \geq 1 $. We denote by $\Gamma^m_{K,K',p}[r]$ the space 
of functions $(U;t,x,\xi)\!\mapsto \!a(U;t,x,\xi)$, defined for $U\in B_{s_0}^K(I;r)$, for some large enough $ s_0$, with complex values such that for any 
$0\leq k\leq K-K'$, any $\s\geq s_0$, there are $C>0$, $0<r(\s)<r$ and for any $U\in B_{s_0}^K(I;r(\s))\cap C^{k+K'}_{*}(I,{\dot{H}}^{\s}(\T;\C^{2}))$ and any $\alpha, \beta \in\N$, with $\alpha\leq \s-s_0$
\begin{equation}\label{simbo}
|{\partial_t^k\partial_x^{\alpha}\partial_{\xi}^{\beta}a(U;t,x,\xi)}|\leq C\langle\xi\rangle^{m-\beta} 
\|U\|^{p-1}_{k+K',s_0}
\|{U}\|_{k+K',\s} \, .
\end{equation}

\smallskip
\item[(iii)]{\bf Symbols.} We denote by $\Sigma\Gamma^{m}_{K,K',p}[r,N]$
the space of functions 
$(U,t,x,\x)\to a(U;t,x,\x)$ such that there are homogeneous symbols $a_{q}\in \widetilde{\Gamma}_{q}^{m}$
for $q=p,\ldots, N-1$ and a non-homogeneous symbol $a_{N}\in \Gamma^{m}_{K,K',N}[r]$
such that
\begin{equation}\label{simbotot1}
a(U;t,x,\x)=\sum_{q=p}^{N-1}a_{q}(U,\ldots,U;x,\x)+a_{N}(U;t,x,\x) \, .
\end{equation}
\end{itemize}
We denote by $\Sigma\Gamma^{m}_{K,K',p}[r,N]\otimes \mathcal{M}_{2}(\C)$ the space $2\times 2$
matrices whose entries are symbols in  $\Sigma\Gamma^{m}_{K,K',p}[r,N]$.
\end{definition}

\begin{remark}\label{2.2}
The translation invariance property \eqref{def:tr-in} means that the dependence with respect to the variable $x$ of the 
symbol $a({\cal U}; x, \xi)$ enters only through the function ${\cal U }(x)$.
As already observed, the translation invariance property \eqref{def:tr-in}   implies the more general assumption \eqref{pomosimbo2} made in \cite{BD} 
(that we report to be consistent with the notation of \cite{BD}).
\end{remark}

Note that 
\be
\begin{aligned}\label{prodottodisimboli2}
a\in \widetilde{\Gamma}_{p}^{m}, \;\; b\in \widetilde{\Gamma}_{q}^{m'} \quad 
& \Rightarrow \qquad ab\in \widetilde{\Gamma}_{p+q}^{m+m'}, \,  \pa_{x}a\in \widetilde{\Gamma}_{p}^{m} \, ,\, 
\pa_{\x}a\in \widetilde{\Gamma}_{p}^{m-1} \, ;\\
a\in \Gamma_{K,K',p}^{m}[r], \; K'+1\leq K \  &\Rightarrow \  
\pa_{t}a\in \Gamma^{m}_{K,K'+1,p}[r], \, 
\pa_x a\in \Gamma_{K,K',p}^{m}[r] \, , \  \pa_{\x}a\in \Gamma_{K,K',p}^{m-1}[r] \, ;\\
a\in \Gamma_{K,K',p}^{m}[r], b\in \Gamma_{K,K',q}^{m'}[r] \quad& \Rightarrow \quad   ab\in \Gamma_{K,K',p+q}^{m+m'}[r] \; \\
 a(\mathcal{U};\cdot) \in \widetilde{\Gamma}_{p}^{m} \quad &\Rightarrow \quad 
a(U,\ldots,U;\cdot) \in \Gamma_{K,0,p}^{m}[r] \, , \  \forall r>0 \, . 
\end{aligned}
\ee

Throughout this paper we will systematically use the following expansions, which are a consequence of
\eqref{def:tr-in} and $u\in \dot{H}^{\infty}(\T;\C)$. 
If $ {\mathtt a}_1 \in \widetilde{\Gamma}_{1}^{m}$ then 
\be\label{eq:F1}
{\mathtt a}_1(U;x,\x)  
=\frac{1}{\sqrt{2\pi}} \sum_{n \in \Z\setminus\{0\}, \s=\pm} (\mathtt{a}_1)^\s_n(\x) u^{\s}_n e^{\ii \s n x } 
\, ,
\ee
for some $(\mathtt{a}_1)^{\s}_{n}(\x) \in \C$,  
and,  if $ {\mathtt a}_2 \in \widetilde{\Gamma}_{2}^{m}$ then
\begin{align} 
 {\mathtt a}_2 (U,{U};x,\x) 
 & = 
\sum_{\substack{n_1, n_2 \in \Z\setminus\{0\}\\ \s=\pm}} (\mathtt{a}_2)^{\s\s}_{n_1, n_2}(\x) u^{\s}_{n_1} u^{\s}_{n_2} 
\frac{e^{\ii \s ( n_1 + n_2) x }}{2\pi} +
\sum_{n_1, n_2 \in \Z\setminus\{0\}} (\mathtt{a}_2)^{+-}_{n_1, n_2}(\x) u_{n_1} 
  \ov{u_{n_2}} \frac{e^{\ii ( n_1 - n_2) x }}{2\pi}
\label{eq:F2} \, 
\end{align}
for some $(\mathtt{a}_2)^{\s\s'}_{n_1, n_2}(\x)\in \C$ with $\s,\s'=\pm$.
In the sequel, for simplicity we may also write  
${\mathtt a}_2 (U;x,\x) $ instead of $ {\mathtt a}_2 (U,U;x,\x) $.

\smallskip
We also define the following classes of functions in analogy with our classes of symbols.

\begin{definition}{\bf (Functions)} \label{apeape} Fix $N\in \N$, 
$p\in \N$ with $p\leq N$,
 $K,K'\in \N$ with $K'\leq K$, $r>0$.
We denote by $\widetilde{\mathcal{F}}_{p}$, resp. $\mathcal{F}_{K,K',p}[r]$,  $\Sigma\mathcal{F}_{p}[r,N]$, 
the subspace of $\widetilde{\Gamma}^{0}_{p}$, resp. $\Gamma^{0}_{p}[r]$, 
resp. $\Sigma\Gamma^{0}_{p}[r,N]$, 
made of those symbols which are independent of $\x$.
We write $\widetilde{\mathcal{F}}^{\R}_{p}$, resp. $\mathcal{F}_{K,K',p}^{\R}[r]$, 
$\Sigma\mathcal{F}_{p}^{\R}[r,N]$,  to denote functions in $\widetilde{\mathcal{F}}_{p}$, 
resp. $\mathcal{F}_{K,K',p}[r]$,  $\Sigma\mathcal{F}_{p}[r,N]$, 
which are real valued.
\end{definition}

Note that  functions  $ {\mathtt a}_1 \in \widetilde{\cal F}_{1} $, $ {\mathtt a}_2 \in \widetilde{\cal F}_{2} $
 expanded as in \eqref{eq:F1}, \eqref{eq:F2} are real valued if and only if
\begin{align}
& \ov{(\mathtt{a}_1)_{n}^{+}}=(\mathtt{a}_1)_{n}^{-},  
\quad \overline{ (\mathtt{a}_2)^{++}_{n_1, n_2}} = (\mathtt{a}_2)^{--}_{n_1, n_2} \, , \quad 
\overline{ (\mathtt{a}_2)^{+-}_{n_1, n_2}} = (\mathtt{a}_2)^{+-}_{n_2, n_1} \, .
\label{condreal2} 
\end{align}

\noindent
{\bf Paradifferential quantization.}
Given $p\in \N$ we consider  smooth functions
  $\chi_{p}\in C^{\infty}(\R^{p}\times \R;\R)$ and $\chi\in C^{\infty}(\R\times\R;\R)$, 
  even with respect to each of their arguments, satisfying, for some $0<\delta\ll 1$,
\begin{align}
&{\rm{supp}}\, \chi_{p} \subset\{(\xi',\xi)\in\R^{p}\times\R; |\xi'|\leq\delta \langle\xi\rangle\} \, ,\qquad \chi_p (\xi',\xi)\equiv 1\,\,\, \rm{ for } \,\,\, |\xi'|\leq \delta \langle\xi\rangle / 2 \, ,\label{cutoff1}\\
&\rm{supp}\, \chi \subset\{(\xi',\xi)\in\R\times\R; |\xi'|\leq\delta \langle\xi\rangle\} \, ,\qquad \quad
 \chi(\xi',\xi) \equiv 1\,\,\, \rm{ for } \,\,\, |\xi'|\leq \delta   \langle\xi\rangle / 2 \, .\label{cutoff2}
\end{align}
For $p=0$ we set $\chi_0\equiv1$. 
We assume moreover that 
\begin{equation}\label{derab}
\begin{aligned}
& |\partial_{\xi}^{\alpha}\partial_{\xi'}^{\beta}\chi_p(\xi',\xi)|\leq C_{\alpha,\beta}\langle\xi\rangle^{-\alpha-|\beta|},\,\,
\quad \forall \alpha\in \N, \,\beta\in\N^{p} \, ,  \\
& |\partial_{\xi}^{\alpha}\partial_{\xi'}^{\beta}\chi(\xi',\xi)|\leq C_{\alpha,\beta}\langle\xi\rangle^{-\alpha-\beta},\,\,
\qquad \forall \alpha, \,\beta\in\N \, . 
\end{aligned}
\end{equation}
A function satisfying the above condition  is $\chi(\xi',\xi):=\widetilde{\chi}(\xi'/\langle\xi\rangle)$
where $\widetilde{\chi}$ is a function in $C_0^{\infty}(\R;\R)$  having a small enough support and equal to one in a neighborhood of zero.

\begin{definition}{\bf (Bony-Weyl quantization)}\label{quantizationtotale}
If a is a symbol in $\widetilde{\Gamma}^{m}_{p}$, respectively in $\Gamma^{m}_{K,K',p}[r]$,
we define its \emph{Weyl} quantization  as the operator
acting on a
$ 2 \pi $-periodic function
$u(x)$ (written as in \eqref{complex-uU})
 as
\be\label{Weil-Q}
Op^{W}(a)u=\frac{1}{\sqrt{2\pi}}\sum_{k\in \Z}
\Big(\sum_{j\in\Z}\hat{a}\big(k-j, \frac{k+j}{2}\big)\hat{u}(j) \Big)\frac{e^{\ii k x}}{\sqrt{2\pi}}
\ee
where $\hat{a}(k,\xi)$ is the $k^{th}-$Fourier coefficient of the $2\pi-$periodic function $x\mapsto a(x,\xi)$.

\noindent
We set, using notation \eqref{ptupla}, 
\[
a_{\chi_{p}}(\mathcal{U};x,\x) =\sum_{\vec{n}\in \N^{p}}\chi_{p}\left(\vec{n},\x\right)a(\Pi_{\vec{n}}\mathcal{U};x,\x),
\quad 
a_{\chi}(U;t,x,\x) =\frac{1}{2\pi}\int_{\R}  \chi\left(\xi',\x\right)\hat{a}(U;t,\xi',\x)e^{\ii \xi' x}d \xi' \, , 
\]
where in the last equality $  \hat a $ stands for the Fourier transform with respect to the $ x $ variable.
 Then we define the \emph{Bony-Weyl} quantization of $ a $ as 
\be\label{quantiz}
\opbw(a(\mathcal{U};\cdot))=Op^{W}(a_{\chi_{p}}(\mathcal{U};\cdot)),\qquad
\opbw(a(U;t,\cdot))=Op^{W}(a_{\chi}(U;t,\cdot)) \, .
\ee
If  $a$ is a symbol in  $\Sigma\Gamma^{m}_{K,K',p}[r,N]$, that we decompose as in \eqref{simbotot1},
we define its \emph{Bony-Weyl} quantization 
\[
\opbw(a(U;t,\cdot))=\sum_{q=p}^{N-1}\opbw(a_{q}(U,\ldots,U;\cdot))+\opbw(a_{N}(U;t,\cdot)) \, . 
\]
\end{definition}

\noindent
$ \bullet $ By the translation invariance property  \eqref{def:tr-in}, we have  
\be\label{tautheta}
\opbw(a_{q}( \tau_\theta U,\ldots, \tau_\theta U;  \cdot , \xi))[ \tau_\theta V] =  
\tau_\theta \big(  \opbw(a_{q}(  U,\ldots, U;\cdot, \xi ))[ V] \big) \, .
\ee

\noindent
$ \bullet $ The operator 
$ \opbw (a) $ acts on homogeneous spaces of functions, 
see Proposition \ref{azionepara}.

\noindent
$\bullet$ 
The action of
$\opbw(a)$ on  homogeneous spaces only depends
on the values of the symbol $ a = a(U;t,x,\xi)$ (or
$a(\mathcal{U};t,x,\xi)$) for $|\xi|\geq 1$.
Therefore, we may identify two symbols $ a(U;t,x,\xi)$ and
$ b(U;t,x,\xi)$ if they agree for $|\xi| \geq 1/2$.
In particular, whenever we encounter a symbol that is not smooth at $\xi=0 $,
such as, for example, $a = g(x)|\x|^{m}$ for $m\in \R\setminus\{0\}$, or $ \sign (\xi) $, 
we will consider its smoothed out version
$\chi(\xi)a$, where
$\chi\in {\cal C}^{\infty}(\R;\R)$ is an even and positive cut-off function satisfying  
\begin{equation}\label{cutoff11}
\chi(\x) =  0 \;\; {\rm if}\;\; |\x|\leq \frac{1}{8}\, , \quad 
\chi (\x) = 1 \;\; {\rm if}\;\; |\x|>\frac{1}{4}, 
\quad  \pa_{\x}\chi(\x)>0\quad\forall  \x\in \big(\frac{1}{8},\frac{1}{4} \big) \, .
\end{equation}
%
\noindent
$ \bullet $  If $ a$ is a homogeneous  symbol, the two definitions  of quantization in \eqref{quantiz}, differ by a  smoothing operator that  we introduce in Definition \ref{omosmoothing} below. 

\smallskip
Definition \ref{quantizationtotale} 
is  independent of the cut-off functions $\chi_{p}$, $\chi$ satisfying \eqref{cutoff1}-\eqref{derab}
up to smoothing operators that we define below (see Definition $ 3.7 $ in \cite{BD}).
Roughly speaking, the class $\widetilde{\mathcal{R}}^{-\rho}_{p}$ contains smoothing operators
which gain $\rho$ derivatives and are homogeneous of degree $p$ in $U$, while the class 
$\mathcal{R}_{K,K',p}^{-\rho}$ contains non-homogeneous $\rho$-smoothing operators which
vanish at degree at least $p$ in $U$, and are $(K-K')$-times differentiable in $t$.

Given  $(n_1,\ldots,n_{p+1})\in \N^{p+1}$ we denote by $\max_{2}(n_1 ,\ldots, n_{p+1})$ 
the second largest among the integers $ n_1,\ldots, n_{p+1}$.

 \begin{definition}{\bf (Classes of  smoothing operators)} \label{omosmoothing}
 Let $K'\leq K\in\N$, $N\in \N$ with $N\geq1$, $\m\in \R$,  $\rho\geq0$ and $r>0$.
 \begin{itemize}
 \item[(i)]{\bf $p$-homogeneous smoothing operators.} We denote by $\widetilde{\mathcal{R}}^{-\rho}_{p}$
 the space of $(p+1)$-linear maps $R$ 
 from the space $(\dot{H}^{\infty}(\T;\C^{2}))^{p}\times \dot{H}^{\infty}(\T;\C)$ to 
 the space $\dot{H}^{\infty}(\T;\C)$ symmetric
 in $(U_{1},\ldots,U_{p})$, of the form
$ (U_{1},\ldots,U_{p+1})\to R(U_1,\ldots, U_p)U_{p+1}$
 that satisfy the following. There are $\mu\geq0$, $C>0$ such that 
$$
 \|\Pi_{n_0}R(\Pi_{\vec{n}}\mathcal{U})\Pi_{n_{p+1}}U_{p+1}\|_{L^{2}}\leq
 C\frac{\max_2( n_1,\ldots, n_{p+1})^{\rho+\mu}}{\max( n_1,\ldots, n_{p+1})^{\rho}}
 \prod_{j=1}^{p+1}\|\Pi_{n_{j}}U_{j}\|_{L^{2}}
$$
  for any 
 $\mathcal{U}=(U_1,\ldots,U_{p})\in (\dot{H}^{\infty}(\T;\C^{2}))^{p}$, any 
 $U_{p+1}\in \dot{H}^{\infty}(\T;\C)$,
 any $\vec{n}=(n_1,\ldots,n_p)\in (\N^{*})^{p}$, any $n_0,n_{p+1}\in \N^*$.
 Moreover, if 
 \begin{equation}\label{omoresti2}
 \Pi_{n_0}R(\Pi_{n_1}U_1,\ldots,\Pi_{n_{p}}U_{p})\Pi_{n_{p+1}}U_{p+1}\neq 0 \, ,
 \end{equation}
 then there is a choice of signs $\s_0,\ldots,\s_{p+1}\in\{\pm 1\}$ such that 
 $\sum_{j=0}^{p+1}\s_j n_{j}=0$. 
 In addition we require the translation invariance property
\be\label{def:R-trin}
R( \tau_\teta {\cal U}) [\tau_\teta U_{p+1}]  =  \tau_\theta \big( R( {\cal U})U_{p+1} \big) \, , 
\quad \forall \theta \in \R \, . 
\ee

 \item[(ii)] {\bf Non-homogeneous smoothing operators.}
  We denote by $\mathcal{R}^{-\rho}_{K,K',N}[r]$ 
  the space of maps $(V,t,U)\mapsto R(V)U$ defined on 
  $B^K_{s_0}(I;r)\times I \times C^K_{*}(I,\dot{H}^{s_0}(\T,\C))$ 
  which are linear in the variable $U$ and such that the following holds true. 
  For any $s\geq s_0$ there exist a constant $C>0$ and $r(s)\in]0,r[$ such that for any $V\in B^K_{s_0}(I;r)\cap C^K_{*}(I,\dot{H}^{s}(\T,\C^2))$, 
  any $  U \in C^K_{*}(I,\dot{H}^{s}(\T,\C))$, any 
  $0\leq k\leq K-K'$ and any $t\in I$, we have 
\begin{equation}
\begin{aligned} \label{piove}
\|{\partial_t^k\left(R(V;t)U\right)(t,\cdot)}\|_{\dot{H}^{s-k+\rho}} 
& \leq \sum_{k'+k''=k}C\Big( \|{U}\|_{k'',s}\|{V}\|_{k'+K',s_0}^{N} 
\\
& \quad \qquad \  \quad +\|{U}\|_{k'',s_0}\|V\|_{k'+K',s_0}^{N-1}\|{V}\|_{k'+K',s}\Big) \, .
\end{aligned}
\end{equation}

\item[(iii)]{\bf Smoothing operators.} We denote by $\Sigma\mathcal{R}^{-\rho}_{K,K',p}[r,N]$
the space of maps $(V,t,U)\to R(V;t)U$ 
that may be written as 
$$
R(V;t) U =\sum_{q=p}^{N-1}R_{q}(V,\ldots,V) U +R_{N}(V;t) U 
$$
for some $R_{q} $ in $ \widetilde{\mathcal{R}}^{-\rho}_{q}$, $q=p,\ldots, N-1$ and $R_{N}$  in 
$\mathcal{R}^{-\rho}_{K,K',N}[r]$.
 \end{itemize}
 We denote by  $\Sigma\mathcal{R}^{-\rho}_{K,K',p}[r,N]\otimes\mathcal{M}_2(\C)$
the space of $2\times 2$ matrices whose entries are  in $\Sigma\mathcal{R}^{-\rho}_{K,K',p}[r,\! N]$.
 \end{definition}

 
$ \bullet $
If $R $ is in $ \widetilde{\mathcal{R}}^{-\rho}_{p}$  then $ (V, U ) \mapsto R(V, \ldots, V) U $
is in $\mathcal{R}^{-\rho}_{K,0,p}[r] $, i.e. \eqref{piove} holds with $ N \rightsquigarrow p, K' = 0 $.

$ \bullet $
If $R_i \in \Sigma\mathcal{R}^{-\rho}_{K,K',p_i}[r,N]$, $ i = 1,2 $,   then
the composition $R_1 \circ R_2 $ is in $ \Sigma\mathcal{R}^{-\rho}_{K,K',p_1+ p_2}[r,N]$. 

\smallskip
The next proposition states
boundedness properties on Sobolev spaces of the  paradifferential operators
(see Proposition 3.8  in \cite{BD}).

\begin{proposition}{\bf (Action of para-differential operator)} \label{azionepara}
Let $r>0$, $m\in \R$, $p\in \N$, $K'\leq K\in \N$. Then:

\smallskip
$(i)$ There is $ s_0 > 0 $ such that for any symbol
 $a\in \widetilde{\Gamma}_{p}^{m}$, 
there is a constant $C>0$, depending only on $s$ and on \eqref{pomosimbo1} with $\alpha=\beta=0$,
such that for any $\mathcal{U}=(U_1,\ldots,U_{p})$
\begin{equation}\label{stimapar}
\|\opbw(a(\mathcal{U};\cdot))U_{p+1}\|_{\dot{H}^{s-m}}\leq C\prod_{j=1}^{p}\|U_{j}\|_{\dot{H}^{s_0}}\|U_{p+1}\|_{\dot{H}^{s}} \, ,
\end{equation}
for $p\geq 1$, while for $p=0$ the \eqref{stimapar} holds by replacing the right hand side with $C\|U_{p+1}\|_{\dot{H}^{s}}$.

\smallskip
$(ii)$
There is $ s_0 > 0 $ such that  for any symbol $a\in \Gamma^{m}_{K,K',p}[r]$
there is  a constant $C>0$, depending  only on $s,r$ and \eqref{simbo}
with $0\leq \alpha\leq 2$, $\beta=0$, such that, for any $t\in I$, any $0\leq k\leq K-K'$,
$$
\|\opbw(\pa_{t}^{k}a(U;t,\cdot))\|_{\mathcal{L}(\dot{H}^{s},\dot{H}^{s-m})}\leq C\|U\|_{k+K',s_0}^{p} \, .
$$
\end{proposition}

$ \bullet $ If $a\in \Sigma\Gamma^{m}_{K,K',p}[r,N]$ with $m\leq 0$ and $p\geq1$, 
then $ \opbw(a(V;t,\cdot))U$
is in $\Sigma\mathcal{R}^{m}_{K,K',p}[r,N]$.

\medskip
Below we deal with classes of operators without keeping track of the number of lost derivatives in a precise way 
(see Definition 3.9  in \cite{BD}).
The class $\widetilde{\mathcal{M}}^{m}_{p}$ denotes multilinear maps that lose $m$ derivatives
and are $p$-homogeneous in $U$, 
while the class $\mathcal{M}_{K,K',p}^{m}$ contains non-homogeneous maps which lose $m$ derivatives,
vanish at degree at least $p$ in $U$, and are $(K-K')$-times differentiable in $t$.

\begin{definition}{\bf (Classes of maps)} \label{smoothoperatormaps}
Let $p,N\in \N $, with $p\leq N$, $N\geq1$, $K,K'\in\N$ with $K'\leq K$
and $ m \geq 0 $. 

\begin{itemize}
\item[(i)] {\bf $p$-homogeneous maps.} We denote by $\widetilde{\mathcal{M}}^{m}_{p}$
 the space of $(p+1)$-linear maps $M$ 
 from the space $(\dot{H}^{\infty}(\T;\C^{2}))^{p}\times \dot{H}^{\infty}(\T;\C)$ to 
 the space $\dot{H}^{\infty}(\T;\C)$ which are symmetric
 in $(U_{1},\ldots,U_{p})$, of the form
$ (U_{1},\ldots,U_{p+1})\to M(U_1,\ldots, U_p)U_{p+1} $
and that satisfy the following. There is $C>0$ such that 
 \[
 \|\Pi_{n_0}M(\Pi_{\vec{n}}\mathcal{U})\Pi_{n_{p+1}}U_{p+1}\|_{L^{2}}\leq
 C( n_0 +  n_1 +\cdots+ n_{p+1})^{m} \prod_{j=1}^{p+1}\|\Pi_{n_{j}}U_{j}\|_{L^{2}}
\] 
  for any 
 $\mathcal{U}=(U_1,\ldots,U_{p})\in (\dot{H}^{\infty}(\T;\C^{2}))^{p}$, any 
 $U_{p+1}\in \dot{H}^{\infty}(\T;\C)$,
 any $\vec{n}=(n_1,\ldots,n_p)\in (\N^*)^{p}$, any $n_0,n_{p+1}\in \N^*$.
 Moreover the properties \eqref{omoresti2}-\eqref{def:R-trin} hold.

 \item[(ii)]{\bf Non-homogeneous maps.}
  We denote by  $\mathcal{M}^{m}_{K,K',N}[r]$ 
  the space of maps $(V,{t},{U})\mapsto M(V;{t}) U $ defined on 
 $B^K_{s_0}(I;r)\times {I\times}C^K_{*}(I,\dot{H}^{s_0}(\T,\C))$ 
  which are linear in the variable $ U $ and such that the following holds true. 
  For any $s\geq s_0$ there exist a constant $C>0$ and 
  $r(s)\in]0,r[$ such that for any 
  $V\in B^K_{s_0}(I;r)\cap C^K_{*}(I,\dot{H}^{s}(\T,\C^2))$, 
  any $ U \in C^K_{*}(I,\dot{H}^{s}(\T,\C))$, any $0\leq k\leq K-K'$ and any $t\in I$, we have 
  $\|{\partial_t^k\left(M(V;{t})U\right)(t,\cdot)}\|_{\dot{H}^{s-k-m}}$ is bounded by the right hand side  of \eqref{piove}.
\item[(iii)]{\bf Maps.}
We denote by $\Sigma\mathcal{M}^{m}_{K,K',p}[r,N]$
the space of maps $(V,t,U)\to M(V,t)U$
that may be written as 
\[
M(V;t)U=\sum_{q=p}^{N-1}M_{q}(V,\ldots,V)U+M_{N}(V;t)U
\]
for some $M_{q} $ in $ \widetilde{\mathcal{M}}^{m}_{q}$, $q=p,\ldots, N-1$ and $M_{N}$  in  
$\mathcal{M}^{m}_{K,K',N}[r]$.
Finally we set $\widetilde{\mathcal{M}}_{p}:=\cup_{m\geq0}\widetilde{\mathcal{M}}_{p}^{m}$,
$\mathcal{M}_{K,K',p}[r]:=\cup_{m\geq0}\mathcal{M}^{m}_{K,K',p}[r]$ and 
$\Sigma\mathcal{M}_{K,K',p}[r,N]:=\cup_{m\geq0}\Sigma\mathcal{M}^{m}_{K,K',p}[r]$.
\end{itemize}

We denote by  $\Sigma\mathcal{M}_{K,K',p}^{m}[r,N]\otimes\mathcal{M}_2(\C)$
 the space of $2\times 2$ matrices whose entries are maps in
the class $\Sigma\mathcal{M}^{m}_{K,K',p}[r,N]$.
We also set $\Sigma\mathcal{M}_{K,K',p}[r,N]\otimes\mathcal{M}_2(\C)=\cup_{m\in \R}
\Sigma\mathcal{M}_{K,K',p}^{m}[r,N]\otimes\mathcal{M}_2(\C)$.
\end{definition}



\vspace{0.3em}
\noindent
$\bullet$
If  $M$ is in $\widetilde{\mathcal{M}}^{m}_{p}$, $p\geq N$, then
$(V,U)\to M(V,\ldots,V)U$ is in $\mathcal{M}^{m}_{K,0,N}[r]$.

\vspace{0.3em}
\noindent
$\bullet$
If $a\in \Sigma\Gamma^{m}_{K,K',p}[r,N]$ for $p\geq1$, then {the map $(V,U)\to$}
$ \opbw(a(V;t,\cdot))U$ 
is in $\Sigma\mathcal{M}^{m'}_{K,K',p}[r,N]$ for some $m'\geq m$.

\vspace{0.3em}
\noindent
$\bullet$
 Any  
$R\in \Sigma\mathcal{R}^{-\rho}_{K,K',p}[r,N]$
defines an element of $\Sigma\mathcal{M}^{m}_{K,K',p}[r,N]$ for some $m\geq0$.

\vspace{0.3em}
\noindent
$\bullet$
 If $ M \in \Sigma\mathcal{M}_{K,K'_1,p}[r,N]$ and $\tilde{M}\in \Sigma\mathcal{M}_{K,K'_2,1}[r,N-p]$,
then  
$(V,t,U)\to M(V+\tilde{M}(V;t)V ;t)[U]$ is in $\Sigma\mathcal{M}_{K,K_1'+K'_2,p}[r,N]$. 

\vspace{0.3em}
\noindent
$\bullet$
If $ M \in \Sigma\mathcal{M}^{m}_{K,K',p}[r,N]$ and $\tilde{M}\in\Sigma\mathcal{M}^{m'}_{K,K',q}[r,N]$,
then 
$M(U;t)\circ \tilde{M}(U;t)$ is in $\Sigma\mathcal{M}^{m+m'}_{K,K',p+q}[r,N]$.

Note that, given $ M_1\in \widetilde{\mathcal{M}}_1 $,   
 the property \eqref{def:R-trin} implies that 
\begin{equation}\label{quadraticTerms200}
M_1 (U)U = \frac{1}{2\pi}\!\!\sum_{\substack{n_1, n_2 \in \Z\setminus\{0\}, \\ \s=\pm}} \!\!(M_{2})^{\s\s}_{n_1, n_2} u^{\s}_{n_1} u^{\s}_{n_2} 
e^{\ii \s ( n_1 + n_2) x} 
+\frac{1}{2\pi}\!\!\sum_{n_1, n_2 \in \Z\setminus\{0\}}\!\! (M_{2})^{+-}_{n_1, n_2} u_{n_1} \ov{{u}_{n_2}} e^{\ii ( n_1 - n_2) x }
\end{equation}
for some coefficients $(M_2)^{\s\s'}_{n_1,n_2}\in \C$ with $\s,\s'=\pm$ and $n_1,n_2\in \Z\backslash\{0\}$.
\\[1mm]
{ \bf Composition theorems.}
Let 
$$ 
\varsigma (D_{x},D_{\x},D_{y},D_{\eta}) := D_{\x}D_{y}-D_{x}D_{\eta} 
$$
where $D_{x}:=\frac{1}{\ii}\pa_{x}$ and $D_{\x},D_{y},D_{\eta}$ are similarly defined. 

\begin{definition}{\bf (Asymptotic expansion of composition symbol)}\label{asyexpexp}
Let $K'\leq K, \rho,p,q$ be in $\N$, $m,m'\in \R$, $r>0$.
Consider $a\in \Sigma\Gamma_{K,K',p}^{m}[r,N]$ and $b\in \Sigma \Gamma^{m'}_{K,K',q}[r,N]$. For $U$ in $B_{\s}^{K}(I;r)$ 
we define, for $\rho< \s- s_0$, the symbol
\begin{equation}\label{espansione2}
(a\#_{\rho} b)(U;t,x,\x):=\sum_{k=0}^{\rho}\frac{1}{k!}
\left(
\frac{\ii}{2} \varsigma(D_{x},D_{\x},D_{y},D_{\eta})\right)^{k}
\Big[a(U;t,x,\x)b(U;t,y,\eta)\Big]_{|_{\substack{x=y, \x=\eta}}}
\end{equation}
modulo symbols in $ \Sigma\Gamma^{m+m' - \rho }_{K,K',p+q}[r,N] $.
\end{definition}

\noindent
$ \bullet $
By  \eqref{prodottodisimboli2} the symbol $ a\#_{\rho} b $ belongs  to $\Sigma\Gamma^{m+m'}_{K,K',p+q}[r,N]$.

\noindent
$ \bullet $
We have  the expansion  $ a\#_{\rho}b =ab+\frac{1}{2 \ii }\{a,b\} + \cdots $, 
up to a symbol in $\Sigma\Gamma^{m+m'-2}_{K,K',p+q}[r,N]$,     
where 
$$
\{a,b\}:=\pa_{\x}a\pa_{x}b-\pa_{x}a\pa_{\x}b
$$ 
denotes the Poisson bracket. 

\noindent
$\bullet$ Note that the terms of even (resp. odd) rank in the asymptotic expansion \eqref{espansione2}
in the Weyl quantization are symmetric (resp. antisymmetric) in $(a,b)$.
Consequently the terms of even rank vanish in the symbol of the commutator $[\opbw(a),\opbw(b)]$.

\begin{proposition}{\bf (Composition of Bony-Weyl operators)} \label{teoremadicomposizione}
Let $K'\leq K, \rho,p,q$ be in $\N$, $m,m'\in \R$, $r>0$. 
Consider 
$a\in \Sigma {\Gamma}^{m}_{K,K',p}[r,N] $ and $b\in \Sigma {\Gamma}^{m'}_{K,K',q}[r, N]$. 
Then
$$
R(U):=\opbw(a(U;t,x,\x))\circ\opbw(b(U;t,x,\x))-\opbw\big(
(a\#_{\rho} b)(U;t,x,\x)
\big)
$$
is a non-homogeneous smoothing remainder in $ \Sigma {\mathcal{R}}^{-\rho+m+m'}_{K,K',p+q}[r,N]$.
\end{proposition}

\begin{proof}
See Propositions 3.12 and 3.15 in \cite{BD}.
Let us justify that the homogeneous components of 
$R(U)$ satisfy the translation invariance property \eqref{def:R-trin}. 
The homogeneous components of the 
symbols $ a $ and $ b $ (that for simplicity we still denote by $ a, b $) satisfy \eqref{def:tr-in}. 
Then, by \eqref{tautheta}, the homogeneous components of the composed operator 
$ \opbw(a(U; \cdot ,\x))\circ\opbw(b(U; \cdot ,\x))$ satisfies \eqref{tautheta} as well.
In addition also the symbol 
 $ a\#_{\rho}b $ defined in \eqref{espansione2} satisfies \eqref{def:tr-in}, and therefore
 $\opbw((a\#_{\rho}b)(U; \cdot ,\x))$
satisfies  \eqref{tautheta}. Thus  the homogeneous components of  $R(U) $ 
satisfy \eqref{def:R-trin} by difference.
\end{proof}

\noindent
$ \bullet $ As proved in the remark after the proof of Proposition 3.12 in \cite{BD}, 
the remainder obtained by the composition of paradifferential operators in Proposition 
\ref{teoremadicomposizione} has actually better estimates than
\eqref{piove}, i.e. it is bounded from $ \dot H^{s} $  to $ \dot H^{s + \rho - (m + m') } $ for {\it any} $ s $,
with operator norm bounded by $ \| U \|_{K, s_0}^{p+q} $.

\begin{proposition}{\bf (Compositions)} \label{composizioniTOTALI}
Let $m,m',m''\in \R$, $K,K',N,p_1,p_2,p_{3},\rho\in \N$ with $K'\leq K$, $p_1+p_{2}<N$, $\rho\geq0$ and $r>0$.
Let $a\in \Sigma\Gamma^{m}_{K,K',p_1}[r,N]$, 
$R\in\Sigma\mathcal{R}^{-\rho}_{K,K',p_{2}}[r,N]$ and $M\in \Sigma\mathcal{M}^{m''}_{K,K',p_{3}}[r,N]$.
Then

\begin{itemize}
\item[(i)]
$ R(U;t)\circ \opbw(a(U;t,x,\x)) $, $  \opbw(a(U;t,x,\x))\circ R(U;t) $
are   in $\Sigma\mathcal{R}^{-\rho+m}_{K,K',p_1+p_{2}}[r,N]$.

\item[(ii)]
$ R(U;t)\circ M(U;t) $ and $ M(U;t)\circ R(U;t) $ 
are smoothing operators in $\Sigma\mathcal{R}^{-\rho+m''}_{K,K',p_2+p_{3}}[r,N]$.

\item[(iii)]
If  $R_{2} \in \widetilde{\mathcal{R}}_{p_2}^{-\rho}$  
then $R_{2}(U, \ldots, U, M(U;t)U)$ belongs to $\Sigma\mathcal{R}^{-\rho+m''}_{K,K',p_{2}+p_3}[r,N]$.

\item[(iv)] Let $c$ be in $\widetilde{\Gamma}_p^{m}$, $p\in \N$. 
Then 
$$
U \rightarrow c_{M}(U;t,x,\x):= c(U,\ldots,U,M(U;t)U;t,x,\x)
$$
is in $\Sigma\Gamma^{m}_{K,K',p+p_3}[r,N]$. If the symbol $c$ is independent of $\x$ (i.e.
$c$ is in $\widetilde{\mathcal{F}}_p$), so is the symbol $c_{M}$
(thus it is a  function in $\Sigma\mathcal{F}_{K,K',p+p_3}[r,N]$). 
Moreover if $c$ is a symbol in $\Gamma^{m}_{K,K',N}[r]$
then the symbol $c_{M}$
is in $\Gamma^{m}_{K,K',N}[r]$.

\item[(v)]
$ \opbw( c(U,\ldots,U,W;t,x,\x))_{|W=M(U;t)U}=\opbw(b(U;t,x,\x))+R(U;t) $
where  
$$
b(U;t,x,\x):= c(U,\ldots,U,M(U;t)U;t,x,\x)
$$
and $R(U; t) $ is in $\Sigma\mathcal{R}^{-\rho}_{K,K',p+p_1}[r,N]$.
\end{itemize}
\end{proposition}
\begin{proof}
See Proposition 3.16,  3.17, 3.18 in \cite{BD}.
The translation invariance properties for the composed operators and symbols 
in items (i)-(v) follow as in the proof of Proposition 
\ref{teoremadicomposizione}.
\end{proof}

\noindent
{\bf Real-to-real operators.} 
Given a linear  operator  
$ R(U) [ \cdot ]$ acting on $ \C $ (it may be a 
smoothing operator in  $\Sigma\mathcal{R}^{-\rho}_{K,K',1}$ or
a map  in $\Sigma\mathcal{M}_{K,K',1}$) 
we associate the linear  operator  defined by the relation 
\begin{equation}\label{opeBarrato}
\ov{R}(U)[V] := \ov{R(U)[\ov{V}]} \, ,   \quad \forall V \in \C \, .
\end{equation}
We say that a matrix of operators acting in $ \C^2 $ is   \emph{real-to-real}, if it has the form 
\begin{equation}\label{vinello}
R(U) =
\left(\begin{matrix} R_{1}(U) & R_{2}(U) \\
\ov{R_{2}}(U) & \ov{R_{1}}(U)
\end{matrix}
\right) \, .
\end{equation}
Note that 

\noindent
$ \bullet $
 if  $R(U)$  is a  real-to-real matrix of operators 
then, given $V=\vect{v}{\ov{v}}$, the vector  $Z:=R(U)[V]$ has the form  $Z=\vect{z}{\bar{z}}$, namely the second component is the complex conjugated of the first one.

\noindent
$ \bullet $
If a  matrix of symbols $A(U;x,\x)$, in some class $\Sigma{\Gamma}^{m}_{K,K',1}\otimes\mathcal{M}_2(\C)$,  
has the form 
\begin{equation}\label{prodotto}
A(U;x,\x) =
\left(\begin{matrix} {a}(U;x,\x) & {b}(U;x,\x)\\
{\ov{b(U;x,-\x)}} & {\ov{a(U;x,-\x)}}
\end{matrix}
\right)
\end{equation}
then the matrix of operators $ \opbw(A(U;x,\x))$ is real-to-real. 

\vspace{0.5em}
\noindent
{\bf Notation.}

\setlength{\leftmargini}{2em}
\begin{itemize}
\item 
To simplify the notation,   we will often omit the dependence on the time $ t $ from the symbols,
smoothing remainders and maps, 
writing $a(U;x,\x)$, $ R(U) $, $ M(U) $ instead of 
$a(U;t,x,\x)$, $ R (U; t)$, $ M (U; t) $.
Moreover, given a 
symbol in $\Sigma\Gamma^{m}_{K,K',p}$
we may omit to write its dependence on $U$, writing $b(x,\x)$ instead of $b(U;x,\x)$,
when this does not cause confusion.

\item Since in the rest of the paper we only need to control
expansions in degrees of homogeneity of symbols, smoothing operators and maps, 
up to cubic terms $O(u^{3})$, we fix once and for all $N=3$. 
We will omit the dependence on $r$ and $N=3$ in the class of symbols, 
writing $\Sigma\Gamma^{m}_{K,K',p}$, 
instead of  $\Sigma\Gamma^{m}_{K,K',p}[r,3]$, 
and similarly for smoothing operators and maps.

\item 
$A\lesssim_{s} B$
means $A \leq C(s) B$
 where $C(s) > 0 $ 
 is a  constant depending on $s \in \R $.

\item 
In this paper we will deal with parameters
$$
s \geq s_0 \gg K \gg \rho \gg 1 \, . 
$$
\end{itemize}

The order of regularization $\rho \gg 1 $ will be chosen large enough to control the loss of derivatives 
coming from the small divisors in the  two steps of Birkhoff normal form, see Section \ref{sec:BNF}. 
More precisely $\rho \sim  N_0$ where  $N_0$ is 
the exponent  appearing in \eqref{stima2}. 
This requires to develop para-differenatial calculus for functions $U$ in  $ \dot H^{s_0}$ with $ s_0 \gg \rho$. 
In order to transform the water waves equations \eqref{eq:113} into  a paradifferential system plus a 
$\rho$-smoothing remainder we perform several para-differential changes of variables
for solutions $U(t)$ which are $K$-times differentiable in time with  $ \pa_t^k U \in \dot H^{s-k}$, $ 0 \leq k \leq K $. 
Since each of the conjugations performed in Section \ref{diagonalizzo} consumes one time derivative, 
we need to require  $K \gg \rho$, more precisely $ K \sim 2 \rho $, see Proposition \ref{teodiagonal}.
We then require that the Sobolev exponents satisfy  $s \geq s_0 \gg K $.

\bigskip
\section{ Complex form of the water waves equations} \label{sec:3}

\subsection{Paralinearization and complex variables}\label{sec:good}

Following \cite{ABZ1,AlDe1}, 
we begin by writing the water waves system \eqref{eq:113} using the good-unknown \eqref{omega0}
\[
\omega =   \psi - \opbw{(B(\eta, \psi)) \eta}
\]
where $ B(\eta, \psi) $ is the real valued function introduced in \eqref{form-of-B}. 
The water-waves equations \eqref{eq:113}, written 
in the new coordinates 
\be\label{Alinach-good} 
\vect{\eta}{\omega} =
{\mathcal G} \vect{\eta}{\psi} := 
 \vect{\eta}{ \psi - \opbw{(B(\eta, \psi)) \eta}  } \, , 
\ee
assume the following paralinearized form derived in \cite{BD}. 

\begin{proposition}{\bf (Water-waves equations in $(\eta, \omega)$ variables)}\label{laprimapara}
Let $ I = [-T, T] $ with $ T > 0 $. 
Let  $K \in \N^* $ and $ \rho \gg 1 $. There exists $ s_0 > 0 $ such that, for any $ s \geq s_0 $, for all 
$ 0 < r \leq r_0(s) $ small enough,  
if $ (\eta, \psi) \in B^K_s (I;r) $ solves \eqref{eq:113}, then 
\begin{align}   \label{eq:1n}
\partial_t \eta & = |D| \omega + \opbw \big(- \ii V \xi - \frac{V_x}{2} \big)\eta + \opbw(b_{-1}(\eta;\cdot))\omega
+ R_1(\eta,\omega)\omega + R'_1(\eta,\omega)\eta \\
\partial_t\omega & = - \eta  + \opbw(  - \ii V \xi + \frac{V_x}{2} )\omega
 -  \opbw( \pa_t B + V B_x  )\eta   \label{eq:2n}
   + R'_2(\eta,\omega)\omega + R''_2(\eta,\omega)\eta 
\end{align}
where the functions $ V, B $ defined in \eqref{def:V}-\eqref{form-of-B} are  in $ \Sigma {\cal F}^\R_{K, 0,1}$, 
the symbol  $b_{-1}(\eta;\cdot) $ belongs to $ \Sigma\Gamma^{-1}_{K,0,1}$, 
and the smoothing  operators $R_1$, $R_{1}'$, $R_2$, $R_{2}'$ are in 
$\Sigma\mathcal{R}^{-\rho}_{K,0,1}$.  
The vector field in the right hand side of \eqref{eq:1n}-\eqref{eq:2n}
is $x$-translation invariant, i.e.  \eqref{commutes-tau} holds. 
\end{proposition}

\begin{proof}
The proof of this proposition follows from the computations in \cite{BD} in 
the absence of capillarity and specified in the case of infinite depth.   
The right hand side in \eqref{eq:1n}
is the paralinearization of the Dirichlet-Neumann operator in  \cite{BD}. 
The approach in \cite{BD} does not make use of a variational method to study the 
Dirichlet-Neumann boundary value problem as in \cite{AlM, ABZ1}, but 
used a paradifferential parametrix \`a la Boutet de
Monvel, 
introducing classes of para-Poisson operators whose symbols have a decomposition in
multilinear terms. The applications of these results to the construction of the good unknown 
and the paralinearization of the water waves 
system provide the expansions \eqref{eq:1n}-\eqref{eq:2n}.
In particular, the explicit expression of the symbols in \eqref{eq:1n} 
follows by developing the computations in Proposition 7.5 and Chapter 8.2 in \cite{BD}.
Note that this expansion agrees 
with the paralinearization of  the Dirichlet-Neumann operator in Theorem 2.12 in \cite{AlM},
in the case of dimension $1$ and using the Bony-Weyl quantization. 
The equation \eqref{eq:2n} follows by 
developing the computations in Proposition 7.6 in \cite{BD}.

The Dirichlet-Neumann operator satisfies the translation invariance property 
$G( \tau_\theta \eta)[\tau_\theta \psi] = \tau_\theta G(\eta)[\psi] $. 
Hence the functions $ V, B $ defined in \eqref{def:V}-\eqref{form-of-B} 
satisfy the $ x $-invariance property \eqref{def:tr-in} as well, 
and so do $ V_x, B_x, \pa_t B $. 
The symbol $b_{-1}(\eta;\cdot) $ satisfies the $ x $-invariance property \eqref{def:tr-in} 
checking  the construction in \cite{BD} (its $x$-dependence enters only through $ (\eta,\omega)$).
Moreover, since $  B (\tau_\teta\eta, \tau_\teta\psi) (x)  = B (\eta, \psi) (x+ \theta) $,  we get 
$$ 
 \opbw{( B(\tau_\teta \eta, \tau_\teta \psi)) } [\tau_\teta\eta] (x)  = \opbw{( B(\eta, \psi)) } [\eta]  (x + \theta ) \, , \quad \forall \theta \in \R \, , 
$$ 
and therefore the good-unknown transformation  $ {\mathcal G} $ defined in \eqref{Alinach-good} satisfies 
 $ {\cal G} \circ \tau_\theta = \tau_\theta \circ {\cal G } $, 
where  $ \tau_\theta  $ is the translation operator
in \eqref{tau-theta}.
This implies that the  whole vector field in the right hand side of \eqref{eq:1n}-\eqref{eq:2n} satisfies the 
$ x $-invariance property  and therefore the smoothing remainders satisfy \eqref{def:R-trin} by difference. 
\end{proof}

In Subsection \ref{expHomogen} we will provide 
explicit expansions for the symbols of non-negative order in \eqref{eq:1n}-\eqref{eq:2n}
in linear and quadratic degrees of homogeneity.

\begin{remark} {\bf (Expansion of the Dirichlet-Neumann operator)}
\label{rem:DN}
\begin{itemize}
\item[(i)] Substituting 
\eqref{Alinach-good} in the right hand side of \eqref{eq:1n}, which is equal to $G(\eta)\psi$, 
we have, using  the remarks under Definition \ref{smoothoperatormaps} 
and the fact that $B(\eta,\psi)\in\Sigma {\cal F}^\R_{K, 0,1}$ is linear in $\psi$, 
that $G(\eta)-|D|$ is a map in  $\Sigma\mathcal{M}_{K,0,1}$
and 
\begin{equation}\label{DNexp}
G(\eta)\psi=|D|\psi+\widetilde{M}_1(\eta)\psi+\widetilde{M}_2(\eta)\psi
+\widetilde{M}_{\geq 3}(\eta)\psi  
\end{equation}
for some maps $\widetilde{M_1}\in \widetilde{\mathcal{M}}_1$, 
$\widetilde{M_2}\in \widetilde{\mathcal{M}}_2$ and $\widetilde{M}_{\geq3}\in\mathcal{M}_{K,0,3}$. 

\item[(ii)] The Dirichlet-Neumann operator admits 
a Taylor expansion (see e.g.  formula (2.5) of \cite{CS}) 
of the form
\begin{equation}\label{DNexp2}
 G(\eta)\psi = |D| \psi+ G_1(\eta) \psi + G_2 (\eta)\psi + G_{\geq3}(\eta)\psi 
 \end{equation}
where, $ D := \frac{1}{\ii} \pa_x $,  
\begin{equation}\label{DNexp3}
\begin{aligned}
G_1(\eta)&:= - \pa_x \eta \pa_x  - |D|  \eta |D| \\
G_2 (\eta) & :=  - \frac12 \Big(  D^2 \eta^2 |D| + |D| \eta^2 D^2 - 2 |D| \eta |D| \eta |D| \Big)  
\end{aligned}
\end{equation}
where $G_{\geq3}$ collects all the terms with homogeneity in $\eta$ greater than $2$.
The notation above $|D|\eta|D|$, resp. $|D| \eta |D| \eta |D| $, means  the composition operator 
$|D| \circ \eta \circ |D|$, resp. $ |D| \circ \eta \circ |D| \circ \eta \circ |D|$, 
of the Fourier multiplier $ |D|$ and the multiplication operator for the function $ \eta $.
We then see that the quadratic and cubic components of the expansions
\eqref{DNexp2} and \eqref{DNexp} coincide, 
namely $G_1=\widetilde{M}_1$ and $G_2=\widetilde{M}_2$.
It follows that $G_{\geq3}$ is in $\mathcal{M}_{K,0,3}$.

\item[(iii)] Performing the paralinearization of $G_1(\eta)\psi$ and $G_2(\eta)\psi$ in  \eqref{DNexp3},
one obtains the expansion 
\[
G(\eta) \psi = |D| \omega  
+ \opbw{\big(  - \ii {\mathtt V}_{\leq 2}   \xi  
- \frac12  ({\mathtt V}_{\leq 2})_x  \big) } \eta
+ {\rm \ quartic \ terms}
\]  
up to smoothing operators, 
where  $ {\mathtt V}_{\leq 2} = \psi_x - \eta_x (|D| \psi) $ contains the linear and quadratic components of the function $ V $
in \eqref{def:V}. 
This formula agrees with \eqref{eq:1n} showing 
that the symbol $b_{-1}$ is zero (at least) at cubic degree of homogeneity.
\end{itemize}

\end{remark}

We now write the equations \eqref{eq:1n}-\eqref{eq:2n} 
in terms of the complex variable $u$ defined by, see \eqref{u0},
\begin{align}
\label{u0'} 
u := \frac{1}{\sqrt{2}}|D|^{-\frac{1}{4}}\eta+\frac{\ii}{\sqrt{2}}|D|^{\frac{1}{4}}\omega \, .
\end{align}

\begin{proposition} {\bf (Water-waves equations in complex variables)}\label{WWcomplexVar}
Let $ K \in \N^* $ and $ \rho \gg 1  $. 
There exists $ s_0 > 0 $ such that, for any $ s \geq s_0 $, for all $ 0 < r \leq r_0(s) $ small enough, 
 if  $ (\eta,\omega)$ solves \eqref{eq:1n}-\eqref{eq:2n} and  
  $U:=\vect{u}{\bar{u}}$ with $ u $ defined in \eqref{u0'}  belongs to $B_{s}^{K}(I;r)$,
then $U$ solves
\begin{equation}\label{eq:415}
\begin{aligned}
\pa_t U=\opbw\big(
\ii A_1(U;x)\x+\ii A_{1/2}(U;x)|\x|^{\frac{1}{2}}+A_0(U;x)+A_{-1}(U;x,\x)
\big)U+R(U)U
\end{aligned}
\end{equation}
where 
\begin{align}\label{matriciinizio}
& A_{1}(U;x) := 
\left(\begin{matrix} -V(U;x) & 0 \\
0 & -V(U;x) \end{matrix}\right) 
\\
& A_{1/2}(U;x)  := 
\left(\begin{matrix}-(1+a(U;x)) & - a(U;x) \\
 a(U;x) & 1+a(U;x) \end{matrix}\right) \, , 
 \qquad a:= 
 \frac12 (  \pa_t B + V B_x ) \, ,  \label{function:a}
 \\
&  A_{0}(U;x) :=  -\frac{1}{4}\left(\begin{matrix}0 & 1\\1&0\end{matrix}\right)V_x (U; x) \, ,
 \label{function:c}
\end{align} 
 $A_{-1} $ is a matrix of symbols in $ \Sigma \Gamma^{-1}_{K,1,1} \otimes {\cal M}_2(\mathbb{C})$, 
 and $R(U) $ is a matrix of smoothing operators belonging to $ \Sigma\mathcal{R}^{-\rho}_{K,1,1}\otimes\mathcal{M}_2(\C) $.
 The vector field in the right hand side of \eqref{eq:415}
is $x$-\emph{invariant}
   and it is \emph{real-to-real} according to  \eqref{vinello}, i.e. the second equation for $ \bar u $ 
   is the complex conjugated of the first equation for $ u $. 
 \end{proposition}

\begin{proof}
We first rewrite  \eqref{eq:1n}-\eqref{eq:2n} as the system
\be\label{sis1}
\begin{aligned}
\partial_t \vect{\eta}{\omega} 
& = \opbw{\left(   \sm{- \ii V \xi - \frac{V_x}{2} }{| \xi | + b_{-1}}{ - (1 +  a_0 ) }{- \ii V \xi + \frac{V_x}{2} }   \right)} \vect{\eta}{\omega} + R(\eta, \omega)  \vect{\eta}{\omega}  , \qquad 
R\in \Sigma\mathcal{R}^{-\rho}_{K,0,1}\otimes\mathcal{M}_2(\C) \, , 
\end{aligned}
\ee
where the function $ a_0 :=  \pa_t B + V B_x $ is in $ \Sigma\mathcal{F}^{\R}_{K,1,1}$.  
We now symmetrize \eqref{sis1} at the highest order, applying  the change of variable
\be\label{FM}
\vect{\eta}{\omega} :=  \sm{| D|^{1/4}}{0}{0}{|D|^{-1/4}} \vect{\tilde \eta}{\tilde \omega}  \, . 
\ee
The conjugated system is, by 
Propositions \ref{teoremadicomposizione} and \ref{composizioniTOTALI}, 
\begin{align}\label{IGU}
\partial_t \vect{\tilde \eta}{ \tilde \omega} 
& = 
\opbw{\left(  
\sm{|  \xi |^{-1/4}}{0}{0}{|  \xi  |^{1/4}}  \#_{\rho} 
\sm{- \ii V  \xi - \frac{V_x}{2} }{| \xi | + b_{-1}}{ - (1 +  a_0 ) }{- \ii V \xi + \frac{V_x}{2}}  
\#_\rho 
\sm{| \xi  |^{1/4}}{0}{0}{| \xi |^{-1/4}} \right)}    
\vect{\tilde \eta}{ \tilde \omega} 
+ R(\tilde \eta, \tilde \omega)  \vect{\tilde \eta}{\tilde \omega}  
\end{align}
for a new smoothing remainder  $ R $ in $ \Sigma\mathcal{R}^{-\rho + 1 }_{K,1,1}\otimes\mathcal{M}_2(\C) $. 
Recalling \eqref{espansione2}  
we expand in decreasing orders the  symbols in \eqref{IGU}.
\\[1mm]
{\sc Diagonal symbols.}
Up to a symbol in $\Sigma\Gamma^{-1}_{K,0,1}$ we have (using Proposition \ref{teoremadicomposizione} and formula \eqref{espansione2})
\be \label{rf2}
\begin{aligned}
& | \xi |^{-1/4 }  \#_\rho  (- \ii V \xi - \frac{V_x}{2} )  \#_\rho |\xi  |^{1/4} 
 =    - \ii V \xi - \frac{V_x}{4} \, , \\
&   | \xi |^{1/4 }  \#_\rho  (- \ii V \xi + \frac{V_x}{2} )  \#_\rho |\xi  |^{- 1/4} 
 =   - \ii V \xi + \frac{V_x}{4}     \, . 
\end{aligned}
\ee
\noindent
{\sc Off-diagonal symbols.}
Up to a symbol in  $\Sigma\Gamma^{-3/2}_{K,0,1}$ we get (using Proposition \ref{teoremadicomposizione} and formula \eqref{espansione2})
\begin{align}\label{rf3}
 | \xi  |^{-1/4} \#_\rho ( | \xi | + b_{-1}) \#_\rho |\xi  |^{-1/4} & = | \xi |^{1/2} 
\end{align}
(recall that $  b_{-1} $ is  in $ \Sigma\Gamma^{-1}_{K,0,1} $)
and, up to a symbol in  $\Sigma\Gamma^{-3/2}_{K,1,1}$, we have 
\begin{align}
- | \xi  |^{1/4} \#_\rho (1 +  a_0 ) \#_\rho | \xi  |^{1/4} 
=  -  (1 +  a_0 ) | \xi  |^{1/2}  . \label{rf4}
\end{align}
The expansion  \eqref{rf2},  \eqref{rf3}, \eqref{rf4} imply that 
the system  \eqref{IGU} has the form 
\be\label{dit2-new}
\partial_t \vect{\tilde \eta}{ \tilde \omega} 
 = 
\opbw{\left(  
\sm{- \ii V \xi - \frac{V_x}{4} }{| \xi |^{1/2} }{-  (1 + a_0 ) | \xi  |^{1/2} }{- \ii V \xi + \frac{V_x}{4} } 
+ A_{-1}  \right)}\vect{\tilde \eta}{ \tilde \omega}  + R( \tilde \eta, \tilde \omega )\vect{\tilde \eta}{ \tilde \omega} 
\ee
where  $ A_{-1} $ is a matrix of symbols in $ \Sigma\Gamma^{-1}_{K,1,1} \otimes {\mathcal M}_2(\C)$
and $ R  $ is in $ \Sigma\mathcal{R}^{-\rho + 1 }_{K,1,1}\otimes\mathcal{M}_2(\C) $. 

Finally  we write  \eqref{dit2-new} in the complex variable \eqref{u0'}, i.e. recalling \eqref{FM},  
\[
\vect{u}{\bar u}   := \frac{1}{\sqrt{2}} \sm{1}{\ii}{1}{-\ii}  \vect{\tilde \eta}{\tilde \omega} \, ,  \quad
{\rm with \ inverse} 
\quad \vect{\tilde \eta}{\tilde \omega} := \frac{1}{\sqrt{2}} \sm{1}{1}{-\ii}{\ii} \vect{u}{\bar u} \, , 
\]
and we deduce \eqref{eq:415} with matrices as in 
\eqref{matriciinizio}, \eqref{function:a}, \eqref{function:c} and a new matrix of symbols 
 $ A_{-1} $ in $ \Sigma\Gamma^{-1}_{K,1,1} \otimes {\mathcal M}_2(\C)$ and a new smoothing 
  operator   $R(U) $ in $ \Sigma\mathcal{R}^{-\rho}_{K,1,1}\otimes\mathcal{M}_2(\C) $, renaming 
  $ \rho - 1 $ as $ \rho $.
   Finally, since the Fourier multiplier transformation \eqref{u0'} trivially 
   commutes with the translation operators $ \tau_\teta $, 
    the water waves vector field in  \eqref{eq:415} is $ x$-invariant 
    as the water waves 
   vector field \eqref{eq:1n}-\eqref{eq:2n}. 
\end{proof}
In some instances we will write the water waves system \eqref{eq:415} as 
 \begin{equation}\label{eq:415tris}
 \pa_t U  =-\ii \Omega U+  {\bf M}(U)[U], \quad 
 \Omega:=\sm{|D|^{\frac{1}{2}}}{0}{0}{-|D|^{\frac{1}{2}}} \, , 
 \end{equation}
where 
$  {\bf M}(U) $ is a real-to-real matrix of maps in  $\Sigma\mathcal{M}^{m_1}_{K,1,1} \otimes {\mathcal M}_2 (\C)$
for some $m_1>0$,
see the Remarks after Definition \ref{smoothoperatormaps}.
We will also write system \eqref{eq:415tris} in Fourier basis as 
\begin{equation}\label{eq:415bis}
\dot u_n = - \ii \omega_nu_n+ \ii ( F_2 (U) + F_{\geq 3}(U) )_{n} \, , \quad n \in \Z \setminus \{0\} \, , 
\end{equation}
where $\omega_n = \sqrt{|n|}$ and 
$F_2 (U) = M_1 (U)[U]$  is the quadratic component of the water-waves vector field  
and $ F_{\geq 3 }(U)$  collects all the cubic terms
(the second equation of \eqref{eq:415tris} for $ \bar u  $ is just the complex conjugated of the  one for $ u $).
Using the $x$-invariance property, the vector field $ F_2 (U) $ can be expanded as
\begin{equation}\label{quadraticTerms2}
F_2 (U)  = \sum_{n_1, n_2 \in \Z\setminus\{0\},\s=\pm} (F_{2})^{\s\s}_{n_1, n_2} u^{\s}_{n_1} u^{\s}_{n_2} 
\frac{e^{\ii \s ( n_1 + n_2) x }}{2\pi}
+\sum_{n_1, n_2 \in \Z\setminus\{0\},} (F_{2})^{+-}_{n_1, n_2} u_{n_1} \ov{{u}_{n_2}} \frac{e^{\ii ( n_1 - n_2) x }}{2\pi} 
\end{equation}
with coefficients $ (F_{2})^{\s \s'}_{n_1, n_2} $ in $ \C $. 
We provide the explicit expression of 
$ \ii F_2 (U) $ in \eqref{goodOmega3}.

\subsection{Homogeneity expansions}\label{expHomogen}

By the expansion of the Dirichlet-Neumann operator in  Remark \ref{rem:DN}, we get
the quadratic approximation of the water waves equations \eqref{eq:113}, 
\begin{equation} \label{WW12}
\begin{cases}
   \partial_t \eta = 
   \dps |D| \psi - \pa_x (\eta \pa_x \psi ) - |D| ( \eta |D|\psi) 
   \, , 
   \cr
   \partial_t \psi = \dps - \eta  -\frac{1}{2} \psi_x^2 +   
    \frac{1}{2} (|D|\psi)^2 \,,
    \end{cases}
\end{equation}
up to functions in  $\mathcal{F}^{\R}_{K,1,3}$.
In this section, using this expansion, 
we compute explicitly the quadratic vector field $ \ii F_2 (U) $ in \eqref{eq:415bis}, and  
the homogeneous expansions up to cubic terms of the functions $ V $ and $ a $ 
appearing in \eqref{matriciinizio}-\eqref{function:c}. We write 
\begin{align}
V&=\mathtt{V}_1+\mathtt{V}_2+\mathtt{V}_{\geq3} \, ,\qquad \mathtt{V}_j\in \widetilde{\mathcal{F}}^{\R}_{j}, \;j=1,2, \;\; 
\mathtt{V}_{\geq3}\in  \mathcal{F}^{\R}_{K,0,3} \, ,\label{busy}\\ 
a&=\mathtt{a}_1+\mathtt{a}_2+\mathtt{a}_{\geq3} \, \qquad  \ \mathtt{a}_j\in \widetilde{\mathcal{F}}^{\R}_{j}, \;j=1,2, \;\; 
\mathtt{a}_{\geq3}\in \mathcal{F}^{\R}_{K,1,3} \label{busy2}  \, . 
\end{align}
In the following it is useful to note that the relation \eqref{u0'} has inverse
\begin{equation}\label{varfin1}
\eta := \frac{1}{\sqrt{2}} |D|^{\frac14}( u + \bar u ) \, , \quad 
\omega = \frac{1}{\ii \sqrt{2} } |D|^{- \frac14}(  u - \bar u ) \, . 
\end{equation} 
We have the following Lemma.

\begin{lemma}{\bf (Expansion of $ V $)}
The function $ V $ defined in \eqref{def:V} admits the expansion 
\be\label{expV}
V = \omega_x + \pa_x \big( \opbw (|D| \omega) \eta \big)  - (|D| \omega) \eta_x  + \mathtt{V}_{\geq 3}
\ee
where  $ \mathtt{V}_{\geq 3} $ is a function in $ \mathcal{F}^{\R}_{K,0,3} $.
Thus, in the complex variable $u$  in \eqref{u0'},\eqref{varfin1},  we have 
\begin{align}
\mathtt{V}_1 & = \frac{1}{\ii \sqrt{2}} \pa_x |D|^{- \frac14} (u - \bar u )  \label{coeV1}\\
 \mathtt{V}_2 & = \frac{1}{ 2\ii } \pa_x  \Big( \opbw \big( |D|^{\frac34} (u- \bar u) \big)  \big[ |D|^{\frac14} (u + \bar u) \big] \Big) 
 - \frac{1}{ 2\ii }  \big( |D|^{\frac34} (u- \bar u) \big) \big( \pa_x |D|^{\frac14} (u+  \bar u) \big) \, .  \label{coeV2}
\end{align}
\end{lemma}
\begin{proof}
By \eqref{form-of-B} and using the expansion \eqref{DNexp2}, we deduce  $  B = |D| \psi  $ up to 
a quadratic function in 
$\mathcal{F}^{\R}_{K,0,2}$. As a consequence, 
by \eqref{def:V} and \eqref{Alinach-good}, we have
$$
V = 	\psi_x - B \eta_x = ( \omega + \opbw (B) \eta )_x - B \eta_x  = 
\omega_x + \pa_x \big( \opbw (|D| \psi) \eta \big) - (|D| \psi) \eta_x  
$$
up to a function in   $\mathcal{F}^{\R}_{K,0,3}$.
Since $ \psi = \omega $ plus a  quadratic function in  $\mathcal{F}^{\R}_{K,0,2}$ (see \eqref{Alinach-good}) we get \eqref{expV}. 
\end{proof}

\begin{lemma}\label{patB}
{\bf (Expansion of $ \pa_t B $)} Let $ B $ the function defined in \eqref{form-of-B}. Then  
\be\label{exppatB}
\pa_t B = - |D| \eta - \eta | D|^2 \eta  + |D| ( \eta |D| \eta  ) + 
|D| \Big( - \frac{1}{2} \omega_x^2 - \frac12 (|D| \omega)^2  \Big)   +
 (|D| \omega)(|D|^2  \omega)
\ee
plus a cubic function in  $ \mathcal{F}^{\R}_{K,1,3} $. 
\end{lemma}

\begin{proof}
Recalling \eqref{form-of-B}, and using \eqref{DNexp2}, we have to compute the expansion of
\be \label{pez1}
\pa_t B  
= \frac{\pa_t(G(\eta) \psi + \eta_x \psi_x )}{1+ \eta_x^2} - \frac{(G(\eta) \psi + \eta_x \psi_x) 2 \eta_x (\eta_t)_x }{(1+ 
\eta_x^2)^2}
 = \pa_t(G(\eta) \psi ) + (\pa_t \eta)_x \psi_x  + \eta_x (\pa_t \psi)_x  
\ee
plus  a cubic function in $\mathcal{F}_{K,1,3}^{\R}$. By \eqref{WW12}
the second plus the third  terms in \eqref{pez1} have the expansion 
\begin{align}
   (\pa_t \eta)_x \psi_x  + \eta_x (\pa_t \psi)_x  = ( |D| \psi_x ) \psi_x  
  -  \eta_x^2 
  \label{pez2} \, ,
\end{align}
up to a function in $\mathcal{F}_{K,1,3}^{\R}$.
For the  first  term in \eqref{pez1} we  use the ``shape derivative''  formula 
(see e.g. \cite{LannesLivre})
\begin{equation} \label{formula shape der}
G'(\eta) [\hat {\eta} ] \psi 
= \lim_{\epsilon \rightarrow 0} \frac{1}{\epsilon} \{ G (\eta+\epsilon \hat \eta ) \psi - G(\eta) \psi \}
= - G(\eta) (B \hat \eta ) -\partial_x (V \hat \eta ) 
\end{equation}
where $ V = \psi_x - B \eta_x $ is defined in \eqref{def:V}.
Then, up to functions in $\mathcal{F}_{K,1,3}^{\R}$,  we have
\be\label{pezz1}
\begin{aligned}
\pa_t(G(\eta) \psi ) & 
=  G'(\eta)[ \eta_t]\psi + G(\eta) (\pa_t \psi )   
 \stackrel{\eqref{formula shape der}} = - G(\eta) (B  \eta_t ) -\partial_x (V \eta_t )+ G(\eta) (\pa_t \psi )  \\
 & \stackrel{\eqref{def:V}, \eqref{form-of-B}, \eqref{WW12}} = - |D|  ((|D| \psi)^2 ) -\partial_x (\psi_x (|D| \psi) ) \\ 
 & \qquad \qquad + 
 |D| \big(- \eta  -\frac{1}{2} \psi_x^2 +    \frac{1}{2} (|D|\psi)^2  \big) 
 + \pa_x (\eta \eta_x ) + |D| ( \eta |D| \eta)  \, . 
\end{aligned}
\ee
Finally, by \eqref{pez1}, \eqref{pez2}, \eqref{pezz1}, we obtain, after simplification, 
\begin{equation}\label{francia}
\pa_t B   
 = - |D|\eta - \frac{1}{2}  |D|  ((|D| \psi)^2 )  -   
\frac{1}{2} |D| \psi_x^2 +  |D| ( \eta |D| \eta)  + \eta \eta_{xx}  -  \psi_{xx} |D| \psi  
\end{equation}
plus a cubic function $\mathcal{F}_{K,1,3}^{\R}$.
Since $ \pa_{xx} = - |D|^2 $ and  $ \psi = \omega$ plus a quadratic function in  $\mathcal{F}_{K,1,2}^{\R}$, we have that 
 \eqref{francia}
implies \eqref{exppatB}.  
\end{proof}

We now expand the function $ a = \frac12 (  \pa_t B + V B_x ) $ which appears in  \eqref{function:a}. 

\begin{lemma}{\bf (Expansion of $ a $)} \label{lem:expa}
We have 
\begin{align*}
a 
& = - \frac12 |D| \eta  
- \frac{\eta}{2} (| D|^2 \eta ) + \frac12 |D| ( \eta |D| \eta  ) - \frac14
 |D| \Big(  \omega_x^2 +  (|D| \omega)^2  \Big)   +
\frac12  (|D| \omega)(|D|^2  \omega)  + \frac12 \om_x    (\pa_x |D| \omega) 
\end{align*}
plus a cubic function in $ \mathcal{F}^{\R}_{K,0,3} $. 
\end{lemma}

\begin{proof}
By
\eqref{expV}  and  \eqref{form-of-B} we have that
$a  = \frac12 (  \pa_t B + V B_x ) = 
\frac12 \pa_t B  + \frac12 \omega_x   ( |D| \psi_x) $ 
plus  a cubic  function in $\mathcal{F}_{K,1,3}^{\R}$.
Hence \eqref{exppatB} implies the lemma. 
\end{proof}

We Fourier develop the functions  $\mathtt{a}_{1},\mathtt{V}_{1}$, $\mathtt{a}_{2},\mathtt{V}_{2}$, 
as in \eqref{eq:F1}, \eqref{eq:F2}.

\begin{lemma}{\bf (Coefficients of $ \mathtt{V}_1 $ and $  \mathtt{V}_2 $)}\label{lem:V1} 
The coefficients of the functions $ \mathtt{V}_1 $ and $ \mathtt{V}_2 $ in \eqref{coeV1}-\eqref{coeV2} are, 
for all $ n \in \Z \setminus \{0\} $,  
\be\label{coV1} 
(\mathtt{V}_1)^{+}_n = (\mathtt{V}_1)^{-}_n = \frac{1}{\sqrt{2}} n |n|^{-1/4} \, , 
\qquad 
(\mathtt{V}_2)^{+-}_{n, n} = n | n | \, , 
\quad (\mathtt{V}_2)^{+-}_{n,-n}=0 \, . 
\ee
\end{lemma}

\begin{proof}
By \eqref{coeV1} (recalling \eqref{complex-uU}) we have
\begin{align*}
{\mathtt V}_1   
& = \frac{1}{\ii \sqrt{2}} \pa_x | D |^{-1/4} (u - \bar u )
= \frac{1}{\sqrt{2\pi}}\frac{1}{ \sqrt{2}}   \sum_{n \neq 0} n | n |^{-1/4} u_n e^{\ii n x} +  n | n |^{-1/4} \ov{u_n} e^{- \ii n x}  
\end{align*}
which implies the expressions  for $ (\mathtt{V}_1)^{\pm}_n $  in  \eqref{coV1}. 
By \eqref{coeV2} an explicit computation using  
Definition \ref{quantizationtotale} of the Bony-Weyl quantitation (and \eqref{Weil-Q}) shows that
\[
(\mathtt{V}_2)_{n_1,n_2}^{+-}=\frac{1}{2}
(n_1-n_2) \chi_1 (n_1, n_2) \Big(|n_1|^{\frac{3}{4}}|n_2|^{\frac{1}{4}}- |n_2|^{\frac{3}{4}}|n_1|^{\frac{1}{4}}\Big)
+\frac{1}{2}\Big(n_2|n_1|^{\frac{3}{4}}|n_2|^{\frac{1}{4}}+n_1 |n_2|^{\frac{3}{4}}|n_1|^{\frac{1}{4}} \Big)
\]
where $ \chi_1 $ is a cut-off function as in \eqref{cutoff1} with $ p = 1 $ (even in all its arguments). 
This formula implies the expressions for $ (\mathtt{V}_2)^{\pm}_n $ in \eqref{coV1}.
\end{proof}

We now compute the coefficients of the linear and quadratic
component of the 
function  $ a $  in \eqref{function:a}.
\begin{lemma}{\bf (Coefficients of $ \mathtt{a}_1 $ and $ \mathtt{a}_2 $)}\label{coefA} 
The coefficients of the functions $ \mathtt{a}_1 $ and $ \mathtt{a}_2 $ in \eqref{busy2} satisfy
\be\label{coea1}
(\mathtt{a}_1)^{+}_n = (\mathtt{a}_1)^{-}_n  = - \frac{1}{2 \sqrt{2}} |n|^{5/4} \, , 
\qquad  ({\mathtt a}_2)_{n, n}^{+-} =  \frac12 |n|^{5/2} \, , \quad \forall n \in \Z \setminus \{0\} \, . 
\ee
\end{lemma}

\begin{proof}
By Lemma \ref{lem:expa} we have 
$
{\mathtt a}_1 = - \frac12 |D| \eta \stackrel{\eqref{varfin1}} =  - \frac{1}{2 \sqrt{2}} |D|^{5/4} (u + \bar u) 
$
and the formulas for $ (\mathtt{a}_1)^{\pm}_n $
in \eqref{coea1} follow. We next compute the coefficients of $ ({\mathtt a}_2)_{n, n}^{+-} $.
We remark that the terms with the operator $ |D| $ in front do not contribute to $ ({\mathtt a}_2)_{n,n}^{+-} $, because
$$
|D| \Big( \sum_{n_1, n_2 }  m_{n_1,n_2}^{+-} u_{n_1} \ov{u_{n_2}} e^{\ii ( n_1 - n_2 ) x }  \Big) = 
\sum_{n_1, n_2 }  |n_1 - n_2 |m_{n_1,n_2}^{+-} u_{n_1} \ov{u_{n_2}} e^{\ii ( n_1 - n_2 ) x }  \, , 
$$ 
whose coefficients vanish for $ n_1 = n_2 $. 
Thus we have to consider the bilinear contribution in the variables $u, \bar u $ defined in 
 \eqref{u0'}, coming from 
the terms 
\begin{align}
- \frac12 \eta (| D|^2 \eta ) & = -  \frac{1}{4} \big(  |D|^{1/4} ( u + \bar u ) \big) \big(  |D|^{9/4}( u + \bar u ) \big) \label{ter1} \\
\frac12 (|D| \omega)(|D|^2  \omega) & = - \frac14 
\big(  |D|^{3/4}(  u - \bar u ) \big) \big( |D|^{7/4}(  u - \bar u ) \big) \label{ter2} \\
\frac12 \om_x    (\pa_x |D| \omega) & = 
- \frac14 \big( \pa_x |D|^{-1/4}(  u - \bar u ) \big)    \big( \pa_x |D|^{3/4}(  u - \bar u ) \big) \, . \label{ter3}
\end{align}
The contribution from \eqref{ter1}-\eqref{ter3}
is
$  ({\mathtt a}_2)_{n,n}^{+-} =  \frac12 |n|^{5/2} $
proving the second formula in \eqref{coea1}. 
\end{proof}

It turns out that $ ({\mathtt a}_2)_{n,-n}^{+-} = |n|^{5/2} $ but we do not use this information in the paper. 

\begin{lemma}{\bf (Quadratic water waves vector field $ \ii F_2 (U) $)} \label{espansioneF2}
The quadratic water waves vector field $ \ii F_2 (U) $ in \eqref{eq:415bis} is 
\be \label{goodOmega3}
\begin{aligned}
\ii F_2 (U) & =  \frac{1}{\sqrt{2}}|D|^{-\frac{1}{4}}\Big(
|D| \opbw(|D|\omega)\eta -\pa_{x}(\eta\pa_{x}\omega)-|D|(\eta|D|\omega)
\Big) \\
&+\frac{\ii}{\sqrt{2}}|D|^{\frac{1}{4}}\Big(
-\frac{1}{2}\omega_{x}^{2} + \frac{1}{2}(|D|\omega)^{2}+\opbw(|D|\eta)\eta-\opbw(|D|\omega)|D|\omega
\Big) 
\end{aligned}
\ee
expressing $ (\eta, \omega) $ in terms of $  (u, \bar u ) $ as in \eqref{varfin1}.  
The coefficients $(F_{2})^{+-}_{n_1,n_2} $ defined 
in  \eqref{quadraticTerms2}  satisfy 
\begin{equation}\label{coeffF2}
(F_{2})^{+-}_{n,-n}=(F_{2})^{+-}_{-n,n}=2^{-\frac{1}{4}}|n|^{\frac{7}{4}} \, .
\end{equation} 
\end{lemma}

\begin{proof}
By  \eqref{form-of-B} (recalling \eqref{DNexp2}) we have  the expansion $ \omega=\psi-\opbw(|D|\psi) \eta$,
up to a cubic function in $\mathcal{F}^{\R}_{K,0,3}$,
of the good unknown in \eqref{omega0}. 
Then the equations in \eqref{WW12} reads
\begin{align*}
  \partial_t \eta &=  \dps |D| \omega  + | D|  \opbw{(|D| \psi)} \eta     - \pa_x (\eta \pa_x \omega ) - |D| ( \eta |D|\omega ) \,,
  \\
 \partial_t \omega & = \psi_t - \opbw( |D|\psi_t   ) \eta -  \opbw (  |D|\psi ) \eta_t \\
&  = \dps - \eta  -\frac{1}{2} \omega_x^2 +     \frac{1}{2} (|D|\omega)^2 
+ \opbw ( |D| \eta ) \eta -  \opbw (  |D| \omega ) [ |D| \omega  ] 
    \end{align*}
    up to cubic functions in $\mathcal{F}^{\R}_{K,1,3}$.
In the  complex  variable $u$ defined  in \eqref{u0'} we obtain the equation  
$ u_t =  - \ii |D|^{\frac12} u + \ii F_2 (U) $ with $ \ii F_2 (U)  $ defined in \eqref{goodOmega3}. 
Expressing in \eqref{goodOmega3} the variables $ (\eta, \omega ) $ in terms of $ (u, \bar u )$ as in 
\eqref{varfin1}
and passing to the Fourier coordinates, we derive  
\eqref{coeffF2} by a direct calculus similar to those in Lemmata \ref{lem:V1} and \ref{coefA}.
\end{proof}

\begin{remark}
The only property of the coefficients $ (F_{2})^{+-}_{n,-n} $ in \eqref{coeffF2} that we are going to use 
is that $ (F_{2})^{+-}_{n,-n}  = \ov{(F_{2})^{+-}_{-n,n}}$, see  the proof 
of Lemma \ref{lem:5.5}. 
This property could be also derived by the 
reversibility and even-to-even property  of the water waves system \eqref{eq:415}, or \eqref{eq:415bis}, 
which are preserved by the good unknown transformation.
The involution $ S $  in \eqref{involution} reads $ u(x) \mapsto \bar u(-x) $,  and in the Fourier basis
$ (u_j) \mapsto ( \ov{u_j}) $.
\end{remark}

\bigskip
\section{ Block-diagonalization}\label{diagonalizzo}

The goal of this section is to transform the water waves system \eqref{eq:415} into the system \eqref{sistemaDiag} below
which is block-diagonal in the variables $ (u, \bar u )$, modulo a smoothing operator $ R (U) $. 

  \begin{proposition}{\bf (Block-Diagonalization)}\label{teodiagonal}
Let  $ \rho \gg 1   $  and  $ K \geq  K':=2\rho+2  $.  There exists $s_0>0$ such that, 
for any $s\geq s_0$, for all $0<r \leq r_0(s)$ small enough,  
and  any solution $U\in B^K_s(I;r)$ of  \eqref{eq:415},  the following holds:

\begin{enumerate}
\item[(i)] 
there is a map 
$ {\bf \Psi}_{diag}^{\theta}(U) $, $ \theta\in [0,1] $,  satisfying, for some $C=C(s,r,K)>0$,
\begin{equation}\label{stime-diagonal}
\|\pa_{t}^{k}{\bf \Psi}_{diag}^{\theta}(U)[V]\|_{\dot{H}^{s-k}}
+\|\pa_{t}^{k}({\bf \Psi}_{diag}^{\theta}(U))^{-1}[V]\|_{\dot{H}^{s-k}}
\leq \big(1+C \|{U}\|_{K,s_0}\big)\|V\|_{k,s} \,, 
\end{equation} 
for any $  0\leq k\leq K-K'$ and any 
$V=\vect{v}{\bar{v}}$ in $C^{K-K'}_{*\R}(I,\dot{H}^s(\T;\C^{2}))$, $\theta\in[0,1]$;

\item[(ii)] the function $W:=({\bf \Psi}_{diag}^{\theta}(U)U)_{|_{\theta=1}}$ solves the system 
\begin{equation}\label{sistemaDiag}
\begin{aligned}
 \pa_{t}W&=
\opbw \left(
\begin{matrix}
d(U;x,\xi) + r_{-1/2}(U;x,\x)&0 \\ 0 & \overline{d(U;x,-\xi)}+\ov{r_{-1/2}(U;x,-\x)} 
\end{matrix}
\right)W +R(U)[W]
 \end{aligned}
\end{equation}
where $d(U;x,\x)$  is a  symbol of the form
\begin{equation}\label{sistemainiziale3}
d(U;x,\xi):=- \ii V(U;x) \xi  -  \ii (1+ a^{(0)}(U;x) ) |\xi|^{1/2} 
\end{equation}
where
 $  a^{(0)}  $ is a function in $  \Sigma {\cal F}^{\R}_{K,1,1}  $, 
$ r_{-1/2}(U;x,\x) $ is a symbol in $  \Sigma {\Gamma}^{-1/2}_{K,2\rho+2,1} $, and 
$R(U) $ is a real-to-real matrix of smoothing operators in $ \Sigma\mathcal{R}^{-\rho}_{K,2\rho+2,1}\otimes\mathcal{M}_2(\mathbb{C})$.
 The function   $a^{(0)} $ has the expansion 
 \be\label{def:a0}
 a^{(0)}=\mathtt{a}_1+ \mathtt{a}^{(0)}_2+\mathtt{a}^{(0)}_{\geq3} \, , \quad 
 \mathtt{a}^{(0)}_2 := \mathtt{a}_2 - \frac12 \mathtt{a}_1^2 \in \widetilde{\mathcal{F}}_2^{\R} \, , 
 \ee
 where
$ \mathtt{a}_1 $ and $\mathtt{a}_2$ are defined in \eqref{busy2}. 
\end{enumerate}
  \end{proposition}

\noindent
Proposition \ref{teodiagonal} is proved 
applying a sequence a  transformations 
which iteratively block-diagonalize \eqref{eq:415} in decreasing orders. 
In Subsection \ref{diagoparaprod} we block-diagonalize \eqref{eq:415} at the order $ 1/ 2 $ and in Subsection
\ref{lowOffdiag} we perform the block-diagonalization until the negative order $ - \rho $.

\subsection{Block-Diagonalization at order $1/2$}\label{diagoparaprod}

The aim of this subsection is to diagonalize the matrix of symbols 
$A_{1/2}(U; x)  |\xi|^{1/2} $ 
in \eqref{eq:415}, up to a matrix of symbols of order $ 0 $. 
We apply a parametrix argument conjugating the system \eqref{eq:415} 
with a paradifferential operator  whose 
principal matrix symbol is 
\begin{equation}\label{TRA}
C:=\left(\begin{matrix}
f & g \\
g & f
\end{matrix}\right) \, , \quad  
f(U; x) :=\frac{1+a+\lambda_{+}}{\sqrt{(1+a+\lambda_{+})^2-a^2}}, \quad g(U; x) :=\frac{-a}{\sqrt{(1+a+\lambda_{+})^2-a^2}} \, , 
\end{equation}
where
\begin{equation}\label{eigenvalues}
\lambda_{\pm}=\lambda_{\pm}(U;x):=\pm \sqrt{(1+a)^{2}-a^{2}} 
\end{equation}
are the eigenvalues of $ A_{1/2} $. We have 
\be\label{invC}
{\rm det}(C)=f^2-g^2=1  \, , \quad 
C^{-1} =
 \left(\begin{matrix}
f & - g \\
- g & f
\end{matrix}\right)  \, , 
\ee
and 
\begin{equation}\label{nuovocoef}
C^{-1}A_{1/2}C=  
\left(\begin{matrix}
- \l_+ & 0 \\
0 & \l_+
\end{matrix}\right)
= 
\left(\begin{matrix}
- (1+a^{(0)}) & 0 \\
0 & 1+a^{(0)}
\end{matrix}\right) \, , \quad 
a^{(0)}:=\lambda_{+}-1\in \Sigma\mathcal{F}^{\R}_{K,1,1} \, . 
\end{equation}

\begin{lemma}\label{lem:step12}
There exists a function $ m_{-1} (U; x) $ in $ \Sigma\mathcal{F}_{K,1,1}$ such that, 
the flow 
\begin{equation}\label{generatore20}
\pa_{\theta}{\bf \Psi}_{-1}^{\theta}(U) =\opbw(M_{-1}){\bf \Psi}_{-1}^{\theta}(U), \ 
{\bf \Psi}_{-1}^{0}(U) = {\rm Id} \, , \quad
M_{-1} :=\left(\begin{matrix}0 & m_{-1}(U;x)\\  \ov{m_{-1}(U;x)} & 0\end{matrix}\right) \, , 
\end{equation}
has the form 
\begin{align}
 ({\bf \Psi}_{-1}^{\theta}(U))_{|_{\theta=1}}& =\opbw(C^{-1}) + R(U), \quad 
R (U) \in \Sigma\mathcal{R}^{-\rho}_{K,1,1}\otimes\mathcal{M}_2(\C) \, ,\label{uguale}\\
 ({\bf \Psi}_{-1}^{\theta}(U))^{-1}_{|_{\theta=1}} &= \opbw(C) + Q(U), \qquad 
Q (U) \in \Sigma\mathcal{R}^{-\rho}_{K,1,1}\otimes\mathcal{M}_2(\C) \, .\label{uguale10}
\end{align}
Moreover,  if $U$ solves \eqref{eq:415},  then the function 
\begin{equation}\label{TRA5}
W_0:=({\bf \Psi}_{-1}^{\theta}(U))_{|_{\theta=1}}U 
\end{equation}
 solves the system
\begin{equation}\label{NuovoParaprod}
\begin{aligned}
\pa_t W_0=\opbw\Big( 
\left(
\begin{matrix}
d(U;x,\xi) &0 \\ 0 & \overline{d(U;x,-\xi)}
\end{matrix}
\right)  +A^{(0)} 
\Big)W_0+R^{(0)}(U)W_0
\end{aligned}
\end{equation}
where $ d(U;x,\xi)  $ is the symbol in  \eqref{sistemainiziale3} with  
 $ a^{(0)}(U;x) $ defined in \eqref{nuovocoef},  
a matrix of symbols 
\begin{equation}\label{lambdazero}
\begin{aligned}
A^{(0)} :=\left(\begin{matrix}  c_0(U;x, \xi)  & b_0(U;x, \xi) \\
 \ov{b_0(U;x, - \xi)} & \ov{c_0(U;x, - \xi)}   \end{matrix}\right)  \, , \ 
c_0 \in \Sigma\Gamma^{-\frac{1}{2}}_{K,2,1},\;\;  b_0 \in \Sigma\Gamma^{0}_{K,2,1}  \, , 
\end{aligned}
\end{equation}
and a real-to-real matrix of smoothing operators $R^{(0)}(U)$ in  
$ \Sigma\mathcal{R}^{-\rho}_{K,2,1}\otimes\mathcal{M}_2(\mathbb{C})$.
Moreover the function  $ a^{(0)} $  has the expansion  \eqref{def:a0}. 
\end{lemma}

\begin{proof}
We prove \eqref{uguale}-\eqref{uguale10} in  Appendix \ref{Psiasflow}.
We conjugate
 \eqref{eq:415} with the flow $({\bf \Psi}_{-1}^{\theta}(U))_{|_{\theta=1}}$ 
 using  formula \eqref{nuovosist} in Lemma \ref{lem:tra.Vec}. 
By
Proposition \ref{composizioniTOTALI} we deduce that, if $U$ solves \eqref{eq:415}, then 
\begin{align*}
\pa_t W_0 &\stackrel{\eqref{uguale},\eqref{uguale10}}{=} 
\pa_{t}\opbw(C^{-1})\opbw(C) W_0 \\
&+ 
\opbw(C^{-1}) \opbw\big( 
\ii A_1 \x+\ii A_{1/2} |\x|^{\frac{1}{2}}+A_0+A_{-1}
\big) \opbw(C) W_0
\end{align*}
up to a matrix of smoothing operators  in 
$\Sigma\mathcal{R}^{-\rho+1}_{K,2,1}\otimes\mathcal{M}_2(\C)$ acting on $W_0$.
Moreover Proposition \ref{teoremadicomposizione} imply that
\begin{equation}\label{eq:41555}
\pa_{t}W_0=\opbw{ \Big( \pa_{t}{C^{-1}} \#_\rho C + 
C^{-1} \#_\rho \big( 
\ii A_1 \x+\ii A_{1/2} |\x|^{\frac{1}{2}}+A_0+A_{-1}
\big) \#_\rho C\Big) }  W_0
\end{equation}
up to terms in  $\Sigma\mathcal{R}^{-\rho+1}_{K,2,1}\otimes\mathcal{M}_2(\C)$.
We now prove that \eqref{eq:41555} has the form \eqref{NuovoParaprod}. 
By \eqref{TRA}, \eqref{invC} we have 
 \begin{equation}\label{TRA9}
\!(\pa_t C^{-1}) \#_\rho C = \left(
\begin{matrix}
(\pa_tf)f-(\pa_tg)g & (\pa_tf)g-(\pa_t g)f\\ (\pa_tf)g-(\pa_tg)f & (\pa_tf)f-(\pa_tg)g 
\end{matrix}
\right) 
= \left(
\begin{matrix}
0 & \!\! \!\!  (\pa_tf)g-(\pa_t g)f\\     (\pa_tf)g-(\pa_tg)f & \!\! \!\!  0 
\end{matrix}
\right) 
\end{equation}
because  differentiating $ f^2-g^2=1 $ we get 
$  (\pa_tf)f-(\pa_tg) g = 0 $.

By \eqref{matriciinizio},  using symbolic calculus and $ f^2-g^2=1 $ (see \eqref{invC}),  we obtain the exact expansion  
\begin{equation}\label{TRA10}
\begin{aligned}
& 
 C^{-1} \#_{\rho}( \ii A_{1} \x)\#_{\rho} C 
&= \left(\begin{matrix}
  -\ii V\x  & V(f_x g-g_x f)  \\
V(f_x g-g_x f) &  -\ii V\x  
\end{matrix}\right).
\end{aligned}
\end{equation}
By \eqref{nuovocoef}  we have
\begin{equation}\label{TRA6}
\begin{aligned}
C^{-1} \#_{\rho} (\ii A_{1/2}|\x|^{\frac{1}{2}}) \#_{\rho}  C
=\ii \left( \begin{matrix} 
-(1+a^{(0)})|\x|^{\frac{1}{2}}  & 0 \\
0 & (1+a^{(0)})|\x|^{\frac{1}{2}}  \end{matrix}\right)
\end{aligned}
\end{equation}
modulo a matrix of symbols $\Sigma\Gamma^{-\frac12}_{K,1,1}\otimes \mathcal{M}_2(\mathbb{C})$. 
Moreover, recalling \eqref{function:c}, we have the paraproduct expansion
\begin{equation}\label{TRA7}
\begin{aligned}
& C^{-1} \#_{\rho}   A_{0} \#_{\rho} C = A_0 = -\frac{1}{4}\sm{0}{1}{1}{0}V_{x}
\end{aligned}
\end{equation}
and finally, since  $ A_{-1} $ is in $ \Sigma\Gamma^{-1}_{K,1,1}\otimes \mathcal{M}_2(\mathbb{C}) $ we deduce
\begin{equation}\label{TRA8}
C^{-1} \#_\rho A_{-1}  \#_\rho C \in \Sigma\Gamma^{-1}_{K,1,1}\otimes \mathcal{M}_2(\mathbb{C}) \, . 
\end{equation}
Formulas \eqref{eq:41555}-\eqref{TRA8}
imply  \eqref{NuovoParaprod},  \eqref{lambdazero}, with a remainder $R^{(0)}(U)$ in  
$ \Sigma\mathcal{R}^{-\rho}_{K,2,1}\otimes\mathcal{M}_2(\mathbb{C})$, renaming $ \rho - 1 $ as $ \rho $. 
Finally, 
by \eqref{nuovocoef}, \eqref{eigenvalues} and \eqref{busy2} 
we get the expansion \eqref{def:a0}.
\end{proof}

\subsection{Block-Diagonalization at negative orders}\label{lowOffdiag}
 The aim of  this subsection is to iteratively 
 block-diagonalize the 
 system \eqref{NuovoParaprod} (which is yet block-diagonal at the orders $ 1$ and $ 1/ 2 $) into  \eqref{sistemaDiag}. 

\begin{lemma}\label{lem:indud}
For $ j = 0, \ldots, 2 \rho $, there are 

$ \bullet $ 
paradifferential operators of the form  
 \begin{align}
\mathcal{Y}^{(j)}(U) :=\opbw \left(
\begin{matrix}
d(U;x,\xi) &0 \\ 0 & \overline{d(U;x,-\xi)} 
\end{matrix}
\right)+\opbw (A^{(j)} )\label{sist2 j-th}
\end{align}
where  $d(U;x,\x)$ is the symbol  defined in Lemma \ref{lem:step12},
$A^{(j)}$ is a  matrix of symbols of the form
\begin{equation}\label{bjbjbj}
A^{(j)} =\left(
\begin{matrix}
 c_j(U;x,\x)& b_j(U;x,\x)\\
 \ov{b_j(U;x,-\x)} & \ov{c_j(U;x,-\x)}
\end{matrix}
\right), \;\; c_j\in \Sigma\Gamma^{-\frac{1}{2}}_{K,j+2,1},\;\;  b_j\in \Sigma\Gamma^{-\frac{j}{2}}_{K,j+2,1} \, , 
\end{equation}

$ \bullet  $  
a real-to-real  matrix of smoothing operators $R^{(j)} (U) $  in  
$\Sigma\mathcal{R}^{-\rho}_{K,j+2,1}\otimes\mathcal{M}_2(\mathbb{C})$,

such that, if $ W_j $, $ j = 0, \ldots, 2 \rho - 1 $, solves 
\be\label{sist1 j-th}
\pa_{t}W_j=\big(\mathcal{Y}^{(j)}(U)+R^{(j)}(U)\big)W_j, \quad  W_j:=\vect{w_j}{\bar w_j} \, , 
\ee
then 
\begin{equation}\label{nuovavarjth}
W_{j+1}:=({\bf \Psi}_{j}^{\theta}(U)W_{j})_{|_{\theta=1}}
\end{equation}
where ${\bf \Psi}_{j}^{\theta}(U)$ is the flow at time $\theta\in [0,1]$ of 
\be\label{flow-jth}
\partial_{\theta} {\bf \Psi}^{\theta}_{j} (U) = \ii  \opbw{( M_j (U; x,\xi) )} {\bf \Psi}_{j}^{\theta}(U) \, ,
 \quad \Psi_{j}^{0}(U) = {\rm Id} \, ,  
\ee
with
\begin{equation}\label{generatore-jth}
\!\! 
M_j(U;x,\xi):=\left(\begin{matrix}
0 & \!\! \!\!  -\ii m_j(U;x,\xi) \\ 
-\ii \ov{m_{j}(U;x,-\xi)}
& \!\! \!\! 0
\end{matrix}\right), \,
m_j =\frac{-\chi(\x)b_{j}(U;x,\x)}{2\ii(1+a^{(0)}(U;x))|\x|^{\frac{1}{2}}}\in 
\Sigma\Gamma^{-\frac{j+1}{2}}_{K,j+2,1} \, ,
\end{equation}
and  $\chi$ defined in \eqref{cutoff11},
satisfies a system of the form \eqref{sist1 j-th} with $ j + 1 $ instead of $ j $.
\end{lemma}

\begin{proof}
The proof proceeds by induction. 
\\[1mm]
{\bf Inizialization.} System \eqref{NuovoParaprod} is \eqref{sist1 j-th} for $ j = 0 $ where  the 
paradifferential operator $ \mathcal{Y}^{(0)}(U)  $ has the form   \eqref{sist2 j-th} with  the matrix of symbols 
$ A^{(0)}$ defined in Lemma \ref{lem:step12}. 
\\[1mm]
{\bf Iteration.} 
We now argue by induction. Suppose that $ W_j $ solves system \eqref{sist1 j-th} with operators
$ \mathcal{Y}^{(j)}(U)  $ of  the form   \eqref{sist2 j-th}-\eqref{bjbjbj} and smoothing operators 
$R^{(j)} (U) $  in   $\Sigma\mathcal{R}^{-\rho}_{K,j+2,1}\otimes\mathcal{M}_2(\mathbb{C})$. 
Let us study the system solved by the function $ W_{j+1} $ defined in  \eqref{nuovavarjth}. 
Note that the symbols of the matrix $ M_j $ defined 
in \eqref{generatore-jth} have negative order for any $ j \geq 0 $.
By 
formula \eqref{nuovosist} the conjugated system has the form 
\be\label{sis:j+1}
\pa_{t}W_{j+1} =\opbw \big( (\pa_t  {\bf \Psi}_{j}^{1}(U)) {\bf \Psi}_{j}^{-1}(U) +
 {\bf \Psi}_{j}^{1}(U)  \mathcal{Y}^{(j)} (U)   {\bf \Psi}_{j}^{-1}(U)  \big) W_{j+1}
\ee
up to a smoothing operator in  $\Sigma\mathcal{R}^{-\rho}_{K,j+2,1}\otimes\mathcal{M}_2(\mathbb{C})$. 
Moreover the operator $(\pa_t  {\bf \Psi}_{j}^{1}(U)) {\bf \Psi}_{j}^{-1}(U)$ admits the Lie expansion
in \eqref{Lie2Vec} specified for $ {\bf A}:=\opbw(M_j (U) )$.
We recall  (see \eqref{espansione2})  that
$$
M_{j}\#_{\rho}\pa_{t}M_{j}-\pa_{t}M_{j}\#_{\rho}M_{j}=\{M_{j},\pa_{t}M_{j}\}\in \Sigma\Gamma^{-({j+1})-1}_{K,j+3,2}
$$ 
up to a symbol in 
$\Sigma\Gamma^{-({j+1})-3}_{K,j+3,2}$. By Proposition \ref{teoremadicomposizione} we have that 
${\rm Ad}_{\ii \opbw(M_{j})}[\ii \opbw(\pa_{t}M_{j})]$ is a paradifferential operator 
with symbol in $\Sigma\Gamma^{-({j+1})-1}_{K,j+3,2}\otimes\mathcal{M}_2(\C)$ plus a smoothing remainder in $\Sigma\mathcal{R}^{-\rho}_{K,j+3,2}\otimes\mathcal{M}_2(\C)$.
As a consequence we deduce, for $k\geq2$,
\[
{\rm Ad}^{k}_{\ii \opbw(M_{j})}[\ii \opbw(\pa_{t}M_{j})]=\opbw(B_{k})+R_{k}, \qquad 
B_{k}\in \Gamma_{K,j+3,k+1}^{-\frac{j+1}{2}(k+1)-k}\otimes\mathcal{M}_2(\C) \, ,
\]
and $R_{k}\in \mathcal{R}^{-\rho}_{K,j+3,k+1}\otimes\mathcal{M}_2(\C)$.
By  taking   $ L $ large enough with respect to $ \rho $, we obtain that 
$(\pa_t  {\bf \Psi}_{j}^{1}(U)) {\bf \Psi}_{j}^{-1}(U)$
is a 
 paradifferential operator with symbol in $ \Sigma\Gamma^{-\frac{j+1}{2}}_{K,j+3,1}\otimes\mathcal{M}_2(\C)$ plus a 
 smoothing operator in $\Sigma\mathcal{R}^{-\rho}_{K,j+3,1}\otimes\mathcal{M}_2(\mathbb{C})$. 
 We now want to apply the expansion \eqref{Lie1Vec} with 
  $ {\bf A}:=\opbw(M_j (U) )$
 and $X:=\mathcal{Y}^{(j)}$ in order to study the second summand in \eqref{sis:j+1}. 
 We claim that
 \begin{align}  \label{eq:428}
&  {\bf \Psi}_{j}^{1}(U) \mathcal{Y}^{(j)}(U)  {\bf \Psi}_{j}^{-1}(U) 
 = \mathcal{Y}^{(j)}(U) + 
[ \opbw(\ii M_j (U) ) , {\mathcal{Y}^{(j)}(U) } ] 
\end{align} 
plus a paradifferential operator with symbol in $ \Sigma\Gamma^{-\frac{j+1}{2}}_{K,j+2,1} \otimes\mathcal{M}_2(\C) $ 
and a 
 smoothing operator in $\Sigma\mathcal{R}^{-\rho}_{K,j+2,1}\otimes\mathcal{M}_2(\mathbb{C})$.
We first give the expansion of 
$ [ \opbw(\ii M_j (U) ) , {\mathcal{Y}^{(j)}(U) } ] $
using the expression of $\mathcal{Y}^{(j)}(U)   $ in  \eqref{sist2 j-th}. We have 
\begin{equation}\label{esp2-jth}
\begin{aligned}
&\Big[
\opbw(\ii M_j(U)), \opbw\Big(\sm{d(U;x,\x)}{0}{0}{\ov{d(U;x,-\x)}}\Big)
\Big]:=\opbw\left(\sm{0}{p_{j}(U;x,\x)}{\ov{p_{j}(U;x,-\x)}}{0}\right)
\\
&\qquad p_{j}:=2\ii m_{j}(U;x,\xi)(1+a^{(0)}(U;x))|\xi|^{\frac{1}{2}}
\end{aligned}
\end{equation}
up to a symbol in $\Sigma\Gamma^{-\frac{j+1}{2}}_{K,j+2,1} \otimes {\cal M}_2(\C)$.
Moreover, since $ A^{(j)} $, is a matrix of symbols of order $ - 1/ 2 $, for $ j \geq 1 $,  respectively $ 0 $ for $ j = 0 $
(see \eqref{bjbjbj}), we have that 
\begin{equation}\label{esp1-jth}
\big[
\opbw(\ii M_j ) , \opbw(A^{(j)}) 
\big] \in 
\left\{
\begin{aligned}
& \Sigma\Gamma^{-\frac{j+2}{2}}_{K,j+2,1}\otimes\mathcal{M}_{2}(\mathbb{C}) \quad {\rm for } \ j \geq 1 \\
& \Sigma\Gamma^{-\frac{1}{2}}_{K, 2,1}\otimes\mathcal{M}_{2}(\mathbb{C}) \qquad {\rm for } \ j =0 
\end{aligned}
\right.
\end{equation}
up to a smoothing operator in $\Sigma\mathcal{R}^{-\rho}_{K,j+2,1}\otimes\mathcal{M}_2(\mathbb{C})$. 
It follows that the off-diagonal symbols of order $ - j / 2 $ in \eqref{eq:428} are 
of the form  $\sm{0}{q_{j}(U;x,\x)}{\ov{q_{j}(U;x,-\x)}}{0}$ with
\begin{equation}\label{equa-jthquatuor}
q_{j}(U;x,\x)\stackrel{\eqref{esp2-jth}}{:=}b_{j}(U;x,\x)+2\ii m_{j}(U;x,\xi)(1+a^{(0)})|\xi|^{\frac{1}{2}} \, . 
\end{equation}
By the definition of $\chi$ in \eqref{cutoff11} and the remark under Definition \ref{azionepara},
the operator $\opbw((1-\chi(\x))b_{j}(U;x,\x))$ is in 
$\Sigma\mathcal{R}^{-\rho}_{K,j+2,1}\otimes\mathcal{M}_2(\mathbb{C})$ for any 
$\rho\geq0$. Moreover, by  
  the choice of $m_{j}(U;x,\x)$ in \eqref{generatore-jth} we have that
  \[
 \chi(\x) b_{j}(U;x,\x)+2\ii m_{j}(U;x,\xi)(1+a^{(0)})|\xi|^{\frac{1}{2}}=0 \, .
  \]
  This implies that $[\ii \opbw(M_{j}) ,\mathcal{Y}^{(j)}(U)]$ is a paradifferential operator with symbol in
  $\Sigma\Gamma_{K,j+2,1}^{-\frac{j+1}{2}}\otimes \mathcal{M}_2(\C)$ plus a remainder in 
  $\Sigma\mathcal{R}^{-\rho}_{K,j+2,1}\otimes\mathcal{M}_2(\mathbb{C})$.
  Now, using Proposition \ref{teoremadicomposizione}, we deduce, for $k\geq2$, 
  \[
{\rm Ad}^{k}_{\ii \opbw(M_{j})}[\mathcal{Y}^{(j)}(U)]=\opbw(\widetilde{B}_{k})+\widetilde{R}_{k}, \qquad 
\widetilde{B}_{k}\in \Gamma_{K,j+2,k+1}^{-\frac{j+1}{2}k}\otimes\mathcal{M}_2(\C) \, , 
\]
where $\widetilde{R}_{k} $ is in $ \mathcal{R}^{-\rho}_{K,j+2,k+1}\otimes\mathcal{M}_2(\C)$.
  Using formula \eqref{Lie1Vec} with $L$ large enough an the estimates of flow in \eqref{flow-jth} (see Lemma \ref{buonflusso}) one obtains the claim  in  \eqref{eq:428}.
We conclude that \eqref{nuovavarjth} solves a system of the form 
\[
\begin{aligned}
\pa_{t}W_{j+1}&=\opbw(\sm{d(U;x,\x)}{0}{0}{\ov{d(U;x,-\x)}})W_{j+1}
+\opbw(A^{(j+1)} )W_{j+1}+R^{(j+1)}(U)W_{j+1}
\end{aligned}
\]
for some matrix of symbol $A^{(j+1)}$ of the form \eqref{bjbjbj} with $j\rightsquigarrow j+1$ 
and smoothing operators $R^{(j+1)}(U) $ in $ \Sigma\mathcal{R}^{-\rho}_{K,j+3,1}\otimes\mathcal{M}_{2}(\mathbb{C})$. 
\end{proof}

\begin{proof}[{\bf Proof of Proposition \ref{teodiagonal}}]
For  $\theta\in [0,1]$ we define 
\begin{equation}\label{finalPsi}
{\bf \Psi}_{diag}^{\theta}(U):={\bf \Psi}_{{2\rho-1}}^{\theta}(U) \circ\cdots 
\circ{\bf \Psi}_{0}^{\theta}(U)\circ{\bf \Psi}_{-1}^{\theta}(U)
\end{equation}
where 
the maps ${\bf \Psi}_{-1}^{\theta}(U)$ and  ${\bf \Psi}_{j}^{\theta}(U)$, $j=0,1,\ldots,2\rho - 1 $
are defined respectively
in  \eqref{TRA5},  \eqref{nuovavarjth}.  
The bound \eqref{stime-diagonal} follows by Lemma \ref{buonflusso}.
Lemmata \ref{lem:step12}, \ref{lem:indud}  imply that 
if $U$ solves \eqref{eq:415} then 
 the function  $ W := W_{2 \rho} = ({\bf \Psi}_{diag}^{\theta}(U)U)_{|_{\theta=1}} $
solves the system \eqref{sist1 j-th} with $j=2\rho  $
which is  \eqref{sistemaDiag} with $ r_{-1/2} := c_{2 \rho} $ and 
\[ 
R(U) := \opbw\left(\sm{0}{ b_{2 \rho}(U;x,\x)}{\ov{ b_{2 \rho}(U;x,-\x)}}{0}\right)
+ R^{(2 \rho) }(U) \, ,\quad  b_{2 \rho} \in \Sigma\Gamma^{-\rho}_{K, 2 \rho +2,1} \, , 
\]
which is a smoothing operator in $ \Sigma {\mathcal R}^{-\rho}_{K, 2 \rho +2,1} \otimes {\mathcal M}_2 (\C) $ by the
remark below Proposition \ref{azionepara}. 
The expansion \eqref{def:a0} is proved in Lemma  \ref{lem:step12}.
\end{proof}

\bigskip
\section{ Reductions to constant coefficients}\label{riduco}
The aim of this section is to conjugate \eqref{sistemaDiag} to a system in which the symbols
of the paradifferential operators are constant in the spatial variable $x$ 
and are ``{\it integrable}" according to Definition \ref{defiintegro} below, 
up to symbols which are ``{\it admissible}" according to Definition \ref{def:admissible}.

\begin{definition}{\bf (Integrable symbol)}\label{defiintegro}
 A homogeneous  symbol $f$ in  $\widetilde{\Gamma}_{2}^{m}$ 
is \emph{integrable}
if it is independent of $x$ and it has the form
\begin{equation}\label{simbointegro}
f(U;x,\x)=f(U;\x):=\frac{1}{2\pi}\sum_{n\in\Z\setminus\{0\}}f^{+-}_{n,n}(\x)|u_n|^{2},\qquad f^{+-}_{n,n}(\x)\in \C \, , \quad n\in\Z\setminus\{0\} \, .
\end{equation}
\end{definition}

\begin{definition}{\bf (Admissible symbol)}\label{def:admissible}
  A non-homogeneous symbol $H_{\geq 3} $ in $  \Gamma^{1}_{K,K',3}$ 
is \emph{admissible} if it has the form
\begin{equation}\label{highordersfinali}
H_{\geq 3}(U;x,\x):=\ii \alpha_{\geq3}(U;x)\x+\ii \beta_{\geq 3}(U;x)|\x|^{\frac{1}{2}}
+ \gamma_{\geq3}(U;x,\x)
\end{equation}
with real valued functions $ \alpha_{\geq3}(U;x), \beta_{\geq3}(U;x) $ in $  \mathcal{F}^\R_{K,K',3}$ 
and a symbol $ \gamma_{\geq3}(U;x,\x) $ in $  \Gamma^{0}_{K,K',3}$.
A matrix of symbols $ {\bf H}_{\geq 3} $ in $ \Gamma^{1}_{K,K',3}\otimes\mathcal{M}_2(\C)$ is admissible 
if it has the form
\begin{equation}\label{def:admissible2}
{\bf H}_{\geq 3}(U;x,\x)=\sm{H_{\geq3}(U;x,\x)}{0}{0}{\ov{H_{\geq3}(U;x,-\x)}}
\end{equation}
for a scalar admissible symbol  $H_{\geq3}$.
\end{definition}

The relevance of Definition \ref{def:admissible} is explained in the next remark. 

\begin{remark}
An equation of the form 
$ \pa_{t} v =\opbw( H_{\geq 3}(U;x,\x))[v] $, 
where  $ H_{\geq 3 }(U;x,\x) $ is an admissible symbol in $ \Gamma^{1}_{K,K',3} $,  admits an 
 energy estimate of the form
$$ 
\pa_{t} \|v(t,\cdot)\|_{\dot H^{s}}^2 
\lesssim_{s} \|U(t,\cdot)\|^{3}_{K,{s_0}}\|v (t,\cdot)\|^{2}_{\dot H^{s}} 
$$
for $s \geq s_0\gg1$, see  Lemma \ref{tempo}.
For this reason 
vector fields of this form 
are ``admissible'' to prove existence of solutions up to times $ O(\e^{-3}) $. 
\end{remark}

The main result of this section is the following.

\begin{proposition}{\bf (Integrability of water waves at cubic degree up to smoothing remainders)}
\label{teodiagonal3}
Fix $ \rho > 0  $ arbitrary and  $ K \geq  K':=2\rho+2  $.  There exists $s_0>0$ such that, 
for any $s\geq s_0$, for all $0<r \leq r_0(s)$ small enough,  
and  any solution $U\in B^K_s(I;r)$ of  \eqref{eq:415},  
there is a family of nonlinear maps ${\bf F}^{\theta} (U)$, $\theta\in [0,1]$, 
such that the function $Z:=  {\bf F}^{1}(U)$ solves the system
\begin{equation}\label{finalsyst}
 \begin{aligned}
 \pa_{t}Z&= -\ii\Omega Z+ \opbw \big(-\ii \mathtt{D}(U;\x)  +  \mathtt{H}_{\geq 3} \big) Z +  \mathtt{R}(U)[Z]
 \end{aligned}
 \end{equation}
 where $\Omega$ is defined in \eqref{eq:415tris} and 
 
 \begin{itemize}
  \item the symbol $\mathtt{D}(U;\x)$ has the form
  \begin{equation}\label{diag-part}
\!\!\! {\mathtt D}(U;\x):=\sm{\zeta(U)\x+\mathcal{D}_{-1/2}(U;\x)}{0}{0}{{\zeta(U)\x}-\ov{\mathcal{D}_{-1/2}(U;-\x)}} \, ,
\;\;\;\;\; \zeta (U) := 
\frac{1}{\pi} \sum_{n \in \Z\setminus\{0\}} n |n | |u_{n}|^2 \, ,
 \end{equation}
with an integrable symbol  $\mathcal{D}_{-1/2}(U;\x)\in \widetilde{\Gamma}^{-\frac{1}{2}}_2$ 
(see Definition \ref{defiintegro});

  \item the matrix of symbols $ \mathtt{H}_{\geq 3} \in \Gamma^{1}_{K,K',3}\otimes \mathcal{M}_2(\C)$ 
is admissible (see Definition \ref{def:admissible});

  \item $\mathtt{R}(U)$ is a real-to-real matrix of smoothing  operators  
    in  $\Sigma\mathcal{R}^{-\rho+4m}_{K,K',1}\otimes\mathcal{M}_2(\C) $ for some $m>0$.

\item The family of transformations has the form
\begin{equation}\label{FTvectPARA}
{\bf F}^{\theta}(U):=\mathfrak{F}^{\theta}(U)[U]
\end{equation}
with $\mathfrak{F}^{\theta}(U)$ real-to-real, bounded and invertible,
and there is a constant $C =C(s,r,K)$, such that, $ \forall \,0\leq k\leq K-K' $,  for any 
$V\in C^{K-K'}_{*\R}(I;\dot{H}^{s}(\T;\C^2))$, 
\begin{equation}\label{stimFINFRAK}
\| \pa_{t}^{k}\mathfrak{F}^{\theta}(U)[V]\|_{\dot{H}^{s-k}}
+\| \pa_{t}^{k}(\mathfrak{F}^{\theta}(U))^{-1}[V]\|_{\dot{H}^{s-k}}\leq \|V\|_{k,s}(1+C\|U\|_{K,s_0}) \, ,
\end{equation}
uniformly in $\theta\in [0,1]$. 
\end{itemize}
\end{proposition}
 
The proof of Proposition \ref{teodiagonal3} above is divided into several steps in Subsections 
\ref{inteord1}-\ref{sec:nega} below. 
We combine these steps in Subsection \ref{Proofteodiagonal3}.

\subsection{Integrability at  order $1$}\label{inteord1}
By  Proposition \ref{teodiagonal} we have obtained, writing only the first line of the
system \eqref{sistemaDiag}-\eqref{sistemainiziale3}, 
\be\label{WWC}
\pa_t w =  \opbw{\big(-\ii V(U; x) \xi-\ii (1+ a^{(0)}(U; x)) |\xi|^{1/2}  +r_{-1/2}
\big)}w  +  R(U)[W] 
\ee
where $ R(U)$ is a  $1\times2$ matrix of smoothing operators in $\Sigma\mathcal{R}^{-\rho}_{K,K',1}$ 
with $K'= 2\rho+2$ and $W=\vect{w}{\bar{w}}$.  The
second component of system \eqref{sistemaDiag} is the complex conjugated of the first one. 
Expanding in degrees of homogeneity the symbol
$$
r_{-1/2} =\mathtt{r}_1 +\mathtt{r}_{2} +\mathtt{r}_{\geq3} \, , \quad 
\mathtt{r}_1\in \widetilde{\Gamma}_{1}^{- \frac{1}{2}} \, ,  \
\mathtt{r}_2\in \widetilde{\Gamma}_{2}^{- \frac{1}{2}} \, , \ \mathtt{r}_{\geq3} \in \Gamma^{- \frac{1}{2}}_{K,K',3} \, ,  
$$
recalling \eqref{busy} and item (ii) in Proposition \ref{teodiagonal}, we rewrite \eqref{WWC} as 
\begin{equation}\label{WWCexp}
\begin{aligned}
\pa_t w &=  \opbw\big(-\ii \big(\mathtt{V}_1+\mathtt{V}_2\big) \xi-\ii \big(1+ \mathtt{a}_1+\mathtt{a}^{(0)}_2\big) 
|\xi|^{1/2} + \mathtt{r}_1+\mathtt{r}_2 + H_{\geq 3} \big)w+ R(U)[W] 
\end{aligned}
\end{equation}
where $ H_{\geq 3}$ is an \emph{admissible} symbol according to Definition \ref{def:admissible}.

\subsubsection{Elimination of the linear symbol of the transport }\label{trasporto1}

The goal of this subsection is to eliminate  the transport operator 
$ \opbw(-\ii \mathtt{V}_1 \xi)  $ in \eqref{WWCexp}.
With this aim we conjugate the equation \eqref{WWCexp} under the flow
\be\label{flow0}
\partial_{\theta} \Phi_{1}^{\theta}(U) = \ii  \opbw{( b (U; \theta, x) \xi )} \Phi_{1}^{\theta}(U)  \, , 
\quad \Phi_{1}^{0}(U) = {\rm Id} \, ,  
\ee
with
\be\label{flow1}
b(U; \theta, x) := \frac{\beta (U; x) }{1 + \theta \beta_x (U; x)} \, , 
\ee
where  $ \beta (U; x) $  is a real valued function 
in $ \widetilde{{\mathcal F}}_1^\R $
of  the same form as $\mathtt V_1 (U; x)$, i.e. 
\be\label{def:beta}
\beta (U; x) = \frac{1}{\sqrt{2\pi}} \sum_{n \in \Z\setminus\{0\}} \beta^+_n u_n e^{\ii n x }
  + \beta^-_n \ov{u_n} e^{- \ii n x } \, .
\ee
The function $ \beta (U; x) $ is real  if a condition like \eqref{condreal2} holds, i.e. 
\be\label{real-beta1}
\overline{\beta^+_n} =  \beta^-_n \, . 
\ee
The flow of the transport equation  \eqref{flow1} is well posed  by Lemma \ref{buonflusso}.
We introduce the new variable
\begin{equation}\label{newcoord}
V_1:=\vect{v_1}{\ov{v_1}}=
\big({\bf \Phi}_{1}^{\theta}(U)[W]\big)_{|_{\theta=1}} = 
\left(\begin{matrix} 
\Phi_{1}^{\theta}(U)[w]\\
\ov{\Phi_{1}^{\theta}(U)}[\overline{w}]
\end{matrix}
\right)_{|_{\theta=1}} \, , 
\qquad W := \vect{w}{\bar{w}} \, ,
\end{equation}
where 
the operator  $\ov{\Phi_{1}^{\theta}(U)}[\cdot]$
is defined as in \eqref{opeBarrato}.

\begin{lemma}\label{lem:5.5}
Define 
$ \beta \in \widetilde{\mathcal{F}}_1^\R  $ in \eqref{def:beta} with coefficients 
\be\label{def:betan}
\beta^+_n := - \frac{({\mathtt V}_1)^+_n}{\ii \omega_n}=\frac{\ii n}{\sqrt{2} |n|^{\frac{3}{4}}}  \, , \quad 
\beta^-_n :=  \frac{({\mathtt V}_1)^-_n}{\ii \omega_n} =-\frac{\ii n}{\sqrt{2} |n|^{\frac{3}{4}}} \, , \quad
n \neq 0 \, ,
\ee
and $(\beta)^{\s}_{0} := 0, \; \s=\pm$.
Then, if $ w $ solves \eqref{WWCexp}, the function  $ v_1 $ defined in \eqref{newcoord} solves 
\begin{equation}\label{WWCexp3}
\pa_t v_1 =   \opbw{ \big( - \ii  \mathtt{V}_{2}^{(1)}} \xi  -
\ii \big(1 +\mathtt{a}_2^{(1)} \big)|\x|^{\frac{1}{2}} 
 +  \mathtt{r}_1^{(1)} +\mathtt{r}^{(1)}_2  + H_{\geq 3}^{(1)} \big)v_1+ R^{(1)}(U)[V_1]
\end{equation}
where:

\begin{itemize}
 \item $ \mathtt{V}_{2}^{(1)} \in \widetilde{\mathcal{F}}_2^\R$ and its coefficients
 (according to the expansion \eqref{eq:F2}) satisfy  
\be\label{nullcoeff211}
(\mathtt{V}_2^{(1)})^{+-}_{n,n} = 2 n |n | \, , \quad (\mathtt{V}^{(1)}_2)^{+-}_{n,-n}=0 \, ;  
\ee

 \item $\mathtt{a}_{2}^{(1)} \in \widetilde{{\mathcal F}}_2^\R$ and its coefficients satisfy
 \be\label{nullcoeff211'}
(\mathtt{a}_2^{(1)})^{+-}_{n,n} = 0 \, ;
\ee

\item $\mathtt{r}_1^{(1)} \in \widetilde{\Gamma}_{1}^{- \frac{1}{2}}$ 
  and $\mathtt{r}_2^{(1)} \in \widetilde{\Gamma}_{2}^{-\frac{1}{2}}$;

\item $H_{\geq 3}^{(1)} \in \Gamma^{1}_{K,K',3} $ is an admissible symbol,
  and $R^{(1)}(U) $ belongs to $ \Sigma\mathcal{R}^{-\rho}_{K,K',1}$.

\end{itemize}

\end{lemma}

Note that the procedure that eliminates the linear term of the transport in \eqref{WWCexp3},
that is, the contribution with degree of homogeneity $1$ to the coefficient of $\xi$,
automatically also eliminates the contribution with degree of homogeneity $1$ to the coefficient of
the symbol of order $1/2$.

\begin{proof}[Proof of Lemma \ref{lem:5.5}]
\emph{Conjugation under the flow in \eqref{flow0}.}
We use Lemmata \ref{conjFlowpara} and \ref{conjFlowpara2}. 
\\[1mm]
{\sc Step 1.}  We apply Lemma  \ref{conjFlowpara2} with 
$ \beta $ in $ \widetilde{{\mathcal F}}_1^\R \subset {\mathcal F}^\R_{K,0,1} $ by the fourth remark in \eqref{prodottodisimboli2}. Then Lemma  \ref{conjFlowpara2} implies that 
$$
\pa_t \Phi_1^{1}(U) (\Phi_1^{1} (U))^{-1} =   \opbw{\big( \ii(\beta_t  - \beta_x \beta_t)  \xi 
+ H_{\geq3})} + R(U)
$$
where $ H_{\geq3} := \ii g_{\geq 3} \xi $ is an admissible symbol in $ \Gamma^{1}_{K,1,3}$  and 
$ R(U) $ belongs to $  \Sigma\mathcal{R}^{-\rho}_{K,1,1} $.
\\[1mm]
{\sc Step 2.}  
We apply Lemma \ref{conjFlowpara} with $ a = - \ii V \xi $. Thus  \eqref{coniugato}-\eqref{coniugato3} 
imply that, noting that $ a_\Phi^{(1)} = 0 $, 
\[
\Phi_1^{1}(U)\opbw(-\ii V\x)(\Phi_1^{1}(U))^{-1}=\opbw\big( 
- \ii V(x+ \beta(x)) ( 1+ \tilde \beta_y (y))_{|y=x + \beta (x)} \xi \big)  + R(U)
\] 
where $ x = y + \tilde \beta (y) $ denotes the inverse diffeomorphism of $ y = x + \beta (x) $ 
and $ R(U)$ is a smoothing operator in $  \Sigma {\mathcal R}^{-\rho+1}_{K,K',2} $.
By \eqref{simbLie} we deduce the expansion
\[
\Phi_1^{1}(U)\opbw(-\ii V\x)(\Phi_1^{1}(U))^{-1} = 
  \opbw \big( -\ii (\mathtt{V}_1 + {\mathtt V}_2)\xi    + 
   \ii ( \mathtt{V}_1 \beta_x -  (\mathtt{V}_1)_x \beta )   \xi  + H_{\geq 3} \big)
  + R(U) 
\]
where $ H_{\geq 3} \in \Gamma^{1}_{K,K',3}$ 
is  an admissible symbol 
and $ R(U) $ belongs to $ \Sigma {\mathcal R}^{-\rho+1}_{K,K',2} $.
\\[1mm]
{\sc Step 3.}  
By Lemma \ref{conjFlowpara} we have that, up to a smoothing remainder   in $ \Sigma\mathcal{R}^{-\rho+\frac{1}{2}}_{K,K',1} $, 
$$
- \Phi_1^{1}(U) \opbw\big(\ii \big(1+ \mathtt{a}_1 +
\mathtt{a}^{(0)}_2 \big) |\xi|^{1/2}\big) (\Phi_1^{1}(U))^{-1} = 
- \opbw(  a_\Phi^{(0)} +  a_\Phi^{(1)} ) 
$$
where  $ a_\Phi^{(0)} \in \Sigma\Gamma^{\frac{1}{2}}_{K,K',0}$ is given by \eqref{coniugato3} 
and  $ a_\Phi^{(1)} \in  \Sigma\Gamma^{-\frac{3}{2}}_{K,K',1}$.  
By \eqref{simbLie} we have the expansion 
$$ 
\begin{aligned}
a_\Phi^{(0)} & = - \ii   \big(1+ \mathtt{a}_1+ \mathtt{a}^{(0)}_2 \big) |\xi|^{1/2} +
\big\{ \beta \xi, - \ii \big(1+ \mathtt{a}_1 \big) |\xi|^{1/2} \big\} \\ 
& \quad + 
\frac12 \Big(  \big\{ \beta \xi,  \big\{ \beta \xi, - \ii |\xi|^{1/2} \} \big\} -
\big\{ \beta \beta_x \xi, - \ii |\xi|^{1/2} \big\} \Big) + {\rm admissible \ symbol}
\end{aligned}
$$
and a direct computation gives
$$
\begin{aligned}
  - \Phi_1^{1}(U)& \opbw\big(\ii \big(1+ \mathtt{a}_1+
\mathtt{a}^{(0)}_2\big) |\xi|^{1/2}\big) (\Phi_1^{1}(U))^{-1} = \\
&\quad\opbw\Big( -\ii \big( 1 + \mathtt{a}_1-\frac{\beta_x}{2}
+ \mathtt{a}^{(0)}_2+(\mathtt{a}_1)_x\beta-\frac{1}{2}\beta_x\mathtt{a}_1
+ \frac{3}{8} \beta_x^{2}  \big) |\xi|^{1/2} + r +  H_{\geq 3} \Big) + R(U)
 \end{aligned}
$$
where 
$ r \in \Sigma\Gamma^{-\frac{3}{2}}_{K,K',1}$,
 $ H_{\geq 3} \in \Gamma^{\frac{1}{2}}_{K,K',3}$ is an admissible 
symbol and $ R(U) $ is in $  \Sigma\mathcal{R}^{-\rho+\frac{1}{2}}_{K,K',1} $.

\medskip
\noindent
{\sc Step 4.} 
By  Lemma \ref{conjFlowpara} 
the conjugated operator 
\begin{align*}
& \Phi_1^{1}(U)  \opbw\big( \mathtt{r}_1 +\mathtt{r}_2 + H_{\geq 3} \big)  (\Phi_1^{1}(U))^{-1} 
 = \opbw\big( \mathtt{r}_1^{(1)} +\mathtt{r}^{(1)}_2  +H_{\geq 3}' \big) + R(U)  
\end{align*}
where $\mathtt{r}_1^{(1)}\in \widetilde{\Gamma}_{1}^{- \frac{1}{2}} $,  $\mathtt{r}^{(1)}_2\in \widetilde{\Gamma}_{2}^{- \frac{1}{2}}$,
a new admissible symbol   $ H_{\geq 3}'  \in \Gamma^{1}_{K,K',3} $, and 
  a smoothing remainder $ R(U)$ in $\Sigma\mathcal{R}^{-\rho+1}_{K,K',1}$. 
  \\[1mm]
{\sc Step 5.} By Lemma \ref{buonflusso}-$(ii)$ we  write $\Phi^{1}(U)={\rm Id}+M(U)$ with
$M(U) $ in $  \Sigma\mathcal{M}^{m}_{K,K',1}$ for some $m>0$. Hence, using Proposition \ref{composizioniTOTALI}-(ii),
we have that the conjugated of the operator $R(U)$ in \eqref{WWCexp} is a smoothing remainder in 
$ \Sigma\mathcal{R}^{-\rho+2m}_{K,K',1}$.
\noindent
In conclusion we get that 
if $ w $ solves \eqref{WWCexp} then  $ v_1 $ defined in \eqref{newcoord} satisfies 
\begin{equation}\label{WWCexp2}
\begin{aligned}
\pa_t v_1 &=   \ii \opbw{\big(  \big( - {\mathtt V}_1  + \pa_t \beta  \big) \xi   }  +   
  { \big( -{\mathtt V}_2   + ( \mathtt{V}_1 \beta_x - (\mathtt{V}_1)_x \beta ) 
    - \beta_x \beta_t  \big) \xi  \big)}v_1\\
&+
\ii\opbw\Big(
-|\x|^{\frac{1}{2}} 
- (\mathtt{a}_1-\frac{\beta_x}{2})|\x|^{\frac{1}{2}}
- \big(\mathtt{a}^{(0)}_2+(\mathtt{a}_1)_x\beta- \frac{1}{2}\beta_x\mathtt{a}_1 + \frac{3}{8} \beta_x^{2} \big) |\x|^{\frac{1}{2}}
\Big)v_1
\\
& + \opbw\big(\mathtt{r}_1^{(1)} +\mathtt{r}^{(1)}_2\big)v_1  +\opbw( H_{\geq 3})v_1+ R^{(1)}(U)[V_1]
\end{aligned}
\end{equation}
where $\mathtt{r}_1^{(1)}\in \widetilde{\Gamma}_{1}^{- \frac{1}{2}}$, 
$\mathtt{r}^{(1)}_2\in \widetilde{\Gamma}_{2}^{- \frac{1}{2}}$,
$ H_{\geq 3 }\in \Gamma^{1}_{K,K',3}$ is admissible according to Definition \ref{def:admissible}
and $ R^{(1)}(U) $ is a $1\times 2$ matrix of smoothing operators in $ \Sigma\mathcal{R}^{-\rho}_{K,K',1}$
(renaming  $ \rho - 2 m $ as $ \rho $).

\medskip
\noindent
\emph{Choice of $\beta$}. Recall that the coefficients $\beta^\pm_n $ defined in \eqref{def:betan} satisfy 
 \eqref{real-beta1} and the function $ \beta (U; x) $ is real.  
Using \eqref{eq:415bis}  we get  
 \begin{align}
 \pa_t \beta(U; x) 
 =  
\frac{1}{\sqrt{2\pi}}\sum_{n\in\Z\setminus\{0\}}
(-\ii\omega_n)\beta_n^{+}e^{\ii nx} u_n +\ii\omega_n\beta_n^{-}e^{-\ii nx} \ov{u_n} +
   {\mathtt h}_2 +\mathtt{h}_{\geq3}
 \label{quadra-rem}
\end{align}
 where $  {\mathtt h}_2 $, $\mathtt{h}_{\geq3}$ are defined as
 \be\label{def:h2}
 \begin{aligned}
 {\mathtt h}_2 &:= 
 \frac{1}{\sqrt{2\pi}} \sum_{n \in \Z\setminus\{0\}}  \beta^+_n \ii {[F_2(U)]_{n}}  e^{\ii n x } -  \beta^-_n \ii  
 \ov{[F_2(U)]_{n}}   e^{- \ii n x}\, , 
 \\
  {\mathtt h}_{\geq 3}  &:= 
 \frac{1}{\sqrt{2\pi}} \sum_{n \in \Z\setminus\{0\}}  \beta^+_n \ii {[F_{\geq3}(U)]_{n}}  e^{\ii n x } -  \beta^-_n \ii 
 \ov{[F_{\geq3}(U)]_{n}}   e^{- \ii n x}  \, .
\end{aligned}
\ee
 By \eqref{quadra-rem} and \eqref{def:betan} we deduce 
 that 
 \begin{equation}\label{pianeta4}
- {\mathtt V}_1 + \pa_t \beta = {\mathtt h}_2+\mathtt{h}_{\geq3} .
\end{equation}
By \eqref{real-beta1} the functions  ${\mathtt h}_2 $ and $\mathtt{h}_{\geq3}$ 
are real. 
Moreover
$ {\mathtt h}_2 \in \widetilde{\mathcal{F}}^{\R}_2$ and $\mathtt{h}_{\geq3}\in \mathcal{F}^{\R}_{K,1,3}$
by item (iv) of Proposition \ref{composizioniTOTALI} and the fact that $ F_2 (U)  + F_{\geq 3} (U) = 
M(U) [U] $ for some $ M $ in $ \Sigma \mathcal{M}_{K,1,1} $, see \eqref{eq:415tris}. 
\smallskip
\noindent
\emph{The new equation.}
From \eqref{pianeta4} and the first line of \eqref{WWCexp2} we deduce 
that $\mathtt{V}_2^{(1)}$ in \eqref{WWCexp3} is given by
\be
\label{eq:V21}
 - \mathtt{V}_2^{(1)}:= {\mathtt h}_2 - {\mathtt V}_2 - (\mathtt{V}_1)_x \beta,
\ee
having used $(\mathtt{V}_1 - \pa_t \beta) \beta_x \in \mathcal{F}^{\R}_{K,1,3}$.
From the second line of \eqref{WWCexp2} we deduce that $\mathtt{a}_2^{(1)}$ in \eqref{WWCexp3} is given by
\be
\label{newcoeff100}
    \mathtt{a}_{2}^{(1)} := \mathtt{a}^{(0)}_2
    +(\mathtt{a}_1)_x\beta-\frac{1}{2}\beta_x\mathtt{a}_1 
    + \frac{3}{8}  \beta_x^{2} \in \widetilde{{\mathcal F}}_2^\R
\ee
having noted that the function $ \mathtt{a}_1-\frac{\beta_x}{2} = 0$ by
\eqref{coea1} and \eqref{def:betan}.

Let us prove  \eqref{nullcoeff211}. By \eqref{eq:V21} 
we have
\be
\label{coeffNuoviV}
((\mathtt{V}_2)^{(1)})_{n_1,n_2}^{+-} = - (\mathtt{h}_2)_{n_1,n_2}^{+-}+(\mathtt{V}_2)^{+-}_{n_1,n_2}
+ \ii  \big( (\mathtt{V}_1)^{+}_{n_1}\beta^{-}_{n_2}n_1 -
(\mathtt{V}_1)^{-}_{n_2}\beta_{n_1}^{+}n_2 \big)\,.
\ee
The coefficients $ (\mathtt{h}_2)^{+-}_{n_1,n_2} $ associated to $ \mathtt{h}_2 $ 
defined in \eqref{def:h2}
 are 
$$
(\mathtt{h}_2)^{+-}_{n_1,n_2} =\ii \beta^{+}_{n_1-n_2}(F_2)_{n_1,n_2}^{+-}-\ii \beta_{-(n_1-n_2)}^{-}\ov{(F_2)_{n_2,n_1}^{+-}}
$$
with $(F_2)_{n_1,n_2}^{+-}$ defined by \eqref{eq:415bis}-\eqref{quadraticTerms2}.
We claim that
 \begin{equation}\label{cancellazione}  
(\mathtt{h}_{2})^{+-}_{n,n}=0, \qquad (\mathtt{h}_{2})^{+-}_{n,-n}=0 \, .
\end{equation}
The first identity in \eqref{cancellazione} is trivial since the coefficients  $\beta^{\s}_{n}$ 
in \eqref{def:betan} are zero for $ n = 0 $.
To prove the second identity in \eqref{cancellazione} we compute by \eqref{def:h2}
and \eqref{coeffF2}
$$
(\mathtt{h}_2)^{+-}_{n,-n} = \ii (F_2)_{n,-n}^{+-}
  \big( \beta^+_{2n} - \beta^{-}_{-2n} \big) 
  = 0 \, 
$$
in view of \eqref{def:betan}.
By \eqref{coeffNuoviV}, \eqref{cancellazione}, \eqref{def:betan}, \eqref{coV1} 
we get $ (\mathtt{V}_2^{(1)})^{+-}_{n,n} =  2 n |n |$ and 
$ (\mathtt{V}_2^{(1)})^{+-}_{n,-n}=0$.

To conclude we prove \eqref{nullcoeff211'}. From \eqref{newcoeff100} we calculate 
\begin{align*}
({\mathtt a}_2^{(1)})^{+-}_{n, n} =
  (\mathtt{a}^{(0)}_2)_{n,n}^{+-}+ \ii n(\mathtt{a}_1)^{+}_{n}\beta^{-}_{n}-
  \ii n(\mathtt{a}_1)_{n}^{-}\beta^{+}_{n}-\frac{\ii n }{2}\big(  \beta^{+}_{n}(\mathtt{a}_1)_{n}^{-}-
  \beta^{-}_{n}(\mathtt{a}_1)^{+}_{n} \big)+ \frac{3}{4}\beta^{+}_{n}\beta^{-}_{n} n^2 
\end{align*}
where $\beta_{n}^{\s}$ are defined in \eqref{def:betan}.
By  \eqref{def:a0} we have 
$ (\mathtt{a}_2^{(0)})^{+-}_{n_1,n_2} =(\mathtt{a}_2)_{n_1,n_2}^{+-}
-(\mathtt{a}_1)^{+}_{n_1}(\mathtt{a}_1)^{-}_{n_2} $ and, using \eqref{coea1}, we calculate
$(\mathtt{a}_2^{(0)})^{+-}_{n,n}=\tfrac{3}{8} |n|^{5/2}$.
Futhermore, one can check directly using the formulas \eqref{coea1} and \eqref{def:betan}, 
that $(\mathtt{a}_2^{(1)})^{+-}_{n,n} = 0$. 
\end{proof}

\subsubsection{Reduction of the quadratic symbol of the transport } 

The  aim of this section is to reduce  the  transport operator 
 $ - \ii \opbw{(\mathtt{V}_{2}^{(1)} (U; x) \xi)} $ in \eqref{WWCexp3} into the  ``integrable'' one
 $ - \ii \opbw{( \zeta (U) \xi)} $ where $ \zeta (U) $ is the function, constant in  $ x $, 
 defined in \eqref{diag-part}. 
To do this we conjugate the equation \eqref{WWCexp3} under the flow of the transport equation 
\begin{equation}\label{flusso2}
\partial_{\theta} \Phi_{{2}}^{\theta} (U) = \ii \opbw( b_{2} (U; \theta, x) \xi ) \Phi_{{2}}^{\theta}(U)
\, , \quad \Phi_{{2}}^{0}(U) = {\rm Id} \, ,  
\end{equation}
where $b_{2} $ is defined as in \eqref{flow1}
in terms of a real valued function $ \beta_{2} (U; x) \in \widetilde{\mathcal F}_2^\R $. 
The flow in \eqref{flusso2}
is well-posed by Lemma \ref{buonflusso}.
We then define the new variable
\begin{equation}\label{newcoord2}
V_2:=\vect{v_2}{\ov{v_2}} =
\big({\bf \Phi}_2^{\theta}(U)[V_1]\big)_{|_{\theta=1}}:=
\left(\begin{matrix} 
\Phi_{2}^{\theta}(U)[v_1]\\
\ov{\Phi_{2}^{\theta}(U)}[\ov{v_1}]
\end{matrix}
\right)_{|_{\theta=1}}
\end{equation}
where 
$\ov{\Phi_{2}^{\theta}(U)} $
is defined as in \eqref{opeBarrato}.

\begin{lemma}\label{lem:56}
Define $ \beta_{2} \in \widetilde{\mathcal F}_2^{\R} $ with coefficients  
for $n_1,n_2\in \Z\setminus\{0\}$, 
\be\label{beta-2}
(\beta_{2})^{\s\s}_{n_1, n_2} := \frac{-(\mathtt{V}_2^{(1)})^{\s\s}_{n_1, n_2}}{\ii \s (\omega_{n_1} + \omega_{n_2})} \, , \, 
\, \s=\pm \, ,
\qquad 
(\beta_{2})^{+-}_{n_1, n_2} :=  \frac{-(\mathtt{V}_2^{(1)})^{+-}_{n_1, n_2}}{\ii (\omega_{n_1} - \omega_{n_2})} \, , 
\, n_1 \neq \pm n_2 \, , 
\ee
and $(\beta_{2})_{0,0}^{\s \s} := 0 $, $ (\beta_2)^{+-}_{n,\s n}:=0$,  $ \s = \pm $, 
where $\mathtt{V}_{2}^{(1)} $ is the real-valued function defined  in Lemma \ref{lem:5.5}. 
If $v_1$ solves \eqref{WWCexp3} then the function $ v_2 $ in \eqref{newcoord2} solves
\begin{equation}\label{WWC3}
\begin{aligned}
\pa_t v_2 & =   \opbw{\big( - \ii  \zeta(U)\xi  - \ii  (1 +\mathtt{a}_2^{(2)} )|\x|^{\frac{1}{2}}
+ \mathtt{r}_1^{(1)} +{\mathtt{r}}^{(2)}_2  +  H_{\geq 3}^{(2)} \big)} v_2+ R^{(2)}(U)[V_2]
\end{aligned}
\end{equation}
where:

\begin{itemize}

\item $\zeta(U)\in \wtilde{{\mathcal F}}^\R_2$ is the integrable function defined in \eqref{diag-part};

\item $\mathtt{a}_2^{(2)} \in \wtilde{{\mathcal F}}^\R_2$ satisfies 
\begin{equation}\label{A22bis}
\mathtt{a}_2^{(2)} :=\mathtt{a}_2^{(1)} -\frac{1}{2}(\beta_2 )_{x} \, , 
\qquad (\mathtt{a}_2^{(2)})^{+-}_{n,n} = 0 \, ;
\end{equation}

\item $\mathtt{r}_1^{(1)}\in \widetilde{\Gamma}_{1}^{- \frac{1}{2}}$ is the same symbol in \eqref{WWCexp3},
 and ${\mathtt{r}}^{(2)}_2 \in \widetilde{\Gamma}^{-\frac{1}{2}}_{2}$;
 
\item ${H}^{(2)}_{\geq 3} \in \Gamma^{1}_{K,K',3}$ is an admissible symbol, 
and  ${R}^{(2)}(U) $ is a $1\times2$ matrix of smoothing operators in $\Sigma\mathcal{R}^{-\rho}_{K,K' ,1}$.
\end{itemize}
\end{lemma}

\begin{proof}
The function $\beta_2$  is real valued since  
the coefficients $(\mathtt{V}_2)_{n_1,n_2}^{\s\s'}$ 
of the real function 
$\mathtt{V}_2^{(1)}$ in \eqref{eq:V21} satisfy \eqref{condreal2}. 
In order to conjugate \eqref{WWCexp3} under the map $ \Phi_2^{\theta} $ in \eqref{newcoord2} 
we apply  Lemmata \ref{conjFlowpara} and \ref{conjFlowpara2}. 
By  \eqref{simbLie} and \eqref{expg}, and since $ \beta_2 $ is quadratic in $ u $, 
the only quadratic contributions  are 
$  \opbw( \{ \beta_2 \xi ,-\ii |\x|^{\frac{1}{2}} \} ) v_2+\ii  \opbw{(\pa_t \beta_2 \xi )} v_2 $, implying  
\be\label{WWCexp4}
\pa_t v_2  = \opbw{\big( \ii  (-  \mathtt{V}_{2}^{(1)} +   \pa_t \beta_2  ) \xi  +
\ii \frac{(\beta_2)_x}{2}   |\x|^{\frac{1}{2}} 
-\ii (1 +\mathtt{a}_2^{(1)} )|\x|^{\frac{1}{2}} 
+  \mathtt{r}_1^{(1)} +\widetilde{\mathtt{r}}^{(1)}_2  +  H_{\geq 3}^{(1)} \big)} v_2  + R(U)[V_2] 
\ee
where $\widetilde{\mathtt{r}}^{(1)}_2 $ is a symbol  in $ \widetilde{\Gamma}_2^{- \frac12}$,  
$H_{\geq 3}^{(1)} \in \Gamma^{1}_{K,K' ,3} $ is a new admissible symbol, 
and  $ R(U)$ is a $1\times 2 $ 
matrix of smoothing operators in  $\Sigma\mathcal{R}^{-\rho}_{K,K',1} $ (by renaming $ \rho - 2m  $ as $ \rho $).

By the choice of $\beta_2$  in \eqref{beta-2}, using \eqref{eq:415bis} and reasoning as in the proof of Lemma 
\ref{lem:5.5} we have 
\[
 -\mathtt{V}_{2}^{(1)} + \pa_t \beta_2 
 =-  \Big(  \frac{1}{2\pi}\sum_{n \in \Z\setminus\{0\}}   
(\mathtt{V}_2^{(1)})^{+-}_{n,n} |u_{n}|^2 + 
(\mathtt{V}_2^{(1)})^{+-}_{n,-n} u_n \ov{u_{- n}} e^{\ii 2 n x }  \Big) +  f_{\geq3}   
\]
where $ f_{\geq3} $ is a function in $ {\mathcal F}^\R_{K, 1, 3} $. 
Therefore, by  \eqref{nullcoeff211}, we get 
\begin{equation}\label{equaomo}
 -\mathtt{V}_{2}^{(1)}  + \pa_t \beta_2   = - \zeta(U)  + f_{\geq 3}  
\end{equation}  
with $ \zeta (U) $ defined in \eqref{diag-part}. 
System \eqref{WWCexp4} and \eqref{equaomo} imply \eqref{WWC3} where
$\mathtt{a}_2^{(2)}$ is the function defined in \eqref{A22bis}. 
Recalling \eqref{nullcoeff211'} we deduce that $ (\mathtt{a}_2^{(2)})^{+-}_{n,n} = 0 $. 
\end{proof}

\begin{remark} \label{rem:conc}
For a cubic vector field of the form $ \opbw (  \ii {\mathtt V}_2 (U; x) \xi ) [u] $ 
with a real valued function $ {\mathtt V}_2 $ in $ {\widetilde {\mathcal F} }^\R_2 $ 
the reversibility and even-to-even properties imply 
$ ({\mathtt V}_2)_{n,-n}^{+-} = 0 $, 
$ ({\mathtt V}_2)_{n,n}^{+-} = - ({\mathtt V}_2)_{-n, -n}^{+-} $,  
$ ({\mathtt V}_2)_{n,n}^{+-} \in \R $, 
in agreement with \eqref{nullcoeff211}.
We remark that the cancellation 
$ ({\mathtt V}_2)_{n,-n}^{+-} = 0 $  is essential for the whole argument to work.
Note also that on the subspace of even functions $u$ one has $\zeta (U) = 0$.
\end{remark}

\subsection{Integrability at order $ 1/2$ and $ 0 $}\label{inteord12}
The first aim of this subsection is to reduce the operator
$ - \ii \opbw \big(\mathtt{a}^{(2)}_2(U; x) |\xi|^{1/2} \big) $ in  \eqref{WWC3}
to an integrable one. 
It actually turns out that, thanks to \eqref{A22bis}, we  reduce it to  the Fourier multiplier
$- \ii  |D|^{1/2} $, see \eqref{WWCexp7}. 
This is done in two steps. 
In \ref{sec:SEMIFIO} we  apply a transformation which is a paradifferential 
``semi-Fourier integral operator", generated as the flow of \eqref{flusso44}. 
Then, in Subsection \ref{scarlatto} 
we apply the para-differential version of a torus diffeomorphism which is ``almost'' time independent, 
see \eqref{quasiconst}-\eqref{choice-beta2}.
Eventually we deal with the operators of order $0$ in Subsection \ref{inteord0}.

\subsubsection{Elimination of the time dependence at order $ 1/2$ up to $O(u^3)$}\label{sec:SEMIFIO}

We conjugate \eqref{WWC3} under the flow
\begin{equation}\label{flusso44}
\partial_{\theta} \Phi_{{3}}^{\theta}(U) = \ii \opbw( \beta_{3} (U; x) |\xi|^{\frac{1}{2}} ) \Phi_{{3}}^{\theta}(U) \, , 
\quad \Phi_{{3}}^{0}(U)  = {\rm Id} \, ,  
\end{equation}
where $ \beta_{3} (U; x) \in \widetilde{\mathcal F}_2^\R $ is a real valued function. 
We set
\begin{equation}\label{newcoord400}
V_3 :=\vect{v_3}{\ov{v_3}}=\big({\bf \Phi}_{3}^{\theta}(U)[V_2]\big)_{|_{\theta=1}}=
\left(\begin{matrix} 
\Phi_{3}^{\theta}(U)[v_2]\\
\ov{\Phi_{3}^{\theta}(U)}[\ov{v_2}]
\end{matrix}
\right)_{|_{\theta=1}} 
\end{equation}
where 
$\ov{\Phi_{3}^{\theta}(U)} $ is defined as in \eqref{opeBarrato}.

\begin{lemma}
Define 
$ \beta_3 \in \widetilde{\mathcal{F}}_2^\R  $  with coefficients 
\be\label{beta-4}
(\beta_{3})^{\s\s}_{n_1, n_2} := \frac{-(\mathtt{a}_{2}^{(2)})^{\s\s}_{n_1, n_2}}{\ii \s(\omega_{n_1} + \omega_{n_2})} \, , \;\;\s=\pm \, ,
\quad 
(\beta_{3})^{+-}_{n_1, n_2} :=  \frac{-(\mathtt{a}_{2}^{(2)})^{+-}_{n_1, n_2}}{\ii (\omega_{n_1} - \omega_{n_2})} \, , \ n_1 \neq \pm n_2 \ , 
\ee
and $(\beta_{3})_{0,0}^{\s \s} := 0 $, $(\beta_3)^{+-}_{n,\s n}:=0$,  $ \s = \pm $, 
where  $\mathtt{a}_2^{(2)} $ is defined in \eqref{A22bis}. 
If $v_2$ solves \eqref{WWC3} then
\begin{equation}\label{WWCexp6}
\pa_t v_3 =   \opbw{\big( - \ii  \zeta(U) \xi   - \ii (1+\mathtt{a}_2^{(3)})|\x|^{\frac{1}{2}}}
 +\ii  \mathtt{b}^{(3)}_2 \sgn (\x)
 + \mathtt{r}_1^{(1)} +{\mathtt{r}}^{(3)}_2 + {H}^{(3)}_{\geq 3} \big)v_3+ R^{(3)}(U)[V_3] 
\end{equation}
where 
\begin{equation}\label{nuovicoeff400}
\mathtt{a}^{(3)}_2 := \frac{1}{2\pi}\sum_{n \in \Z\setminus\{0\}} 
(\mathtt{a}^{(2)}_{2})_{n, - n}^{+-} u_n \ov{u_{-n}} e^{\ii 2 n x } , \qquad 
\mathtt{b}^{(3)}_2 :=\frac{1}{2}(\beta_{3})_x \, ,
\end{equation}
$\mathtt{r}_1^{(1)}\in \widetilde{\Gamma}_{1}^{- \frac{1}{2}}$ is the same symbol in \eqref{WWCexp3},
${\mathtt{r}}^{(3)}_2 \in \widetilde{\Gamma}^{-\frac{1}{2}}_{2}$, 
${H}^{(3)}_{\geq 3} \in \Gamma^{1}_{K,K',3}$ is admissible,
and $ R^{(3)}(U) $ is a $1\times2$ matrix of smoothing operators in  
$\Sigma\mathcal{R}^{-\rho}_{K,K',1}$.
Moreover 
\begin{equation}\label{condcoeffB400}
({\mathtt b}_{2}^{(3)})^{+-}_{n, n}=({\mathtt b}_{2}^{(3)})^{+-}_{n, -n}=0 \, . 
\end{equation}
\end{lemma}

\begin{proof}
By \eqref{beta-4} and  \eqref{condreal2} we deduce that $ \beta_{3} $  is a real function.
To conjugate system  \eqref{WWC3} 
we apply Lemmata \ref{flusso12basso} and \ref{flusso12basso2} with $m\rightsquigarrow 1/2$ and $m'\rightsquigarrow1$.
The only new contribution at quadratic degree of homogeneity and positive order is 
 $ \opbw( \{ \beta_3 |\x|^{\frac{1}{2}} , -\ii |\x|^{\frac12} \} ) $ and 
 $ \ii \opbw( \pa_t \beta_3 |\x|^{\frac{1}{2}} ) $. 
 Then we have 
\begin{equation}\label{WWCexp5}
\pa_t v_3 =    \opbw{ \Big( - \ii \zeta(U) \xi  
-\ii (1+\mathtt{a}_2^{(2)}} -  \pa_t \beta_3  )|\x|^{\frac{1}{2} }  + 
\ii  \frac{(\beta_3)_x}{2}  \sgn(\x) 
+ \mathtt{r}_{1}^{(1)} + \widetilde{\mathtt{r}}^{(3)}_2 
 + H_{\geq 3} \big)v_3+ R(U)[V_3]
\end{equation}
where $\widetilde{\mathtt{r}}^{(3)}_2 \in \widetilde{\Gamma}^{-\frac{1}{2}}_{2}$, the symbol
$ H_{\geq 3} \in \Gamma^{1}_{K,K',3}$ is  admissible 
and $ R(U) $ is a $1\times2$ matrix of smoothing operators in  
$\Sigma\mathcal{R}^{-\rho}_{K,K',1}$. 
By \eqref{beta-4} and \eqref{eq:415bis} we have
\begin{equation}\label{equaomo40}
\begin{aligned}
-\ii \mathtt{a}^{(2)}_2  +\ii \pa_t \beta_3  =
-
\ii \, {  \sum_{n \in \Z\setminus\{0\}} (\mathtt{a}^{(2)}_{2})_{n, n}^{+-} |u_n|^2 + 
(\mathtt{a}^{(2)}_{2})_{n,-n}^{+-} u_n \ov{u_{-n}} e^{\ii 2 n x }  }
\end{aligned}
\end{equation}
up to a function $f_{\geq 3} $ in $ \mathcal{F}_{K,1,3}^\R $. The conjugation of the remainder $R^{(2)}(U)$ in \eqref{WWC3}
is performed as in {\sc Step 5} of Lemma \ref{lem:5.5}. 
In conclusion, \eqref{WWCexp5}-\eqref{equaomo40} and the vanishing of the coefficients 
\eqref{A22bis} 
imply  \eqref{WWCexp6}-\eqref{nuovicoeff400}. Finally, \eqref{condcoeffB400} follows from
$ ({\mathtt b}_{2}^{(3)})^{+-}_{n_1, n_2}= 
\frac{1}{2}(\beta_{3})^{+-}_{n_1,n_2}(\ii n_1-\ii n_2)$.
\end{proof}

\begin{remark}\label{vedizeta}
The cancellation $(\mathtt{a}^{(2)}_{2})_{n, n}^{+-} = 0$ in \eqref{A22bis}
does not follow from the properties of reality, parity and reversibility of the water waves equations.
This appears to be an intrinsic property of the gravity water waves system \eqref{eq:113}
and is, of course, 
in agreement with the normal form identification of Section \ref{sec:NFI}.
We notice however that one would  not need to prove this property 
for the sequel of the proof, since the symbol  
$ \ii \, \sum_{n \in \Z\setminus\{0\}} (\mathtt{a}^{(2)}_{2})_{n,n}^{+-} |u_n|^2 $ is integrable and  
the coefficients $ (\mathtt{a}^{(2)}_{2})_{n,n}^{+-} $ are real 
(by \eqref{condreal2} and since the function $\mathtt{a}_{2}^{(2)}$ is  real).
\end{remark}

\subsubsection{Elimination of the $x$-dependence at order $ 1/2$ up to $ O(u^3)$}\label{scarlatto}

The aim of this section is to  cancel out the operator
\begin{equation}\label{QI1}
- \ii \opbw \Big(\frac{1}{2\pi} \sum_{n\in \mathbb{Z}}(\mathtt{a}^{(2)}_{2})_{n, - n}^{+-} u_n \ov{u_{-n}} e^{\ii 2 n x }\Big) 
\end{equation}
arising by the non-integrable part of the function  $\mathtt{a}^{(3)}_2(U;x)$ in \eqref{nuovicoeff400}.
We argue in a way inspired by Section 12 in \cite{BBHM}, noticing 
that the symbol in \eqref{QI1} is a prime integral up to cubic terms $O(u^3)$.

We conjugate \eqref{WWCexp6} under the flow
\begin{equation}\label{quasiconst}
\partial_{\theta} \Phi_{4}^{\theta}(U) = \ii  \opbw{( b_{4} (U; \theta, x) \xi )} \Phi_{4}^{\theta}(U) \, , \quad 
\Phi_{4}^{0}(U) = {\rm Id} \, ,  
\end{equation}
where $ b_{4} $ is defined as in \eqref{flow1} in terms of a real valued 
function $ \beta_{4} (U; x)  \in \widetilde{\mathcal F}_2^\R $ 
of the same form of the symbol in \eqref{QI1}, i.e.  
\be\label{choice-beta2}
\beta_4 (U; x) = \frac{1}{2\pi}\sum_{n \in \Z\setminus\{0\} } (\beta_4)_{n, - n}^{+-} u_n \ov{u_{-n}} e^{\ii 2 n x } \, . 
\ee
The flow in \eqref{quasiconst} is well-posed by Lemma \ref{buonflusso}. We set
\begin{equation}\label{newcoord500}
V_4 :=\vect{v_4}{\ov{v_4}}=\big({\bf \Phi}_{4}^{\theta}(U)[V_3]\big)_{|_{\theta=1}}=
\left(\begin{matrix} 
\Phi_{4}^{\theta}(U)[v_3]\\
\ov{\Phi_{4}^{\theta}(U)}[\ov{v_3}]
\end{matrix}
\right)_{|_{\theta=1}}
\end{equation}
where $\ov{\Phi_{4}^{\theta}(U)} $ is defined as in \eqref{opeBarrato}. 

\begin{lemma}\label{lem:510}
Define  the function $ \beta_4  \in \widetilde{\cal F}_2^\R $ as in \eqref{choice-beta2}  with coefficients  
\be\label{zero-12}
(\beta_4)^{+-}_{n, - n} := \frac{(\mathtt{a}^{(2)}_2)_{n, - n}^{+-} }{\ii  n } \, , \, \quad n\neq 0 \, , 
\qquad (\beta_4)^{+-}_{0, 0}:=0  \, .
\ee
If  $v_3$ solves  \eqref{WWCexp6} then 
\begin{equation}\label{WWCexp7}
\begin{aligned}
\pa_t v_4 &=   \opbw{\big( -\ii  \zeta(U) \xi  - \ii |\x|^{\frac{1}{2}}
+ \ii  \mathtt{b}^{(3)}_2 \sgn(\x)
+  \mathtt{r}_1^{(1)}+{\mathtt{r}}^{(3)}_2 + {H}^{(4)}_{\geq 3} \big)}v_4+ R^{(4)}(U)[V_4] 
\end{aligned}
\end{equation}
where the symbols  $ \mathtt{b}_2^{(3)}$, $\mathtt{r}_1^{(1)}$, $ \mathtt{r}^{(3)}_2$
are the same of equation \eqref{WWCexp6},  
the symbol ${H}^{(4)}_{\geq 3} \in \Gamma^{1}_{K,K',3}$ is admissible
and $ R^{(4)}(U) $ is a $1\times2$ matrix of smoothing operators in  
$\Sigma\mathcal{R}^{-\rho}_{K,K',1}$.  
\end{lemma}

\begin{proof}
In order  to  conjugate \eqref{WWCexp6} we apply Lemmata \ref{conjFlowpara} and \ref{conjFlowpara2}.
The contribution coming from the conjugation of $ \pa_t $ is $ \ii  \opbw{  \big((\pa_t \beta_{4} ) \xi  \big)} v_4 $
plus a paradifferential operator with  symbol $ \ii (- (\b_4)_x (\b_4)_t + g_{\geq 3}) \xi $ (see \eqref{expg}),
which is admissible,  
 and a smoothing remainder in 
$ \Sigma {\mathcal R}_{K, 1,1}^{- \rho }$.
Recalling \eqref{eq:415bis} we have
\begin{equation} \label{quasi-pi}
\begin{aligned}
\frac{d}{dt} \sum_{n \in \Z\setminus\{0\} } (\beta_4)_{n, - n}^{+-} u_n \ov{u_{-n}} e^{\ii 2 n x }  
&  =  \sum_{n \in \Z\setminus\{0\} } (\beta_4)_{n, - n}^{+-}  
 \big( - \ii \omega_n u_n   \ov{u_{-n}}  +   u_n    \ii \omega_{-n} \ov{ u_{-n}}  \big)  e^{\ii 2 n x }+
 h_{\geq 3} \\
 &  =  h_{\geq 3}
 \end{aligned}
 \end{equation}
 because $ \omega_{-n} = \omega_n $ and where, arguing as in the proof of Lemma \ref{lem:5.5},
$ h_{\geq 3} $ is a   function  in $ \mathcal{F}^{\R}_{K,1, 3} $.
This implies that 
the function $ \pa_t \beta_4  $ is in $ \mathcal{F}^{\R}_{K,1, 3} $
and therefore  $\ii (\pa_t \beta_{4} ) \xi $ is an admissible symbol.

Lemma \ref{conjFlowpara} implies that 
 the conjugation of the spatial  operator in  \eqref{WWCexp6} is a paradifferential operator with  symbol 
\be\label{dallosp62}
- \ii  \zeta(U) \xi   - \ii (1+\mathtt{a}_2^{(3)})|\x|^{\frac{1}{2}} + 
 \{ \beta_4 \xi,- \ii |\xi|^{1/2}\} 
 +\ii  \mathtt{b}^{(3)}_2 \sgn (\x)
 + \mathtt{r}_1^{(1)} +{\mathtt{r}}^{(3)}_2 
\ee 
 plus a symbol in $ \Sigma \Gamma^{-3/2}_{K,K',1} $ an admissible symbol and a smoothing operator in 
 $ \Sigma {\mathcal R}_{K, 1,1}^{- \rho + 1 }$.  
Notice that  $ \{ \beta_4 \xi,- \ii |\xi|^{\frac{1}{2}}\} = \frac{\ii}{2} (\beta_4)_x |\xi |^{\frac12 }$
and that this equals $\mathtt{a}_2^{(3)} \ii |\x|^{1/2}$
 in view of the definitions of $\beta_4$ in \eqref{choice-beta2} and \eqref{zero-12}, 
 and of $\mathtt{a}_2^{(3)} $ in \eqref{nuovicoeff400}.
 It follows that the symbol in \eqref{dallosp62} reduces to  
$$
- \ii  \zeta(U) \xi - \ii
  | \x |^{\frac{1}{2}} + \ii  \mathtt{b}^{(3)}_2 \sgn (\x) + \mathtt{r}_1^{(1)} +{\mathtt{r}}^{(3)}_2 \, . 
$$  
We have therefore obtained \eqref{WWCexp7} (after slightly redefining $\rho$) as desired. 
\end{proof}

\subsubsection{Integrability at order $ 0 $}\label{inteord0}
Our aim here is to  eliminate in \eqref{WWCexp7} the zero-th order paradifferential operator 
$ \opbw\big( \ii  \mathtt{b}^{(3)}_2 \sgn (\x)\big) $.
We conjugate \eqref{WWCexp7} with  the flow 
\begin{equation}\label{flusso668}
\partial_{\theta} \Phi_{5}^{\theta}(U) =   
\opbw{(  \ii \b_5(U;x) \,  \sign (\xi) )} \Phi_5^{\theta}(U) \, , \quad \Phi_5^{0}(U) = {\rm Id} \, ,  
\end{equation}
where $ \b_{5} (U; x)  \in \widetilde{\mathcal F}_2^\R $
is a real valued function. 
We introduce the variable
\begin{equation}\label{newcoord700}
V_5 :=\vect{v_5}{\ov{v_5}}=\big({\bf \Phi}_5^{\theta}(U)[V_4]\big)_{|_{\theta=1}}=
\left(\begin{matrix} 
\Phi_{5}^{\theta}(U)[v_4]\\
\ov{\Phi_{5}^{\theta}(U)}[\ov{v_4}]
\end{matrix}
\right)_{|_{\theta=1}}
\end{equation}
where 
$\ov{\Phi_{5}^{\theta}(U)} $ is defined as in \eqref{opeBarrato}. 

\begin{lemma}
Define $  \b_{5}  \in \widetilde{\mathcal F}_2^\R $ (of the form \eqref{eq:F2}) with 
\begin{align}
& \label{betaquadratico}
(\b_{5})^{\s\s}_{n_1, n_2} := \frac{(\mathtt{b}_{2}^{(3)})^{\s\s}_{n_1, n_2}}{\ii \s(\omega_{n_1} + \omega_{n_2})} \, , 
\;\;\s=\pm \, ,
\quad 
(\b_{5})^{+-}_{n_1, n_2} :=  \frac{(\mathtt{b}_{2}^{(3)})^{+-}_{n_1, n_2}}{\ii (\omega_{n_1} - \omega_{n_2})} \, , \ n_1 \neq \pm n_2 \, ,
\end{align}
and $(\beta_{5})_{0,0}^{\s \s} := 0 $, $(\beta_5)^{+-}_{n,\s n}:=0$,  $ \s = \pm $.
If $v_4 $ solves  \eqref{WWCexp7} then 
\begin{equation}\label{WWCexp9}
\begin{aligned}
\pa_t v_5 &=  
 \opbw{\big( -\ii  \zeta (U) \xi 
 - \ii  |\x|^{\frac{1}{2}} +\mathtt{r}_1^{(5)}+  \mathtt{r}^{(5)}_2 +H^{(5)}_{\geq 3}\big)}v_5 + R^{(5)}(U)[V_5] 
\end{aligned}
\end{equation}
where 
$\mathtt{r}_1^{(5)}\in \widetilde{\Gamma}_{1}^{-\frac{1}{2}}$, $\mathtt{r}^{(5)}_2\in \widetilde{\Gamma}_{2}^{-\frac{1}{2}}$, 
the symbol
$ H_{\geq 3}^{(5)} \in \Gamma^{1}_{K,K',3}$ is admissible, and
$ R^{(5)} (U) $ is a $1\times 2$ matrix of smoothing  operators 
in  $\Sigma\mathcal{R}^{-\rho}_{K,K',1}$.
\end{lemma}

\begin{proof}
To conjugate \eqref{WWCexp7}
we apply Lemmata \ref{flusso12basso} and \ref{flusso12basso2}. 
By \eqref{betaquadratico} we have that
$$
\opbw\big(
\ii (\mathtt{b}^{(3)}_2  +  \pa_t \b_5 )  \sign (\xi) \big) = \ii \opbw\big( \mathtt{b}_2^{(5)} \sgn (\x)
\big),
$$
up to symbols with degree of homogeneity greater than $3$, and where
$$
\mathtt{b}_{2}^{(5)}(U;x) := \frac{1}{2\pi}\sum_{n \in \Z\setminus\{0\}} (\mathtt{b}^{(3)}_{2})_{n, n}^{+-} |u_n|^2 + 
(\mathtt{b}^{(3)}_{2})_{n, - n}^{+-} u_n \ov{u_{-n}} e^{\ii 2 n x }\stackrel{\eqref{condcoeffB400}}{=}0 \, .
$$
The lemma is proved.
\end{proof}

\begin{remark}
For a cubic vector field of the form $ \opbw ( {\mathtt c} (U; x) + \ii {\mathtt b} (U; x) \sign (\xi) ) [u] $ 
with real valued functions $ {\mathtt c}, {\mathtt b} $ 
in $ {\widetilde {\mathcal F} }^\R_2 $ the reversibility and even-to-even properties imply 
$ {\mathtt c}_{n,n}^{+-} = $ 
$ {\mathtt c}_{n,-n}^{+-} =  $ $ {\mathtt b}_{n,-n}^{+-} = 0 $, 
$  {\mathtt b}_{n,n}^{+-}  = - {\mathtt b}_{-n, -n}^{+-}  $, $  {\mathtt b}_{n,n}^{+-} \in \R $. 
The fact that actually $ ({\mathtt b}^{(3)}_2)_{n,n}^{+-} =0 $ as stated in \eqref{condcoeffB400} 
follows by other properties of the water waves equations, 
and, once again, is in agreement with the normal form identification of Section \ref{sec:INF}. 
\end{remark}

In the following subsection we will be dealing with negative order operators,
and will not need additional algebraic information about the coefficients and their vanishing.

\subsection{Integrability at negative orders}\label{sec:nega}
In this section we  algorithmically reduce the linear and quadratic symbols 
$ \mathtt{r}_1^{(5)}+  \mathtt{r}^{(5)}_2 $ of order $ - 1/ 2 $ in \eqref{WWCexp9} into an integrable one, 
plus an admissible symbol. 

\begin{proposition}\label{prop:NO}
For any $ j = 0, \ldots, 2 \rho - 1 $, there exist

\begin{itemize}

\item integrable symbols $\mathtt{p}_2^{(j)} \in {\widetilde \Gamma}_2^{- \frac12} $ 
(Definition \ref{defiintegro}), 
symbols 
$$
\mathtt{q}^{(j)}(U;x,\x) \in  \Sigma\Gamma^{-m_j}_{K,K',1} \, , \quad m_{j} :=  \frac{j+1}{2} \, , 
$$
admissible symbols 
$ H_{\geq 3}^{(j)} $ in $ \Gamma^{1}_{K,K',3} $, 
and a $1\times 2$ matrix of smoothing operators $ R^{(j)} (U) $ in  $\Sigma\mathcal{R}^{-\rho}_{K,K',1}$,

\item bounded maps $ {\bf \Upsilon}^{\theta}_{j+1}(U) $, $\theta \in [0,1]$, 
  defined as the compositions of three flows generated by paradifferential 
  operators with symbols of order $ \leq 0 $,   
  (see \eqref{Upsilonjth} and \eqref{newcoordj1}, \eqref{newcoordj1bis} and \eqref{newcoord1000}) 
\end{itemize}

\noindent
such that:  if $z_j$ solves
\begin{equation}\label{WWCexpjth}
\pa_t z_j =   \opbw\big(   - \ii  \zeta (U) \xi -\ii |\x|^{\frac{1}{2}}  +  \mathtt{p}_2^{(j)}(U;\x)  
+  \mathtt{q}^{(j)}(U;x,\x) +  H_{\geq 3}^{(j)}\big)z_{j} + R^{(j)}(U)[Z_j] \, , 
\end{equation}
then the first component  $ z_{j+1} $ of the vector defined by  
\be\label{coordZj}
Z_{j+1} = \vect{z_{j+1}}{\ov{z_{j+1}}} := \big( {\bf \Upsilon}^{\theta}_{j+1}(U) \big)_{\theta=1} Z_j  
\ee
solves an equation of the form \eqref{WWCexpjth} with $ j + 1 $ instead of $ j $. 
\end{proposition}

The proof proceeds by induction.

\smallskip
\noindent
{\bf Initialization.} 
Notice that equation \eqref{WWCexp9} has the form \eqref{WWCexpjth} with $j=0$, 
denoting $z_0 := v_5$, 
$\mathtt{p}_2^{(0)} := 0$,
$\mathtt{q}^{(0)} :=\mathtt{r}_1^{(5)} +\mathtt{r}^{(5)}_2 \in \Sigma\Gamma^{- 1/2 }_{K,K',1}  $, 
and renaming $ H_{\geq 3}^{(0)} $ the admissible symbol $ H_{\geq 3}^{(5)}$ in  \eqref{WWCexp9}
and $ R^{(0)}(U)$ the smoothing operator $ R^{(5)}(U) $. 

We remark that the \emph{integrable} 
corrections $\mathtt{p}_2^{(j)}$ in \eqref{WWCexpjth} (initially $\mathtt{p}_2^{(0)} = 0 $) 
are generated by the reductions on quadratic symbols made in Lemma \ref{def:ln2} below.

\smallskip
\noindent
{\bf Iteration.} The aim of the iterative procedure is to  
cancel out the symbol $\mathtt{q}^{(j)}$ up to a symbol of order $-m_j-1/2$. This is done in two steps.

\smallskip
\noindent
{\bf Step 1: Elimination of the linear symbols of negative order.} 
We expand the symbol 
$  {\mathtt q}^{(j)}  =  {\mathtt q}_1^{(j)} + {\mathtt q}_2^{(j)}  + \cdots $
with $ {\mathtt q}_l^{(j)}  \in \Gamma^{-m_j}_l $, $ l = 1,2 $. 
In order to  eliminate the operator  $ \opbw{( {\mathtt q}_1^{(j)}(U; x, \xi)}) $   in \eqref{WWCexpjth}
we conjugate it by the flow 
\begin{equation}\label{flussojth}
\partial_{\theta} \Phi_{\g^{(1)}_{j+1}}^{\theta}(U) =   \opbw{( \g^{(1)}_{j+1}(U;x,\x) )} 
\Phi_{\g^{(1)}_{j+1}}^{\theta}(U) \, , 
\quad \Phi_{\g^{(1)}_{j+1}}^{0}(U) = {\rm Id} \, ,  
\end{equation}
where $ \g^{(1)}_{j+1}(U; x;\x) $ is a symbol in $ \widetilde{\Gamma}^{-m_j}_1 $.  
The flow  \eqref{flussojth} is well posed because the order of $ \g^{(1)}_{j+1} $ is negative.  
  We introduce the new variable 
\begin{equation}\label{newcoordj1}
\widetilde{Z}_{j+1}:=
\vect{\widetilde{z}_{j+1}}{\ov{\widetilde{z}_{j+1}}} =
\big(\mathcal{A}^{\theta}_{j+1,1}(U)[Z_{j}]\big)_{|_{\theta=1}}=
\left(\begin{matrix} 
\Phi_{\g^{(1)}_{j+1}}^{\theta}(U)[{z}_{j}]\\
\ov{\Phi_{\g^{(1)}_{j+1}}^{\theta}(U)}[\ov{{z}_j}]
\end{matrix}
\right)_{|_{\theta=1}}
\end{equation}
where the map $\ov{\Phi_{\gamma^{(1)}_{j+1}}^{\theta}(U)} $
is defined as in \eqref{opeBarrato}. 

\begin{lemma}\label{def:ln1}
Define $ \g^{(1)}_{j+1} \in \widetilde{\Gamma}_1^{-m_j}$  with coefficients  
\be\label{def:gamma}
(\g^{(1)}_{j+1})^+_n :=  \frac{(\mathtt{q}^{(j)}_1)^+_n}{\ii \omega_n}  \, , 
\quad (\g_{j+1})^-_n :=  \frac{-(\mathtt{q}^{(j)}_1)^-_n}{\ii \omega_n}  \, , 
\quad n \neq 0 \, , \quad (\g^{(1)}_{j+1})^{\s}_0 :=0 \, , \;\;\s=\pm \, .
\ee
If $ z_{j} $ solves  \eqref{WWCexpjth} then 
\begin{equation}\label{WWCexpjthfalso2}
\begin{aligned}
\pa_t \widetilde{z}_{j+1} & =   \opbw(   - \ii \zeta(U) \xi  - \ii |\x|^{\frac{1}{2}} + 
 \mathtt{p}_2^{(j)}(U;\x)  
 +  \widetilde{\mathtt{q}}^{(j)}_2(U;x,\x) 
 +  \widetilde{\mathtt{k}}^{(j)}_1(U;x,\x) + \widetilde{\mathtt{k}}^{(j)}_2(U;x,\x)  \big)\widetilde{z}_{j+1}
 \\& \quad   +\opbw\big(  H^{(j)}_{\geq 3} \big)\widetilde{z}_{j+1}
 + R^{(j)}(U)[\widetilde{Z}_{j+1}]
\end{aligned}
\end{equation}
where $  \mathtt{p}_2^{(j)}(U;\x)  \in {\widetilde \Gamma}^{- \frac12}_2 $ is the same of \eqref{WWCexpjth}, 
$$ 
\widetilde{\mathtt{q}}^{(j)}_2   \in \widetilde \Gamma_{2}^{-m_j} \, , 
\quad \widetilde{\mathtt{k}}^{(j)}_1\in\widetilde{\Gamma}_{1}^{-m_j-\frac{1}{2}} \, , \quad
\widetilde{\mathtt{k}}^{(j)}_2\in\widetilde{\Gamma}_{2}^{-m_j-\frac{1}{2}} \, , 
$$ 
the symbol $ H^{(j)}_{\geq 3} \in \Gamma^{1}_{K,K'+1,3}$ 
is admissible 
and $ R^{(j)} (U) $ 
is a $1\times2$ matrix of smoothing operators 
 in  $\Sigma\mathcal{R}^{-\rho}_{K,K',1}$.
\end{lemma}

\begin{proof}
In order to conjugate \eqref{WWCexpjth}
we apply Lemmata \ref{flusso12basso} and \ref{flusso12basso2}.
The only contributions at homogeneity degree $ 1 $ and order $ - m_j $ are given by 
$ \opbw{( \mathtt{q}^{(j)}_1  +  \pa_t\g^{(1)}_{j+1}  )}  $
up to smoothing remainders. From the time contribution a symbol which has homogeneity
2 and order less or equal $-m_{j}-1/2$ appears (see the term $r_1$ in \eqref{coniugato1001} of Lemma \ref{flusso12basso2}).
By \eqref{def:gamma} and \eqref{eq:415bis} we have that
$$
\mathtt{q}^{(j)}_1   +  \pa_t \g^{(1)}_{j+1} = {\mathtt q}_{2,j} + {\mathtt q}_{\geq 3} \, , 
\quad {\mathtt q}_{2,j}  \in {\widetilde \Gamma}^{-m_j}_{2} \, , \quad
{\mathtt q}_{\geq 3}  \in \Gamma^{-m_j}_{K,1,3} \, ,
$$
and we set $ \widetilde{\mathtt{q}}^{(j)}_2  := \mathtt{q}^{(j)}_2 + {\mathtt q}_{2,j}$,
and absorb ${\mathtt q}_{\geq 3}$ in the admissible symbol  $H^{(j)}_{\geq 3}$.
The contributions in \eqref{WWCexpjthfalso2} 
at order less or equal  $-m_j-\frac{1}{2}$, and homogeneity $\leq 2$
come from the conjugation of the spatial operator $-\ii|\x|^{1/2}$.
In particular, using formula \eqref{coniugato1000}, we can set $\widetilde{\mathtt{k}}^{(j)}_1 := - 
\frac{\ii}{2} (\gamma_{j+1}^{(1)})_x |\x|^{-\frac{1}{2}}\sign(\x)$  
and obtain \eqref{WWCexpjthfalso2} with some $\widetilde{\mathtt{k}}^{(j)}_2$ in $\Gamma^{-m_{j}-1/2}_2$. 
\end{proof}

\smallskip
\noindent
{\bf Step 2: Reduction of the quadratic symbols of negative order.} 
We now  cancel out
 the symbol $\widetilde{\mathtt{q}}^{(j)}_2$ in \eqref{WWCexpjthfalso2}, up to an integrable one
 and a lower order symbol.
Following Section \ref{inteord12} we use two different
 transformations.

\smallskip
\noindent
{\sc Elimination of the time dependence up to $O(u^{3})$.} 
We consider the flow  generated by
\begin{equation}\label{flussojth3}
\partial_{\theta} \Phi_{\g^{(2)}_{j+1}}^{\theta}(U) =   
\opbw{( \g^{(2)}_{j+1}(U;x,\x) )} \Phi_{\g^{(2)}_{j+1}}^{\theta}(U) \, , 
\quad \Phi_{\g^{(2)}_{j+1}}^{0}(U) = {\rm Id} \, ,  
\end{equation}
where $ \g^{(2)}_{j+1}(U; x;\x) $ 
is a symbol in  $\widetilde{\Gamma}^{-m_j}_{2}$. 
We introduce the new variable
\begin{equation}\label{newcoordj1bis}
\breve{Z}_{j+1}:=\vect{\breve{z}_{j+1}}{\ov{\breve{z}_{j+1}}}=
\big(\mathcal{A}^{\theta}_{j+1,2}(U)[\widetilde{Z}_{j}]\big)_{|_{\theta=1}}=
\left(\begin{matrix} 
\Phi_{\g^{(1)}_{j+1}}^{\theta}(U)[\widetilde{z}_{j}]\\
\ov{\Phi_{\g^{(1)}_{j+1}}^{\theta}(U)}[\ov{\widetilde{z}_j}]
\end{matrix}
\right)_{|_{\theta=1}} 
\end{equation}
where  the map $\ov{\Phi_{\gamma^{(2)}_{j+1}}^{\theta}(U)} $
is defined as in \eqref{opeBarrato}. 

\begin{lemma}\label{def:ln2}
Let $ \g^{(2)}_{j+1}(U; x;\x) $ be a symbol in  $\widetilde{\Gamma}^{-m_j}_{2} $  of the form \eqref{eq:F2} 
with coefficients 
\begin{equation}\label{genegamma}
(\g^{(2)}_{j+1})^{\s\s}_{n_1, n_2} := \frac{(\widetilde{{\mathtt q}}^{(j)}_2)^{\s\s}_{n_1, n_2}}{\ii \s(\omega_{n_1} + \omega_{n_2})} \, , \;\;\s=\pm \, ,
\quad 
(\g^{(2)}_{j+1})^{+-}_{n_1, n_2} := 
 \frac{(\widetilde{{\mathtt q}}^{(j)}_2)^{+-}_{n_1, n_2}}{\ii (\omega_{n_1} - \omega_{n_2})} \, , \;\;n_1 \neq \pm n_2 \, .
\end{equation}
If $\widetilde{z}_{j}$ solves  \eqref{WWCexpjthfalso2} then 
 \begin{equation}\label{WWCexpjthfalso3}
\begin{aligned}
\pa_t \breve{z}_{j+1} &=   \opbw \big( - \ii |\x|^{\frac{1}{2}} - \ii \zeta (U) \xi  +
 \mathtt{p}_2^{(j)}(U;\x) \big) \breve{z}_{j+1}
\\
&
\quad  +\opbw{\Big(  \big( \sum_{n \in \Z\setminus\{0\}} (\widetilde{{\mathtt q}}^{(j)}_2)_{n, n}^{+-}(\x) |u_n|^2 + 
(\widetilde{{\mathtt q}}^{(j)}_2)_{n, - n}^{+-}(\x) u_n \ov{u_{-n}} e^{\ii 2 n x } \big)  \Big)}\breve{z}_{j+1}\\
& \quad +\opbw \big( \breve{\mathtt{k}}_1^{(j)} (U; x, \xi) + \breve{\mathtt{k}}_2^{(j)}(U; x, \xi)  +  H_{\geq 3}^{(j)}\big)\breve{z}_{j+1}+
 R^{(j)}(U)[\breve{Z}_{j+1}]
\end{aligned}
\end{equation}
where  
$\breve{\mathtt{k}}_1^{(j)} \in \widetilde{\Gamma}_1^{- m_j - \frac12} $, 
$\breve{\mathtt{k}}_2^{(j)} \in \widetilde{\Gamma}_2^{- m_j - \frac12} $, 
the symbol  $ H_{\geq 3}^{(j)} \in \Gamma^{1}_{K,K',3}$ is admissible and
$ R^{(j)}(U) $ is a $1\times2$ matrix of smoothing operators  in  $\Sigma\mathcal{R}^{-\rho}_{K,K',1}$.  
\end{lemma}

\begin{proof}
In order to conjugate  \eqref{WWCexpjthfalso2}
we apply  Lemmata  \ref{flusso12basso} and \ref{flusso12basso2}. The 
contributions at order $ - m_j $ and degree $ 2 $  are given by
$ \opbw\big(\widetilde{\mathtt{q}}_2^{(j)} + \pa_t\gamma_{j+1}^{(2)} \big) $.
All the other contributions have homogeneity greater or equal $3$ and are admissible.
By the choice of $ \gamma_{j+1}^{(2)} $ in \eqref{genegamma} we have 
$$
\widetilde{\mathtt{q}}_2^{(j)} +
\pa_t\gamma_{j+1}^{(2)} =\frac{ 1}{2\pi}
 \sum_{n \in \Z\setminus\{0\}} (\widetilde{{\mathtt q}}^{(j)}_2)_{n, n}^{+-}(\x) |u_n|^2 + 
(\widetilde{{\mathtt q}}^{(j)}_2)_{n, - n}^{+-}(\x) u_n \ov{u_{-n}} e^{\ii 2 n x } 
$$
up to a symbol  in  $ \Gamma^{-m_j}_{K,1,3}$.
\end{proof}

\smallskip
\noindent
{\sc Elimination of the $x$-dependence up to $O(u^{3})$.}
In order to eliminate the non-integrable symbol 
\be\label{we-to-kill}
\frac{1}{2\pi} \sum_{n\in\Z\setminus\{0\}}(\widetilde{{\mathtt q}}^{(j)}_2)_{n, - n}^{+-}(\x) u_n \ov{u_{-n}} e^{\ii 2 n x }
 \ee
 in \eqref{WWCexpjthfalso3}  we follow the same strategy used in Subsection  \ref{scarlatto}. We
 conjugate \eqref{WWCexpjthfalso3} by  the flow 
 \begin{equation}\label{quasiconstjth}
\partial_{\theta} \Phi_{\g^{(3)}_{j+1}}^{\theta}(U)  = 
\ii  \opbw{( \g^{(3)}_{j+1} (U; x,\x)  )} \Phi_{\g^{(3)}_{j+1}}^{\theta}(U) \, , \quad 
\Phi_{\g^{(3)}_{j+1}}^{0}(U) = {\rm Id} \, ,  
\end{equation}
where $\g^{(3)}_{j+1}(U; x,\x) $ is a symbol in $ \widetilde{\Gamma}_2^{-m_j+\frac{1}{2}}  $ of the same form
\eqref{we-to-kill}, i.e.
\be\label{ultimodiffeo}
\g^{(3)}_{j+1}(U; x,\x)  := 
\frac{1}{2\pi}\sum_{n \in \Z\setminus\{0\} } (\g^{(3)}_{j+1})_{n, - n}^{+-}(\x) u_n \ov{u_{-n}} e^{\ii 2 n x }
\, . 
\ee
We introduce the new variable
\begin{equation}\label{newcoord1000}
Z_{j+1}:=\vect{z_{j+1}}{\bar{z}_{j+1}}=\big(\mathcal{A}^{\theta}_{j+1,3}(U)[\breve{Z}_{j+1}]\big)_{|_{\theta=1}} =
\left(\begin{matrix} 
\Phi_{\g^{(3)}_{j+1}}^{\theta}(U)[\breve{z}_{j+1}]\\
\ov{\Phi_{\g^{(3)}_{j+1}}^{\theta}(U)}[\ov{\breve{z}_{j+1}}]
\end{matrix}
\right)_{|_{\theta=1}}
\end{equation}
where the map $\ov{\Phi_{\gamma^{(3)}_{j+1}}^{\theta}(U)} $
is defined as in \eqref{opeBarrato}.

\begin{lemma}\label{def:ln3}
Define $ \g^{(3)}_{j+1} $ in $ \widetilde{\Gamma}_{2}^{-m_j+\frac{1}{2}} $ 
as in \eqref{ultimodiffeo} with coefficients 
\begin{equation}\label{GAMMA3JTH}
(\g^{(3)}_{j+1})_{n, - n}^{+-}(\x):=|\xi|^{\frac{1}{2}}
  \sgn (\x)\frac{1}{n}(\widetilde{{\mathtt q}}^{(j)}_2)_{n, - n}^{+-}(\x) \, , \, \quad n\neq 0\, .
\end{equation}
If  $\breve{z}_{j}$ solves \eqref{WWCexpjthfalso3} then
\begin{equation}\label{WWCexp(j+1)th}
\pa_t {z}_{j+1} =  \opbw(   -\ii  \zeta(U) \xi - \ii |\x|^{\frac{1}{2}} +  \mathtt{p}_2^{(j+1)}(U;\x)  
 +  \mathtt{q}^{(j+1)}(U;x,\x) +  H_{\geq 3}^{(j+1)}\big){z}_{j+1}+  R^{(j+1)}(U)[{Z}_{j+1}]
\end{equation}
where  $\mathtt{p}_2^{(j+1)}(U;\x) $ is an integrable symbol  in $   {\widetilde \Gamma}^{- \frac12}_2 $,  
 $ \mathtt{q}^{(j+1)}(U;x,\x) $ is in $ \Sigma\Gamma^{-m_{j+1}}_{K,K',1}$,
 the symbol $  H_{\geq 3}^{(j+1)} \in \Gamma^{1}_{K,K',3}$ is admissible, 
and 
$ R^{(j+1)}(U) $ is a $1\times2$ matrix of smoothing operators in  $\Sigma\mathcal{R}^{-\rho}_{K,K',1}$.
\end{lemma}

\begin{proof}
Reasoning as in  \eqref{quasi-pi}, we have  
$ \frac{d}{dt}\g^{(3)}_{j+1}(U; x,\x) = 0 $
up to a cubic  symbol in $ \Gamma^{-m_j+\frac{1}{2}}_{K,1,3}$.
In order to conjugate \eqref{WWCexpjthfalso3}
we apply  Lemmata \ref{flusso12basso} and \ref{flusso12basso2}.
The only contributions with homogeneity $2$ and order $-m_j$ are 
\[
\opbw{ \Big(\frac{\ii}{2}(\g^{(3)}_{j+1} )_x|\x|^{-\frac{1}{2}}\sgn (\x)+\frac{1}{2\pi}
 \sum_{n\in\Z\setminus\{0\}}(
 \widetilde{{\mathtt q}}^{(j)}_2)_{n, n}^{+-}(\x) |u_n|^2 
 + (\widetilde{{\mathtt q}}^{(j)}_2)_{n, - n}^{+-}(\x) u_n \ov{u_{-n}} e^{\ii 2 n x } \Big)} \, .
 \]
By the choice of $ \g^{(3)}_{j+1}  $ in \eqref{ultimodiffeo}, \eqref{GAMMA3JTH} we have 
 \[
 \frac{\ii}{2}(\g^{(3)}_{j+1} )_x|\x|^{-\frac{1}{2}}\sgn(\x)+\frac{1}{2\pi}
 \sum_{n\in\Z\setminus\{0\}}(\widetilde{{\mathtt q}}^{(j)}_2)_{n, - n}^{+-}(\x) u_n \ov{u_{-n}} e^{\ii 2 n x }=
 0\,.
 \]
Then \eqref{WWCexp(j+1)th}  follows with the new integrable symbol
\[
\mathtt{p}_2^{(j+1)}(U;\x) := \mathtt{p}_2^{(j)}(U;\x)  + 
 \sum_{n\in\Z\setminus\{0\}}  (\widetilde{{\mathtt q}}^{(j)}_2)_{n, n}^{+-}(\x) |u_n|^2 
 \]
 and a symbol 
  $ \mathtt{q}^{(j+1)}(U;x,\x) $ in $ \Sigma\Gamma^{-m_{j+1}}_{K,K',1}$ where $ m_{j+1} = m_j + \frac12 $.
 \end{proof}

 \noindent
Lemmata \ref{def:ln1}, \ref{def:ln2}, \ref{def:ln3} imply Proposition \ref{prop:NO} by defining the map 
  \begin{equation}\label{Upsilonjth}
{\bf \Upsilon}^{\theta}_{j+1}(U) := \mathcal{A}^{\theta}_{j+1,3}(U)\circ\mathcal{A}^{\theta}_{j+1,2}(U)
 \circ\mathcal{A}^{\theta}_{j+1,1}(U)  
 \end{equation}
 where  $\mathcal{A}^{\theta}_{j+1,k}(U)$, for $k=1,2,3$, 
 are defined respectively in \eqref{newcoordj1}, \eqref{newcoordj1bis}, \eqref{newcoord1000}.

\subsection{Proof of Proposition \ref{teodiagonal3}}\label{Proofteodiagonal3}
We set
\be\label{fr-flow}
\mathfrak{F}^{\theta}(U):=
{\bf \Upsilon}_{fin}^{\theta}(U)\circ {\bf \Phi}_{5}^{\theta}(U)
\circ\cdots {\bf \Phi}_{1}^{\theta}(U)\circ {\bf \Psi}_{diag}^{\theta}(U) 
\ee
 and ${\bf F}^{\theta}(U):=\mathfrak{F}^{\theta}(U)[U]$ as in \eqref{FTvectPARA},
where ${\bf \Psi}_{diag}^{\theta} (U) $ is defined in 
Proposition \ref{teodiagonal}, the maps ${\bf \Phi}_{j}^{\theta} (U) $, $j=1,\ldots,5$
are given  respectively in 
\eqref{newcoord}, \eqref{newcoord2},  \eqref{newcoord400}, \eqref{newcoord500}, 
\eqref{newcoord700}, and 
$ {\bf \Upsilon}^{\theta}_{fin}(U):={\bf \Upsilon}_{2\rho}^{\theta}(U)\circ \cdots \circ{\bf \Upsilon}^{\theta}_{1}(U) $
where  ${\bf \Upsilon}_{j+1}^{\theta}(U) $, $j=0,\ldots,2\rho-1$, are defined
 in \eqref{Upsilonjth}. 
Then, by the construction in Subsections \ref{inteord1}-\ref{sec:nega}, we have that  
$ Z:=(  {\bf F}^{\theta}(U))_{\theta=1} $ solves the system 
\eqref{WWCexpjth} with $j=2\rho-1$ which has the form 
\eqref{finalsyst} with $\mathcal{D}_{-1/2}(U;\x)\rightsquigarrow\mathtt{p}_2^{(2\rho-1)}(U;\x) $, 
$\mathtt{H}_{\geq3}\rightsquigarrow H_{\geq 3}^{(2\rho-1)}$ and $\mathtt{R}(U)\rightsquigarrow R^{(2\rho-1)}(U)$. 
The bounds \eqref{stimFINFRAK} follow since $\mathfrak{F}^{\theta}(U)$ 
is the composition of maps constructed using Lemma \ref{buonflusso} (see bounds  \eqref{est1}).

\bigskip 
\section{ Poincar\'e-Birkhoff Normal Forms}\label{sec:BNF} 

The aim of this section is to eliminate all the terms of the system  \eqref{finalsyst}
up to cubic degree of homogeneity which are not yet in Poincar\'e-Birkhoff normal form.
Such terms appear only in the smoothing remainder 
$ \mathtt{R}(U)[Z] $ that we write  as
 \begin{align}\label{smooth-terms}
\mathtt{R}(U) &=\mathtt{R}_1(U)+\mathtt{R}_2(U)+\mathtt{R}_{\geq3}(U) \, ,\quad 
  \mathtt{R}_{\geq3}(U) \in \mathcal{R}^{-\rho}_{K,K',3}\otimes\mathcal{M}_2(\C)  \, , \\ 
\mathtt{R}_i(U) 
& =\left(
\begin{matrix}
(\mathtt{R}_i(U))_{+}^{+} & (\mathtt{R}_i(U))_{+}^{-}\\
(\mathtt{R}_i(U))_{-}^{+} & (\mathtt{R}_i(U))_{-}^{-}
\end{matrix}
\right)\,,  \quad 
(\mathtt{R}_i(U))_{\s}^{\s'}
\in \widetilde{\mathcal{R}}^{-\rho}_i \, ,  \quad 
 (\mathtt{R}_{i}(U))_{\s}^{\s'}=\ov{ (\mathtt{R}_{i}(U))_{-\s}^{-\s'}} \, , 
\label{smooth-terms2}
 \end{align}
 for $\s,\s'=\pm$ and $i=1,2$.
  The third identity in \eqref{smooth-terms2} means that the matrix of operators
  $\mathtt{R}(U)$ is \emph{real-to-real} (see \eqref{vinello}). 
 For any $\s,\s'=\pm$ we expand
\begin{equation}\label{BNF2}
(\mathtt{R}_1(U))_{\s}^{\s'} =\sum_{\ep=\pm}(\mathtt{R}_{1,\ep}(U))_{\s}^{\s'} \,, \;\;\;\;\;
 (\mathtt{R}_{2}(U))_{\s}^{\s'}=\sum_{\ep =\pm}
 (\mathtt{R}_{2,\ep,\ep}(U))_{\s}^{\s'}+(\mathtt{R}_{2,+,-}(U))_{\s}^{\s'} \, , 
\end{equation}
where $(\mathtt{R}_{1,\ep}(U))_{\s}^{\s'}\in \widetilde{\mathcal{R}}^{-\rho}_{1} $, 
$ (\mathtt{R}_{2,\ep,\ep'}(U))_{\s}^{\s'}\in \widetilde{\mathcal{R}}^{-\rho}_{2} $ with $ \ep,\ep'=\pm$, are the 
homogeneous smoothing operators
\begin{align}\label{espansio1}
(\mathtt{R}_{1,\ep}(U))_{\s}^{\s'} z^{\s'} & = 
 \frac{1}{\sqrt{2\pi}}\sum_{j\in \Z\setminus\{0\}} 
\Big( \sum_{k\in \Z\setminus\{0\}}  (\mathtt{R}_{1,\ep}(U))_{\s,j}^{\s',k} z_{k}^{\s'}\Big) e^{\ii \s jx} 
\end{align}
with entries 
 \begin{align}\label{BNF3}
&  (\mathtt{R}_{1,\ep}(U))_{\s,j}^{\s',k} :=\frac{1}{\sqrt{2\pi}}\sum_{\substack{n\in \Z\setminus\{0\} \\ \ep n+\s'k=\s j}}
 (\mathtt{r}_{1,\ep})^{\s,\s'}_{n,k}u_{n}^{\ep} \, , \quad j,k\in \Z\setminus\{0\} \, , 
 \end{align}
for suitable scalar coefficients $ (\mathtt{r}_{1,\ep})^{\s,\s'}_{n,k} \in \C $, 
 and
\be\label{R2epep'}
(\mathtt{R}_{2,\ep,\ep'}(U))_{\s}^{\s'} z^{\s'} =\frac{1}{\sqrt{2\pi}}
\sum_{j\in \Z\setminus\{0\}} 
\Big( \sum_{k\in \Z\setminus\{0\}}   (\mathtt{R}_{2,\ep,\ep'}(U))_{\s,j}^{\s',k} z_{k}^{\s'} \Big) e^{\ii \s jx} 
\ee
with entries
 \begin{align} \label{BNF5}
 (\mathtt{R}_{2,\ep,\ep'}(U))_{\s,j}^{\s',k} :=\frac{1}{2\pi}\sum_{\substack{n_1,n_2\in\Z\setminus\{0\} \\ \ep n_1+\ep' n_2+\s' k=\s j }}
  (\mathtt{r}_{2,\ep,\ep'})^{\s,\s'}_{n_1,n_2,k}u_{n_1}^{\ep}u_{n_2}^{\ep'}  \, ,   \quad  
  j,k\in \Z\setminus\{0\}   \, , 
 \end{align}
 and suitable scalar coefficients $ (\mathtt{r}_{2,\ep,\ep'})^{\s,\s'}_{n_1,n_2,k} \in \C $.

\begin{definition}{\bf (Poincar\'e-Birkhoff Resonant smoothing operator)}\label{resterm}
Let $\mathtt{R}(U) $ be a real-to-real smoothing operator in $\widetilde{\mathcal{R}}^{-\rho}_2 \otimes\mathcal{M}_2(\C ) $ 
with $\rho\geq0$ and scalar coefficients $  (\mathtt{r}_{\ep,\ep'})^{\s,\s'}_{n_1,n_2,k} \in \C $ defined as in \eqref{BNF5}.
We denote by $\mathtt{R}^{res}(U) $ the real-to-real smoothing operator in 
$\widetilde{\mathcal{R}}^{-\rho}_2 \otimes\mathcal{M}_2(\C ) $ 
with  coefficients 
 \begin{equation}\label{resterm2}
  (\mathtt{R}^{res}_{\ep,\ep'}(U))_{\s,j}^{\s',k} :=
 \sum_{\substack{n_1,n_2\in\Z\setminus\{0\} \\ \ep n_1+\ep' n_2+\s' k -\s j = 0 \\
 \ep \omega({n_1})+\ep' \omega({n_2})+\s' \omega({k})-\s\omega({j}) = 0}}\!\!\!
(\mathtt{r}_{\ep,\ep'})^{\s,\s'}_{n_1,n_2,k}u_{n_1}^{\ep}u_{n_2}^{\ep'} \, , \quad  j, k \in \Z \setminus \{0\} \, , 
 \end{equation}
where we recall that $ \omega({j})=|j|^{\frac{1}{2}} $.
 \end{definition}

In  Subsections \ref{primostepBNF} and \ref{secondostepBNF} we will reduce
the remainder $\mathtt{R}(U) $ in \eqref{smooth-terms} to its Poincar\'e-Birkhoff resonant component.
The key result of this section is the following.

   \begin{proposition}{\bf (Poincar\'e-Birkhoff normal form of the water waves at cubic degree)}\label{cor:BNF}
 There exists $ \rho_0 > 0 $ such that, for all $ \rho \geq \rho_0 $, $ K \geq K' = 2 \rho + 2 $, 
there exists $s_0>0$ such that, 
for any $s\geq s_0$, for all $0<r \leq r_0(s)$ small enough,  
and  any solution $U\in B^K_s(I;r)$ of the water waves system \eqref{eq:415}, there is a 
nonlinear map 
$ {\bf F}_T^{\theta}(U) $, $ \theta \in [0,1] $, of the form
\begin{equation}\label{FTvect}
{\bf F}_{T}^{\theta}(U):=\mathfrak{C}^{\theta}(U)[U]
\end{equation}
where  $\mathfrak{C}^{\theta}(U)$ is a real-to-real, bounded and invertible operator,
such that the function
$Y :=  \vect{y}{\bar{y}} = {\bf F}_T^{1}(U)$ 
solves
\begin{equation}\label{sistemaBNF1000}
 \pa_{t}Y= -\ii\Omega Y- \ii \opbw (\mathtt{D}(Y;\x))[Y]+ \tilde{\mathtt{R}}^{res}(Y)[Y] + \mathcal{X}_{\geq 4 }(U,Y) 
\end{equation}
where:

\begin{itemize}
\item $\Omega $ is the diagonal matrix of Fourier multipliers defined in \eqref{eq:415tris}, 
 and $\mathtt{D}(Y;\x)$ is the diagonal matrix of  integrable symbols  
 $\widetilde{\Gamma}^{1}_2 \otimes {\mathcal M}_2 (\C)$  defined   in \eqref{diag-part};
 
\item the smoothing operator $ \tilde{\mathtt{R}}^{res} (Y) \in 
 \widetilde{\mathcal{R}}^{- (\rho - \rho_0)}_2\otimes\mathcal{M}_2(\C) $ 
 is Poincar\'e-Birkhoff resonant according to Definition \ref{resterm}; 
  
\item $ \mathcal{X}_{\geq 4}(U,Y) $ has the form 
 \begin{equation}\label{Stimaenergy100}
 \mathcal{X}_{\geq 4}(U,Y)=\opbw({\mathfrak{H}}_{\geq3}(U; x, \xi) )[Y]+\mathfrak{R}_{\geq 3}(U)[Y]
 \end{equation}
where 
 ${\mathfrak{H}}_{\geq3}(U; x, \xi ) \in \Gamma^{1}_{K,K',3}\otimes\mathcal{M}_2(\C)$ is an admissible matrix
 of symbols
 (Definition \ref{def:admissible}) 
 and $\mathfrak{R}_{\geq 3}(U)$ is a matrix of real-to-real smoothing operators in 
 $\mathcal{R}^{-(\rho - \rho_0)}_{K,K',3}\otimes\mathcal{M}_2(\C)$. 
\end{itemize}

\smallskip
Furthermore, the map ${\bf F}_{T}^{\theta}(U)$ defined in \eqref{FTvect} satisfies the following properties:
\\[1mm]
(i) There is a constant $C$  depending on $s$, $r$ and $K$, such that, for $s\geq s_0$,
\begin{equation}\label{stimafinaleFINFRAK}
\| \pa_{t}^{k}\mathfrak{C}^{\theta}(U)[V]\|_{\dot{H}^{s-k}}+\| \pa_{t}^k (\mathfrak{C}^{\theta}(U))^{-1}[V]\|_{\dot{H}^{s-k}}\leq \|V\|_{k,s}(1+C\|U\|_{K,s_0}) 
+ C\|V\|_{k,s_0}\|U\|_{K,s} \,,
\end{equation}
for any $0\leq k\leq K-K'$,  $V\in C^{K-K'}_{*\R}(I;\dot{H}^{s}(\T;\C^2))$ and uniformly in $\theta\in [0,1]$;
\\[1mm]
(ii)
The function 
$ Y =  {\bf F}_T^\theta (U)_{|_{\theta=1}} $ 
satisfies
\begin{equation}\label{equiv100}
C^{-1}\|U\|_{{\dot H}^{s}}\leq \|Y\|_{ \dot H^{s}}\leq C \|U\|_{{\dot H}^{s}} \, . 
\end{equation}
(iii) The map ${\bf F}_{T}^{\theta}(U)$ admits an expansion as 
$$
{\bf F}^{\theta}_T (U)=U
+\theta \big(M_{1}(U)[U]+M_2^{(1)}(U)[U]\big) + \theta^{2}{M}^{(2)}_2(U)[U]  +M_{\geq3}(\theta;U)[U] \, , 
$$
where  $ M_{1}(U) $ is  in $ \widetilde{{\mathcal{M}}}_{1}\otimes\mathcal{M}_2(\C)$, 
the maps $ M_{2}^{(1)}(U), M^{(2)}_2(U) $ are in 
$ \widetilde{{\mathcal{M}}}_{2}\otimes\mathcal{M}_2(\C)$, 
 and $ M_{\geq3} (\theta; U) $ 
 is in $ \mathcal{M}_{K,K',3}\otimes\mathcal{M}_2(\C) $ 
 with estimates uniform in $ \theta \in [0,1] $.
\end{proposition}

In the following subsection we provide lower bounds on the ``small divisors'' 
which appear in the Poincar\'e-Birkhoff reduction procedure. 
Then, in Subsection \ref{BirReduco}, we  prove  Proposition \ref{cor:BNF}.

\subsection{Cubic and quartic wave interactions}\label{smalldivis}
We study in this section the cubic and quartic resonances among the linear frequencies
$\omega(n)=|n|^{\frac{1}{2}}$. 
 
\begin{proposition}{\bf (Non-resonance conditions)}\label{stimedalbasso}
There are constants $\mathtt{c}>0$ and $N_0>0$ such that
\begin{itemize}
\item
{\bf (cubic resonances)} for any 
$\s,\s'=\pm$ and $n_1,n_2,n_3\in \Z\setminus\{0\} $ satisfying 
\begin{equation}\label{moment}
 n_1+\s n_2+\s' n_3=0 \,,
 \end{equation}
  we have 
 \begin{equation}\label{stima1}
 |\omega({n_1})+\s \omega({n_2})+\s' \omega({n_3})|\geq\mathtt{c}\,. 
 \end{equation}
 \item 
 {\bf (quartic resonances)}
For any $\s,\s',\s''=\pm$ and $n_1,n_2,n_3,n_4\in \Z\setminus\{0\}$ such that
 \begin{equation}\label{moment2}
n_1+\s n_2+\s' n_3+\s'' n_4=0 \, , \quad 
\omega({n_1})+\s \omega({n_2})+\s' \omega({n_3})+\s'' \omega({n_4})\neq0 \, ,
 \end{equation}
 we have 
 \begin{equation}\label{stima2}
 |\omega({n_1})+\s \omega({n_2})+\s' \omega({n_3})+\s'' \omega({n_4})|\geq
 \mathtt{c}\max\{|n_1|,|n_2|,|n_3|,|n_4|\}^{-N_0} \,.
 \end{equation}
 \end{itemize}
 \end{proposition}
 
 \begin{proof} We first consider the cubic and then the quartic resonances. 
 \\[1mm]
{\sc  Proof of \eqref{stima1}}. If $\s=\s'=+$ then the bound \eqref{stima1} is trivial. Assume $\s=+$ and $\s'=-$.
By  \eqref{moment} we have that
 $ |n_3|\leq |n_1|+|n_2| $
and therefore
 \begin{equation*}
 \begin{aligned}
 |\sqrt{|n_1|}+\sqrt{|n_2|}-\sqrt{|n_3|}| 
 & =\frac{||n_1|+|n_2|-|n_3|+2\sqrt{|n_1||n_2|}|}{\sqrt{|n_1|}+\sqrt{|n_2|}+\sqrt{|n_3|}}\\
 &\geq \frac{2\sqrt{|n_1||n_2|}}{\sqrt{|n_1|}+\sqrt{|n_2|}+\sqrt{|n_1|+|n_2|}} \geq \frac{2 }{2+\sqrt{2}}
 \end{aligned}
 \end{equation*}
 since $|n_1|, |n_2|\geq1$.
The  bound \eqref{stima1} in the case $\s=-$ and $\s'=+$ is the same.
 \\[1mm]
{\sc Proof of \eqref{stima2}}. The case $\s=\s'=\s''=+$ is trivial. Assume $\s=\s'=+$ and $\s''=-$. We have 
$$
 |\omega({n_1})+\omega({n_2})+ \omega({n_3})-\omega({n_4})|=\frac{ | |n_1|+|n_2|+|n_3|-|n_4|  +2\sqrt{|n_1n_2|}+
 2\sqrt{|n_2n_3|}+2\sqrt{|n_1n_3|}  |}{ \omega({n_1})+\omega({n_2})+ \omega({n_3})+\omega({n_4})} \,.
$$
The first (momentum) condition in \eqref{moment2} implies that 
$ |n_1|+|n_2|+|n_3|-|n_4|\geq0 $
and hence \eqref{stima2} follows (actually with $ N_0 = 0 $).
It remains to study the case $\s=\s''=-$ and $\s'=+$, 
i.e. we have to prove that the phase
\begin{align}
\psi(n_1,n_2,n_3,n_4) & := |n_1|^{\frac{1}{2}}-|n_2|^{\frac{1}{2}}
+|n_3|^{\frac{1}{2}}-|n_4|^{\frac{1}{2}} \label{phace}\\ 
& =
\frac{|n_1|-|n_2|+|n_3|-|n_4|+2\sqrt{|n_1n_3|}-2\sqrt{|n_2n_4|}}{|n_1|^{\frac{1}{2}}
+|n_2|^{\frac{1}{2}}+|n_3|^{\frac{1}{2}}+|n_4|^{\frac{1}{2}}} \label{phace2}
\end{align}
satisfies \eqref{stima2}. 
Notice that the first (momentum) equality in \eqref{moment2}
becomes
\begin{equation}\label{moment3}
n_1-n_2+n_3-n_4=0\,.
\end{equation}
Let $|n_1|:=\max\{|n_1|,|n_2|,|n_3|,|n_4|\}$ and assume, without loss of generality, that $n_1>0$ and
$ |n_2| \geq |n_4| $ (the phase \eqref{phace} is symmetric in $ |n_2|, |n_4| $). 
We consider 
different cases.

\vspace{0.5em}
\noindent
\emph{Case a)} Assume that $n_1=|n_2|$. 
Then by \eqref{phace} 
\[
|\psi(n_1,n_2,n_3,n_4)|=\big|{|n_3|}^{\frac{1}{2}}-{|n_4|}^{\frac{1}{2}}\big|= \frac{||n_3|-|n_4||}{|n_3|^{\frac{1}{2}}+|n_4|^{\frac{1}{2}}}.
\]
Since $\psi\neq 0$ then $ |n_3|-|n_4| $ is a non-zero integer and 
we get  \eqref{stima2}.
Thus in the sequel we suppose 
\be\label{n1>n2}
n_1 > |n_2| \geq | n_4 | \, . 
\ee
\vspace{0.5em}
\noindent
\emph{Case b)} Assume that $  |n_3| \geq |n_4| $.
Then by \eqref{phace} we get
$$
\psi(n_1,n_2,n_3,n_4) \geq 
|n_1|^{\frac{1}{2}} -|n_2|^{\frac{1}{2}} = \frac{|n_1|- |n_2|}{ |n_1|^\frac{1}{2} + |n_2|^\frac{1}{2} }
\stackrel{\eqref{n1>n2}} \geq \frac{1}{  |n_1|^\frac{1}{2} + |n_2|^\frac{1}{2}  }
$$
which implies \eqref{stima2}. Thus in the sequel we suppose, in addition to \eqref{n1>n2}, that 
\be\label{caso3}
n_1 > | n_2 | \geq |n_4| > |n_3|  \, . 
\ee
The case
 $n_2<0$ is not possible. Indeed, if $ n_2 < 0 $ then \eqref{moment3} implies
$ n_4 = n_1+|n_2|+n_3  > n_1   $ by \eqref{caso3}
which is in contradiction with $ n_1 > |n_4| $.
Hence from now on we assume that 
\be\label{assump}
n_1>n_2 \geq |n_4| > |n_3| > 0 \, .
\ee
\noindent
\emph{Case c1)} 
Assume that all the frequencies have all the same sign, i.e. $n_1>n_2 \geq n_4 >  n_3>0$. 
In this case, by \eqref{phace2}-\eqref{moment3}, we get
\[
|\psi(n_1,n_2,n_3,n_4)|=\frac{|2\sqrt{n_1n_3}-2\sqrt{n_2n_4}|}{|n_1|^{\frac{1}{2}}
+|n_2|^{\frac{1}{2}}+|n_3|^{\frac{1}{2}}+|n_4|^{\frac{1}{2}}}\geq\frac{2}{\sum_{i=1}^{4}|n_i|^{\frac{1}{2}}}
\frac{|n_1n_3-n_2n_4|}{\sqrt{n_1n_3}+\sqrt{n_2n_4}}\,.
\]
Since $\psi\neq 0$ we have $n_1n_3\neq n_2n_4$, 
and therefore \eqref{stima2} follows.

\noindent
\emph{Case c2)} Assume now that two frequencies are positive and two are negative, i.e.
$ n_4 < n_3 <0<n_2<n_1$. 
The momentum condition \eqref{moment3} 
becomes
$ n_1-n_2=-|n_4|+|n_3| $ and,
since $n_1>n_2$, then  $|n_3|>|n_4|$ contradicting \eqref{assump}. 

\noindent
\emph{Case c3)} Assume that three frequencies have the same sign and one has the opposite sign. 
By \eqref{moment3} and \eqref{assump} we then  have
\begin{equation}\label{assump1}
n_1>n_2>n_4>0>n_3 \, , \quad n_4 > |n_3| \, . 
\end{equation}
By \eqref{phace2} 
we get
\begin{align}
\psi(n_1,n_2,n_3,n_4)&=\frac{n_1-n_2+|n_3|-n_4+2\sqrt{n_1|n_3|}-2\sqrt{n_2n_4}}{\sum_{i=1}^{4}|n_i|^{\frac{1}{2}}}
\nonumber \\
&\stackrel{\eqref{moment3}, \eqref{assump1}}{=}\frac{2}{\sum_{i=1}^{4}|n_i|^{\frac{1}{2}}}\Big(
|n_3|+\sqrt{n_1|n_3|}-\sqrt{n_2n_4}\Big) \nonumber \\
&=\frac{2}{\sum_{i=1}^{4}|n_i|^{\frac{1}{2}}}
\frac{n_3^{2}+n_1|n_3|-n_2 n_4+2|n_3|\sqrt{n_1|n_3|}}{|n_3|+\sqrt{n_1|n_3|}+\sqrt{n_2n_4}}\,. \label{phace100}
\end{align}
If $n_2 n_4\leq n_1|n_3|$ then \eqref{phace100} implies the bound \eqref{stima2}.
If instead $n_2 n_4> n_1|n_3|$
we reason as follows.
Notice that  
$$
B  := n_3^{2}+n_1|n_3|-n_2 n_4-2|n_3|\sqrt{n_1|n_3|}  \leq 
 n_3^{2}-2|n_3|\sqrt{n_1|n_3|}\leq -|n_3|\sqrt{n_1|n_3|}  \leq - 1 \, , 
$$
in particular $B \neq 0 $.
Then we rationalize again \eqref{phace100} to obtain
$$
\psi(n_1,n_2,n_3,n_4)=C\cdot A\cdot B^{-1} 
$$
where 
$$
A  := (n_3^{2}+n_1|n_3|-n_2 n_4)^{2}-4|n_3|^{3}n_1 \, , \quad
C  :=\frac{2}{\sum_{i=1}^{4}|n_i|^{\frac{1}{2}}}\frac{1}{|n_3|+\sqrt{n_1|n_3|}+\sqrt{n_2n_4}} \, .
$$
Since $ \psi \neq 0 $ then $A$ is a non zero integer and so  $|\psi| \geq C |B|^{-1}$. Moreover 
$|B|\leq  c n_1^{2} $, for some constant $c>0$, 
and  \eqref{stima2} follows.
\end{proof}

\subsection{Poincar\'e-Birkhoff reductions}\label{BirReduco}
The proof of Proposition \ref{cor:BNF} is divided into two steps: in the first (Subsection \ref{primostepBNF}) 
we eliminate all the quadratic terms in \eqref{finalsyst}; 
in the second one (Subsection \ref{secondostepBNF}) we eliminate all the non resonant cubic terms.

\subsubsection{Elimination of the  quadratic vector field}\label{primostepBNF}
 
 In this section we cancel out 
 the smoothing term $\mathtt{R}_1(U)$ in \eqref{smooth-terms} of  system \eqref{finalsyst}.
  We conjugate \eqref{finalsyst} with  the flow 
 \be\label{BNFstep1}
\partial_{\theta} \mathcal{B}_1^{\theta}(U)  = \mathtt{Q}_1(U) \mathcal{B}_{1}^{\theta}(U) \, , 
\quad \mathcal{B}_1^{0}(U) = {\rm Id} \, ,  
\ee
with $ \mathtt{Q}_1(U)\in \widetilde{\mathcal{R}}^{-\rho}_1 \otimes\mathcal{M}_2(\C)$ of the same form of $\mathtt{R}_1(U)$ in \eqref{smooth-terms2}-\eqref{BNF3}, to be determined. 
We introduce the new variable
\begin{equation}\label{newcoordBNF}
Y_1:=\vect{y_1}{\ov{y_1}}=
\big(\mathcal{B}_1^{\theta}(U)[Z]\big)_{|_{\theta=1}} \, .  
\end{equation}

 \begin{lemma}{\bf (First Poincar\'e-Birkhoff step)} \label{lem:BNF1}
Assume that $\mathtt{Q}_1(U)\in \widetilde{\mathcal{R}}^{-\rho}_1\otimes\mathcal{M}_{2}(\C)$ solves 
the homological equation
\begin{equation}\label{omoBNF}
\mathtt{Q}_1(-\ii \Omega U)+ \big[  \mathtt{Q}_1( U), -\ii \Omega \big]+\mathtt{R}_1(U)=0 \, .
\end{equation}
Then  
 \begin{equation}\label{BNF12}
 \begin{aligned}
 \pa_{t}Y_1&=-\ii \Omega Y_1+ \opbw( -\ii {\mathtt D}(U;\x) + {\mathtt H}_{\geq3})[Y_1] +
 \big( \mathtt{R}^{+}_2(U)  
 + \mathtt{R}^{+}_{\geq 3}(U) \big)[Y_1] 
 \end{aligned}
 \end{equation}
 where $  \Omega $ is defined in \eqref{eq:415tris}, $ {\mathtt D}(U;\x) $  in \eqref{diag-part}, 
 $  {\mathtt H}_{\geq3} $ is an admissible symbol in $ \Gamma^{1}_{K,K',3}\otimes \mathcal{M}_2(\C)$, 
 and 
 $$ 
 \mathtt{R}^{+}_2 (U) \in \widetilde{\mathcal{R}}^{-\rho + m_1}_{2} \otimes\mathcal{M}_2(\C) \, , 
 \quad 
 \mathtt{R}^{+}_{\geq 3}(U)\in 
\mathcal{R}^{-\rho + m_1}_{K,K',3}\otimes\mathcal{M}_{2}(\C) \, , 
$$
with   $m_1\geq1$ as in \eqref{eq:415tris}. 
 \end{lemma}
 
 \begin{proof}
To conjugate \eqref{finalsyst}  we apply Lemma \ref{lem:tra.Vec} with 
$ \mathtt{Q}_1(U)  =  \ii {\bf A}(U)  $. 
By  \eqref{Lie1Vec} with $ L = 1 $
  we have 
   \begin{align}
  -\ii \mathcal{B}_{1}^{1}(U)\Omega (\mathcal{B}_{1}^{1}(U))^{-1} & =
 -\ii\Omega+ \big[ \mathtt{Q}_1(U),-\ii \Omega \big]  \nonumber \\
 & \ +\int_{0}^{1} (1- \theta) 
\mathcal{B}_{1}^{\theta}(U) \big[ \mathtt{Q}_1(U), \big[ \mathtt{Q}_1(U),-\ii \Omega
\big] \big](\mathcal{B}_{1}^{\theta}(U))^{-1} d \theta \label{core4} \, .
\end{align}
Using that $\mathtt{Q}_1(U)$ belongs to 
 $\widetilde{\mathcal{R}}_{1}^{-\rho} \otimes\mathcal{M}_2(\C) $
 and applying Proposition \ref{composizioniTOTALI}, and Lemma \ref{buonflusso2}, 
 the  term in \eqref{core4}
 is a smoothing operator  in $ \Sigma\mathcal{R}^{-\rho+ \frac12}_{K,K',2}\otimes\mathcal{M}_2(\C)$.
Similarly we obtain
 \begin{equation}\label{core2}
  -\ii \mathcal{B}_{1}^{1}(U)\opbw( {\mathtt D}(U;\x)) (\mathcal{B}_{1}^{1}(U))^{-1}=-\ii \opbw( {\mathtt D}(U;\x))
 \end{equation}
  up to a  term in 
  $ \Sigma\mathcal{R}^{-\rho+1}_{K,K',2}\otimes\mathcal{M}_2(\C)$, and 
 \begin{equation}\label{core} 
 \begin{aligned}
 \mathcal{B}_1^{1}(U)& \opbw({\mathtt H}_{\geq3}) (\mathcal{B}_{1}^{1}(U))^{-1}
=\opbw({\mathtt H}_{\geq3})
 \end{aligned}
 \end{equation}
 up to a matrix of smoothing operators in $\Sigma\mathcal{R}^{-\rho+1}_{K,K',2}\otimes\mathcal{M}_2(\C)$.
Finally 
  \begin{equation}\label{core3}
  \mathcal{B}_{1}^{1}(U)\big( \mathtt{R}_1(U)+\mathtt{R}_2(U)+\mathtt{R}_{\geq3}(U)\big)(\mathcal{B}_{1}^{1}(U))^{-1}=
  \mathtt{R}_1(U) 
 \end{equation}
plus a smoothing operator in $ \Sigma\mathcal{R}^{-\rho+1}_{K,K',2}\otimes\mathcal{M}_2(\C)$.
 
Next we consider the contribution coming from the conjugation of $ \pa_t $. 
Applying formula \eqref{Lie2Vec}  with $L=2$,
 we get
 \begin{align}
  \pa_t \mathcal{B}_{1}^{1} (U) (\mathcal{B}_{1}^{1} (U))^{-1} &=  
  \pa_t \mathtt{Q}_1(U)+\frac{1}{2} \big[ \mathtt{Q}_1(U), \pa_t \mathtt{Q}_1(U) \big] \nonumber 
\\
  &+\frac{1}{2}\int_0^{1} (1-\theta)^{2} \mathcal{B}_{1}^{\theta}(U)
\big[ \mathtt{Q}_1(U), \big[ \mathtt{Q}_1(U) , \pa_t \mathtt{Q}_1(U)  \big] \big] 
 (\mathcal{B}_{1}^{\theta} (U))^{-1} d \theta \, . \label{core5}
\end{align}
Recalling \eqref{eq:415tris} we have 
 \begin{equation}\label{core8}
 \pa_t  \mathtt{Q}_1(U) =  \mathtt{Q}_1(-\ii \Omega U+{\bf M}(U)[U]) =
   \mathtt{Q}_1(-\ii \Omega U) +\mathtt{Q}_1({\bf M}(U)[U])  =  \mathtt{Q}_1(-\ii \Omega U)
 \end{equation}
 up to a term in $ \Sigma\mathcal{R}^{-\rho+m_1}_{K,K',2}\otimes\mathcal{M}_{2}(\C)$, 
 where we used item (iii) of Proposition \ref{composizioniTOTALI}.
 By  \eqref{core8} and the fact that 
$ \mathtt{Q}_1(-\ii \Omega U) $ is in 
$ \widetilde{\mathcal{R}}_{1}^{-\rho + \frac12} \otimes\mathcal{M}_2(\C)   $  we have that the second line \eqref{core5} belongs to 
$ \Sigma\mathcal{R}^{-\rho+m_1}_{K,K',2}\otimes\mathcal{M}_{2}(\C)$.

In conclusion, by  \eqref{core4}, \eqref{core2}, \eqref{core}, \eqref{core3}, \eqref{core5}  
and the assumption that  $ \mathtt{Q}_1 $ solves \eqref{omoBNF}  we deduce \eqref{BNF12}. 
\end{proof}

  \begin{itemize}\item {\bf Notation.}\label{max2}
  Given $p\in \N$ we denote by
  $\max_{2}(|n_1|,|n_2|,\ldots,|n_p|)$ and   $\max(|n_1|,|n_2|,\ldots,|n_p|)$  respectively
  the second largest and the largest among $ | n_1|, \ldots, |n_p| $.
\end{itemize}

\noindent
We have the following lemma.
  \begin{lemma}\label{decadimento}
An operator $ {\mathtt R}_1 (U) $ of the form \eqref{smooth-terms2}-\eqref{BNF3} 
belongs to $\widetilde{\mathcal{R}}^{-\rho}_1\otimes\mathcal{M}_2(\C)$ if and only if, for some $\mu>0$,
  \begin{equation}\label{BNF6}
  | (\mathtt{r}_{1,\ep})^{\s,\s'}_{n,k}|\leq \frac{\max_2(|n|,|k|)^{\rho+\mu}}{\max(|n|,|k|)^{\rho}} \, ,
    \quad \forall\, \ep,\s,\s'=\pm,\;\; n,k\in\Z\setminus\{0\}\,. 
  \end{equation}
  An operator $ {\mathtt R}_2 (U) $ of the form \eqref{smooth-terms2}-\eqref{BNF2} as in 
  \eqref{R2epep'}-\eqref{BNF5} belongs to 
  $\widetilde{\mathcal{R}}^{-\rho}_2\otimes\mathcal{M}_2(\C)$ if and only if, for some $\mu>0$, 
  \begin{equation}\label{BNF7}
  | (\mathtt{r}_{2,\ep,\ep'})^{\s,\s'}_{n_1, n_2, k}|\leq 
  \frac{\max_2(|n_1|, |n_2|, |k|)^{\rho+\mu}}{\max(|n_1|, |n_2|, |k|)^{\rho}}
  \, , 
  \quad \forall\, \ep,\ep',\s,\s'=\pm,\;\; n_1,n_2,k\in\Z\setminus\{0\}\,.
  \end{equation}  
    \end{lemma}
  \begin{proof}  
By the definition of smoothing homogeneous operators given in Definition \ref{omosmoothing}. 
\end{proof}

We now solve the homological equation \eqref{omoBNF}.

\begin{lemma}{\bf (First homological equation)} 
The operator $ \mathtt{Q}_{1} $ of the form \eqref{smooth-terms2}-\eqref{BNF3} 
with coefficients 
\begin{equation}\label{omoBNF5}
  (\mathtt{q}_{1,\ep})^{\s,\s'}_{n,k} :=
  \frac{-(\mathtt{r}_{1,\ep})^{\s,\s'}_{n,k}}{\ii \big( \s|j|^{ \frac{1}{2}}-\s'  |k|^{\frac{1}{2}}- \ep |n|^{\frac{1}{2}} \big)} \, , 
  \quad  \s j -\s' k-\ep n=0 \, ,   
\end{equation}
with $\s, \s', \ep  = \pm $, $ j, n, k \in \Z \setminus \{0\} $ 
solves the homological equation \eqref{omoBNF} and $ \mathtt{Q}_{1} $  is in 
$ \widetilde{\mathcal{R}}_1^{-\rho}\otimes\mathcal{M}_2(\C)$. 
\end{lemma}

\begin{proof}
First note that the coefficients in \eqref{omoBNF5} are well-defined
since $\s|j|^{ \frac{1}{2}}-\s'  |k|^{\frac{1}{2}}- \ep |n|^{\frac{1}{2}}\neq 0$ for any $\s, \s', \ep  = \pm $, 
$ n, k \in \Z \setminus \{0\} $, 
by Proposition \ref{stimedalbasso}, in particular  \eqref{stima1}.
Moreover,  by  \eqref{stima1} and Lemma \ref{decadimento} we have 
\begin{equation*}
 | (\mathtt{q}_{1,\ep})^{\s,\s'}_{n,k}|\leq  \frac{{\max}_2 (|n|,|k|)^{\rho+\mu}}{\max(|n|,|k|)^{\rho}} \, ,
 \end{equation*}
and therefore the operator $\mathtt{Q}_1$
is in $\widetilde{\mathcal{R}}_1^{-\rho}\otimes\mathcal{M}_2(\C)$.

Next, recalling \eqref{smooth-terms2} and \eqref{eq:415tris},  the homological 
equation \eqref{omoBNF} amounts to the equations
 \begin{equation*}
(\mathtt{Q}_1(-\ii \Omega U))_{\s}^{\s'}-  (\mathtt{Q}_1( U))_{\s}^{\s'}  \s' \ii |D|^\frac12 +
 \s\ii |D|^{\frac{1}{2}}  (\mathtt{Q}_1( U))_{\s}^{\s'} +
(\mathtt{R}_1(U))_{\s}^{\s'}=0\,, \ \forall \s,\s'=\pm \, , 
\end{equation*}
and expanding 
 $(\mathtt{Q}_{1}(U))_{\s}^{\s'}$ as in \eqref{BNF2}-\eqref{BNF3} with entries 
\be\label{expQ1}
(\mathtt{Q}_{1,\ep}(U))_{\s,j}^{\s',k}=\frac{1}{\sqrt{2\pi}}\sum_{\substack{n\in \Z\setminus\{0\} \\ \ep n+\s'k=\s j}}
 (\mathtt{q}_{1,\ep})^{\s,\s'}_{n,k}u_{n}^{\ep} \, , \quad j,k\in \Z\setminus\{0\}\, , 
\ee
to the equations, for any $j,k\in \Z\setminus\{0\}$, $\ep=\pm$, 
\begin{equation}\label{omoBNF3}
(\mathtt{Q}_{1,\ep}(-\ii \Omega U))_{\s,j}^{\s',k}+  (\mathtt{Q}_{1,\ep}( U))_{\s,j}^{\s',k} \big(
\s\ii |j|^{\frac{1}{2}}-\s' \ii |k|^{\frac{1}{2}}
\big)
+
(\mathtt{R}_{1,\ep}(U))_{\s,j}^{\s', k}=0 \, .
\end{equation}
By \eqref{expQ1} and \eqref{eq:415tris} we have 
 \[
 (\mathtt{Q}_{1,\ep}(-\ii \Omega U))_{\s,j}^{\s',k}=\frac{1}{\sqrt{2\pi}}
 \sum_{\substack{n\in \Z\setminus\{0\} \\ \ep n+\s'k=\s j}}
 (\mathtt{q}_{1,\ep})^{\s,\s'}_{n,k} (-\ii \ep |n|^{\frac{1}{2}}) u_{n}^{\ep} \, , 
 \]
 and \eqref{omoBNF3} becomes, 
 for $j,k,n\in \Z\setminus\{0\}$ and  $\s,\s',\ep=\pm$ with $ \ep n+\s'k=\s j$,
 \begin{equation*} 
  (\mathtt{q}_{1,\ep})^{\s,\s'}_{n,k} \ii \big(\s|j|^{\frac{1}{2}}-\s' |k|^{\frac{1}{2}}- \ep |n|^{\frac{1}{2}}\big)+
   (\mathtt{r}_{1,\ep})^{\s,\s'}_{n,k}=0 \, , 
 \end{equation*}
 which is solved by the  coefficients $(\mathtt{q}_{1,\ep})^{\s,\s'}_{n,k} $
defined in \eqref{omoBNF5}.
\end{proof}

\subsubsection{Elimination of the cubic vector field}\label{secondostepBNF}

In this section we reduce to  Poincar\'e-Birkhoff normal form 
the smoothing term $\mathtt{R}_{2}^{+} (U) \in \widetilde{\mathcal{R}}_2^{-\rho+m_1}\otimes\mathcal{M}_2(\C)$
 in  \eqref{BNF12}.
We conjugate \eqref{BNF12} with  the flow 
 \be\label{BNFstep2}
\partial_{\theta} \mathcal{B}_2^{\theta}(U)= \mathtt{Q}_2(U) \mathcal{B}_{2}^{\theta}(U)  \, , 
\quad \mathcal{B}_2^{0}(U) = {\rm Id} \, ,  
\ee
where $ \mathtt{Q}_2(U)$ is a matrix of smoothing operators in $\widetilde{\mathcal{R}}^{-\rho+N_0+m_1}_2\otimes\mathcal{M}_2(\C)$ 
of the same  form
 of $\mathtt{R}_2^{+}(U) $ to be determined.
 We introduce the new variable
\begin{equation}\label{newcoordBNF2}
Y_2:=\vect{y_2}{\ov{y_2}}=
\big(\mathcal{B}_2^{\theta}(U)[Y_1]\big)_{|_{\theta=1}}\,.
\end{equation}
\begin{itemize}
\item 
{\bf Notation.}
Given the operator $\mathtt{Q}_2(U)$ in \eqref{BNFstep2},
we denote by $\mathtt{Q}_2(-\ii\Omega U)$ the operator of the form 
\eqref{smooth-terms2}, \eqref{BNF2}, \eqref{R2epep'}-\eqref{BNF5} with coefficients defined as
 \begin{equation}\label{BNFcoef5}
 (\mathtt{Q}_{2,\ep,\ep'}(-\ii \Omega U))_{\s,j}^{\s',k}=\frac{1}{2\pi} \!\! \!\!\!\!
 \sum_{\substack{n_1, n_2 \in \Z\setminus\{0\} \\ \ep n_1 + \ep' n_2 +\s' k = \s j}} \!\!\!\!\!\!
 (\mathtt{q}_{2,\ep,\ep'})^{\s,\s'}_{n_1,n_2,k} 
 \big( -\ii \ep |n_1|^{\frac{1}{2}}-\ii\ep'|n_2|^{\frac{1}{2}} \big) u_{n_1}^{\ep}u_{n_2}^{\ep'}\, .
\end{equation}
\end{itemize}

\begin{lemma}{\bf (Second Poincar\'e-Birkhoff step)} \label{lem:BNF2}
Assume that  $\mathtt{Q}_2(U) \in \widetilde{\mathcal{R}}^{-\rho+N_0+m_1}_2 \otimes\mathcal{M}_{2}(\C) $ solves the homological equation 
 \begin{equation}\label{omostep2}
 \mathtt{Q}_2(-\ii \Omega U)+ \big[  \mathtt{Q}_2( U), -\ii \Omega \big] +\mathtt{R}_2^{+}(U)=
 (\mathtt{R}_{2}^+)^{res}(U)  \,.
\end{equation}
Then  
  \begin{equation}\label{BNF122}
 \begin{aligned}
 \pa_{t}Y_2&=-\ii \Omega Y_2+ \opbw( -\ii {\mathtt D}(U;\x) + {\mathtt H}_{\geq3} )[Y_2] +
\big( (\mathtt{R}_{2}^+)^{res}(U)  +  \mathtt{R}'_{\geq 3}(U)\big) [Y_2]
 \end{aligned}
 \end{equation}
  where $\Omega$ is defined in \eqref{eq:415tris} and ${\mathtt D}(U;\x)$  in \eqref{diag-part},  
     $ {\mathtt H}_{\geq3} $ is an admissible symbol in $ \Gamma^{1}_{K,K',3}\otimes \mathcal{M}_2(\C)$,  
  $ (\mathtt{R}_{2}^+)^{res}(U) $ 
  is a Poincar\'e-Birkhoff resonant smoothing operator according to Definition \ref{resterm} in 
  $ \widetilde{\mathcal{R}}^{-\rho+m_1}_2 \otimes\mathcal{M}_{2}(\C) $,   
  and $  \mathtt{R}'_{\geq 3}(U)$ is a matrix of smoothing operators in 
  $ \mathcal{R}^{-\rho+N_0+2m_1}_{K,K',3}\otimes\mathcal{M}_2(\C)$
   with $m_1\geq1$ as in \eqref{eq:415tris}. 
   \end{lemma}

\begin{proof}
To conjugate system \eqref{BNF12} we  apply Lemma \ref{lem:tra.Vec} with $ \mathtt{Q}_2(U) = \ii {\bf A}(U)   $.
Applying formula \eqref{Lie1Vec} with $ L = 1 $, 
the fact that $\mathtt{Q}_2(U)$ is a smoothing operator in 
$\widetilde{\mathcal{R}}^{-\rho+N_0+m_1}_2\otimes\mathcal{M}_2(\C) $,  
 Proposition \ref{composizioniTOTALI} and Lemma \ref{buonflusso2},
we have that 
 \[
 \mathcal{B}_2^{1}(U) ( -\ii \Omega )   ( \mathcal{B}_2^{1}(U))^{-1}=-\ii \Omega+ \big[ \mathtt{Q}_2(U),-\ii \Omega \big]
 \]
 plus a smoothing operator in $ \mathcal{R}^{-\rho+N_0+m_1+1}_{K,K',3}\otimes\mathcal{M}_2(\C)$.
Similarly 
 \[
 \begin{aligned}
&  \mathcal{B}_2^{1}(U) \big(
 \opbw( -\ii  {\mathtt D}(U;\x) + {\mathtt H}_{\geq3} )+ \mathtt{R}^{+}_2(U)
 + \mathtt{R}^{+}_{\geq 3}(U)\big)
   ( \mathcal{B}_2^{1}(U))^{-1} \\
& =  \opbw \big(  -\ii  {\mathtt D}(U;\x) + {\mathtt H}_{\geq3} \big)  + \mathtt{R}^{+}_2(U)
 \end{aligned}
 \]
 up to a smoothing operator in $ \mathcal{R}^{-\rho+N_0+m_1+1}_{K,K',3}\otimes\mathcal{M}_2(\C)$.

Next we consider the contribution coming from the conjugation of $ \pa_t $. 
First, note that, using equation \eqref{eq:415tris},  
\be\label{Q2t}
 \pa_{t}\mathtt{Q}_{2}(U) =\mathtt{Q}_{2}(\pa_t U) = \mathtt{Q}_{2}(-\ii \Omega U) 
 \ee
(defined in \eqref{BNFcoef5})
up to a smoothing operator in $ \mathcal{R}^{-\rho+N_0+2m_1}_{K,K',3}\otimes\mathcal{M}_2(\C)$.
The operator $ \mathtt{Q}_{2}(-\ii \Omega U) $ is in 
$ \widetilde{\mathcal{R}}^{-\rho+N_0+m_1 + \frac12}_2 \otimes\mathcal{M}_{2}(\C)  $. 
Then, applying formula \eqref{Lie2Vec} with $ L = 2 $ we have 
 \begin{align*}
  \pa_t \mathcal{B}_{2}^{1} (U) (\mathcal{B}_{2}^{1} (U))^{-1} &=  
  \pa_t \mathtt{Q}_2(U)+\frac{1}{2} \big[ \mathtt{Q}_2(U), \pa_t \mathtt{Q}_2(U) \big] \\
  & \ \ +\frac{1}{2}\int_0^{1} (1-\theta)^{2} 
   \mathcal{B}_{2}^{\theta} (U) \big[ \mathtt{Q}_2(U), \big[ \mathtt{Q}_2(U) , \pa_t\mathtt{Q}_{2}(U)
  \big] \big](\mathcal{B}_{2}^{\theta} (U))^{-1} d\theta
\stackrel{\eqref{Q2t}} = \mathtt{Q}_{2}(-\ii \Omega U)  \nonumber 
\end{align*}
 up to a  smoothing operator  in $ \mathcal{R}^{-\rho+N_0+2m_1}_{K,K',3}\otimes\mathcal{M}_2(\C)$.

In conclusion 
 $ \mathtt{Q}_2(-\ii \Omega U)+[  \mathtt{Q}_2( U), -\ii \Omega ]$ 
 $+\mathtt{R}_2^{+}(U)$ 
collects all the non integrable terms quadratic in $ U $ in the transformed system.
Since $\mathtt{Q}_2$ solves \eqref{omostep2} 
we conclude that $ Y_2 $ solves \eqref{BNF122}.
 \end{proof}
 
 We  now solve the homological equation \eqref{omostep2}.

 \begin{lemma}{\bf (Second homological equation)} 
The operator $ \mathtt{Q}_{2} $  of the form \eqref{smooth-terms2}-\eqref{BNF2}, \eqref{R2epep'}-\eqref{BNF5}
with coefficients
 \begin{equation}\label{omoBNF5step2}
\!\!  (\mathtt{q}_{2,\ep,\ep'})^{\s,\s'}_{n_1,n_2,k}
:=
\left\{
\begin{aligned}
& \frac{-(\mathtt{r}^{+}_{2,\ep,\ep'})^{\s,\s'}_{n_1,n_2,k}}{\ii(\s|j|^{\frac{1}{2}}-\s' |k|^{\frac{1}{2}}- \ep |n_1|^{\frac{1}{2}}- \ep' |n_2|^{\frac{1}{2}})} \, , & \!\! \!\!  \s|j|^{\frac{1}{2}}-\s' |k|^{\frac{1}{2}}- \ep |n_1|^{\frac{1}{2}}- \ep |n_2|^{\frac{1}{2}}\neq 0 \\
& 0 & \!\!  \!\! \s|j|^{\frac{1}{2}}-\s' |k|^{\frac{1}{2}}- \ep |n_1|^{\frac{1}{2}}- \ep |n_2|^{\frac{1}{2}} = 0  
\end{aligned}
\right.
\end{equation}
with  $\s,\s',\ep,\ep'=\pm$, $ n_1, n_2, k \in \Z \setminus \{0\} $, satisfying  $\s j-\s' k-\ep n_1-\ep' n_2=0 $, 
solves the homological equation \eqref{omostep2}.  We have that $ \mathtt{Q}_{2} $  is in
$ \widetilde{\mathcal{R}}^{-\rho+N_0+m_1}_2\otimes\mathcal{M}_2(\C) $. 
\end{lemma}
 
 \begin{proof}
 First note that the 
 coefficients in \eqref{omoBNF5step2} are well-defined thanks to Proposition \ref{stimedalbasso}, in 
 particular \eqref{stima2}, and satisfy, using also $ |j| \leq |k| + |n_1| + |n_2| $,   
\be\label{decadimento2}
|(\mathtt{q}_{2,\ep,\ep'})^{\s,\s'}_{n_1,n_2,k}|\leq 
C |(\mathtt{r}^{+}_{2,\ep,\ep'})^{\s,\s'}_{n_1,n_2,k}| \max(|n_1|,|n_2|,|k|)^{N_0} \leq
C   \frac{\max_2(|n_1|, |n_2|, |k|)^{\rho - m_1- N_0 + \mu'}}{\max(|n_1|, |n_2|, |k|)^{\rho- m_1 - N_0}}
\ee
with $ \mu' = \mu + N_0 $, because $(\mathtt{r}^{+}_{2,\ep,\ep'})^{\s,\s'}_{n_1,n_2,k}$ are the coefficients of a remainder in 
$\widetilde{\mathcal{R}}_2^{-\rho+m_1}\otimes\mathcal{M}_2(\C)$,
and so they satisfy the bound \eqref{BNF7} with $\rho \rightsquigarrow \rho-m_1$. 
 The estimate  \eqref{decadimento2} and Lemma \ref{decadimento} 
imply  that $\mathtt{Q}_2(U)$ 
belongs to the class $\widetilde{\mathcal{R}}^{-\rho+m_1+N_0}_2\otimes\mathcal{M}_2(\C)$. 
 
 Next, the homological equation \eqref{omostep2} amounts to,  for any $\s,\s',\ep,\ep'=\pm$,
   \begin{equation}\label{omoBNF3step2}
(\mathtt{Q}_{2,\ep,\ep'}(-\ii \Omega U))_{\s,j}^{\s',k}+  (\mathtt{Q}_{2,\ep,\ep'}( U))_{\s,j}^{\s',k} 
\big( \s\ii |j|^{\frac{1}{2}}-\s' \ii |k|^{\frac{1}{2}} \big)
+
(\mathtt{R}^{+}_{2,\ep,\ep'}(U))_{\s,j}^{\s', k}= \big((\mathtt{R}^+_{2,\ep,\ep'} \big)^{res}(U) \big)_{\s,j}^{\s', k}
\end{equation}
for any $j,k\in\Z\setminus\{0\}$. 
Recalling  \eqref{BNFcoef5} and \eqref{resterm2}, 
the left hand side of  \eqref{omoBNF3step2} is given by  
 \begin{equation*}
 (\mathtt{q}_{2,\ep,\ep'})^{\s,\s'}_{n_1,n_2,k} \ii \big(\s|j|^{\frac{1}{2}}-\s' |k|^{\frac{1}{2}}
 - \ep |n_1|^{\frac{1}{2}}- \ep |n_2|^{\frac{1}{2}}\big)+
   (\mathtt{r}^{+}_{2,\ep,\ep'})^{\s,\s'}_{n_1,n_2,k}   \, , 
 \end{equation*}
  for $j,k,n_1,n_2\in \Z\setminus\{0\}$,  $\s,\s',\ep,\ep'=\pm$ and  $ \ep n_1+\ep' n_2 +\s'k=\s j$.
We deduce,  recalling Definition \ref{resterm}, 
that the operator 
$ \mathtt{Q}_{2} $  with coefficients $ (\mathtt{q}_{2,\ep,\ep'})^{\s,\s'}_{n_1,n_2,k} $ defined 
in \eqref{omoBNF5step2} solves the homological equation 
\eqref{omostep2}. 
 \end{proof}
 
We can now prove the main result of this section.
 
 \begin{proof}[{\bf Proof of Proposition \ref{cor:BNF}}]
 Let $Z$ be the function given by Proposition \ref{teodiagonal3} which solves \eqref{finalsyst}. We set 
 \be\label{def:Btheta}
Y:=(\mathcal{B}^{\theta}(U))_{\theta=1}[Z] \, , \quad {\rm where} \quad 
\mathcal{B}^{\theta}(U) :=\mathcal{B}_2^{\theta}(U)\circ\mathcal{B}^{\theta}_1(U) \, , \quad \theta\in [0,1]  \, , 
\ee
and  $\mathcal{B}_i^{\theta}(U)$, $i=1,2$, are the flow maps 
defined  respectively in \eqref{BNFstep1},  
\eqref{BNFstep2}  
with generators $  \mathtt{Q}_1(U) $, $  \mathtt{Q}_2(U) $
defined respectively in Lemmata \ref{lem:BNF1} and \ref{lem:BNF2}. 
Therefore 
 the function $Y$ defined in \eqref{def:Btheta} solves the system (recall  \eqref{BNF122})
  \begin{equation}\label{sistemaBNF}
 \pa_t Y = -\ii \Omega Y+ \opbw(-\ii {\mathtt D}(U;\x) + {\mathtt H}_{\geq3})[Y]
 + \tilde{\mathtt{R}}^{res} (U)[Y]+  \mathtt{R}'_{\geq 3}(U)[Y]
\end{equation}
where $\Omega$ and ${\mathtt D}(U;\x) $ are defined respectively in \eqref{eq:415tris} and \eqref{diag-part}, 
the smoothing operator $ \tilde{\mathtt{R}}^{res} (U) :=(\mathtt{R}_{2}^+)^{res} (U) $ in 
$\widetilde{\mathcal{R}}^{-\rho+m_1}_2\otimes\mathcal{M}_2(\C) $
(where  $m_1\geq1$ is the loss in \eqref{eq:415tris})  
is  Poincar\'e-Birkhoff resonant according to Definition \ref{resterm}, 
the symbol ${\mathtt H}_{\geq3}\in \Gamma^{1}_{K,K',3}\otimes \mathcal{M}_2(\C) $ is admissible, and 
$ \mathtt{R}'_{\geq 3}(U)$ is in
$ \mathcal{R}^{-\rho+N_0+2m_1}_{K,K',3}\otimes\mathcal{M}_2(\C)$ where
the constant $N_0$  is defined by Proposition \ref{stimedalbasso}.

 We set $\mathfrak{C}^{\theta}(U):=\mathcal{B}^{\theta}(U)\circ\mathfrak{F}^{\theta}(U)$
 where $ \mathcal{B}^{\theta}(U) $ is the map defined in \eqref{def:Btheta} and $ \mathfrak{F}^{\theta}(U) $ 
 in \eqref{FTvectPARA}. Then  
we define ${\bf F}_T^{\theta}(U):=\mathfrak{C}^{\theta}(U)[U]$
as in \eqref{FTvect}.
The maps $\mathcal{B}_i^{\theta}$, $i=1,2$ are constructed as flows of smoothing remainders,
hence, by Lemma \ref{buonflusso2}, they are well-defined and 
satisfy the bounds \eqref{est1VEC}, \eqref{est1quatris}.
Then, since the map $\mathfrak{F}^{\theta}(U)$ satisfies \eqref{stimFINFRAK},
the composition map $\mathfrak{C}^{\theta}(U)$  satisfies 
\eqref{stimafinaleFINFRAK} and \eqref{equiv100}.
Moreover the map 
$\mathfrak{F}^{\theta}(U)$ 
is the composition of flows of para-differential operators (see its definition in 
\eqref{fr-flow}), hence, by Lemma \ref{buonflusso},
it  admits  multilinear expansions as in \eqref{est2}. In the same way, 
by Lemma \ref{buonflusso2}
the map $\mathcal{B}^{\theta}(U)$
admits a multilinear expansion as in \eqref{est2}, and therefore
${\bf F}_{T}^{\theta}(U)$ admits an expansion like \eqref{est2} as well, 
implying item $(iii)$ of Proposition \ref{cor:BNF}. 
Moreover
\be\label{YMU}
Y = ({\bf F}_{T}^{\theta}(U))_{|_{\theta=1}} =U+\mathtt{M}(U)[U]  \qquad {\rm where} \qquad 
\mathtt{M}(U) \in  \Sigma\mathcal{M}_{K,K',1}\otimes\mathcal{M}_{2}(\C) \, . 
\ee 
Then, substituting \eqref{YMU} in \eqref{sistemaBNF}, 
we obtain \eqref{sistemaBNF1000}-\eqref{Stimaenergy100}  with 
\begin{align}
{\mathfrak H}_{\geq 3} (U; x, \xi) &  := -\ii \big( {\mathtt D}(U;\xi) - {\mathtt D}(U+\mathtt{M}(U)[U];\xi)\big) + 
{\mathtt H}_{\geq 3} (U; x, \xi) \,, \label{DEF1}\\
{\mathfrak R}_{\geq 3} (U) & := 
  \tilde{\mathtt{R}}^{res} (U) - \tilde{\mathtt{R}}^{res} (U+\mathtt{M}(U)[U]) + 
 \mathtt{R}'_{\geq 3}(U) \, .\label{DEF2}
\end{align}
Since the integrable symbol $\mathtt{D}(U;\x) $  in \eqref{diag-part} 
is   homogeneous of degree $ 2 $, 
the quadratic terms in the right hand side of \eqref{DEF1} cancel out
and, by \eqref{YMU} and item $(iv)$ of Proposition \ref{composizioniTOTALI}, we deduce that 
${\mathfrak H}_{\geq 3} (U; x, \xi)\in\Gamma^{1}_{K,K',3}\otimes\mathcal{M}_2(\C)$ is an 
admissible symbol. 
Similarly, since 
$ \tilde{\mathtt{R}}^{res} (U) $ is  a smoothing operator  in 
$\widetilde{\mathcal{R}}^{-\rho+m_1}_2\otimes\mathcal{M}_2(\C) $, 
we deduce,  by \eqref{YMU} and 
item $(iii)$ of Proposition \ref{composizioniTOTALI},   
that  ${\mathfrak R}_{\geq 3} (U)$ defined in \eqref{DEF2} is a smoothing operator in 
$ \Sigma\mathcal{R}^{- (\rho - \rho_0)}_{K,K',3}\otimes\mathcal{M}_{2}(\C)$
where  $ \rho_0 := N_0 + 2 m_1  $.
\end{proof}

\bigskip
\section{ Long time existence}\label{sec:NFI}

The system 
\begin{equation}\label{sistemaBNFleq3}
\pa_{t}Y= -\ii\Omega Y- \ii \opbw\big(\mathtt{D}(Y;\x)\big)[Y]+ \tilde{\mathtt{R}}^{res}(Y)[Y]  \, , 
\end{equation}
obtained retaining only the vector fields in \eqref{sistemaBNF1000} up to degree $3$ of homogeneity,  
is in Poincar\'e-Birkhoff normal form. 
In Section \ref{sec:INF} we will actually prove that this is uniquely determined 
and that \eqref{sistemaBNFleq3} coincides with the Hamiltonian system 
generated by the  fourth order 
Birkhoff normal form Hamiltonian $H_{ZD} $ computed by a formal expansion in 
\cite{Zak2, CW, DZ, CS}, see Section \ref{sec:formal}.
 Such normal form is integrable and its corresponding Hamiltonian system preserves all Sobolev norms,
 see Theorem \ref{Zakintegro}. 
 The key new relevant information 
in Proposition \ref{cor:BNF} is that the quartic remainder in \eqref{Stimaenergy100} 
satisfies  energy estimates (see Lemma \ref{tempo}). 
This allows us to prove in Section \ref{sec:EE} 
energy estimates for the whole system \eqref{sistemaBNF1000} and thus 
the long time existence result of Theorem \ref{thm:main}.

\subsection{The formal Birkhoff normal form}\label{sec:formal}
We introduce, as in formula (2.7) of \cite{CS}, the complex symplectic variable 
\be\label{CVWW}
\left(\begin{matrix}
w  \\
\overline{w}
\end{matrix} 
\right) 
= 
\Lambda 
\left(\begin{matrix}
\eta \\
\psi
\end{matrix}
\right)
 : =
\frac{1}{\sqrt{2}} 
\left(\begin{matrix}
|D|^{-\frac{1}{4}}\eta+ \ii |D|^{\frac{1}{4}}\psi   \\
|D|^{-\frac{1}{4}}\eta - \ii |D|^{\frac{1}{4}}\psi  
\end{matrix} 
\right), \,    
\left(\begin{matrix}
\eta \\
\psi
\end{matrix}
\right) = \Lambda^{-1} 
\left(\begin{matrix}
w \\
 \bar w 
\end{matrix}
\right) = 
\frac{1}{\sqrt{2}} 
\left(\begin{matrix}
|D|^{\frac14}( w + \bar w ) \\
 - \ii  |D|^{- \frac14}(  w - \bar w )
\end{matrix}
\right).
\ee
Compare this formula with \eqref{u0} and recall that, in view of \eqref{uave0},  we may disregard the
zero frequency in what follows.
In the new complex variables $(w, \bar w ) $, a vector field $ X (\eta, \psi ) $ becomes 
\be\label{notationC} 
X^\C  :=
 \Lambda^\star X := \Lambda X \Lambda^{-1} \, . 
\ee
 The push-forward acts naturally on the commutator of nonlinear vector fields \eqref{exp:XY}, namely 
 \[
\Lambda^\star \bral X, Y\brar =  \bral\Lambda^\star X, \Lambda^\star Y\brar = \bral X^\C, Y^\C \brar  \, . 
\]
The Poisson bracket in \eqref{poissonBra} assumes the form 
$$
\{F,H\}=\frac{1}{\ii}\sum_{k\in \Z\setminus\{0\}}
  \big(\pa_{w_{k}}H\pa_{\bar{w_k}}F - \pa_{\bar{w_k}}H\pa_{w_k}F\big) \, .
$$
Given a Hamiltonian $ F (\eta, \psi )$ we denote by $ F_{\C}:= F \circ\Lambda^{-1}$
the same Hamiltonian expressed in terms of the complex variables $ (w, \bar w )$. The associated Hamiltonian 
vector field 
$ X_{F_{\C}} $ is 
\be\label{HSFC}
X_{F_{\C}} 
= \frac{1}{\sqrt{2 \pi}} \sum_{k\in\Z\setminus\{0\}} 
  \left(\begin{matrix}
- \ii  \pa_{ \ov{w_k}} F_{\C} \, e^{\ii k x}  \\
 \ii  \pa_{w_k} F_{\C}   \, e^{- \ii k x} 
\end{matrix} 
\right) \, ,
\ee
that we also identify, using the  standard vector field notation, with 
\be\label{HSuk}
X_{F_{\C}} 
=\sum_{k\in\Z\setminus\{0\}, \sigma=\pm} -\ii \s \pa_{w^{-\s}_{k}} F_\C  \, \pa_{w^{\s}_{k}} \, . 
\ee
Note that, if $X_{F}$ is the Hamiltonian vector field of $F$ in the real variables, then, using \eqref{notationC}, we have 
\begin{equation}\label{compreal}
X_{F}^{\C} := \Lambda^\star X_F =  X_{F_{\C}}  \, , \quad   F_{\C}:= F \circ\Lambda^{-1} \, , 
\end{equation}
and
\be\label{Poi-Bra}
\bral X^\C_H , X^\C_K \brar = X^\C_{ \{ H, K \} } = X_{ \{ H_\C, K_\C \} }  \, .
\ee
We now describe the formal Birkhoff normal form procedure performed in \cite{Zak2, DZ,CW,CS}.
One first expands the water waves Hamiltonian \eqref{Hamiltonian}, written in the complex variables $ (w, \bar w ) $,
 in degrees of homogeneity 
\begin{equation}\label{HVF2ham}
H_\C := H  \circ\Lambda^{-1} = H^{(2)}_\C +H^{(3)}_\C +H^{(4)}_\C + H^{(\geq5)}_\C,
\qquad  H^{(2)}_\C = \sum_{j \in \Z\setminus\{0\}} \om_j  w_j \bar{w_j}  \, , \quad \om_j  := \sqrt{|j|} \, ,
\end{equation}
 where 
\begin{align}
H^{(3)}_\C &= \sum_{\s_1 j_1 + \s_2 j_2 + \s_3 j_3 = 0 } H_{j_1, j_2, j_3}^{\s_1, \s_2, \s_3} w_{j_1}^{\s_1} w_{j_2}^{\s_2} w_{j_3}^{\s_3} \, , \label{cubic3} \\
H^{(4)}_\C &= \sum_{\s_1 j_1 + \s_2 j_2 + \s_3 j_3 + \s_4 j_4 = 0 } H_{j_1, j_2, j_3, j_4}^{\s_1, \s_2, \s_3, \s_4} 
w_{j_1}^{\s_1} w_{j_2}^{\s_2} w_{j_3}^{\s_3} w_{j_4}^{\s_4} \, , \label{vectHam}
\end{align}
can be explicitly computed. 
The Hamiltonian  
$H^{(\geq 5)}_\C $ collects all the monomials
of homogeneity greater or equal $5 $. 
The  Hamiltonians $H^{(3)}_\C$, $H^{(4)}_\C$ are real valued if and only if
their  coefficients satisfy 
\begin{equation}\label{real1000}
\ov{H_{j_1, j_2, j_3}^{\s_1, \s_2, \s_3}}=H_{j_1, j_2, j_3}^{-\s_1, -\s_2, -\s_3}, \, \quad 
\ov{H_{j_1, j_2, j_3,j_4}^{\s_1, \s_2, \s_3,\s_4}}=H_{j_1, j_2, j_3,j_4}^{-\s_1, -\s_2, -\s_3,-\s_4}.
\end{equation}
{\sc Step 1. Elimination of cubic Hamiltonian.}
One looks for a  symplectic transformation $ \Phi^{(3)} $ 
as the (formal) time $1$ flow  generated by a cubic real Hamiltonian $ F^{(3)}_\C $ of the form  \eqref{cubic3}.
Then a Lie  expansion gives 
\be\label{Hnew30}
H_\C  \circ  \Phi^{(3)}  = 
H^{(2)}_\C + \{  F^{(3)}_\C , H^{(2)}_\C  \} + H^{(3)}_\C + H^{(4)}_\C + 
\frac12 \{F^{(3)}_\C, \{ F^{(3)}_\C, H^{(2)}_\C \} \} + \{ F^{(3)}_\C, H^{(3)}_\C \} + \cdots 
\ee 
up to terms of quintic degree. 
The cohomological equation
\begin{equation}\label{formalomo}
H^{(3)}_\C +\{F^{(3)}_\C, H^{(2)}_\C \}=0  
\end{equation}
has a unique solution since
$$ 
\{  w_{j_1}^{\s_1} w_{j_2}^{\s_2} w^{\s_3}_{j_3}, H_{\C}^{(2)} \} = 
\ii (\s_1 \omega({j_1}) + \s_2 \omega({j_2}) + \s_3 \omega({j_3})  ) 
w_{j_1}^{\s_1} w_{j_2}^{\s_2} w^{\s_3}_{j_3}  \, ,
$$
and the system
\begin{equation}\label{ord3res}
 \s_1 j_1 + \s_2 j_2 + \s_3 j_3 = 0 \, , \qquad  \s_1 \omega({j_1}) + \s_2 \omega({j_2}) + \s_3 \omega({j_3})= 0 \, , 
\end{equation}
has no integer solutions, see Proposition \ref{stimedalbasso}. 
Hence, defining the cubic real valued Hamiltonian (see \eqref{real1000})
\[
F_{\C}^{(3)}= \sum_{\s_1 j_1 + \s_2 j_2 + \s_3 j_3 = 0 } 
\frac{-H_{j_1, j_2, j_3}^{\s_1, \s_2, \s_3}}{\ii ( \s_1 \omega({j_1}) + \s_2 \omega({j_2}) + \s_3 \omega({j_3}))} w_{j_1}^{\s_1} w_{j_2}^{\s_2} w_{j_3}^{\s_3} \, ,
\]
the Hamiltonian in \eqref{Hnew30} reduces to 
\be\label{Hnew3}
H_\C  \circ  \Phi^{(3)}  = H^{(2)}_\C + H^{(4)}_\C + \frac12 \{F^{(3)}_\C, H^{(3)}_\C  \} + {\rm quintic \ terms} \, . 
\ee 
\noindent
{\sc Step 2. Normalization of the quartic Hamiltonian.}
Similarly, one  can finds a  symplectic transformation $ \Phi^{(4)} $, 
defined  as the (formal) time $1$ flow generated by a real quartic Hamiltonian $ F^{(4)}_\C $ of the form \eqref{vectHam}, such that
\be\label{H4BNF}
H_{\C} \circ  \Phi^{(3)} \circ  \Phi^{(4)}  = H_{\C}^{(2)} +  {\it \Pi}_{\ker} \big(
 H^{(4)}_\C + \frac12 \{ F^{(3)}_\C, H^{(3)}_\C\} \big) + {\rm quintic \ terms} \, ,  
\ee
where, given a quartic monomial 
$ w_{j_1}^{\s_1} w_{j_2}^{\s_2} w_{j_3}^{\s_3} w_{j_4}^{\s_4} $
satisfying $\s_1 j_1+\s_2 j_2+\s_3 j_3 +\s_4 j_4 = 0$,
we define 
\begin{equation}\label{piker2}
{\it \Pi}_{\ker} \Big(  w_{j_1}^{\s_1} w_{j_2}^{\s_2} w_{j_3}^{\s_3} w_{j_4}^{\s_4}  \Big) := 
\begin{cases}
 w_{j_1}^{\s_1} w_{j_2}^{\s_2} w_{j_3}^{\s_3} w_{j_4}^{\s_4} \quad {\rm if} \ 
  \s_1 \omega({j_1}) + \s_2 \omega({j_2}) + \s_3 \omega({j_3}) +  \s_4 \omega({j_4})  = 0 \\
0 \quad \qquad\qquad \quad \ {\rm otherwise} \, . 
\end{cases}
\end{equation}
The fourth order (formal) Birkhoff normal form Hamiltonian in \eqref{H4BNF}, that is,
\be\label{H4tilde}
H_{ZD} = H^{(2)}_{ZD} + H^{(4)}_{ZD} \, , \qquad H^{(2)}_{ZD}:=H^{(2)}_{\C} \, , \quad
 H_{ZD}^{(4)} :=
  {\it \Pi}_{\ker} \big( H^{(4)}_\C + \frac12 \{F^{(3)}_\C, H^{(3)}_\C \} \big) \, ,
\ee
has been computed explicitly  in \cite{Zak2,CW, DZ,CS}, and it is completely integrable.
In \cite{CW}  this is expressed as
\be\label{new}
H_{ZD} = 
\sum_{k>0} \Big(  2 \omega_k I_1 (k) - \frac{k^3}{2 \pi}( I_1^2 (k) - 3 I_2^2 (k) ) \Big) +
\frac{4}{\pi}  \sum_{0 < k < l} k^2 l I_2 (k) I_2 (l)  
\ee
with actions 
\begin{equation}\label{ActionCW}
 \!\!I_1 (k) := \frac{z_k \ov{z_k} + z_{-k} \ov{ z_{-k}} }{2} \, , \quad 
 I_2 (k) := \frac{z_k \ov{ z_k} - \ov{ z_{-k}}  z_{-k}}{2  } \, ,  
\end{equation}
where 
 $z_{k}$  denote the Fourier coefficients of $ z = \frac{1}{\sqrt{2}}|D|^{-\frac{1}{4}}\eta+\frac{\ii}{\sqrt{2}}|D|^{\frac{1}{4}}\psi $ defined in \eqref{CVWW}.
 When expressed in terms of the complex variables $ (z_k, \bar{z_k}) $, the Hamiltonian $ H_{ZD} $ is given by
$H^{(2)}_{ZD}+H^{(4)}_{ZD} $ as in \eqref{theoBirHfull}-\eqref{theoBirH}.

The associated Hamiltonian system 
is (see \eqref{HSuk}) 
\be\label{HCS1}
\begin{aligned}
\dot z_n 
 = - \ii \omega_n z_n  & + \frac{\ii}{\pi}\!\!\! \sum_{\substack{|k_4| < |n|, \\ - \sign(n) = \sign( k_4 )}} \!\!\! |n| |k_4|^2  |z_{k_4}|^2 z_n 
  - \frac{\ii}{\pi} \!\!\!\sum_{\substack{|k_4| < |n|, \\ \sign(n) = \sign( k_4 )}} \!\!\! |n| |k_4|^2  |z_{k_4}|^2 z_n + [R(z)]_n
\end{aligned}
\ee
where
\begin{align}\label{primRn}
[R(z)]_n := -  \frac{\ii}{ 2 \pi}  |n|^3 \big( | z_n |^2 - 2  | z_{-n} |^2 \big) z_n 
 & +  \frac{ \ii }{\pi}  \!\!\!  \sum_{\substack{|n| < |k_1|, \\ \sign(k_1) = \sign( n )}} \!\!\! |k_1| |n|^2 z_{n} \big(|z_{-k_1}|^2- |z_{k_1}|^2\big) \, .
   \end{align}
   Note in particular that $ | z_n |^2 $ are prime integrals, as stated in Theorem \ref{Zakintegro}.

 Although it is not necessary for the paper,  for completeness
we compare explicitly the structure of the normal form vector field 
\eqref{HCS1}-\eqref{primRn} with \eqref{sistemaBNFleq3}. 

\begin{lemma}\label{sec:CS}
The   Hamiltonian system \eqref{HCS1}-\eqref{primRn} has the form 
\be\label{CSFIN}
\dot z_n = - \ii \omega_n z_n 
- \frac{\ii}{\pi} \Big( \sum_{|j| < \epsilon |n|}  j |j|  |z_{j}|^2 \Big)   n z_n 
+ [\mathfrak R (z)]_n
\ee
where $0<\epsilon<1$ and  $ {\mathfrak R}(z) $ is a smoothing vector field in the sense that, 
for any $0\leq \rho\leq 2 s-3 $, 
\be \label{mfRs}
\| \mathfrak R (z)\|_{s+\rho} \leq C(s) \| z \|^{3}_{s}  \, ,
\ee
where for a sequence $a=\{a_j\}_{j\in\Z}$ we define
$\|a\|_{s}^{2}:=\sum_{j\in\Z}\langle j\rangle^{2s}|a_j|^{2}$.
Note that 
\be\label{paratra}
- \sum_{n} \frac{\ii}{\pi} \Big( \sum_{|j| < \epsilon |n|}  j |j|  |z_{j}|^2 \Big)   n z_n \frac{1}{\sqrt{2 \pi}} e^{\ii nx}=
\opbw{  (- \ii  \zeta (Z) \xi ) } z 
\ee
where $ \zeta (Z)  $ is defined in \eqref{diag-part}. 
\end{lemma}

\begin{proof}
Note that 
\begin{align}
&   \sum_{|k_4| < |n|, - \sign(n) = \sign( k_4 )} |n| |k_4|^2  |z_{k_4}|^2 z_n 
-  \sum_{|k_4| < |n|, \sign(n) = \sign( k_4 )} |n| |k_4|^2  |z_{k_4}|^2 z_n \nonumber \\
& \qquad = 
-  \sum_{|k_4| < |n|, - \sign(n) = \sign( k_4 )} n \,  \sign (k_4)  |k_4|^2  |z_{k_4}|^2 z_n 
-  \sum_{|k_4| < |n|, \sign(n) = \sign( k_4 )} n \sign (k_4) |k_4|^2  |z_{k_4}|^2 z_n \nonumber  \\
& \qquad = - n \Big( \sum_{|k_4| < |n|}   \sign (k_4)  |k_4|^2  |z_{k_4}|^2 \Big) z_n = - n \Big( \sum_{|j| < |n|}   j |j|  |z_j|^2 \Big) z_n  \, . \label{finale}
\end{align}
Then \eqref{finale} allows to write \eqref{HCS1} as \eqref{CSFIN} 
where
$$
[\mathfrak R (z)]_n  :=  -  \frac{\ii}{\pi}  \Big( \sum_{ \epsilon |n| \leq |j| < |n|}  j |j|  |z_{j}|^2 \Big)   n z_n + [R(z)]_n \, . 
$$
The vector $ \mathfrak R (z) $ satisfies the (super) smoothing estimate \eqref{mfRs} since it contains three high comparable
frequencies.
\end{proof}

Note that \eqref{paratra}
is the transport paradifferential operator of order $ 1 $   in \eqref{sistemaBNF1000} and \eqref{diag-part}.
Note also that $ \zeta (Z) $  in \eqref{diag-part} vanishes on the subspace of functions even in $ x $, coherently with the fact that
the Hamiltonian in the second line of \eqref{theoBirH}
vanishes as well on even functions.  
We also remark that \eqref{CSFIN} does not contain paradifferential operators 
at non-negative orders, in agreement with the form of 
the cubic terms in the Poincar\'e-Birkhoff normal form  in \eqref{sistemaBNF1000} and \eqref{diag-part}.  
In the next section we actually prove that these have to coincide.  

\subsection{Normal form identification}\label{sec:INF}
In Sections \ref{sec:3}-\ref{sec:BNF} we have
transformed the water waves system \eqref{eq:113} into \eqref{sistemaBNF1000}, 
whose cubic component \eqref{sistemaBNFleq3} is in Poincar\'e-Birkhoff normal form.  
All the conjugation maps that we have used
have an expansion in homogeneous components up to degree $ 4 $.  
In this section we identify the cubic monomials left in
the Poincar\'e-Birkhoff normal form \eqref{sistemaBNFleq3}. 
The main result is the following.

\begin{proposition}{\bf (Identification of normal forms)} \label{teosez9}
The cubic vector field component in \eqref{sistemaBNF1000}, i.e.
\begin{equation}\label{finalcubic}
\mathcal{X}_{Res}(Y):=- \ii \opbw\big(\mathtt{D}(Y;\x)\big)[Y]+ \tilde{\mathtt{R}}^{res}(Y)[Y] \, , 
\end{equation}
coincides with the Hamiltonian vector field 
\begin{equation}\label{ident}
\mathcal{X}_{Res} 
=  X_{ {\it \Pi}_{\ker} ( H_{\C}^{(4)} + \frac12 \{ F_{\C}^{(3)}, H_{\C}^{(3)}\})} = X_{H^{(4)}_{ZD}}
\end{equation}
where the Hamiltonians $  H^{(l)}_\C $, $ l = 3, 4 $, are  defined in \eqref{HVF2ham},  
$ F^{(3)}_\C $ is the unique solution of \eqref{formalomo},  and 
$ {\it \Pi}_{\ker} $ is defined in \eqref{piker2}. 
\end{proposition}

The rest of the section is devoted to the proof of Proposition \ref{teosez9},
which is based on a uniqueness argument for the Poincar\'e-Birkhoff normal form up to quartic remainders. 
The idea is the following. We first expand  the water waves Hamiltonian vector field in  \eqref{eq:113},\eqref{HS}
in degrees of homogeneity
\be\label{HVF2}
X_H =  X_1 + X_2 + X_3 + X_{\geq 4}   \qquad {\rm where} \qquad 
X_{1}:=X_{H^{(2)}}, \  X_{2}:=X_{H^{(3)}}, \  X_{3}:=X_{H^{(4)}} \, ,
\ee
where $X_{\geq 4}  $ collects of the higher order terms
and $H^{(p)} := H_{\mathbb{C}}^{(p)}\circ\Lambda$, $p=2,3,4$, see \eqref{HVF2ham}.
Then, in order to identify the cubic monomial vector fields in \eqref{finalcubic}
we express the transformed system \eqref{sistemaBNF1000}, obtained conjugating 
\eqref{eq:113}  via the
good-unknown transformation
$ {\mathcal G} $ in \eqref{Alinach-good}  and $ {\bf F}_{T}^{1} $ in Proposition \ref{cor:BNF}, 
by a Lie commutator expansion up to terms of homogeneity at least $4$. See Lemma \ref{quasi-flusso}.
Note that the quadratic and cubic terms in \eqref{finalcubic} may arise by only the conjugation of
 $ X_1 + X_2 + X_3 $ under the
homogeneous components up to cubic terms of the paradifferential transformations $ {\mathcal G} $ 
and $ {\bf F}_{T}^{1}$. 
Then, after some algebraic manipulation, we obtain  the formulas
\eqref{DefG2}-\eqref{DefG4}. 
Since the adjoint operator $ {\rm Ad}_{X_{H^{(2)}}^\C} := [ \, \cdot \, , X_{H^{(2)}}^\C ] $
acting on quadratic monomial vector fields 
satisfying the momentum conservation property is injective and surjective we then obtain the identity \eqref{F3},
and can eventually deduce  \eqref{ident}.

\begin{itemize}
\item {\bf Notation.} 
We use the Lie  expansion \eqref{espansionecampo} 
 induced by a time-dependent vector field $ S $, which contains quadratic and cubic terms. 
Given a homogeneous vector field $ X $, we denote
by $ \Phi_{S}^\star X $ the induced (formal) push forward  
\be\label{push-forward}
\Phi_{S}^\star X = X + \bral S,X \brar_{|\theta = 0} + \frac{1}{2} \bral S, \bral S,X \brar \brar_{|\theta = 0} + 
\frac12 \bral \pa_\theta S_{|\theta=0},X \brar + \cdots 
\ee
where $\bral \cdot, \cdot \brar$ is the non linear commutator defined in  \eqref{exp:XY}.
\end{itemize}

\noindent
{\bf Step $ 1$. The good unknown change of variable $ {\mathcal G} $ in \eqref{Alinach-good}}. 
We first provide the Lie expansion  up to degree four  of the  vector field 
in \eqref{eq:1n}-\eqref{eq:2n}, which is 
obtained by transforming the
water waves vector field 
$ X_1 + X_2 + X_3 $ in \eqref{HVF2}  under the nonlinear map $ {\mathcal G} $ in \eqref{Alinach-good}.

We first note that $ {\mathcal G}(\eta,\psi)=(\Phi^{\theta}(\eta, \psi))_{\theta=1} $ where
$\Phi^{\theta} (\eta,\psi) := (\eta, \psi-\theta\opbw(B(\eta,\psi))\eta) $, $\theta\in [0,1]$.  
Since $ B(\eta, \psi) $ is a function in $ \Sigma {\mathcal F}^\R_{K,0,1} $ we have,
using the remarks under Definition  \ref{smoothoperatormaps}, that 
the map $\Phi^{\theta} (\eta, \psi ) $ has  the form  \eqref{espmultilin}
in which $ U $ denotes the real variables $(\eta,\psi)$, plus a map in $ \mathcal{M}_{K,0,3}\otimes\mathcal{M}_2(\C) $.
By Lemma \ref{compolinearflows}
we regard the inverse of the map $ {\mathcal G}_{\leq 3}$,  obtained 
approximating $ {\cal G} $ up to quartic remainders, as  the (formal) 
 time one flow of a non-autonomous vector field of the form 
\be\label{S34}
 S(\theta) := \mathtt{S}_2 +  \theta \mathtt{S}_3    \;\;\;\;\;{\rm where}\;\;\;\;\;
\mathtt{S}_2 := S_1(\eta,\psi)\vect{\eta}{\psi} \, , \quad \mathtt{S}_3 := {S}_2(\eta,\psi)\vect{\eta}{\psi}  \, , 
\ee
where
 $ S_{1}(\eta,\psi) $ is in $ \widetilde{\mathcal{M}}_{1}\otimes\mathcal{M}_2(\C) $ and 
 $ S_2(\eta,\psi) $ is in $  \widetilde{\mathcal{M}}_{2}\otimes\mathcal{M}_2(\C) $.
By \eqref{HVF2}, \eqref{push-forward} and  \eqref{S34}, we get
\begin{align}
\Phi_{S}^\star (X_1+X_2+X_3) &  = X_1 + X_{2,1}+ 
X_{3,1} + \cdots \label{VFB}
\end{align}
where
\be\label{K3K4}
X_{2,1} := X_2 + \bral\mathtt{S}_2, X_1\brar \, , \ \ 
X_{3,1} := X_3 + \bral \mathtt{S}_2, X_2\brar +  \frac12 \bral\mathtt{S}_2, \bral \mathtt{S}_2, X_1 \brar\brar 
+ \frac12 \bral   \mathtt{S}_3 , X_1\brar\, . 
\ee

\vspace{0.5em}
\noindent
{\bf Complex coordinates $\Lambda$ in \eqref{CVWW}}. In the complex 
coordinates \eqref{CVWW},  the vector field \eqref{VFB} reads, recalling the notation \eqref{notationC}, 
\be\label{VFComplex}
\Lambda^\star \Phi_{S}^\star (X_1+X_2+X_3) = 
\Lambda^\star X_1 + \Lambda^\star X_{2,1}  + \Lambda^\star X_{3,1} + \cdots =
X_1^\C + X_{2,1}^\C + X_{3,1}^\C + \cdots 
\ee
where $X_{1}^{\C} $ is the linear Hamiltonian 
vector field $ X_{1}^{\C} =X_{H^{(2)}}^{\C} =  - \ii  \sum_{j, \s} \sigma \omega_j u_j^\s \pa_{u_j^\s} $.

\vspace{0.5em}
\noindent
{\bf  Step $ 2 $. The transformation $ {\bf F}_{T}^{1}$ in Proposition \ref{cor:BNF}}. 
We consider the nonlinear  map $ ({\bf F}_{T}^1)_{\leq 3}  $ obtained retaining only the terms
of the map $ {\bf F}_{T}^{1} :=( {\bf F}_{T}^{\theta})_{|_{\theta=1}} $ up to quartic remainders. 
The approximate inverse of the map $ ({\bf F}_{T}^1)_{\leq 3}  $ 
provided by Lemma \ref{quasi-inversa}, 
can be regarded, by Lemma \ref{compolinearflows},  as the (formal) 
approximate time-one flow of a non-autonomous vector field 
$$ 
 {T}(\theta) :=  \mathtt{T}_2 +   \theta \mathtt{T}_3    \;\;\;\;\; {\rm where}\;\;\;\;\;  
\mathtt{T}_2(U):=T_1(U)[U] \, , \quad   \mathtt{T}_3(U):=T_2(U)[U] \, , 
$$
 for some 
 $ T_{1}(U) $ in $ \widetilde{\mathcal{M}}_{1}\otimes\mathcal{M}_2(\C) $ and 
 $ T_2(U) \in   \widetilde{\mathcal{M}}_{2}\otimes\mathcal{M}_2(\C) $.
We transform the system obtained retaining only the components   
$ X_1^\C + X_{2,1}^\C + X_{3,1}^\C $ in  \eqref{VFComplex}. 
By \eqref{push-forward} we get  
\begin{equation}\label{push1}
\Phi_T^\star \Lambda^\star \Phi_{S}^\star (X_1+X_2+X_3)  =  X_1^\C + X_{2,2}^\C + X_{3,2}^\C + \cdots  
\end{equation}
where 
$$
X_{2,2}^\C := X_{2,1}^\C + \bral\mathtt{T}_2, X_1^\C\brar \, , \ \ 
X_{3,2}^\C := X_{3,1}^\C + \bral \mathtt{T}_2, X_{2,1}^\C\brar +  
\frac12 \bral\mathtt{T}_2, \bral \mathtt{T}_2, X_1^\C \brar\brar + 
\frac12 \Bral  \mathtt{T}_3, X_1^\C \Brar 
$$
and, recalling the expressions of $ X_{2,1}, X_{3,1} $ in \eqref{K3K4}, 
the quadratic and the cubic components of the vector field \eqref{push1} are given by
\begin{align}\label{M3}
& X_{2,2}^{\C}   = 
X_2^\C + \bral  \mathtt{S}_2^\C + \mathtt{T}_2, X_1^\C\brar 
\end{align}
and
\be\label{M4}
\begin{aligned}
X_{3,2}^{\C} & =  X_3^\C + \bral  \mathtt{S}_2^\C+ \mathtt{T}_2, X_2^\C\brar +  
\frac12 \bral \mathtt{S}_2^\C, \bral  \mathtt{S}_2^\C, X_1^\C \brar\brar 
+ \bral\mathtt{T}_2,  \bral  \mathtt{S}_2^\C, X_1^\C\brar\brar + 
\frac12 \bral \mathtt{T}_2, \bral \mathtt{T}_2, X_1^\C \brar\brar    \\
& \ \ +  \frac12 \Bral   \mathtt{S}_3^\C +    \mathtt{T}_3, X_1^\C \Brar \, . 
\end{aligned}
\ee

\begin{lemma} 
Given vector fields $ X, Y, Z $ we have the identity
\be\label{identXYZ}
\frac12 \bral Y, \bral Y,X  \brar \brar
+ \bral Z,  \bral Y, X \brar \brar +  \frac12 \bral Z, \bral Z, X \brar \brar = 
\frac12 \bral Y+Z, \bral Y+ Z, X \brar \brar  + \frac12 \bral  \bral Y+Z, Y\brar , X\brar  \, .
\ee
\end{lemma}
\begin{proof}
Use the Jacobi identity $\bral X, \bral Y, Z\brar \brar+
\bral Z, \bral X, Y\brar \brar+\bral Y, \bral Z, X\brar \brar=0$.
\end{proof}

Using the identity \eqref{identXYZ}, the term $ X_{3,2}^{\C} $ in \eqref{M4} is
\be\label{M4-bis}
X_{3,2}^{\C}  = X_3^\C + \bral  \mathtt{S}_2^\C+ \mathtt{T}_2, X_2^\C\brar +  
\frac12 \Bral \mathtt{S}_2^\C + \mathtt{T}_2, \bral  \mathtt{S}_2^\C + \mathtt{T}_2, X_1^\C\brar \Brar 
+   \frac{1}{2}  \Bral \bral \mathtt{S}_2^\C + \mathtt{T}_2, \mathtt{S}_2^\C\brar 
+  \mathtt{S}_3^\C +  \mathtt{T}_3 , X_1^\C\Brar \, .
\ee
{\bf Step $ 3$. Identification of  quadratic and cubic vector fields. }
The vector field $ \Phi_T^\star \Lambda^\star  \Phi_{S}^\star  (X_1+X_2+X_3) $ in \eqref{push1} 
is 
the  vector field in the right hand side of \eqref{sistemaBNFleq3}, up to quartic remainders. 
Thus, recalling the expression of the 
quadratic, respectively cubic, vector field in \eqref{M3}, respectively  \eqref{M4-bis}, the expansion \eqref{HVF2},
formula \eqref{compreal}, 
and the definition of $\mathcal{X}_{Res}$ in \eqref{finalcubic}, we have the identification order by order: 
\begin{align}
& X_{1}^{\C}(Y) = -  \ii\Omega Y  \label{DefG2} \\
& X_{H^{(3)}}^\C + \bral \mathtt{S}_2^\C + \mathtt{T}_2 , X_{H^{(2)}}^\C\brar  = 0 \label{DefG3} \\ 
& X_{H^{(4)}}^\C + \bral  \mathtt{S}_2^\C+ \mathtt{T}_2, X_2^\C\brar +  
\frac12 \bral \mathtt{S}_2^\C + \mathtt{T}_2, \bral  \mathtt{S}_2^\C + \mathtt{T}_2, X_1^\C\brar \brar 
+ \frac12 \Bral   \bral \mathtt{S}_2^\C + \mathtt{T}_2, \mathtt{S}_2^\C\Brar 
+  \mathtt{S}_3^\C +  \mathtt{T}_3 , X_1^\C\Brar  =\mathcal{X}_{Res}  \, . \label{DefG4}
\end{align}

\noindent
{\bf Quadratic vector fields.} 
Since  $F_{\C}^{(3)}$ solves \eqref{formalomo}, by \eqref{Poi-Bra}, we have
\be\label{Poisson3}
X_{H^{(3)}}^{\C} + \bral X_{F_{\C}^{(3)}} , X_{H^{(2)}_\C} \brar = 0 \, . 
\ee
Subtracting  \eqref{DefG3} and \eqref{Poisson3}, and since $ X_{H^{(2)}_\C} = X_{H^{(2)}}^\C  $, we deduce 
$$
\bral \mathtt{S}_2^\C + \mathtt{T}_2  - X_{F_{\C}^{(3)}} , X_{H^{(2)}}^\C \brar  = 0 \, . 
$$
The adjoint operator
$ {\rm Ad}_{X_{H^{(2)}}^\C} := [ \, \cdot \, , X_{H^{(2)}}^\C ] $
acting on quadratic monomial vector fields $ u_{j_1}^{\s_1} u_{j_2}^{\s_2}  \pa_{u^{\s}_j} $
satisfying the momentum conservation property
$ \s j = \s_1 j_1 + \s_2 j_2  $, 
is injective and surjective. Indeed we have that
$ \bral  u_{j_1}^{\s_1} u_{j_2}^{\s_2}  \pa_{u^{\s}_j} , X_{H^{(2)}}^\C \brar = $ 
$ -\ii (\s \omega(j) - \s_1 \omega({j_1}) - \s_2 \omega({j_2}) ) 
u_{j_1}^{\s_1} u_{j_2}^{\s_2} \pa_{u^{\s}_j}   $
and the system \eqref{ord3res} has no solutions.
As a consequence 
\be\label{F3}
\mathtt{S}_2^\C + \mathtt{T}_2 = X_{F_\C^{(3)}} \, .
\ee
\noindent
{\bf Cubic vector fields.} 
The vector field $\mathcal{X}_{Res}$ defined in \eqref{finalcubic}  is in Poincar\'e-Birkhoff normal form,
since the symbol $ {\mathtt D}(Y;\x)$ is \emph{integrable} (Definition \ref{defiintegro}) 
and $\tilde{\mathtt{R}}^{res}(U) $    is \emph{ Birkhoff resonant} (Definition \ref{resterm}). 
Therefore, defining the linear operator 
$ \Pi_{\ker}  $ acting on a cubic monomial vector field
$ u_{j_1}^{\s_1} u_{j_2}^{\s_2} u_{j_3}^{\s_3} \pa_{u^{\s}_j}$ 
as
\be\label{Pikerdi}
\Pi_{\ker} \Big( u_{j_1}^{\s_1} u_{j_2}^{\s_2} u_{j_3}^{\s_3} \pa_{u^{\s}_j} \Big) := 
\begin{cases}
u_{j_1}^{\s_1} u_{j_2}^{\s_2} u_{j_3}^{\s_3} \pa_{u^{\s}_j} \quad {\rm if} \ - \s\omega(j) + \s_1 \omega({j_1}) + \s_2 \omega({j_2}) + \s_3 \omega({j_3}) = 0 \\
0 \qquad \qquad\quad \quad \ {\rm otherwise} \, , 
\end{cases}
\ee
 we have 
\begin{equation}\label{sononelKer}
\Pi_{\ker}(\mathcal{X}_{Res}) = \mathcal{X}_{Res} \, .
\end{equation}
In addition, since 
$ \bral  u_{j_1}^{\s_1} u_{j_2}^{\s_2} u_{j_3}^{\s_3}  \pa_{u^{\s}_j} , X_{H^{(2)}}^\C \brar = $
$ -\ii (\s \omega(j) - \s_1 \omega({j_1}) - \s_2 \omega({j_2}) - \s_3 \omega({j_3}) ) 
u_{j_1}^{\s_1} u_{j_2}^{\s_2}  u_{j_3}^{\s_3}  \pa_{u^{\s}_j} $
we deduce that, for any cubic vector field $ G_3 $, 
\begin{equation}\label{kerQ4}
\Pi_{{\rm Ker}} \bral G_3, X_{H^{(2)}}^\C \brar =  0 \, .  
\end{equation}
We can then calculate
\begin{equation*}
\begin{aligned}
\mathcal{X}_{Res}  \stackrel{\eqref{sononelKer}} =   \Pi_{\ker}  (\mathcal{X}_{Res})  & 
\stackrel{\eqref{DefG4}, \eqref{kerQ4}} =
\Pi_{\ker} 
\Big( X_{H^{(4)}}^\C + \bral \mathtt{S}_2^\C+ \mathtt{T}_2, X_2^\C\brar  + 
\frac12 \bral \mathtt{S}_2^\C + \mathtt{T}_2,
 \bral \mathtt{S}_2^\C + \mathtt{T}_2, X_1^{\C}\brar \brar 
\Big) \\
& \stackrel{(\ref{F3}),\eqref{HVF2}}{=}
\Pi_{\ker} 
\Big( 
X_{H^{(4)}}^\C + \bral X_{F_{\C}^{(3)}},X_{H^{(3)}}^\C \brar+   \frac12 \bral X_{F_{\C}^{(3)}}, \bral X_{F_{\C}^{(3)}} , X_{H_{\C}^{(2)}} \brar \brar
\Big) \\
& \stackrel{\eqref{compreal},\eqref{Poi-Bra}} =
\Pi_{\ker} 
\Big( 
X_{H_{\C}^{(4)} + \{F_{\C}^{(3)}, H_{\C}^{(3)}\} 
+ \frac12 \{F_{\C}^{(3)} ,\{ F_{\C}^{(3)}, H_{\C}^{(2)} \} \}} \Big)  \\
& \stackrel{\eqref{formalomo}} = 
\Pi_{\ker} \big( X_{H_{\C}^{(4)} + \frac12 \{F_{\C}^{(3)}, H_{\C}^{(3)} \}}\big)  \\
& \stackrel{\eqref{Pikerdi}, \eqref{piker2}} =    X_{ {\it \Pi}_{\rm ker}  (H_{\C}^{(4)} + \frac12 \{F_{\C}^{(3)}, H_{\C}^{(3)}\})}
\end{aligned}
\end{equation*}
which is  \eqref{ident}; the second identity follows by the definition of $ H^{(4)}_{ZD}$ in \eqref{H4tilde}.

\subsection{Energy estimate  and proof of Theorem \ref{BNFtheorem}}\label{sec:EE}

We first prove the following lemma.

 \begin{lemma}\label{equiv-tempo}
 Let $ K \in \N^* $.
There is $ s_0 > 0 $ such that, 
for any $ s \geq s_0 $, for all $ 0 < r \leq r_0(s) $ small enough, 
 if  $U$  belongs to $B_{s}^{K}(I;r)$
 and solves \eqref{eq:415}, then
 there is a constant $ C_{s,K}>0$  such that
\begin{equation}\label{aslong1}
\|\pa_{t}^{k}U(t,\cdot)\|_{\dot{H}^{s-k}}\leq C_{s,K} \|U(t,\cdot)\|_{\dot{H}^{s}} \, , \quad 
\forall \, 0\leq k \leq K
\, .
\end{equation}
In particular the norm 
$\|U (t, \cdot) \|_{K,s}$ defined in \eqref{normaT} is equivalent to the norm 
$\|U(t, \cdot) \|_{\dot{H}^{s}}$
for $U(t,\cdot)$ a solution of \eqref{eq:415}.
\end{lemma}

\begin{proof}
For $k=0$ the estimate \eqref{aslong1} is trivial. 
We are going to estimate $ \pa_t^k U$ by \eqref{eq:415}. Since the matrix of symbols
$ \ii A_1(U;x)\x+\ii A_{1/2}(U;x)|\x|^{\frac{1}{2}}+A_0(U;x,\x)+A_{-1/2}(U;x,\x) $ in \eqref{eq:415} 
belongs to $ \Sigma\Gamma_{K,1,0}^{1} \otimes\mathcal{M}_2(\C)  $
and the smoothing operator $R(U)$ is in 
$ \Sigma\mathcal{R}^{-\rho}_{K,1,1}\otimes\mathcal{M}_2(\C) $,
applying 
Proposition \ref{azionepara}-$(ii)$ (with $K'=1$, $k=0$),  
the estimate \eqref{piove} for $ R(U) $ (with $K'=1$, $k=0$, $ N = 1 $),
and recalling  \eqref{normaT}, 
we deduce,  for $ s \geq s_0$ large enough,  
\begin{equation}\label{tempo1}
\begin{aligned}
\|\pa_t U(t,\cdot)\|_{\dot{H}^{s-1}} & \lesssim_s \|U(t,\cdot)\|_{\dot{H}^{s}}
\big(1+\|U(t,\cdot)\|_{\dot{H}^{s_0}}+\|\pa_{t}U(t,\cdot)\|_{\dot{H}^{s_0-1}}\big) \\
& \ \quad +\|\pa_{t}U(t,\cdot)\|_{\dot{H}^{s-1}}\|U(t,\cdot)\|_{\dot{H}^{s_0}} \, .
\end{aligned}
\end{equation}
Evaluating \eqref{tempo1} at $s=s_0$ 
and since  $\|U(t,\cdot)\|_{\dot{H}^{s_0}}$ is small, we get
$\|\pa_tU(t,\cdot)\|_{\dot{H}^{s_0-1}}\lesssim_s\|U(t,\cdot)\|_{\dot{H}^{s_0}}.$ 
The latter inequality and \eqref{tempo1} imply \eqref{aslong1} for $k=1$, for any $s \geq s_0$ . 
Differentiating in $t$ the system \eqref{eq:415} 
and arguing by induction on $k$, one proves similarly \eqref{aslong1} for any $ k \geq 2 $. 
\end{proof}

We now prove the following energy estimate.

\begin{lemma}{\bf (Energy estimate)}\label{tempo}
Under the same assumptions of Proposition \ref{cor:BNF}
the vector field 
$\mathcal{X}_{\geq4} (U,Y) = \vect{\mathcal{X}_{\geq4}^+ (U,Y) }{\overline{\mathcal{X}_{\geq4}^+ (U,Y)} } $  in
\eqref{Stimaenergy100} satisfies, for any $t\in I$,
the  energy estimate
\begin{equation}\label{stimasuY}
{\rm Re}\int_{\T}|D|^{s}\mathcal{X}_{\geq4}^+(U,Y)\cdot \ov{|D|^{s} y }\; dx
\lesssim_{s} \| y \|_{\dot{H}^{s}}^{5} \, .
\end{equation}
\end{lemma}

\begin{proof}
By \eqref{Stimaenergy100} and 
\eqref{def:admissible2}, we have that 
$ \mathcal{X}_{\geq4}^+(U,Y) = \opbw( H_{\geq3})[y]  + \mathfrak{R}_{\geq3}^+ (U)[Y] $
where  $H_{\geq3} $ is an admissible symbol as in \eqref{highordersfinali} that we write 
\be\label{decoH3}
H_{\geq3} 
=  h^+_{\geq3} (U; x, \xi)
+ \gamma_{\geq3}(U;x,\x) \, , \quad 
h^+_{\geq3} (U; x, \xi) := \ii \alpha_{\geq3}(U;x)\x+\ii \beta_{\geq 3}(U;x)|\x|^{\frac{1}{2}} \, , 
\ee
 and 
$ \mathfrak{R}_{\geq3}^+ (U) $  denotes  the first row  of $ \mathfrak{R}_{\geq3} $.
Then the left hand side of \eqref{stimasuY} is equal to 
\begin{align}
& \ \frac12 \big(|D|^{s} y, |D|^{s}\opbw(h^+_{\geq3})[y]\big)_{L^{2}}
+ \frac12 \big( |D|^{s} \opbw(h^+_{\geq3})[y], |D|^{s} y\big)_{L^{2}}\label{linea2}\\
& + {\rm Re}\int_{\T}|D|^{s}  \opbw(\gamma_{\geq 3})[y]  \cdot \ov{|D|^{s} y }\; dx 
+ {\rm Re}\int_{\T}|D|^{s}  \mathfrak{R}_{\geq3}^+ (U)[Y] \cdot \ov{|D|^{s} y }\; dx \, .    
\label{linea3}
\end{align}
Since $ \gamma_{\geq3} \in   \Gamma_{K,K',3}^{0} $ and $  \mathfrak{R}_{\geq3}^+ (U) $ is a 
$ 1 \times 2 $ matrix of smoothing operators in $ \mathcal{R}^{0}_{K,K',3} $,   
Cauchy-Schwarz inequality, Proposition \ref{azionepara} and \eqref{piove} imply that
\be\label{stimabb}
|\eqref{linea3}| \lesssim_s \| y(t,\cdot)\|_{\dot{H}^{s}}^{2}\|U(t,\cdot)\|^{3}_{K,s} \, .
\ee
We now prove that  \eqref{linea2} satisfies the same bound. 
Since  the symbol $ h^+_{\geq 3} $ has positive order we  decompose it 
according to 
\begin{align}
\eqref{linea2} & =\tfrac{1}{2}
\big(|D|^{s}y,|D|^{s}( \mathcal{H}_{\geq 3}+ \mathcal{H}_{\geq 3}^{*})y\big)_{L^2}  +\tfrac{1}{2}\big(
|D|^{s}y, [  \mathcal{H}_{\geq 3}^{*}, |D|^{s}]y
\big)_{L^2} \! +\tfrac{1}{2} \!
\big( [|D|^{s},  \mathcal{H}_{\geq 3}] y, |D|^{s}y  \big)_{L^2}\label{linea5} 
\end{align}
where $\mathcal{H}_{\geq 3} :=\opbw(h_{\geq3}^+ (U;x,\x)) $ and 
$ \mathcal{H}_{\geq 3}^{*}=\opbw(\ov{h_{\geq3}^+(U;x,\x)}) $ is its 
adjoint 
with respect to the  $L^{2}$-scalar product. 
Recalling \eqref{decoH3} and that the functions $ \alpha_{\geq3}(U;x) $, $ \beta_{\geq 3}(U;x) $ are real 
 we have 
\be\label{fazero}
\mathcal{H}_{\geq 3} + \mathcal{H}^{*}_{\geq 3} = \opbw 
\big( h^+_{\geq3} +\ov{h^+_{\geq3}} \, \big) = 0  
\ee
Furthermore, by Proposition \ref{teoremadicomposizione} and the remark below it,
the commutators 
$ [\mathcal{H}_{\geq 3}^*, |D|^{s}] $, $ [|D|^{s}, \mathcal{H}_{\geq 3}]$   are
paradifferential operators with symbol  in $\Gamma^{s}_{K,K',3} $,  up to a bounded 
operator in $ {\cal L}(\dot H^s, \dot H^{0}) $ with operator norm bounded by $ \| U \|_{K,s_0}^3 $. 
Then applying Proposition \ref{azionepara} we get 
 \be\label{sti2}
\big| \big(
|D|^{s}y, [  \mathcal{H}_{\geq 3}^{*}, |D|^{s}]y
\big)_{L^2} \big| +
\big| \big( [|D|^{s},  \mathcal{H}_{\geq 3} ] y, |D|^{s}y  \big)_{L^2} \big| 
\lesssim_s \| y(t,\cdot)\|_{\dot{H}^{s}}^{2}\|U(t,\cdot)\|^{3}_{K,s} \, .
\ee
In conclusion, by \eqref{stimabb}, \eqref{linea5}, \eqref{fazero}, \eqref{sti2}, 
and using Lemma \ref{equiv-tempo}
we deduce 
$$
{\rm Re}\int_{\T}|D|^{s}\mathcal{X}^+_{\geq4}(U,Y)\cdot \ov{|D|^{s}y} \, dx
 \lesssim_s
\| y(t,\cdot)\|^{2}_{\dot{H}^s}\|U(t,\cdot)\|^{3}_{\dot{H}^s} \lesssim_s \|y(t,\cdot)\|_{\dot{H}^s}^{5}
$$
by \eqref{equiv100}, proving the energy estimate \eqref{stimasuY}. 
\end{proof}

We can now prove Theorem \ref{BNFtheorem}.
\begin{proof}[{\bf Proof of Theorem \ref{BNFtheorem}}]
By \eqref{u01}, 
the function $U=\vect{u}{\bar{u}}$, where
$u$ is the variable defined  in \eqref{u0} and  $\omega$ in \eqref{omega0},
belongs to the ball  $B_{N}^{K}(I;r)$  (recall \eqref{palla}) 
 with $r=\bar{\e}\ll 1$ and $I=[-T,T]$. 
By Proposition \ref{WWcomplexVar} the function
 $U$ solves system \eqref{eq:415}.
 Then we
apply Poincar\'e-Birkhoff Proposition \ref{cor:BNF}  with $s\rightsquigarrow N\gg K\geq 2\rho+2\geq 2\rho_0+2$.
The map ${\bf F}_{T}^{1}(U)=\mathfrak{C}^{1}(U)[U]$ in \eqref{FTvect} 
transforms the water waves system \eqref{eq:415} into \eqref{sistemaBNF1000}, which, thanks to  Proposition \ref{teosez9}, 
is expressed in terms of the Zakharov-Dyachenko Hamiltonian $H_{ZD} $ in \eqref{theoBirHfull}, as
\begin{equation*}
 \pa_{t}Y= X_{H_{ZD}} (Y)  + \mathcal{X}_{\geq 4 }(U,Y) \, .
\end{equation*}
Renaming $ y \rightsquigarrow z $ and recalling \eqref{HSFC},
the first component of the above system is the equation \eqref{theoBireq},
and denoting $\mathfrak{B}(u)u$ 
the first component of $\mathfrak{C}^{1}(U)[U]$.
The bound \eqref{Germe} follows by 
\eqref{stimafinaleFINFRAK} with $s\rightsquigarrow N $ and $ k=0$,
and  Lemma \ref{equiv-tempo}.
The energy estimate \eqref{theoBirR} is proved  in  Lemma \ref{tempo}.
\end{proof}

\subsection{Proof of Theorem \ref{thm:main}}\label{sec:prLTE} 
The next bootstrap Proposition \ref{proapriori}  is the main ingredient 
for the proof of the long time existence Theorem \ref{thm:main}.
 Proposition \ref{proapriori} is a consequence of 
Theorem \ref{BNFtheorem} and 
the integrability of the fourth order Hamiltonian $H^{(4)}_{ZD}$ in \eqref{theoBirH}.

By time reversibility we may, without loss of generality, only look at positive times $t>0$.

\begin{proposition}{\bf (Main bootstrap)} \label{proapriori}
Fix the constants 
$\bar{\e}, K,N$ as in Theorem \ref{BNFtheorem}  and let the function $ u \in C^{0}([0,T]; H^{N}) $  be defined as in \eqref{u0},  
 with $\omega$  in \eqref{omega0} and $(\eta,\psi)$ solution of \eqref{eq:113}
 satisfying \eqref{etapsi}, \eqref{etaave0}. The function $ u $ satisfies \eqref{uave0}.  
Then there exists $c_0>0$  such that, for any $0<\e_1 \leq \bar{\e} $, if 
\be\label{apriorih}
{\| u(0)\|}_{H^N} \leq c_0 \e_1 \, , \qquad \sup_{t\in[0,T]} \sum_{k=0}^K {\| \pa_t^k u(t) \|}_{H^{N-k}} \leq \e_1 \, ,\qquad T \leq
 c_0 \e_1^{-3} \,,
\ee 
then
we have the improved bound
\be\label{aprioric}
\sup_{t\in[0,T]} \sum_{k=0}^K {\| \pa_t^k u(t) \|}_{H^{N-k}} \leq \dfrac{\e_1}{2} \, .
\ee 
\end{proposition}

\begin{proof}
In view of \eqref{apriorih} the smallness condition \eqref{u01} holds and 
we can apply Theorem \ref{BNFtheorem}
obtaining the new variable $z =\mathfrak{B}(u)u$ satisfying the equation \eqref{theoBireq}-\eqref{theoBirR}.
The integrability of $H^{(4)}_{ZD}$ in Theorem \ref{Zakintegro} 
(see the Hamiltonian system written in \eqref{HCS1}-\eqref{primRn}) gives 
$$
\Re \int_{\T} |D|^N \big(\ii \partial_{\bar{z}}H^{(4)}_{ZD}\big) \cdot \bar{|D|^N z} \, dx = 0 \, .
$$
From this, \eqref{theoBireq} and \eqref{theoBirR} we obtain the energy estimate
\begin{equation}\label{aprioriEne1}
\frac{d}{dt} {\|z(t)\|}^2_{\dot{H}^{N}} \lesssim_{N} \|z(t)\|^{5}_{\dot{H}^{N}} \, .
\end{equation}
Using \eqref{Germe} and \eqref{apriorih} we deduce that, for all $ 0 \leq t \leq T $, 
$$
\begin{aligned}
{\|u(t)\|}_{\dot{H}^{N}}^{2} \lesssim_N  {\|z(t)\|}^{2}_{\dot{H}^{N}}  
& \stackrel{\eqref{aprioriEne1}}{\lesssim_N }  {\| z(0) \|}^{2}_{\dot{H}^N} 
+ \int_0^t {\| z(\tau)\|}^5_{\dot{H}^N} \, d\tau \leq C {\|u(0)\|}^{2}_{\dot{H}^{N}} + C \int_{0}^{t} \|u(\tau)\|^{5}_{\dot{H}^{N}} \, d \tau 
\end{aligned}
$$
for some $C=C(N)>0$.
Then, by the a priori assumption \eqref{apriorih}
we get, for all $0\leq t\leq T\leq  c_0\e_{1}^{-3}$,
\begin{align}
\label{aprioriEne12}
{\|u(t)\|}^{2}_{\dot{H}^{N}} \leq C c^{2}_0 \e^{2}_1 + C \, {T} \, \e_1^5 
\leq  \e_1^{2}(C c_0^2+Cc_0) \, .
\end{align}
The desired conclusion \eqref{aprioric}
on the norms $C^k_t  H^{N-k}_x$
follows by Lemma \ref{equiv-tempo}, \eqref{aprioriEne12},  and recalling that 
$ \int_{\T} u(t, x )  d x = 0 $, 
choosing $c_0$  small enough depending on $N$. 
\end{proof}

We now prove the long-time existence  Theorem \ref{thm:main}, 
by Theorem \ref{BNFtheorem} and Proposition \ref{proapriori}.

\medskip
\noindent
{\it Step 1: Local existence}.
Let $ s > 3 / 2 $. By the assumption \eqref{thm:main1}, 
Theorem \ref{thm:local}  guarantees the existence of a time $T_{\mathrm{loc}}>0$ 
and a unique classical solution 
$(\eta,\psi)\in C^{0}([0,T_{\mathrm{loc}}]; 
H^{s+\frac{1}{2}}\times H^{s+\frac{1}{2}})$  
of \eqref{eq:113}, with initial data as in \eqref{thm:main1}, 
such that 
\be
\label{prLTE1}
\sup_{t\in[0,T_{\mathrm{loc}}]} {\|(\eta,\psi,V,B)(t) \|}_{X^s} \leq C \e \, ,
\qquad \int_\T \eta(t,x) \, dx \,=0 \, . 
\ee
  
\medskip
\noindent
{\it Step 2: Preliminary estimates in high norms}.
We now show that, 
for any $K>0$, if $s\geq K+ \s_0 $, for some $ \s_0 $ large enough,  and $\e$ is small enough, then  
the time derivatives $ (\pa_t^k \eta, \pa_t^k \psi ) $,  $ k=0,\ldots,K,$  satisfy,  
 for all $t\in[0,T_{\rm loc}]$,
\begin{equation}\label{pat114}
\|\pa_{t}^{k}\eta\|_{H^{s+\frac{1}{2}-k}}+\|\pa_{t}^{k}\psi\|_{H^{s+\frac{1}{2}-k}}\lesssim_s 
\|\eta\|_{H^{s+\frac{1}{2}}}+\|\psi\|_{H^{s+\frac{1}{2}}}\lesssim_{s}{\e} \, .
\end{equation}
 One argues by induction on $k$. For $k=0$ the second 
 estimate in \eqref{pat114} is  \eqref{prLTE1}. Assume that \eqref{pat114}
holds for any  $0\leq j\leq k-1\leq K-1$, $k\geq1$.
By differentiating in $t$ the water waves system \eqref{eq:113} we get
\begin{equation}\label{pat113}
\pa_{t}^{k}\eta=\pa_{t}^{k-1}\big(G(\eta)\psi\big) \, , \quad \pa_{t}^{k}\psi=\pa_{t}^{k-1}\big(\mathcal{F}(\eta,\eta_{x},\psi_x,G(\eta)\psi)\big) \, , \quad k=1,\ldots,K \, ,
\end{equation}
where $\mathcal{F}$ is an analytic function vanishing at the origin. 
Then, using that $G(\eta)\psi$ is expressed from the right hand side
 of \eqref{eq:1n}, Proposition \ref{azionepara}, 
 \eqref{piove} and the inductive hypothesis, we get
\[
\|\pa_{t}^{k-1} (G(\eta)\psi )\|_{H^{s+\frac{1}{2}-k}}\lesssim_{s}\sum_{k'\leq k-1}
\|\pa_t^{k'}\psi\|_{H^{s+\frac{1}{2}-k+1}}+
\|\pa_t^{k'}\eta\|_{H^{s+\frac{1}{2}-k+1}} \lesssim_s 
\|\eta\|_{H^{s+\frac{1}{2}}}+\|\psi\|_{H^{s+\frac{1}{2}}}  \, . 
\]
This implies,
in view of the first equation in \eqref{pat113},
 that $\|\pa_t^{k}\eta\|_{H^{s+ \frac12 -k}}$ is bounded as in \eqref{pat114}.
To estimate  $ \| \pa_t^{k}\psi\|_{H^{s+ \frac12 -k}} $, we use the second equation in 
\eqref{pat113}, the inductive estimates 
for $ (\pa_t^j \eta, \pa_t^j \psi ) $,  $0\leq j\leq k-1 $,  the 
previous bound on $ \|\pa_{t}^{k-1} (G(\eta)\psi )\|_{H^{s+\frac{1}{2}-k}} $
 and 
the fact that for $s\geq K+ \s_0 $ the space  $H^{s+ \frac12 -K}$ 
 is an algebra.

\medskip
\noindent
{\it Step 3: A priori estimate for the basic diagonal complex variable}.
We now look at the complex variable 
\be\label{u0''}
u = \frac{1}{\sqrt{2}}|D|^{-\frac{1}{4}}\eta+\frac{\ii}{\sqrt{2}}|D|^{\frac{1}{4}}\omega 
\ee
defined in \eqref{u0} where $ \om = \psi-\opbw(B(\eta,\psi))\eta  $ is the good unknown  defined 
in \eqref{omega0}.
Since the function  
$ B  $ 
is  in $ \Sigma {\cal F}^\R_{K, 0,1}$ (Proposition \ref{laprimapara}), we deduce,
applying  Proposition \ref{azionepara} for $s \geq s_0 $ large enough,  that 
$ \om   $ is in  $ C^{0}([0,T_{\mathrm{loc}}]; \dot H^{s+ \frac12}) $,
 and so
\begin{equation}\label{legoNS}
u\in C^{0}([0,T_{\mathrm{loc}}]; H^{N}) \, ,\qquad N:=s+\frac14 \, .
\end{equation}
Moreover, using  \eqref{thm:main1},
\eqref{prLTE1}-\eqref{pat114}, we estimate   $ \| \pa_t^k u \|_{H^{N-k}} $, $k=0,\ldots, K $ for $ N \gg K $, by
\[
\begin{aligned}
\|\pa_{t}^{k} u\|_{H^{N-k}} & \stackrel{\eqref{u0''}}{\lesssim_N} 
\|\pa_{t}^{k}\eta\|_{H^{N-\frac{1}{4}-k}}+\|\pa_{t}^{k}\omega\|_{H^{N+\frac{1}{4}-k}} \\
& \stackrel{\eqref{omega0}, {\rm Prop.} \ref{azionepara}}{\lesssim_N} \|\pa_{t}^{k}\eta\|_{H^{N+\frac{1}{4}-k}}
+\|\pa_{t}^{k}\psi\|_{H^{N+\frac{1}{4}-k}} \stackrel{\eqref{pat114}}{\lesssim_N}{\e} \, ,
\end{aligned}
\]
for any  $t\in [0,T_{\rm loc}]$. In conclusion, there is $C_1=C_1(N)>0$ such that
\be
\label{prLTE2}
{\|u(0) \|}_{H^{N}} \leq 2\e \, , 
\qquad \sup_{t\in[0,T_{\mathrm{loc}}]} \sum_{k=0}^{K} {\| \pa_t^k u(t) \|}_{H^{N - k}} \leq C_1 \e \, ,
\qquad \int_\T u(t,x) \, dx = 0 \, .
\ee

\noindent
{\it Step 4: Bootstrap argument and continuation criterion}.
With $ \bar \e, K, N  $ given by Theorem \ref{BNFtheorem}, and $c_0$ by  Proposition \ref{proapriori},
we choose $\e_0$ in \eqref{thm:main1} small enough so that, for 
$0<\e\leq \e_0$ we have $2\e \leq c_0 \bar{\e}$, $C_1 \e\leq \bar{\e}$ where
$ C_1 $ is the constant in \eqref{prLTE2}. 
Moreover we take $s\geq s_0$  large enough in such a way that 
\eqref{legoNS}-\eqref{prLTE2} hold with $N$ given by Theorem \ref{BNFtheorem}.
Hence 
the first two assumptions in \eqref{apriorih}
hold with $\e_1 = \e \max\{2c_0^{-1},C_1 \} $ on the time interval $ [0, T_{\rm loc}] $.
Then Proposition \ref{proapriori} 
and  a standard bootstrap argument guarantee that $ u (t)  $ can be extended up to a time 
$$
T_\e := c_0\e_1^{-3} 
\, ,
$$
consistently with the existence time \eqref{time} of the statement, 
and that
\be
\label{prLTE5}
\sup_{[0,T_\e]} {\| u(t) \|}_{H^N} \leq \e_1 \, , \qquad \int_\T u(t,x)\,dx =0 \, .
\ee

Finally, we prove that the solution of \eqref{eq:113} satisfies \eqref{thm:main2} 
and  that $(\eta,\psi,V,B)(t)  $ takes values in $ X^s $ for all $ t \in [0, T_\e ] $.
Expressing $(\eta, \omega)$ in terms of $ u, \bar u $ as in  \eqref{varfin1}, 
we deduce by \eqref{prLTE5} that 
\be\label{prLTE10}
\sup_{[0,T_\e]}\big( {\|\eta(t)\|}_{H^{s}} + {\|\omega(t)\|}_{H^{s+\frac12}} \big) \lesssim_{s} \e \, .
\ee
Then we estimate 
\be\label{prLTE11}
\sup_{[0,T_\e]}{\|\psi(t)\|}_{H^{s}} \lesssim_{s} \e
\ee
by  \eqref{omega0}, \eqref{prLTE10}  and Proposition \ref{azionepara}, 
and
\be\label{prLTE12}
\sup_{[0,T_\e]}{\|(V,B)(t)\|}_{H^{s-1} \times H^{s-1}} \lesssim_{s} \e \, ,
\ee
using \eqref{eq:1n} for $G(\eta) \psi $.
The estimates \eqref{prLTE10}-\eqref{prLTE12} imply \eqref{thm:main2} and, in 
particular,  that  
$$ \sup_{[0,T_\e]}{\|(\eta,\psi,V,B)(t)\|}_{X^{s-1}} \lesssim_{s} \e $$
thus guaranteeing \eqref{thm:localcont}, for $s-1\geq 5 $, on the time interval $[0, T_\e] $. The continuation criterion in
 Theorem \ref{thm:local}-(2)  implies   that the solution 
 $(\eta,\psi,V,B) $ is in $ C^0 ([0, T], X^s) $ for $T\geq T_\e $.  
$\hfill \Box$

\appendix
\bigskip
\section{Flows and conjugations}\label{ConjRules}

In this Appendix  we study the conjugation rules of a vector field under flow maps.

\subsection{Conjugation rules}

We first give this simple lemma that we use  in sections \ref{diagonalizzo} and \ref{sec:BNF}. 

\begin{lemma}\label{lem:tra.Vec}
For $U=\vect{u}{\bar{u}}$
consider a system $ \pa_t U  =X(U) U $ 
with $X(U) $ in $ \Sigma\mathcal{M}_{K,K',0}\otimes\mathcal{M}_2(\C) $
and let $ {\bf \Phi}^{\theta} (U) $ be the flow of
\begin{equation}\label{flussoZERO}
\pa_\theta {\bf \Phi}^{\theta} (U)  = \ii {\bf A}(U) {\bf \Phi}^{\theta} (U) \, , \quad {\bf \Phi}^{0} (U) = {\rm Id}  \,  ,
\end{equation}
where ${\bf A}:={\bf A}(U) $ is in $  \Sigma\mathcal{R}^{0}_{K,K',1}\otimes\mathcal{M}_2(\C)$.
Under the change of variable $ V := ({\bf \Phi}^{\theta} (U))_{\theta=1} U$,  the new system becomes 
\begin{equation}\label{nuovosist}
\pa_t V =X^{+}(U) V \,  , \qquad 
X^{+}(U):=  (\pa_t {\bf \Phi}^{1} (U)) 
({\bf \Phi}^{1}(U) )^{-1} + {\bf \Phi}^{1} (U) X(U)   ({\bf \Phi}^{1}(U) )^{-1}  \, .
\end{equation}
The operator $X^{+}(U)  $ is in $ \Sigma\mathcal{M}_{K,K'+1,0}\otimes\mathcal{M}_2(\C)$ and,
setting ${\rm Ad}_{\ii {\bf A}}[X]:=[\ii {\bf A},X]$, 
it  admits the Lie expansion
\begin{align} 
 \label{Lie1Vec}
  {\bf \Phi}^{1} (U) X(U)  ({\bf  \Phi}^{1} (U))^{-1} & = X+\sum_{q=1}^{L}\frac{1}{q!}{\rm Ad}_{\ii {\bf A}}^{q}[X]
 +  \frac{1}{L!} \int_{0}^{1}  (1- \theta)^{L} {\bf \Phi}^{\theta}(U) {\rm Ad}_{\ii{\bf A}}^{L+1}[X] ({\bf \Phi}^{\theta}(U))^{-1} 
d \theta \\
  \label{Lie2Vec}
 (\pa_t {\bf \Phi}^{1} (U) )({\bf  \Phi}^{1}(U))^{-1}   
& =\ii \pa_t {\bf A}+
\sum_{q=2}^{L}\frac{1}{q!}{\rm Ad}_{\ii {\bf A}}^{q-1}[\ii \pa_{t}{\bf A}] \nonumber \\ 
& \quad +\frac{1}{L!}\int_{0}^{1}
(1- \theta )^{L} 
{\bf \Phi}^{\theta}(U){\rm Ad}_{\ii {\bf A}}^{L}[\ii \pa_{t}{\bf A}]({\bf  \Phi}^{\theta}(U))^{-1} d \theta \, .
\end{align}
\end{lemma}
 
\begin{proof} 
The expression \eqref{nuovosist} follows by an explicit computation.
In order to prove \eqref{Lie1Vec} note that the vector field
$ P(\theta) := {\bf \Phi}^{\theta} (U) X(U)   ({\bf \Phi}^{\theta} (U))^{-1} $ 
satisfies the Heisenberg equation 
$$
\partial_\theta P(\theta) = [ \ii {\bf A}, P (\theta)]  =  \ii {\bf A}  P (\theta) - P (\theta) \ii  {\bf A}\,, \qquad P(0)=X(U) \,.
$$
Since the vector field $ {\bf A} $ is independent of $ \theta $ we also have
$$
\partial_\theta P(\theta) =  {\bf \Phi}^{\theta} (U)  {\rm Ad}_{\ii {\bf A}}[X] {\bf \Phi}^{\theta} (U)^{-1} \, . 
$$
Then  \eqref{Lie1Vec} follows by a Taylor expansion. 
To prove \eqref{Lie2Vec} we reason as follows.
We have that
\begin{equation}\label{triangolo}
{\bf \Phi}^{1} (U)\circ \pa_{t}\circ({\bf \Phi}^{1} (U))^{-1}=\pa_{t}+{\bf \Phi}^{1} (U)\big[\pa_{t}({\bf \Phi}^{1} (U))^{-1}\big]=
\pa_{t}-(\pa_{t}{\bf \Phi}^{1} (U))({\bf \Phi}^{1} (U))^{-1}.
\end{equation}
Moreover, using the Lie expansion obtained by formally replacing $X$ with $\pa_{t}$ in \eqref{Lie1Vec},
we have
\[
\begin{aligned}
{\bf \Phi}^{1} (U)\circ \pa_{t}\circ({\bf \Phi}^{1} (U))^{-1}
&=\pa_{t}-\sum_{q=1}^{L}\frac{1}{q!}{\rm Ad}^{q-1}_{\ii {\bf A}}[\ii \pa_{t}{\bf A}]-
\frac{1}{L!}\int_{0}^{1} (1-\theta)^{L} 
{\bf \Phi}^{\theta} (U) {\rm Ad}^{L}_{\ii {\bf A}}[\ii \pa_{t}{\bf A}]({\bf \Phi}^{\theta} (U))^{-1}d\theta
\end{aligned}
\]
which, together with  \eqref{triangolo}, implies \eqref{Lie2Vec}.
By Taylor expanding $ {\bf \Phi}^{1}(U)$ using \eqref{flussoZERO}, we derive that 
$ {\bf \Phi}^{1}(U) - {\rm Id} $ is in $ \Sigma\mathcal{M}_{K,K',1}\otimes\mathcal{M}_2(\C)$. 
The translation invariance property \eqref{def:R-trin} of the homogeneous components of 
$ {\bf \Phi}^{1}(U) $ follows since the generator $ {\bf A}(U) $ satisfies \eqref{def:R-trin}. 
Then,  the operator  $X^{+}(U)$ in \eqref{nuovosist} belongs 
to $ \Sigma\mathcal{M}_{K,K'+1,0}\otimes\mathcal{M}_2(\C) $ by Proposition \ref{composizioniTOTALI} 
and the remarks after Definition \ref{smoothoperatormaps}.
Let us justify the translation invariance property 
of the homogeneous components of $X^{+}(U)$.
Denoting by
 $ {\bf \Phi}_{\leq 2}^{1}(U)$ the sum of  its homogeneous components of degree
less or equal to 2 we have that, for any $ \vartheta\in \R$,  $\tau_{\vartheta} {\bf \Phi}_{\leq2}^{1}(U)= 
{\bf \Phi}_{\leq 2}^{1}(\tau_{\vartheta}U) \tau_\vartheta $, and so 
\be\label{pass-int}
 \tau_\vartheta d_U {\bf \Phi}_{\leq 2}^{1}(U)[  \hat H ] = 
 d_U {\bf \Phi}_{\leq 2}^{1}(\tau_{\vartheta}U)[ \tau_\vartheta \hat H ] \tau_\vartheta 
 \, . 
\ee
Then 
\begin{equation*}
\begin{aligned}
\tau_{\vartheta} ( \pa_{t} {\bf \Phi}_{\leq 2}^{1}(U)) = 
\tau_{\vartheta}\big( d_{U} {\bf \Phi}_{\leq2}^{1}(U)[X(U)U] \big)&
\stackrel{\eqref{pass-int}} = d_{U}{\bf \Phi}_{\leq 2}^{1}(\tau_{\vartheta}U)[\tau_{\vartheta}X(U)U]\tau_{\vartheta} 
=(\pa_{t} {\bf \Phi}_{\leq 2}^{1}(\tau_{\vartheta}U)) \tau_{\vartheta} 
\end{aligned}
\end{equation*}
using the translation invariance of  $X(U) U$. 
By composition we deduce that the homogeneous components of $ X^+ (U) $ in \eqref{nuovosist} 
satisfy \eqref{def:R-trin}.
\end{proof}

In the next subsection  we analyze how paradifferential operators change under the flow maps
generated by  paradifferential operators.

\subsection{Conjugation of paradifferential operators via flows}\label{sez:A2}

We consider the  flow equation 
 \begin{equation}\label{linprob}
 \pa_{\theta} \Phi^\theta = \ii\opbw(f(\theta, U ;x,  \x)) \Phi^\theta, \quad  \Phi^0 = {\rm Id} \, , 
\end{equation}
where $ f $ is a symbol assuming one of the following forms: 
\begin{align}
& f(\theta, U;x,  \x) := b(\theta, U;  x) \xi := \frac{\beta (U; x)}{1 + \theta \beta_x (U; x)} \xi \, , \quad \, \beta (U; x) \in 
\Sigma {\mathcal F}_{K,K',1}^\R \, , \label{sim1} \\
&  f(\theta, U;x, \x) := f(U;x, \x):=
\beta(U;x)|\x|^{\frac{1}{2}}, \qquad \quad \  \beta(U;x) \in \Sigma\mathcal{F}^{\R}_{K,K',1} \, , \label{sim2}\\
&  f(\theta, U;x, \x) := f(U;x, \x) \in \Sigma\Gamma^{m}_{K,K',1} \, \, , \quad \quad
\ \qquad m \leq 0 \, . \label{sim3}
\end{align}
Note that  \eqref{linprob}  with $ f $ as in \eqref{sim1} is a para-differential transport equation.
This is used  in  Section \ref{inteord1} and Subsection \ref{scarlatto}. 
Flows with $ f $ as in \eqref{sim2}  are used 
in Section \ref{inteord12} 
and with $f$ as in \eqref{sim3}  in Subsection \ref{inteord0} and Section \ref{sec:nega}. 

\begin{lemma}{\bf (Linear flows generated by a para-differential operator)}\label{buonflusso}
Assume that $ f $ has the form \eqref{sim1} or \eqref{sim2} or \eqref{sim3}.
Then, there is $s_0 > 0, r > 0 $ such that,  for any $U\in C^{K}_{*\R}(I; {\dot H}^{s}) \cap B^{K}_{s_0}(I;r)$, 
for any $s>0$,
the equation \eqref{linprob} has a unique solution $  \Phi^{\theta} (U) $ 
satisfying:
\\[1mm]
$(i)$ 
the linear map $\Phi^{\theta}(U)$  is invertible  and, 
for some $C_{s}>0$, 
\begin{align}\label{est1}
\|\pa_{t}^{k}\Phi^{\theta}(U)[v]\|_{\dot{H}^{s-k}}&+\|\pa_{t}^{k}(\Phi^{\theta}(U))^{-1}[v]\|_{\dot{H}^{s-k}} \leq 
\|v\|_{k,s}\big(1+C_s\|U\|_{K,s_0}\big) \, , \ \forall  0\leq k\leq K-K' \, , 
\\
C_{s}^{-1}\|v\|_{\dot{H}^s}&\leq\|\Phi^{\theta}(U)[v]\|_{\dot{H}^s} \leq C_s \|v\|_{\dot{H}^s} \, , 
\label{est1quatris}
\end{align}
for any $v\in C_{*}^{K-K'}(I;\dot{H}^{s})$ and uniformly in  $\theta \in [0,1]$;

\noindent
$(ii)$ the map $\Phi^{\theta}(U)$
admits an expansion in multilinear maps as 
$ \Phi^{\theta}(U)-{\rm Id}\in \Sigma\mathcal{M}_{K,K',1} $, $\theta\in[0,1]$.  
More precisely 
there are $M_1(U) $ in $  \widetilde{\mathcal{M}}_{1} $, 
and $M^{(1)}_2(U), M^{(2)}_2(U) $ in $  \widetilde{\mathcal{M}}_{2} $ (independent of $ \theta $) such that
\begin{equation}\label{est2}
\Phi^{\theta}(U)[U]=U+ \theta \big(M_1(U)[U] + M_{2}^{(1)}(U)[U]\big)+\theta^{2}M_2^{(2)}(U)[U]
+ M_{\geq3}(\theta;U)[U] 
\end{equation}
where $M_{\geq3}(\theta;U)$ 
is in $\mathcal{M}^{m}_{K,K',3}$ with estimates uniform in $\theta\in [0,1]$.

The same result holds for a matrix valued system
$ \pa_{\theta} {\bf \Phi}^\theta (U)  = {\bf B}(U) {\bf \Phi}^\theta (U) $, 
$ {\bf \Phi}^0 (U) = {\rm Id} $, 
where $ {\bf B}(U)$ $=\opbw{(B(U;x,\x))}$ and $B(U;x,\x) $ is a matrix of 
symbols in $ \Sigma{\Gamma}^{0}_{K,K',1} \otimes {\mathcal M}_2 (\C) $. 
\end{lemma}

\begin{proof}
See 
Lemma 3.22 in \cite{BD}. The translation invariance 
property \eqref{def:R-trin} of the flow map $\Phi^{\theta}(U)$ defined by \eqref{linprob} follows by the fact 
that the homogeneous components of the symbol $f(\theta, U ;x,  \x)$ satisfy \eqref{def:tr-in}. 
\end{proof}

The proof of the 
next lemma  follows by standard theory of Banach space ODEs.

\begin{lemma}{\bf (Linear flows generated by a smoothing operator)}\label{buonflusso2}
Assume that ${\bf A}(U)$ in \eqref{flussoZERO}
 is a smoothing operator in $\Sigma\mathcal{R}_{K,0,1}^{-\rho}\otimes\mathcal{M}_2(\C)$ 
for some $\rho\geq0$. 
Then, there is $s_0 > 0, r > 0 $ such that,  for any $U\in  B^{K}_{s}(I;r)$, 
for any $s>s_0$,
the equation \eqref{flussoZERO} has a unique solution $  {\bf \Phi}^{\theta} (U) $ 
satisfying, for some $C_{s}>0$,
\begin{equation}\label{est1VEC}
\|\pa_{t}^{k}({\bf \Phi}^{\theta}(U))^{\pm 1}[v]\|_{\dot{H}^{s+\rho-k}} \leq 
\|v\|_{k,s}\big(1+C_s\|U\|_{K,s_0}\big) +C_s\|v\|_{k,s_0}\|U\|_{K,s} \,  ,
\end{equation}
for any $v\in C_{*}^{K-K'}(I;\dot{H}^{s})$,  $0\leq k\leq K-K' $, and uniformly in 
$\theta \in [0,1]$.
Moreover $  {\bf \Phi}^{\theta} (U) $  satisfies 
a bound like \eqref{est1quatris}  and $(ii)$ of Lemma \ref{buonflusso}.
\end{lemma}

We now provide the conjugation rules  of a
paradifferential operator under the flow  $\Phi^{\theta}(U) $ in \eqref{linprob}.
We first give the result in the case when $ f $ has the form  \eqref{sim1}, i.e.  \eqref{linprob} is a transport equation. 

\begin{lemma}{\bf (Conjugation of a paradifferential operator under transport flow)} \label{conjFlowpara}
Let $\Phi^{\theta}(U)$ be the flow of \eqref{linprob} given by Lemma \ref{buonflusso} 
with $ f(\theta, U;  x, \xi )$ as in \eqref{sim1} and $U\in C^{K}_{*\R}(I; {\dot H}^{s_0}) \cap B^{K}_{s_0}(I;r)$. 
Consider the diffeomorphism of $ \T $ given by 
$$ 
\psi_U : x\mapsto x + \beta (U; x ) \, .
$$
Let $a(U;x,\x)$ be a symbol in $\Sigma\Gamma^{m}_{K,K',q}$
for some $q\in \N$, $ q \leq 2 $, $ K'\leq K$, $ r>0$ and  $m\in \R $. 
If $s_0$ is large enough and $r$ small enough
then there is a symbol $a_{\Phi}(U;x,\x) $ in $ \Sigma\Gamma^{m}_{K,K',q}$
such that
\begin{equation}\label{coniugato}
\Phi^{1}(U)\opbw(a(U;x,\x))(\Phi^1(U))^{-1}=\opbw(a_{\Phi}(U;x,\x))+R(U)
\end{equation}
where $R(U) $ is a smoothing remainder in $ \Sigma\mathcal{R}^{-\rho+m}_{K,K',q+1}$.
Moreover $a_{\Phi} $ admits  an expansion  as
\begin{equation}\label{coniugato2}
a_{\Phi}(U;x,\x)=a^{(0)}_{\Phi}(U;x,\x)+a^{(1)}_{\Phi}(U;x,\x)
\end{equation}
where
\begin{equation}\label{coniugato3}
a_{\Phi}^{(0)}(U;x,\x)=a\big(U; \psi_U(t,x),\x \pa_{y}( \psi^{-1}_{U}(t,y))_{|_{y=\psi_U(t,x)}}\big)\in \Sigma\Gamma^{m}_{K,K',q}
\end{equation}
and $a^{(1)}_{\Phi}(U;x,\x) $ is a symbol in $ \Sigma\Gamma^{m-2}_{K,K',q+1}$.
In addition, if $ a (U; x, \xi) = g(U; x) \xi $  then  $ a^{(1)}_{\Phi} = 0 $. 

Furthermore, the symbol $ a^{(0)}_{\Phi} $ in \eqref{coniugato3} admits an expansion in degrees of homogeneity as
\be\label{simbLie}
a^{(0)}_{\Phi} = a + \{ \beta  \xi, a  \} + \frac12 
\Big(  \big\{ \beta  \xi, \{ \beta \xi, a \} \big\} +
\{ - \beta  \beta_x \xi, a \}  \Big) 
\ee
up to a symbol in $ \Gamma^{m}_{K,K',3}$ 
which is real valued like $a_{\Phi}^{(0)}$ if $a(U;x,\xi)$ is real valued.
\end{lemma}

\begin{proof}
Formulas \eqref{coniugato}-\eqref{coniugato3} are proved
in Theorem  3.27 of \cite{BD} (with homogeneity degree $ N = 3 $), where it is shown  that 
the symbol
$ a^{(0)}_{\Phi} (U; x, \xi) = a_0 (\theta, U;  x, \xi )_{| \theta=1} $ and $ a_0 (\theta) $
solves the transport equation 
\be\label{dthetaa0}
\frac{d}{d \theta} a_0 (\theta) = \{ b(\theta, U;  x) \xi, a_0 (\theta)  \} \, , \quad a_0 (0) = a \, . 
\ee
The claim  that, if $ a (U; x, \xi) = g(U; x) \xi $  then  $ a^{(1)}_{\Phi} = 0 $ follows because in formula 
(3.5.37) of \cite{BD}, the symbol $ r_{-\rho,3} = 0 $.
Finally we deduce \eqref{simbLie}
by  a Taylor expansion in $ \theta $  using \eqref{dthetaa0} (note that $ b $ and $  \beta $ have degree of homogeneity 
$ 1 $ in $ u $). 
Since  the homogeneous components of 
$\beta(U; x )$ satisfy the invariance 
condition \eqref{def:tr-in}, the flow $ \Phi^{1}(U) $ satisfies \eqref{def:R-trin}, and so the 
left hand side in \eqref{coniugato}.
The proof 
shows that the symbol  $a_{\Phi}$ 
 in \eqref{coniugato2} satisfies the invariance 
condition \eqref{def:tr-in} and therefore
the remainder $R(U)$ in \eqref{coniugato} satisfies \eqref{def:R-trin} 
by difference.
\end{proof}

\begin{remark}
If the symbol $a(U;x,\x)$ in Lemma \ref{conjFlowpara} is real  then 
the whole symbol $a_{\Phi}$ in \eqref{coniugato2} is real  as well. 
In this paper we shall not use this information 
 because in our application $m\leq 1$, and hence the symbol $a_{\Phi}^{(1)}$
has negative order.
\end{remark}

\begin{lemma}{\bf (Conjugation of $\pa_t$ under transport flow)} \label{conjFlowpara2}
Let $\Phi^{\theta}(U)$ be the flow of \eqref{linprob} given by Lemma \ref{buonflusso} 
with $f(\theta, U;x,  \x)$ 
  as in \eqref{sim1}.
Then 
\begin{equation}\label{pes2010}
\begin{aligned}
\big( \pa_t \Phi^1(U)\big) \big(\Phi^1(U)\big)^{-1} =  \ii \opbw(g(U;x) \xi ) + R(U)
\end{aligned}
\end{equation}
where  $ g (U;x) $ is a function in 
$\Sigma {\mathcal F}^\R_{K,K'+1,1} $ and $ R(U) $ is a smoothing operator in 
$ \Sigma\mathcal{R}^{-\rho}_{K,K'+1,1}$.

\noindent
In addition, the function $ g(U; x) $ admits the expansion in degrees of homogeneity 
\be\label{expg}
g(U; x) = \b_t - \b_x \b_t 	+ {g}_{\geq 3} (U; x) \, , \quad  {g}_{\geq 3} (U; x) \in {\mathcal F}^\R_{K,K'+1,3} \, . 
\ee
\end{lemma}

\begin{proof}
By the proof of Proposition 3.28  of \cite{BD} 
(see formul\ae \, \eqref{triangolo} and (3.5.55) in \cite{BD})
the operator  
$
P(\theta) := ( \pa_t  \Phi^{\theta} (U) )(\Phi^{\theta}(U))^{-1}
$
solves
\be\label{eqPtheta}
\frac{d}{d \theta} P(\theta) = \big[ \ii \opbw( b(\theta, U;  x) \xi) ,  P(\theta) \big] + \ii \opbw( \pa_t  b(\theta, U;  x) \xi ) \, , 
\quad P(0) = 0 \, .
\ee
We claim that the  
solution of \eqref{eqPtheta} is, up to smoothing remainders,  
 $ P(\theta)  = \opbw(p_0(\theta, x, \xi)) $,  where 
the symbol $p_0(\theta, x, \xi) $ solves the forced transport equation 
\be\label{eq:intq0}
\frac{d}{d \theta } p_0(\theta, x, \xi) = \{ b(\theta,U;  x) \xi,  p_0(\theta, x, \xi) \} + \ii \pa_t b(\theta, U;  x) \xi \, ,
\quad p_0 (0) = 0 \,  . 
\ee
Indeed, the solution of \eqref{eq:intq0} is 
\begin{equation}\label{p0exact}
\begin{aligned}
 p_0 (\theta, x, \xi ) &  =   \ii  \int_0^\theta \pa_t  f(s, U;  \phi^{\theta, s}  (x, \xi) ) \, ds 
 \quad {\rm where} \quad 
 f(s, U; x, \xi) := b(s, U; x) \xi 
 \end{aligned}
\end{equation}
and $ \phi^{\theta, s}  (x, \xi ) $ is the solution of the characteristic Hamiltonian system 
$$ 
\begin{cases}
\frac{d}{ds} x (s) = - b(s,  x(s) ) \cr
\frac{d}{ds} \xi (s) = b_x (s, x(s) ) \xi (s) \, .
\end{cases} 
$$
with initial condition $ \phi^{\theta, \theta}  = {\rm Id }$. 
Note that $\phi^{\theta,s}(x,\xi)=\phi^{0,s}\phi^{\theta,0}$ where
\[
\phi^{\theta,0}(x,\x)=\Big(x+\theta\beta(U;x), \xi (1+\pa_{y}\gamma(U;\theta,y)_{|y=x+\theta\beta(U;x)}\Big)
\]
where $y+\gamma(U;\theta,y)$ is the inverse diffeomorphism of $x+\theta\beta(U;x)$
(see Lemma 3.21 in \cite{BD}). Then $\phi^{\theta,s}(x,\xi)$
 is linear in $\x$ and 
hence
also $p_0(\theta,x,\x)$ in \eqref{p0exact} is linear in $\xi$.
Since both $b(\theta, U;x)\x$ and $p_0(\theta,x,\x)$ are linear in $\x$
we deduce that 
the commutator $\big[ \ii \opbw( b(\theta, U;  x) \xi) ,  \opbw(p_0(\theta,x,\x)) \big]$
is given by 
$\opbw( \{ b(\theta,U;  x) \xi,  p_0(\theta, x, \xi) \})$ up to smoothing operators.
Moreover, by Lemma 3.23 in \cite{BD}, $f(s, U;   \phi^{\theta, s}  (x, \xi) )$ is in $\Sigma\Gamma^{1}_{K,K',1}$
with estimates uniform in $|\theta|, |s|\leq 1$.
Then \eqref{pes2010} follows with $ \ii g(U;x) \xi  := p_0 (1, x, \xi )  $. 
Finally we deduce \eqref{expg} by a Taylor expansion in $ \theta $ 
of the symbol $ p_0 (\theta) $, using \eqref{eq:intq0}. 
The function $  \b_t - \b_x \b_t  $
satisfies the translation invariance property  \eqref{def:tr-in} as $ \beta$.
As in Lemma \ref{lem:tra.Vec}  the operator $ ( \pa_t \Phi^1(U)) (\Phi^1(U))^{-1} $
 in \eqref{pes2010} is translation invariant 
and $R(U)$ satisfies the  property \eqref{def:R-trin}  by difference.  
\end{proof}

We now provide the conjugation  of a
paradifferential operator under the flow  $\Phi^{\theta}(U) $ in \eqref{linprob}, if  $ f $ has the form  \eqref{sim2} or \eqref{sim3}.

\begin{lemma}\label{flusso12basso}
{\bf (Conjugation of a paradifferential operator)}
Let $\Phi^{\theta}(U)$ be the flow of \eqref{linprob} given by Lemma \ref{buonflusso} 
with symbol $ f(U;  x, \xi ) $ in $  \Sigma\Gamma^{m}_{K,K',1} $ with $ m \leq 1/ 2  $, 
of the form  \eqref{sim2} or \eqref{sim3}. 
Let $a(U;x,\x)$ be a symbol in $\Sigma\Gamma^{m'}_{K,K',q}$
for some $q\in \N$, $ q\leq 2 $, $K'\leq K$, $ r>0$ and  $m'\in \R $. 
Then 
\begin{equation}\label{coniugato1000}
\begin{aligned}
& \Phi^{1}(U)\opbw(a(U;x,\x))(\Phi^{1}(U))^{-1} = \\
& 
\opbw\Big( a +  \{ f, a\} +  \frac12 \{ f, \{ f, a \} \} + r_1 + r_2 +r_3\Big)  + R(U)
\end{aligned}
\end{equation}
where  $ r_1 \in \Sigma\Gamma^{m+m'-3}_{K,K',q+1}$, 
$ r_2  \in \Sigma\Gamma^{2 m + m'- 4}_{K,K',q+2}$, $ r_3  \in \Gamma^{3 m + m'- 3}_{K,K',3}$  and 
 $ R(U) \in \Sigma\mathcal{R}^{-\rho}_{K,K',q+1} $. 
In addition  if $a(U;x,\x)$ is real 
 then also the symbols $r_i$, $i=1,2,3$, are real valued as well.
\end{lemma}

\begin{proof}
The result follows by a Lie  expansion. 
Using \eqref{Lie1Vec} we have,  for $L\geq 3$,
$$
\begin{aligned}
& \Phi^{1}(U) \opbw(a)(\Phi^{1}(U))^{-1}=\opbw(a)+ \big[ \opbw(\ii f), \opbw(a) \big]
 + \frac{1}{2}{\rm Ad}^{2}_{\opbw(\ii f)}[\opbw(a)] + \label{int-Lie} \\
& \sum_{k=3}^{L}\frac{1}{k!}{\rm Ad}^{k}_{\opbw(\ii f)}[\opbw(a)]
+\frac{1}{L!}\int_{0}^{1} (1-\theta)^{L} \Phi^{\theta}(U)
\big({\rm Ad}^{L+1}_{\opbw(\ii f)}[\opbw(a)] \big)(\Phi^{\theta}(U))^{-1} d \theta\,.  
\end{aligned}
$$
By applying Propositions \ref{teoremadicomposizione}, \ref{composizioniTOTALI} 
replacing the smoothing index $\rho$ 
by some $\tilde{\rho}$ to be chosen below large enough, we get
$$
{\rm Ad}_{\opbw(\ii f)}[\opbw(a)]  =\big[\opbw(\ii f), \opbw(a)\big] 
 = \opbw\big( \{  f, a \} + r_1 \big) \, , \quad  r_1\in\Sigma\Gamma^{m+m' -3}_{K,K',q+1} \, , 
$$
up to a smoothing operator in $\Sigma\mathcal{R}^{-\tilde{\rho}+m+m'}_{K,K',q+1}$.
Moreover
$$
{\rm Ad}^{2}_{\opbw(\ii f)}[\opbw(a)] =\opbw( \{ f, \{  f, a \} \} + r_2) \, , \quad  
r_2 \in\Sigma\Gamma^{2m+m' -4}_{K,K',q+2} \, , 
$$
up to a smoothing operator in $\Sigma\mathcal{R}^{-\tilde{\rho}+2m+m'}_{K,K',q+2}$.
By  induction, for $ k \geq 3 $ we have  
\[
{\rm Ad}^{k}_{\opbw(\ii f)}[\opbw(a)]=\opbw( b_k),\quad 
b_k\in
\Sigma\Gamma^{k(m-1)+m'}_{K,K',q+k}  \, , 
\]
up to a smoothing operator in $\Sigma\mathcal{R}^{-\tilde{\rho}+m'+ k m}_{K,K',q+k}$. 
We choose $L$  in such a way that  
$  (L+1)(1-m) - m' \geq \rho$ and $ L + 1 \geq 3 $,  so that the operator $ \opbw( b_{L+1}) $ belongs to 
$ \mathcal{R}^{- \rho}_{K,K',3} $. The integral Taylor remainder in \eqref{int-Lie} 
belongs to $ \mathcal{R}^{- \rho}_{K,K',3} $ as well, see Lemma 5.6 in \cite{BD}.  
Then we choose $\tilde{\rho}$  large enough so that $ \tilde{\rho} - m' - ( L + 1) m \geq  \rho $
and the remainders are  $\rho$-smoothing.
By the third remark 
under Definition \ref{asyexpexp} we deduce that if $a(U;x,\x)$ is real
then the symbol of $\big[\opbw(\ii f), \opbw(a)\big] $ is real, and so $r_1,r_2,r_3$ are real valued as well.
\end{proof}

\begin{lemma}{\bf (Conjugation of $\pa_t $)} \label{flusso12basso2}
Let $\Phi^{\theta}(U)$ be the flow of \eqref{linprob}  
with symbol $ f(U;  x, \xi ) $ in $  \Sigma\Gamma^{m}_{K,K',1} $ with $ m \leq 1/ 2  $, 
of the form  \eqref{sim2} or \eqref{sim3}. Then 
\begin{equation}\label{coniugato1001}
\big( \pa_t \Phi^{1}(U) \big) \big( \Phi^{1}(U) \big)^{-1} = 
\ii  \opbw \big( \pa_t f + \frac{1}{2} \{  f,  \pa_t f  \}\big)  + \opbw(r_1 + r_2) + R(U) 
\end{equation}
where $ r_1 \in \Sigma \Gamma^{2 m - 3}_{K,K'+1,2}$,
$ r_2 \in \Gamma^{3 m - 2}_{K,K'+1,3}$ and $ R(U) \in \Sigma\mathcal{R}^{-\rho}_{K,K'+1,2}$.
\end{lemma}
\begin{proof}
By the Lie expansion \eqref{Lie2Vec} we have 
$$
\begin{aligned}
\big( \pa_t \Phi^{1}(U) \big) \big( \Phi^{1}(U) \big)^{-1} & = \opbw(\ii \pa_t f )
+\sum_{k=2}^{L}\frac{1}{k!}{\rm Ad}^{k-1}_{\opbw(\ii f)}[\opbw(\ii\pa_t f )] \\
&+\frac{1}{L!}\int_{0}^{1} (1-\theta)^{L} \Phi^{\theta}(U)
\Big({\rm Ad}^{L}_{\opbw(\ii f)}[\opbw(\ii\pa_t f )] \Big)(\Phi^{\theta}(U))^{-1} d\theta
\end{aligned}
$$
and the lemma  follows noting that 
$ f \#_\rho f_t - f_t  \#_\rho f = \frac{1}{ \ii}\{ f, f_t \}  $ plus a symbol of order $ 2 m - 3 $. The translation invariance properties \eqref{def:tr-in}, \eqref{def:R-trin} follow since $f(U;x,\x)$ and $(\pa_{t}f)(U;x,\x)$
satisfy \eqref{def:tr-in} as well, and then arguing as in the proof of Lemma \ref{lem:tra.Vec}.
\end{proof}

\subsection{Lie expansions of vector fields up to quartic degree}\label{sec:Lie}
In  this subsection the variable $U$ may denote
both the couple of complex variables $(u,\bar{u})$ or the  real variables $(\eta,\psi)$.

\begin{lemma}\label{quasi-inversa} 
{\bf (Inverse of $ {\bf F}_{\leq 3}^{\theta}(U) $ up to $ O(u^4)$)}
Consider  a map $ \theta \mapsto {\bf F}_{\leq 3}^{\theta}(U) $, $ \theta \in [0,1] $, 
of the form 
\begin{equation}\label{espmultilin}
{\bf F}^{\theta}_{\leq 3} (U)=U
+\theta \big(M_{1}(U)[U]+M_2^{(1)}(U)[U]\big) + \theta^{2}{M}^{(2)}_2(U)[U]  
\end{equation} 
where  $ M_{1}(U) $ is  in $ \widetilde{{\mathcal{M}}}_{1}\otimes\mathcal{M}_2(\C)$ and 
the maps $ M_{2}^{(1)}(U), M^{(2)}_2(U) $ are in 
$ \widetilde{{\mathcal{M}}}_{2}\otimes\mathcal{M}_2(\C)$. 
Then there is a family of maps ${\bf G}_{\leq 3}^{\theta}(V) $ of the form 
\begin{equation}\label{espmultilin3}
{\bf G}_{\leq 3}^{\theta}(V)=V- \theta \big( M_{1}(V)[V] + M_{2}^{(1)}(V)[V]\big) + \theta^{2}\breve{M}_2^{(2)}(V)[V]
\end{equation}
where 
$  \breve{M}^{(2)}_2(V) $ is in $ \widetilde{{\mathcal{M}}}_2\otimes\mathcal{M}_2(\C) $, 
such that
\begin{equation}\label{espmultilin5}
{\bf G}_{\leq 3}^{\theta}\circ{\bf F}_{\leq 3}^{\theta}(U)= U + M_{\geq 3}(\theta; U)[U] \, ,  \quad
{\bf F}_{\leq 3}^{\theta}\circ{\bf G}_{\leq 3}^{\theta}(V)= V + M_{\geq 3}(\theta; U)[U] \, , 
\end{equation}
where $ M_{\geq 3}(\theta; U) $ is a polynomial in $ \theta $ and finitely many monomials  
$ M_p (U)[U] $ for maps $ M_p (U) \in  \widetilde{{\mathcal{M}}}_{p} \otimes\mathcal{M}_2(\C) $, $ p \geq 3 $. 
\end{lemma}

\begin{proof}
Setting  $V={\bf F}_{\leq 3}^{\theta}(U)$  we have, by \eqref{espmultilin},  
\begin{equation}\label{rel2}
U=V-\theta \big(M_{1}(U)[U]+M_2^{(1)}(U)[U]\big) - \theta^{2}{M}^{(2)}_2(U)[U]  \, .  
\end{equation}
Substituting iteratively twice the relation \eqref{rel2} into itself, and 
using  the last two remarks under Definition \ref{smoothoperatormaps}, we get  
$$
U = V - \theta  M_{1}(V)[V] + \theta^2 M_{1}(V)[ M_{1}(V)[V]] +  \theta^2 M_{1}( M_{1}(V)[V])[V] -
\theta^{2}M_2^{(2)} (V)[V]- \theta M_2^{(1)} (V)[V]
$$
up to a polynomial $ M_{\geq 3}(\theta; U)[U]$ in $ \theta $ and $ U $ 
which has degree of homogeneity at least three in
$ U $ (recall that $ {\bf F}^\theta_{\leq 3}  (U)$ in \eqref{espmultilin} is a polynomial in $ U $).  
This expansion defines  $ {\bf G}_{\leq 3}^{\theta} $ in \eqref{espmultilin3}, and proves \eqref{espmultilin5}. 
\end{proof}

We regard  the map $ \theta \mapsto {\bf G}_{\leq 3}^{\theta}(V)$ in \eqref{espmultilin3} as the formal 
flow of a  non-autonomous vector field
$ S(\theta; U )$ up a remainder of degree of homogeneity four, see \eqref{linprob200}.

\begin{lemma}\label{compolinearflows}
Consider a map ${\bf F}_{\leq 3}^{\theta}(U)$ as in \eqref{espmultilin} and let $ {\bf G}_{\leq 3}^\theta (V) $ be its 
approximate inverse as in \eqref{espmultilin3} up to quartic remainders.
Then 
 \begin{equation}\label{linprob200}
\pa_{\theta}{\bf G}_{\leq 3}^{\theta}(V) 
= S(\theta;{\bf G}_{\leq 3}^{\theta}(V))+M_{\geq 3}(\theta;U)[U] \, , \quad 
{\bf G}_{\leq 3}^{0}(V) = V \, , 
\end{equation}
where 
$ S(\theta;U) $  is a  vector field  of the form
 \begin{equation}\label{essone}
S(\theta; U) = S_1(U)[U]+  \theta S_2(U)[U] 
\end{equation}
where $ S_{1}(U) $ is a map in $ \widetilde{\mathcal{M}}_{1}\otimes\mathcal{M}_2(\C) $ and 
$ S_2(U) $ in $  \widetilde{\mathcal{M}}_{2} \otimes\mathcal{M}_2(\C) $,  and 
$ M_{\geq 3}(\theta; U) $ is a polynomial in $ \theta $ and finitely many monomials  
$ M_p (U)[U] $ for maps $ M_p (U) \in  \widetilde{{\mathcal{M}}}_{p} \otimes\mathcal{M}_2(\C) $, $ p \geq 3 $. 
\end{lemma}

\begin{proof}
Differentiating  \eqref{espmultilin3} we have
\begin{equation}\label{compoeq}
\pa_\theta {\bf G}_{\leq 3}^{\theta}(V)=  - M_1(V)[V] - {M}^{(1)}_2(V)[V]+ 2  \theta \breve M_2^{(2)}(V)[V] \, . 
\end{equation}
Then set  
\be\label{breveS2}
\breve S (\theta; U ) := -  M_1({\bf F}_{\leq 3}^{\theta}(U))[{\bf F}_{\leq 3}^{\theta}(U)] - 
 M_2^{(1)}({\bf F}_{\leq 3}^{\theta}(U))[{\bf F}_{\leq 3}^{\theta}(U)]
+2\theta \breve M_2^{(2)}({\bf F}_{\leq 3}^{\theta}(U))[{\bf F}_{\leq 3}^{\theta}(U)]   \, .
\ee
By \eqref{espmultilin} and 
 the last two remarks under Definition \ref{smoothoperatormaps}, 
we deduce that $ \breve S (\theta; U )  $ 
is equal to a vector field  $ S(\theta; U )$ as in \eqref{essone} 
plus  a term $ M_{\geq 3}(\theta; U)[U]$ 
which is a polynomial in $ \theta $ of monomials   
$ M_p (U)[U] $ for maps $ M_p (U) \in  \widetilde{{\mathcal{M}}}_{p} \otimes\mathcal{M}_2(\C) $, $ p \geq 3 $. 
By \eqref{breveS2}  and the second identity in  \eqref{espmultilin5} we deduce
\be\label{passofina}
 \breve S (\theta;  {\bf G}_{\leq 3}^{\theta}(V) ) =
 - M_1(V)[V] - {M}^{(1)}_2(V)[V]+ 2  \theta \breve M_2^{(2)}(V)[V]   
\ee
plus another polynomial $ M_{\geq 3}(\theta; U)[U]$ of degree at least three. 
Comparing \eqref{passofina} with \eqref{compoeq}
the lemma follows.
\end{proof}

Given polynomials  vector fields $X(U)$ and $Y (U)$ we define the nonlinear commutator 
\be\label{exp:XY}
\bral X, Y\brar (U)  := d_U Y (U) [X(U)] -  d_U X (U) [Y(U)] \, . 
\ee
Under the same notation of Lemmata \ref{quasi-inversa}, \ref{compolinearflows}, we have the following result. 

\begin{lemma}\label{quasi-flusso} {\bf (Lie expansion)}
Consider a  vector field 
$ X $ of the form $ X(U) = M(U)U $ for some map $ M (U) = M_0 + M_1 (U) + M_2 (U) $ 
where  $ M_0 $ is in $ \widetilde{\mathcal M}_0 \otimes {\mathcal M}_2 (\C) $, 
$ M_1 (U) $ is in $ \widetilde{\mathcal M}_1 \otimes {\mathcal M}_2 (\C) $ and $ M_2 (U)  $ in  
 $ \widetilde{\mathcal M}_2 \otimes {\mathcal M}_2 (\C) $. 
Consider a transformation $ {\bf F}_{\leq 3}^{\theta}(U)$ as in \eqref{espmultilin} and let 
$ S(\theta; U) $ be the vector field of the form \eqref{essone} such that  \eqref{linprob200}
holds true. 
Then,
if $U$ solves 
\begin{equation}\label{problemaU}
\pa_t U =X(U) \,,
\end{equation}
the function  $V := {\bf F}^{1}_{\leq 3} (U) $ solves
\begin{equation}\label{espansionecampo}
\pa_t V = 
 X(V) + \bral S, X\brar_{|\theta=0}(V) + \frac{1}{2} \bral S, \bral S, X\brar \brar_{|\theta=0}(V) + \frac12 \bral \pa_\theta S_{|\theta=0}, X \brar (V)+\cdots 
\end{equation}
up to terms of degree of homogeneity greater or equal to $4$.
\end{lemma}

\begin{proof}
In order to find the quadratic and cubic components of the transformed system, it is sufficient  to write 
$ V := {\bf F}_{\leq 3}^\theta (U) $, $ \theta \in [0,1] $,  and the first identity 
in \eqref{espmultilin5}  as 
$ U = {\bf G}_{\leq 3}^\theta (V) - M_{\geq 3}(\theta; U )[U] $.  
Then, differentiating with $ \pa_t $ 
the first identity in \eqref{espmultilin5}, and using \eqref{problemaU}, we obtain,  up to a quartic term, 
\be\label{traspuA}
X( {\bf G}^\theta_{\leq 3} (V) ) = d {\bf G}^\theta_{\leq 3} (V)[ V_t ]   = ({\rm Id} - M(\theta; V))[V_t ]
\ee
where $ M(\theta; V)  = \theta (\check M_1 (V)  + \check M_2^{(1)} (V)  )+ \theta^2 \check M_2 (V) $ for suitable maps 
$  \check M_1 (V) $ in $ \widetilde{\mathcal{M}}_{1}\otimes\mathcal{M}_2(\C) $ 
and  $ \check M_2 (V), \check M_2^{(1)} (V)  $ in  $  \widetilde{\mathcal{M}}_{2} \otimes\mathcal{M}_2(\C) $,
recall \eqref{espmultilin3}. 
Applying in \eqref{traspuA} the ``pseudo-inverse" 
$$
\big( d {\bf G}_{\leq 3}^\theta (V) \big)^{-1} :=  {\rm Id} + M(\theta; V)  + M^2 (\theta; V) \, , 
$$
and since, by \eqref{problemaU}, we have 
 $ \pa_t V  = \pa_t U $ plus a quadratic  term in $ U $, we deduce that, up to a quartic term,   
$$
\big( d {\bf G}_{\leq 3}^\theta (V) \big)^{-1} X( {\bf G}^\theta_{\leq 3} (V) ) = V_t  \, . 
$$
The left hand side of this formula can be expanded in Taylor at $ \theta = 0 $ up to degree $ 2 $,  
obtaining, using  \eqref{linprob200},  
the usual Lie formula (see e.g. \cite{He})
\be\label{XLie}
X(V)+ \theta \bral S, X\brar_{|\theta=0}(V) + 
\frac{\theta^2}{2}   \Big( \bral S, \bral S, X\brar \brar_{|\theta=0}(V) 
+ \bral  (\pa_{\theta} S(\theta))_{|\theta=0}, X\brar (V)\Big) 
\ee
up to terms of degree $ 4 $. Evaluating \eqref{XLie} at $ \theta = 1 $ we get  \eqref{espansionecampo}.
\end{proof}

\subsection{Proof of \eqref{uguale}-\eqref{uguale10}. } 
\label{Psiasflow}

\begin{proof}
For any function $ m_{-1}(U; x) $ in $ \Sigma\mathcal{F}_{K,1,1}[r] $, by Lemma \ref{buonflusso}  
the flow in \eqref{generatore20}
is well defined. 
We claim that 
\begin{equation}\label{exp}
({\bf \Psi}_{-1}^{\theta}(U))_{|_{\theta=1}}=\exp\{\opbw(M_{-1})\}=\opbw(\exp\{ M_{-1}\})+R(U)
\end{equation}
where $R $ is in $ \Sigma\mathcal{R}^{-\tilde\rho}_{K,1,1}\otimes\mathcal{M}_{2}({\C})$ for any $\tilde\rho>0$.
 Indeed 
$$
 \exp\{\opbw(M_{-1})\} ={\rm Id}+\opbw( M_{-1})
 +\frac{1}{2}\opbw( M_{-1})\opbw(M_{-1})+\sum_{k\geq3}\frac{1}{k!}\big(\opbw( M_{-1})\big)^{k}\,.
$$
 By Proposition \ref{teoremadicomposizione} (applied with some $\tilde{\rho}$ to be chosen later)
we have that $ \opbw( M_{-1})\opbw(M_{-1}) $  is equal to $ \opbw(( M_{-1})^{2})$
plus  a smoothing remainder in $\Sigma\mathcal{R}^{-\tilde{\rho}}_{K,1,2}$.
Furthermore, by Proposition 3.6 in \cite{FI1} we deduce 
that 
\[
\sum_{k\geq3}\frac{1}{k!}\Big(\big(\opbw( M_{-1})\big)^{k}-\opbw\big((M_{-1})^{k}\big)\Big)
\]
belongs to the class of non-homogeneous smoothing remainders $\mathcal{R}^{-\tilde{\rho}}_{K,1,3}$. This proves \eqref{exp}.
By an explicit computation (see the proof of Corollary
$3.1$ in \cite{FI1}) we have 
\be
 \exp\{M_{-1}\}:=\left(
\begin{matrix}g_1 & g_{2} \\ \ov{g_{2}} & g_{1}
\end{matrix}
\right) \, , \quad 
g_1 := 1 + \Psi_1 (|m_{-1}|^2  ) \, , \  \quad g_{2}:= m_{-1} \Psi_2 (| m_{-1}|^2)  \, , 
\label{iperbolici}
\ee
where $ \Psi_1 (y), \Psi_2 (y) $ are the analytic functions
$$
\Psi_1 (y) :=  \frac{y}{2} + \sum_{k\geq 2} \frac{y^k}{(2k)!} \, , \quad \Psi_2 (y) := 
1 + \sum_{k \geq 1}^{\infty}\frac{y^k}{(2k+1)!}  \, . 
$$
We now choose $m_{-1}(U;x)$ in such a way that 
$ \exp\{M_{-1}\} :=C^{-1}(U;x)$, namely  (recalling \eqref{invC}, \eqref{TRA}) 
we have to solve the following equations 
$$
\Psi_1 (|m_{-1}|^2)= f - 1  \, , \quad m_{-1} \Psi_2 (| m_{-1}|^2) = - g \, .
$$
Note that  $ \Psi_1(y)  $ is locally invertible near $ 0 $ and 
 that, since $1+a+\lambda_{+} \geq1/2 $, the function $f$ in \eqref{TRA} satisfies 
$ f^{2}-1=\frac{|a|^{2}}{2\lambda_{+}(1+a+\lambda_{+})}\geq 0 $, thus $ f - 1 \geq 0 $.
 Therefore we have 
 \begin{equation}\label{definizioneC2}
|m_{-1}|^{2} = \Psi_1^{-1} ( f - 1  ) \, , \quad m_{-1} = - \frac{g}{ \Psi_2 (| m_{-1}|^2) } \, . 
\end{equation}
Since $ f - 1 $ and $ g $ belong to $\Sigma\mathcal{F}_{K,1,1} $ and $ \Psi_1^{-1}, \Psi_2 $ are analytic, it follows that  
the function $ m_{-1} $ belongs to $\Sigma\mathcal{F}_{K,1,1}$ as well. 
Formulas \eqref{exp}-\eqref{definizioneC2} prove \eqref{uguale}.

Let us prove  \eqref{uguale10}. The flow $ ({\bf \Psi}_{-1}^{\theta}(U))_{|_{\theta=1}} $ 
is invertible and, 
setting  
$Q(U):= ({\bf \Psi}_{-1}^{\theta}(U))^{-1}_{|_{\theta=1}} - \opbw(C) $,  we have 
\[
{\rm Id}= ({\bf \Psi}_{-1}^{\theta}(U))^{-1}_{|_{\theta=1}} ({\bf \Psi}_{-1}^{\theta}(U))_{|_{\theta=1}}
\stackrel{\eqref{uguale}}{=}
\opbw(C)\opbw(C^{-1})+\opbw(C)R(U)+Q(U)({\bf \Psi}_{-1}^{\theta}(U))_{|_{\theta=1}}.
\]
Hence, using Propositions \ref{teoremadicomposizione}, \ref{composizioniTOTALI}, we deduce that
$
Q(U)=\tilde{R}(U) ({\bf \Psi}_{-1}^{\theta}(U))^{-1}_{|_{\theta=1}},
$
for some $\tilde{R}$ in the class $\Sigma\mathcal{R}^{-\tilde{\rho}}_{K,1,1}$. 
We conclude that 
$Q$ is in $\Sigma\mathcal{R}^{-{\rho}}_{K,1,1}$, using that 
$({\bf \Psi}_{-1}^{\theta}(U))^{-1}_{|_{\theta=1}} $ is in $ \Sigma {\mathcal M}_{K,1,0} \otimes {\mathcal M}_2 (\C) $ 
by item $(ii)$ of 
Lemma \ref{buonflusso},  and  choosing $\tilde\rho$ large enough. 
\end{proof}

\end{document}